\documentclass[12pt,oneside]{amsart}

\usepackage{geometry}                
\usepackage{graphicx}
\usepackage{amssymb}
\usepackage{color}
\usepackage{bbm, dsfont}
\usepackage[hidelinks]{hyperref}

\hypersetup{
	colorlinks=false,
	pdfborder={0 0 0},
	pdfborderstyle={/S/U/W 0},
}

\usepackage[symbol]{footmisc}

\usepackage{enumitem}

\numberwithin{equation}{section}

\usepackage{verbatim}

\usepackage[dvipsnames]{xcolor}

\usepackage{ulem}
\newtheorem{theorem}{Theorem}[section]
\newtheorem{lemma}[theorem]{Lemma}
\newtheorem{corollary}[theorem]{Corollary}
\newtheorem{proposition}[theorem]{Proposition}

\theoremstyle{definition}
\newtheorem{definition}[theorem]{Definition}

\theoremstyle{remark}
\newtheorem{remark}[theorem]{Remark}
\newtheorem*{remark*}{Note}

\numberwithin{equation}{section}

\usepackage{amsmath}
\usepackage{mathrsfs}

\usepackage{verbatim}
\usepackage{amsmath,amssymb,amsthm}           
\usepackage{amssymb}            
\usepackage{amsfonts}            
\usepackage{mathrsfs}          
\usepackage{amsthm}
\usepackage{algorithm}  
\usepackage{algorithmicx}  
\usepackage{algpseudocode}  
\usepackage{mathtools}
\usepackage{commath}
\usepackage{physics}
\usepackage{bm}
\usepackage{graphicx}

\usepackage{float}
\usepackage{listings}
\usepackage{subfigure}
\usepackage{multirow}
\usepackage{color}
\usepackage{bbm}
\usepackage[export]{adjustbox}
\usepackage{bbm}

\usepackage{amsmath}
\usepackage{mathrsfs}

\newcommand{\RNum}[1]{\uppercase\expandafter{\romannumeral #1\relax}}

\newcommand{\specificthanks}[1]{\@fnsymbol{#1}}

\DeclareFontFamily{OML}{rsfs}{\skewchar\font'177}
\DeclareFontShape{OML}{rsfs}{m}{n}{ <5> <6> rsfs5 <7> <8> <9>
	rsfs7 <10> <10.95> <12> <14.4> <17.28> <20.74> <24.88> rsfs10 }{}
\DeclareMathAlphabet{\mathfs}{OML}{rsfs}{m}{n}

\newcounter{cnstcnt}
\newcommand{\cl}{%
	\refstepcounter{cnstcnt}%
	\ensuremath{c_{\thecnstcnt}}}
\newcommand{\cref}[1]{\ensuremath{c_{\ref*{#1}}}}

\newcounter{newcnstcnt}
\newcommand{\Cl}{%
	\refstepcounter{newcnstcnt}%
	\ensuremath{C_{\thenewcnstcnt}}}
\newcommand{\Cref}[1]{\ensuremath{C_{\ref*{#1}}}}

 \newcounter{newnewcnstcnt}
\newcommand{\bl}{%
	\refstepcounter{newnewcnstcnt}%
	\ensuremath{\beta_{\thenewnewcnstcnt}}}
\newcommand{\bref}[1]{\ensuremath{\beta_{\ref*{#1}}}}

\DeclareFontFamily{U}{mathx}{}
\DeclareFontShape{U}{mathx}{m}{n}{<-> mathx10}{}
\DeclareSymbolFont{mathx}{U}{mathx}{m}{n}
\DeclareMathAccent{\widehat}{0}{mathx}{"70}
\DeclareMathAccent{\widecheck}{0}{mathx}{"71}

 \makeatletter
\let\saved@makefntext\@makefntext
\def\@setthanks{%
  \def\thanks##1{\par\noindent\ignorespaces##1\@addpunct.}%
  \thankses}
\let\saved@maketitle\maketitle
\def\maketitle{%
  \def\@makefntext{\noindent}%
  \saved@maketitle
  \let\@makefntext\saved@makefntext}
\makeatother

\begin{document}

	\title{ Scaling Limit of Critical Loop Soup Clusters in Three, Four, and Five Dimensions}


		\author{Zhenhao Cai$^1$}
		\address[Zhenhao Cai]{Faculty of Mathematics and Computer Science, Weizmann Institute of Science}
		\email{zhenhao.cai@weizmann.ac.il}
\thanks{$^1$Faculty of Mathematics and Computer Science, Weizmann Institute of Science}

		\author{Jian Ding$^2$}
		\address[Jian Ding]{New Cornerstone Science Laboratory, School of Mathematical Sciences, Peking University}
		\email{dingjian@math.pku.edu.cn}
		\thanks{$^2$New Cornerstone Science Laboratory, School of Mathematical Sciences, Peking University}

	

	\maketitle

	 	\begin{abstract}

 We prove that on the metric graph of $\mathbb{Z}^d$ for $d\in \{3,4,5\}$, clusters of the critical loop soup (or equivalently, clusters of the critical Gaussian free field level-set) admit a scaling limit. To the best of our knowledge, this presents the first scaling limit result for a percolation model on a three-dimensional lattice.

 We further prove that these scaling limits are conformally invariant and can be constructed by randomly gluing a Brownian loop soup of intensity $1/2$. Notably, the random gluing strictly increases the fractal dimension of the clusters. This construction hints at the possibility that in three dimensions, the scaling limits of some classical critical models can also be described via Brownian loop soups, in analogy with the two-dimensional case.

      \end{abstract}

    \tableofcontents

\section{Introduction}\label{section_intro}

The \textit{loop soup} model has been an active object of study in probability theory and statistical physics since its introduction \cite{lawler2004brownian, lawler2007random}. Its origins lie in the study of the loop-erased random walk (LERW) \cite{1077314188}, where the loop soup provides a rigorous description of the loops removed by the loop-erasure procedure. It also played an important role in the development of Schramm--Loewner evolution (SLE), which was initiated by Schramm \cite{schramm2000scaling} and then became a central tool for analyzing scaling limits of the interfaces of statistical physics models in two dimensions, including Bernoulli percolation \cite{smirnov2001critical}, the Ising model and its FK representation \cite{smirnov2010conformal, chelkak2012universality, chelkak2014convergence}, the Gaussian free field (GFF) \cite{schramm2009contour}, etc. The conformal loop ensemble (CLE) \cite{sheffield2009exploration} extends SLE from a single interface to the collection of all interfaces: $\mathrm{CLE}_\kappa$ is a random family of non-crossing loops, each of which behaves locally like an $\mathrm{SLE}_\kappa$ curve. In particular, it was shown in \cite{sheffield2012conformal} that CLE can be constructed from the Brownian loop soup \cite{lawler2004brownian}: for $\kappa\in(\tfrac{8}{3},4]$, $\mathrm{CLE}_\kappa$ has the same law as the outer boundaries of the clusters of a Brownian loop soup with intensity 
\begin{equation*} 
	\alpha  =(4\kappa)^{-1} (3\kappa -8) (6-\kappa), 
\end{equation*} 
 where $\kappa$ is the SLE parameter, and $\alpha$ is the intensity of the Brownian loop soup as a Poisson point process (its definition is given below). This correspondence covers the critical $q$-state Potts model with $q\in \{2,3,4\}$. The spin interfaces in the $q=2$ case (Ising model) converge to $\mathrm{CLE}_3$ \cite{benoist2019scaling} ($\alpha=\tfrac14$), while in the $q=3,4$ cases they are conjectured to converge to $\mathrm{CLE}_{\kappa}$ with $\kappa=\tfrac{10}{3},4$ respectively ($\alpha=\tfrac25,\tfrac12$). The intensity $\alpha$ provides a unifying parameter with additivity that is not apparent at the discrete level. In this sense, the Brownian loop soup links these models at the level of scaling limits in two dimensions. In dimensions three and higher, however, such connections remain far from clear. Nevertheless, these results have motivated a systematic study of the geometric properties of the Brownian loop soup, as well as its discrete analogue---the random walk loop soup \cite{lawler2007random}. Percolation provides a classical and natural perspective for exploring these properties.

    The main object of this paper, the loop soup on metric graphs (also known as cable graphs), was first introduced in \cite{lupu2016loop} as a continuous extension of random walk loop soups. In recent years, it has emerged as one of the few percolation models for which substantial progress at criticality has been made beyond the planar setting, where powerful tools from complex analysis are no longer available, but below the regime where mean-field behavior appears (namely, $3\le d\le 5$ in the context of this paper). For clarity, we first recall the definitions of loop soups and metric graphs. Let $\{X_t^{\Omega}\}_{t\ge 0}$ be a transient Markov process on $\Omega$ with transition density $q_t^{\Omega}(x,y)$, with respect to a measure $\mathrm{m}^{\Omega}$. Let $\mathbb{P}^{\Omega}_{x,y,t}$ denote the bridge measure of $X_\cdot^{\Omega}$ from $x$ to $y$ with duration $t$ (its transition density is given by $\frac{q_s^{\Omega}(x,\cdot )q_{t-s}^{\Omega}(\cdot ,y)}{q_t^{\Omega}(x,y)}$ for $0\le s\le t$). The loop measure (associated with $X_\cdot^{\Omega}$) is then defined by
\begin{equation}\label{def_1.1}
	\mu^{\Omega}(\cdot):= \int_{x\in \Omega}  \mathrm{dm}^{\Omega}(x) \int_{t>0} t^{-1} q_t^{\Omega}(x,x) \mathbb{P}^{\Omega}_{x,x,t}(\cdot) \mathrm{d}t. 
\end{equation}
    For $\alpha>0$, the Poisson point process with intensity $\alpha\mu^{\Omega}$ is called the loop soup of intensity $\alpha$ and denoted by $\mathcal{L}_{\alpha}^{\Omega}$. The three types of loop soups mentioned above all fit into this framework, with different choices of the underlying Markov process.
    \begin{itemize}

    	\item Brownian loop soup: for $\Omega=\mathbb{R}^d$ with $d\ge 3$, take $X_\cdot^{\mathbb{R}^d}$ to be standard Brownian motion on $\mathbb{R}^d$, and let $\mathrm{m}^{\mathbb{R}^d}$ be the Lebesgue measure on $\mathbb{R}^d$. (In two dimensions, due to recurrence, an absorbing boundary is needed to avoid divergences. The same applies to the loop soups defined below.)

    	\item Random walk loop soup: for $\Omega=\mathbb{Z}^d$ with $d\ge 3$, take $X_\cdot^{\mathbb{Z}^d}$ to be the (continuous-time) simple random walk on $\mathbb{Z}^d$. I.e., for $x,y\in \mathbb{Z}^d$ and $t\ge 0$,  
  \begin{equation*}
	  	\mathbb{P}\big(  X_{t+\Delta t}^{\mathbb{Z}^d} =y  \mid  X_t^{\mathbb{Z}^d} =x \big) =   \left\{
\begin{aligned}
&(2d)^{-1}\Delta t\cdot \mathbbm{1}_{\{x,y\}\in \mathbb{L}^d}+ o(\Delta t)   & \text{if}\ x\neq y;  \\
&  1- \Delta t + o(\Delta t)    & \text{if}\ x= y.
\end{aligned}
\right.
 \end{equation*}
Here $\mathbb{L}^d:=\{ \{x,y\}:x\ \text{and}\ y\ \text{are adjacent vertices of}\ \mathbb{Z}^d \}$ denotes the edge set of $\mathbb{Z}^d$. Let $\mathrm{m}^{\mathbb{Z}^d}$ be the counting measure.

 	 \item  Metric graph loop soup: for each $e=\{x,y\}\in \mathbb{L}^d$, we assign a compact interval $I_e$ of length $d$ whose endpoints are identified with $x$ and $y$ (the choice of the common interval length $d$ is made for convenience and does not affect the geometric properties studied in this paper). The metric graph $\widetilde{\mathbb{Z}}^d$ is defined as the union of these intervals, glued at their common endpoints. The Markov process $X_\cdot^{\widetilde{\mathbb{Z}}^d}$ is defined as follows. Within each interval $I_e$, it behaves as standard one-dimensional Brownian motion. When it reaches a vertex $x\in \mathbb{Z}^d$, it uniformly selects one of the incident intervals and then evolves as a Brownian excursion along this interval. Let $\mathrm m^{\widetilde{\mathbb Z}^d}$ be the measure on $\widetilde{\mathbb{Z}}^d$ whose restriction to each interval is the Lebesgue measure.

    \end{itemize}

     When the loop soup $\mathcal{L}_{\alpha}^{\widetilde{\mathbb{Z}}^d}$ is viewed as a percolation model, a point $v \in \widetilde{\mathbb{Z}}^d$ is called open if it is contained in at least one loop of $\mathcal{L}_{\alpha}^{\widetilde{\mathbb{Z}}^d}$, and closed otherwise. According to the isomorphism theorem \cite{le2011markov, lupu2016loop}, when $\alpha=\frac{1}{2}$, the collection of the closed points has the same distribution as the zero set of the GFF $\{\phi_v\}_{v\in \widetilde{\mathbb{Z}}^d}$. Here $\phi$ is the mean-zero Gaussian field with covariance given by the Green's function: 
      \begin{equation*}
    	\mathbb{E}\big[\phi_{v_1}\phi_{v_2} \big] = G(v_1,v_2)= \int_{t>0} q_t^{\widetilde{\mathbb{Z}}^d}(v_1,v_2)\mathrm{d}t, \ \forall v_1,v_2\in \widetilde{\mathbb{Z}}^d. 
    \end{equation*}
    One can equivalently construct it by first sampling a discrete GFF on $\mathbb{Z}^d$ and then independently interpolating along each edge by a Brownian bridge. In fact, the open clusters of $\mathcal{L}_{1/2}^{\widetilde{\mathbb{Z}}^d}$ and the sign clusters of $\phi$ are not only equivalent, but are also both at their respective percolation thresholds: $\alpha_*=\frac{1}{2}$ is the critical intensity for $\mathcal{L}_{\alpha}^{\widetilde{\mathbb{Z}}^d}$ (see \cite{chang2024percolation}), and $h_*=0$ is the critical level for the GFF level-set $E^{\ge h}:=\{v \in \widetilde{\mathbb{Z}}^d: \phi_v \ge h\}$ (see \cite{lupu2016loop}). In the remainder of this paper, we focus only on the critical intensity $\frac{1}{2}$. For brevity, we write $\mathcal{L}:= \mathcal{L}_{1/2}^{\widetilde{\mathbb{Z}}^d}$, and let $\mathfrak{C}$ denote the collection of all open clusters of $\mathcal{L}$. We write $A\xleftrightarrow{} A'$ for the event that $A$ and $A'$ intersect the same cluster in $\mathfrak{C}$. Applying the distributional identity above, it was derived in \cite[Proposition 5.2]{lupu2016loop} that for any $v_1,v_2\in \widetilde{\mathbb{Z}}^d$, 
      \begin{equation}\label{lupu_two_point}
    	\mathbb{P}(v_1  \xleftrightarrow{} v_2 ) = \frac{2}{\pi} \arcsin  \Big(\frac{G(v_1,v_2)}{\sqrt{G(v_1,v_1)G(v_2,v_2)}}   \Big) \asymp (|v_1-v_2|+1)^{2-d}, 
    \end{equation}
    where $|v_1-v_2|$ denotes the Euclidean distance between $v_1$ and $v_2$ (throughout this paper, points and subsets of $\widetilde{\mathbb{Z}}^d$ are identified with their images under the map obtained by first embedding $\mathbb{Z}^d$ canonically into $\mathbb{R}^d$ and then linearly interpolating along each edge), and $f \asymp g $ means that there exist constants $C>c>0$ depending only on $d$ such that $cg\le f\le Cg$.

   In his inspiring note \cite{werner2020clusters}, Werner conjectured that for $3\le d\le 5$, the clusters in $\mathfrak{C}$ have a scaling limit; in addition, this limit is supported on the families of clusters of fractal dimension $\frac{d}{2}+1$. It was further predicted that this scaling limit could be constructed from the Brownian loop soup. To be precise, we denote $\mathcal{L}^{\mathrm{B}}:=\mathcal{L}_{1/2}^{\mathbb{R}^d}$, and define $\mathfrak{C}^{\mathrm{B}}$ as the collection of all clusters of $\mathcal{L}^{\mathrm{B}}$(note that for $d\ge 4$, each cluster of $\mathcal{L}^{\mathrm{B}}$ includes a single loop because the Brownian loops in $\mathbb{R}^d$ cannot intersect).
    Using couplings between random walks and Brownian motion, it was shown that the macroscopic loops in $\mathcal{L}$ converge to the Brownian loops in $\mathcal{L}^{\mathrm{B}}$ (see \cite{lawler2007random, sapozhnikov2018brownian, qian2026coupling}). Here ``macroscopic'' refers to objects with diameters $\asymp N$ within a fixed region of scale $N$, and the convergence is understood in the sense that after rescaling space by $\frac{1}{N}$, the objects converge as $N\to \infty$. Meanwhile, one can show that microscopic loops cannot by themselves form macroscopic clusters. Thus, heuristically, the scaling limit of $\mathfrak{C}$ should be described by $\mathfrak{C}^{\mathrm{B}}$ (serving as the skeleton of the limiting clusters) together with a gluing relation (encoding the limiting connectivity effects of microscopic loops), which is formally a (possibly random) equivalence relation on $\mathfrak{C}^{\mathrm{B}}$ (those in the same equivalence class are declared to be connected). A blueprint for the behavior of this gluing relation in different dimensions was proposed in \cite{werner2020clusters}:  
   \begin{itemize}

   	\item[$\spadesuit$] $d=3$. No additional gluing is needed. In other words, the scaling limit of $\mathfrak{C}$ is exactly $\mathfrak{C}^{\mathrm{B}}$ (note that the Brownian loops in $\mathbb{R}^3$ can intersect).

   	 \noindent (P.S. As established in \cite{lupu2018convergence}, this is the case in dimension two.)

   	\item[$\heartsuit$] $d=4$. The gluing relation is non-trivial and measurable with respect to the Brownian loop soup $\mathcal{L}^{\mathrm{B}}$.

   	\item[$\clubsuit$] $d=5$. The gluing relation is non-trivial and involves additional randomness beyond $\mathcal{L}^{\mathrm{B}}$.

   \end{itemize} 
   For $d>6$, it was predicted in \cite{werner2020clusters} (with insightful heuristics)  that $\mathcal{L}$ should exhibit mean-field behavior, analogous to Bernoulli percolation:   
   \begin{itemize}
   	\item[$\diamondsuit$] $d>6$. The scaling limit of $\mathfrak{C}$ consists of clusters of fractal dimension $4$, whose law is described by integrated super-Brownian excursions.   	  
   \end{itemize}
 By now, a substantial part of this broad picture has been clarified. The conjectured cluster dimensions have been established via estimates of the one-arm probability $\theta_d(N):=\mathbb{P}(\bm{0}\xleftrightarrow{} \partial [-N,N]^d)$, where $\bm{0}$ is the origin and $\partial A$ is the boundary of $A$ in $\mathbb{R}^d$. Precisely, it was derived in successive works \cite{ding2020percolation, cai2025one, drewitz2026arm, drewitz2025critical, cai2024one} that  
 \begin{align}
&\text{when $3\le d \le 5$,}\  \ \   \theta_d(N)\asymp N^{-\frac{d}{2}+1};  \label{one_arm_lowd}\\
&\text{when $d>6$,}\ \ \ \ \ \ \ \   \   \theta_d(N)\asymp N^{-2}.\label{one_arm_highd}
\end{align}
At the critical dimension ($d_c=6$), the exponent of $\theta_6(N)$ was also obtained in  \cite{cai2024one}: $\theta_6(N)=N^{-2+o(1)}$, with the upper and lower bounds differing by a sub-polynomial multiplicative factor. These estimates further motivated investigations into cluster volumes \cite{cai2024quasi, cai2024incipient, drewitz2024cluster}. In particular, it was shown that a macroscopic cluster in $\mathfrak{C}$ typically contains $\asymp N^{\frac{d}{2}+1}$ (resp. $\asymp N^{4}$) vertices when $3 \le d \le 5$ (resp. $d>6$), consistent with the conjectured cluster dimensions in \cite{werner2020clusters}.

\subsection{Main result}

    The main result of this paper settles (arguably) the central conjecture in \cite{werner2020clusters}. In short, we prove that 
    \begin{itemize}
    
    	\item[(i)] For $3\le d\le 5$, the scaling limit of $\mathfrak{C}$ exists and can be obtained from $\mathfrak{C}^{\mathrm{B}}$ by imposing a random gluing relation;

    	\item[(ii)] This scaling limit is invariant under conformal maps (i.e., compositions of translations, dilations, rotations, and inversions);

    	\item[(iii)] This gluing relation is not measurable with respect to $\mathfrak{C}^{\mathrm{B}}$ (in other words, the scenario in Item $\clubsuit$ is the correct one for all $3\le d\le 5$).

    \end{itemize}

   Prior to the present work, a prominent result of this type was the scaling limit of loop-erased random walk \cite{kozma2007scaling}, whose proof strategy substantially inspired our approach. The result in \cite{kozma2007scaling} also played a central role in the derivation of the scaling limit of uniform spanning tree \cite{angel2021scaling}. Next, we present the precise statement of our main theorem. We begin with some basic notation needed for its formulation.

    \vspace{0.1cm}

    \begin{itemize}[nosep, itemsep=5pt, leftmargin=1em]

    \item[$\blacktriangleright$] (Hausdorff distance) For $A_1,A_2\subset \mathbb{R}^d$, their Euclidean distance is 
    \begin{equation*}
    	\mathrm{d}(A_1,A_2):=\inf\{|x_1-x_2|: x_1 \in A_1 ,x_2 \in A_2\}, 
    \end{equation*}
  and their Hausdorff distance is  
     \begin{equation*}
    	\mathrm{d}_{\mathrm{H}}(A_1,A_2):= \inf \big\{r>0: A_i  \subset  B(A_{3-i},r) \ \text{for}\ i\in \{1,2\} \big\},
    \end{equation*}
    where $B(A,r):=\{x\in \mathbb{R}^d: \mathrm{d}(\{x\},A)<r \}$. For any finite collections $\mathcal{A}_1,\mathcal{A}_2$ of subsets of $\mathbb{R}^d$, if $|\mathcal{A}_1|=|\mathcal{A}_2|$, we define their induced Hausdorff distance by 
     \begin{equation*}
    	\mathrm{d}_{\mathrm{H}}^*(\mathcal{A}_1,\mathcal{A}_2):= \min_{\sigma\in \mathrm{bij}(\mathcal{A}_1,\mathcal{A}_2) } \max_{A \in \mathcal{A}_1 } 	\mathrm{d}_{\mathrm{H}}\big(A , \sigma(A) \big), 
    \end{equation*} 
    where $\mathrm{bij}(\mathcal{A}_1,\mathcal{A}_2)$ is the collection of all bijections from $\mathcal{A}_1$ to $\mathcal{A}_2$; otherwise (i.e., $|\mathcal{A}_1|\neq |\mathcal{A}_2|$), $\mathrm{d}_{\mathrm{H}}^*(\mathcal{A}_1,\mathcal{A}_2):=+\infty$.

   \item[$\blacktriangleright$] (Restrictions and images) For any $A\subset \mathbb{R}^d$, its $\ell^{\infty}$-diameter is given by $|A|_{\infty}:=\sup \{|x-y|_{\infty}: x,y\in A\}$, where $|\cdot|_{\infty}$ denotes the $\ell^{\infty}$-norm. For $D\subset \mathbb{R}^d$, $\epsilon>0$ and a collection $\mathcal{A}$ of subsets of $\mathbb{R}^d$, let $\mathcal{A}_D^{> \epsilon}:=\{ A\in \mathcal{A}: A\subset D, |A|_{\infty} > \epsilon \}$ denote the subcollection of $\mathcal{A}$ consisting of sets contained in $D$ with $\ell^{\infty}$-diameters greater than $\epsilon$. For any map $f:\mathbb{R}^d\to\mathbb{R}^d$, let $f(\mathcal{A}):=\{f(A):A\in\mathcal{A}\}$ denote the collection of images of the sets in $\mathcal{A}$. In particular, for the dilation $f(x)=\delta \cdot  x$ with $\delta>0$, we write $\delta\cdot\mathcal{A}$ for $f(\mathcal{A})$.

  \item[$\blacktriangleright$] (Gluing operation) For a collection $\mathcal{A}$ of subsets of $\mathbb{R}^d$ and an equivalence relation $\sim$ on $\mathcal{A}$, let $\mathcal{A}_{\sim }$ be the collection obtained from $\mathcal{A}$ by replacing each equivalence class with the union of its elements. For any map $f:\mathbb{R}^d\to\mathbb{R}^d$, we define $f(\mathcal{A},\sim):=(f(\mathcal{A}),\sim')$, where $\sim'$ is the equivalence relation on $f(\mathcal{A})$ given by 
 \begin{equation*}
 	f(A_1)\sim' f(A_2)
 \iff  A_1\sim A_2,\ \forall  A_1,A_2\in\mathcal{A}. 
 \end{equation*}

  \item[$\blacktriangleright$] (Boundary regularity) We define $\mathfrak{D}$ to be the collection of bounded open sets $D\subset\mathbb{R}^d$ with $C^0$ boundaries. More explicitly, for every $x\in\partial D$, one can choose a neighborhood $U_x$ of $x$ and an orthonormal coordinate system centered at $x$ such that $D\cap U_x$ is the region below the graph of a continuous function (see Section~\ref{section_preliminaries} for the corresponding formulation).

    \end{itemize}

       \begin{theorem}\label{new_thm1}
        For any $d\in \{3,4,5\}$, there exists a probability space carrying the Brownian loop soup $\mathcal{L}^{\mathrm{B}}$ and an equivalence relation $\sim_{\mathrm{g}}$ on $\mathfrak{C}^{\mathrm{B}}$ such that the following hold:
       	\begin{itemize}

       		\item (Scaling limit) For any $D\in \mathfrak{D}$ and $\epsilon>0$, the clusters $(\delta\cdot \mathfrak{C})^{> \epsilon}_{D}$ converge in distribution to $\big(\mathfrak{C}^{\mathrm{B}}_{\sim_{\mathrm{g}}}\big)^{> \epsilon}_{D}$ as $\delta \downarrow 0$, with respect to the distance $\mathrm{d}^{*}_{\mathrm{H}}$.

       		\item  (Conformal invariance) For any conformal map $f$, the image $f(\mathfrak{C}^{\mathrm{B}},\sim_{\mathrm{g}})$ has the same law as $(\mathfrak{C}^{\mathrm{B}},\sim_{\mathrm{g}})$.

       	\item  (Extra randomness) The relation $\sim_{\mathrm{g}}$ is not measurable with respect to $\mathcal{L}^{\mathrm{B}}$.

       	\end{itemize}

       \end{theorem}

       \begin{remark}\label{newremark_1.2}

  (1) In the scaling limit part of Theorem \ref{new_thm1}, the condition $D\in\mathfrak{D}$ is imposed to rule out loop clusters that are contained in $D$ but nearly touch $\partial D$. In fact, the argument can be extended to domains whose boundaries are $C^0$ except on a compact set of Hausdorff dimension strictly less than $\frac{d}{2}-1$: the probability that macroscopic clusters occur near such an exceptional set can be controlled using the crossing probability estimates in \cite{cai2024one}.

%
%
%
%
%
%
%

 (2) We expect that the method in the proof of Theorem~\ref{new_thm1} can be extended to more general graphs, for instance to periodic graphs in dimensions $d\in\{3,4,5\}$. In fact, our argument essentially requires that the random walk converges to Brownian motion at a uniform rate, which in particular holds for Voronoi networks (see Section~\ref{subsection_isometric_graph}). The extension to this class is explained in Remark~\ref{rmk:universality}; moreover,  the scaling limit on these graphs coincides with that in Theorem~\ref{new_thm1}. This demonstrates a certain level of universality for the limit. To keep the exposition focused, we do not pursue the other extensions in full detail here; indeed, the case of $\mathbb{Z}^d$ already captures the main ideas and involves substantial technical difficulties.

       \end{remark}

       Notably, the gluing relation $\sim_{\mathrm{g}}$ in Theorem \ref{new_thm1} is not only non-trivial, but also substantially changes the geometry of the Brownian loop clusters being glued, particularly their fractal dimensions. When $d\in \{4,5\}$, this is immediate: each Brownian loop cluster consists of a single loop and hence has dimension $2$, which is strictly smaller than the dimension $\frac{d}{2}+1$ of the metric graph loop clusters. In three dimensions, however, this is far from obvious. Since Brownian loops can now intersect, their clusters may have dimension strictly greater than $2$, and it is not clear a priori whether this dimension is strictly smaller than $\frac{5}{2}$, the dimension of the metric graph loop clusters. The following result\footnote[2]{In our earlier arXiv post \cite{cai2025gap}, we established this dimension gap under the assumption that the scaling limit of the metric graph loop clusters exists. Since the existence of the scaling limit is proved in the present paper, we decided to incorporate the main result of \cite{cai2025gap} as well as its proof here to make the presentation more unified and logically coherent; accordingly, \cite{cai2025gap} will not be published anywhere.} establishes a strict dimension gap between these two types of clusters.

       \begin{theorem}\label{theorem_dimension_gap}

       (i) For $d=3$, there exists $\zeta>0$ such that almost surely the upper box-counting dimension of every element in $\mathfrak{C}^{\mathrm{B}}$ is at most $\frac{5}{2}-\zeta$.

    \noindent (ii) For $d\in \{3,4,5\}$, almost surely the box-counting dimension of every element in $\mathfrak{C}^{\mathrm{B}}_{\sim_{\mathrm{g}}}$ is $\frac{d}{2}+1$.


       \end{theorem}

      As a direct consequence, for $d\in\{3,4,5\}$, every equivalence class of the gluing relation $\sim_{\mathrm{g}}$ contains infinitely many Brownian loop clusters.

    \begin{remark*}
    (1) Item (i) first appeared in our earlier arXiv
preprint~\cite{cai2025gap}, building on \cite{cai2025heterochromatic, cai2025separation}; since we merge \cite{cai2025gap} into this paper, we will present the proof in Section~\ref{section_dimension_gap}. In~\cite{cai2025gap}, we also proved the existence of clusters with box-counting dimension $\frac{d}{2}+1$ in any scaling limit of metric graph loop clusters (which, together with Item (i), already suffices to reveal the dimension gap). Here we strengthen the latter result to Item (ii), showing that all clusters have this dimension.

   (2) Building on our estimate in \cite{cai2025gap} for the crossing probability of $\mathcal{L}^{\mathrm{B}}$ in $\mathbb{R}^3$, the box-counting dimension and the more delicate Hausdorff dimension of the clusters in $\mathfrak{C}^{\mathrm{B}}$ were derived in \cite{gao2026loop}. Both dimensions were also recently established in \cite{drewitz2026fractal} for any subsequential scaling limit of the metric graph loop clusters. In particular, this result covers Item~(ii), which we prove here via a different approach.

    \end{remark*}





   As a byproduct of the proof of Theorem \ref{new_thm1}, we also derive the tail asymptotics for the $\ell^{\infty}$-diameters of critical clusters. Let $\mathcal{C}_{\bm{0}}$ be the cluster in $\mathfrak{C}$ containing $\bm{0}$.

      \begin{theorem}\label{thm_asymptotic}
     	For any $d\in \{3,4,5\}$, the limit 
     	\begin{equation}\label{ineq_thm_asymptotic1}
      		 \xi_{d,\infty} :=  \lim_{N\to \infty}	N^{\frac{d}{2}-1}\mathbb{P}\big(  |\mathcal{C}_{\bm{0}}|_{\infty} \ge N \big) 
     	\end{equation}
     	exists. Moreover, there exist constants $\cl\label{const_asymptotic1}(d),\Cl\label{const_asymptotic2}(d)>0$ such that 
     	\begin{equation}\label{ineq_thm_asymptotic2}
     	 \big|	N^{\frac{d}{2}-1}\mathbb{P}\big(  |\mathcal{C}_{\bm{0}}|_{\infty}  \ge N \big) -   \xi_{d,\infty}   \big| \le  \Cref{const_asymptotic2}N^{-\cref{const_asymptotic1}}. 
     	\end{equation}
     \end{theorem}

\begin{remark}[Extensions]
	The ideas underlying the proof of Theorem \ref{thm_asymptotic} can be used to obtain further results. We mention several possible extensions below.

	(1) The proof of Theorem \ref{thm_asymptotic} can be adapted to obtain the asymptotics of the one-arm probability $\theta_d(N)$, and the resulting argument is even more concise. It can also be extended to the Euclidean diameter (or diameters defined via other equivalent norms), albeit with additional technical effort. Precisely, if enlarging a cluster within a small region increases its $\ell^\infty$-diameter, then there must exist a hyperplane orthogonal to one of the coordinate axes and intersecting this region such that the original cluster lies entirely on one side. There are only $d$ possible choices for this coordinate axis; in the Euclidean case, by contrast, the normal direction of the hyperplane is unrestricted. This makes the application of Russo's formula to the tail probabilities of cluster diameters more delicate.


	(2) The proof of Theorem \ref{thm_asymptotic} can also be adapted to show that the crossing probability for a tube converges as the mesh size $ \delta \to  0$. Here a tube is a subset of $\mathbb R^d$ homeomorphic to $[0,1]^d$ with two distinguished faces, and is said to be crossed if there exists a finite chain of loops contained in the tube that connects the two distinguished faces. Such convergence, together with suitable tightness results, should be sufficient to derive the scaling limit of $\mathfrak{C}$ in the Schramm--Smirnov topology \cite{schramm2011scaling, berestycki2015coalescing}. As pointed out in \cite[Section 1.1]{schramm2011scaling}, scaling limits in this topology are not equivalent to Hausdorff-type scaling limits (such as that in Theorem \ref{new_thm1}) in the presence of pivotal points (for $\mathcal{L}$, \cite{cai2025separation} established the existence of pivotal points and also computed their fractal dimension).

	(3) The argument may be adapted to establish the asymptotics of multi-arm probabilities. Although multiple clusters are involved, these clusters are separated with high probability at a fixed microscopic scale (see \cite{cai2025heterochromatic}). Hence, heuristically, the required analysis should not differ substantially from that in Theorem \ref{thm_asymptotic} for a single cluster. Following the approach of \cite{garban2013pivotal} for two-dimensional Bernoulli percolation, the asymptotics of multi-arm probabilities could then be used to construct the scaling limit of the normalized counting measure on pivotal points. This would lead to a deeper understanding of the interfaces between clusters in $\mathfrak{C}$.

\end{remark}

  Building on the argument used in the proof of Theorem \ref{new_thm1}, we also establish the asymptotics of the $k$-point connecting probabilities for every $k\ge 3$, and show that the limiting $k$-point functions are conformally covariant. To state the result, for any $x\in\mathbb{R}^d$ and $\delta>0$, we denote by $\hat{x}^{\delta}$ the closest point to $\delta^{-1}\cdot x$ in $ \mathbb{Z}^d$ (with ties broken in a predetermined manner).



        \begin{theorem}\label{thm_k_point}
       For any $3\le d\le 5$, any integer $k\ge 3$ and any distinct points $\{x_i\}_{1\le i\le k}$ in $\mathbb{R}^d$, the limit 
        	\begin{equation}\label{def_gamma_k}
      \gamma_k(x_1,x_2,...,x_k):= 	\lim\limits_{\delta \downarrow 0} \delta^{-(\frac{d}{2}-1)k}\mathbb{P}(\hat{x}_1^{\delta}\xleftrightarrow{} \hat{x}_2^{\delta} \xleftrightarrow{} ... \xleftrightarrow{} \hat{x}_k^{\delta}) 
        	\end{equation}
        	exists and lies in $(0,\infty)$. Moreover, $\gamma_k$ is conformally covariant: for any conformal map $f=(f_1,...,f_d)$, 
        	\begin{equation}\label{good_113}
        		\gamma_k\big(f(x_1),f(x_2),...,f(x_k) \big)=    \gamma_k(x_1,x_2,...,x_k) \cdot \prod_{1\le i\le k} [\lambda_f(x_i)]^{-\frac{d}{2}+1}, 
        	 \end{equation}
  where $\lambda_f(x):= \big[d^{-1}\sum_{1\le i,j\le d} \big( \partial_i f_j(x) \big)^2\big]^{1/2}$ is the conformal scale factor of $f$ at $x$. 
   \end{theorem}

  As a corollary, we obtain the Delfino--Viti formula \cite{delfino2011three} for three-point functions:

  \begin{corollary}
  	  For any $3\le d\le 5$, there exists a constant $C_{\mathrm{DV}}(d)>0$ such that for any distinct $x,y,z\in \mathbb{R}^d$,
  \begin{equation}\label{newthree_point_115}
  	 \gamma_3(x,y,z)= C_{\mathrm{DV}}  |x-y|^{-\frac{d}{2}+1}|y-z|^{-\frac{d}{2}+1}|z-x|^{-\frac{d}{2}+1}.
  \end{equation}
  \end{corollary}

		    In two dimensions, the Delfino--Viti ratios for critical models can be described by the imaginary DOZZ formula (see e.g., \cite{ang2021integrability} for Bernoulli percolation and \cite{picco2013connectivities, cai2025three} for the Potts model). In dimensions three and higher, by contrast, our understanding remains very limited. Computing the constant $C_{\mathrm{DV}}$ in (\ref{newthree_point_115}) may be a plausible starting point. Furthermore, the $k$-point functions for $k\ge 4$ theoretically fall within the scope of the conformal bootstrap (see e.g., \cite{el2012solving, el2014solving} for discussions of the three-dimensional Ising model). Developing a rigorous conformal bootstrap framework for the critical loop soup percolation (which potentially leads to explicit expressions for $\gamma_k$ with $k\ge 4$) would represent a major advance.

        In sufficiently high dimensions, the analogue of $\gamma_k$ in (\ref{def_gamma_k}) for critical Bernoulli percolation (with $\delta^{-(\frac{d}{2}-1)k}$ replaced by $\delta^{-(d-4)k+d-6}$) was shown to exist and was computed in \cite{chatterjee2026convergence}. We conjecture that the analogous result should hold for $\mathcal L$ when $d>6$. Building on the techniques developed in \cite{werner2020clusters,cai2025one} to control the contribution of large loops, it seems plausible that this conjecture could be proved by adapting the argument of \cite{chatterjee2026convergence}. By contrast, the method developed in this paper (at least in its present form) appears to rely on features specific to the regime $d<6$.

   \subsection{Discussions}

   
     In this subsection, we briefly discuss several questions that naturally arise from the results of this paper. These questions are largely driven by the following unifying motivation: 
      \begin{center}
   	\textit{Is there a unified characterization of scaling limits } 
   
   	\textit{ of critical models in three dimensions?}
   \end{center} 
  We first ask whether the scaling limit obtained in this paper admits a more explicit description, which may shed light on this question.

    \textbf{Solvability of the gluing relation $\sim_{\mathrm{g}}$.} Under suitable regularity assumptions, Schramm showed that the domain Markov property and conformal invariance force the Loewner driving function to be Brownian motion, giving rise to the family of $\mathrm{SLE}_\kappa$ curves. The clusters $\mathfrak{C}$ studied here can also be viewed as the sign clusters of the metric graph GFF, and thus inherit a domain Markov property from the field (for instance, if we partially explore a cluster along with GFF values thereon, then the conditional law of the field on the unexplored region is again a GFF; in particular if we completely explore a cluster, then in the remaining region we have sign clusters for a GFF with zero boundary condition). In addition, Theorem \ref{new_thm1} shows that the scaling limit of $\mathfrak{C}$ is conformally invariant. As in two dimensions, we feel that these two properties might lead to an explicit description of the scaling limit. Since $\mathfrak{C}^{\mathrm{B}}$ is already explicit, the remaining question becomes:
        \begin{center}
    	\textit{Is there an explicit characterization of the gluing relation $\sim_{\mathrm{g}}$?}
    \end{center}


         \textbf{Potts models.} As reviewed in the introduction, in two dimensions the outer boundaries of clusters of Brownian loop soups with intensities $\frac{1}{4}$, $\frac{2}{5}$ and $\frac{1}{2}$ form $\mathrm{CLE}_3$, $\mathrm{CLE}_{10/3}$ and $\mathrm{CLE}_4$ respectively. These ensembles correspond to the scaling limits of the $q$-state Potts models for $q\in \{2,3,4\}$ respectively (where the case $q=2$ is the Ising model). The $\mathrm{CLE}_4$ ensemble ($\alpha=\frac{1}{2}$) also describes the scaling limit of the sign clusters of the metric graph GFF in two dimensions, while in three dimensions, Theorem \ref{new_thm1} shows that this scaling limit can be constructed from a Brownian loop soup of the same intensity together with a random gluing relation. This suggests that Brownian loop soups, equipped with suitable gluing mechanisms, may also play a role for other critical models in three dimensions.

  We now turn to the Potts model in three dimensions. For $q\ge 3$, the $q$-state Potts model on $\mathbb{Z}^3$ is expected to undergo a discontinuous phase transition; see \cite{janke1997three} for numerical evidence in the case $q=3$ and \cite{PhysRevE.108.064120} for a renormalization group analysis. This suggests that for these values of $q$, one may not expect a non-trivial scale-invariant continuum limit at the transition. In contrast, the Ising model ($q=2$) on $\mathbb{Z}^3$ undergoes a continuous phase transition (established in the celebrated work \cite{aizenman2015random}), so its FK representation potentially admits a scaling limit at criticality. Moreover, the Ising model has profound connections to the metric graph GFF (see discussions in \cite{lupu2016note, lupu2019inverting}). We hope that the framework developed here can serve as a starting point for understanding the scaling limit of critical FK-Ising clusters in three dimensions.

    SLE and CLE have led to remarkable progress in the study of critical scaling limits in two dimensions. In light of this, it is tempting to hope that a similar unified picture might one day emerge in three dimensions.

\vspace{0.2cm}

\textbf{Convention for constants.} In this paper, $C$ and $c$ (sometimes with subscripts or superscripts) denote positive constants, with $C$ reserved for large constants and $c$ for small ones. In particular, we always require $C>1$ and $c<1$. Numerically indexed constants, such as $C_1,C_2,c_1,c_2,...$, remain fixed throughout the paper, whereas unindexed constants may change from line to line. Unless stated otherwise, all constants depend only on the dimension $d$. Any additional dependence will be indicated explicitly in parentheses.

 \section{Proof outline}\label{subsection2.1_sketch}

 
 In this section, we outline the proofs of Theorems \ref{new_thm1}, \ref{theorem_dimension_gap}, \ref{thm_asymptotic} and \ref{thm_k_point}.

 \subsection{Scaling limit}\label{subsection_scaling_limit}

 We begin by outlining the key convergence argument, which constitutes the core of our proofs. At a high level, we prove the convergence by establishing the Cauchy property: we compare the collections of macroscopic loop clusters on $\delta \cdot\mathbb{Z}^d$ and $\delta' \cdot\mathbb{Z}^d$ and show that under a suitable coupling, their distance tends to zero in probability as $\delta \vee\delta' \to 0$. Without loss of generality, we assume that $\delta>\delta'$. According to the heuristic picture described in \cite{werner2020clusters}, macroscopic loops form the skeletons of the clusters, while microscopic loops connect different components of these skeletons. To implement this picture, we use geometric killing to decompose the loop soup into two parts. The loops that survive the geometric killing, which we refer to as ``small loops'', form a loop soup associated with a massive GFF. Among the killed loops, we call those whose diameters exceed a certain threshold ``large loops''; as for the remaining killed loops, we will prove in Section \ref{section_control_massive} that their influence is negligible. Since the large loops converge to the Brownian loop soup, the main issue is to understand the scaling limit of the connectivity induced by the small loops. The isomorphism theorem then allows us to reformulate the convergence of the connectivity induced by these loops as the convergence of connecting probabilities for the sign clusters of the massive GFF. To establish the latter convergence, we employ the isotropic interpolation scheme introduced in \cite{kozma2007scaling}: we refine $\delta \cdot \mathbb{Z}^d$ patch by patch until the entire graph is transformed into $\delta' \cdot \mathbb{Z}^d$ (see Figure \ref{fig_graphnew} for an illustration of the case $\delta'=\frac{\delta}{2}$). Roughly speaking, at each step, we replace the lattice inside a small box by a lattice of mesh size $\delta'$ and adjust the conductances there so that the modification has only a small effect on the harmonic measures of the associated random walk (i.e., its exit distributions from the relevant domains). The main task is to control the error incurred at each step. Compared with
\cite{kozma2007scaling}, where this scheme was applied to the LERW as a single random set, here we need to track the whole collection of clusters. On top of that, even for a single cluster, the effect of a modification is multi-scale---it alters both the configuration of the macroscopic loops and the connections induced by the microscopic loops. A key ingredient of our proof is to separate these two effects in a suitable way.

 \begin{figure}[h]
	\centering
	\includegraphics[width=  0.97\textwidth]{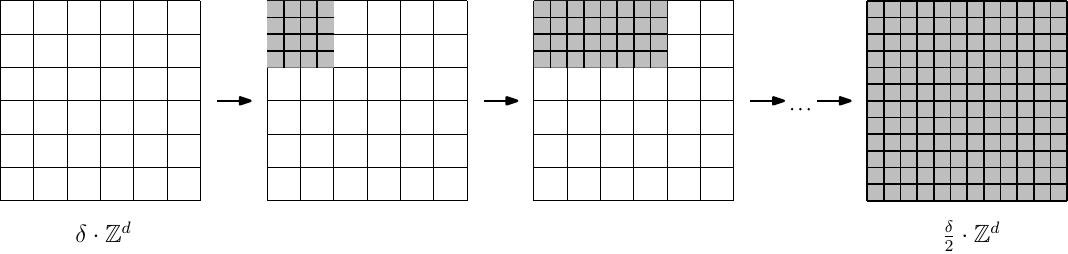}
	\caption{ 	Isotropic interpolation from $\delta\cdot\mathbb{Z}^d$ to $\frac{\delta}{2}\cdot\mathbb{Z}^d$: the lattice is refined patch by patch until the whole graph becomes $\frac{\delta}{2}\cdot\mathbb{Z}^d$.  }  \label{fig_graphnew}
\end{figure}


 In what follows, we examine the sources of these errors in more detail and outline how they can be controlled. Before that, we would like to emphasize that our main concern is how the graph modification affects the law of the connectivity relation among large loops induced by small loops. Moreover, we only need to consider the loops contained entirely in the unit ball $\mathbb{B}:=B(\{\bm{0}\},1)$. In fact, the same argument also yields the scaling limit for loops contained in $B(R):=B(\{\bm{0}\},R)$ for every $R>0$. Together with the consistency of these limits under restriction, this allows us to apply the Kolmogorov extension theorem and thereby derive the full limit in $\mathbb{R}^d$, as stated in Theorem \ref{new_thm1}.

   \textbf{Part I: Fluctuations in the coupling between large loops.} Suppose that the graph is modified inside $B(x,\delta^{\beta}):=B(\{x\},\delta^{\beta})$, where $\delta$ is the mesh size and $\beta\in (0,1)$ is a small parameter to be determined later. Loops that do not enter this ball remain unchanged. By choosing the diameter threshold for large loops appropriately, a large loop that enters $B(x,\delta^{\beta})$ will also cross the annulus $B(x,\delta^{\beta'})\setminus B(x,\delta^{\beta})$ (where $\beta'\in(0,\beta)$ will be determined later). For such loops, we construct a coupling between the loop soups before and after the modification such that with high probability, corresponding loops in each pair coincide outside $B(x,\delta^{\beta})$ and their Hausdorff distance is $o(\delta^{\beta})$ (see Lemma \ref{lemma_coupling_crossingloop} below). It remains to estimate the probability that the local modification alters the connectivity relation among the large loops. In fact, the probability of having a loop and a disjoint loop cluster, both crossing the annulus $B(x,\delta^{\beta'})\setminus B(x,\delta^{\beta})$, is $O(\delta^{(\beta-\beta')(d+c_{\dagger})})$, where $c_{\dagger}>0$ is a universal constant (we will prove this estimate later). Since $\mathbb{B}$ can be covered by $O(\delta^{-d\beta})$ balls of radius $\delta^{\beta}$, the total probability of these exceptional events is $O(\delta^{a_{\dagger}})$ with $a_{\dagger}:= c_\dagger(\beta-\beta')- d\beta'$ (where we choose the parameters $\beta$ and $\beta'$ such that $a_{\dagger}>0$). This estimate implies that except on an event of vanishing probability, regardless of how the portion of large loops inside $B(x,\delta^{\beta})$ changes under the modification, every loop cluster crossing $B(x,\delta^{\beta'})\setminus B(x,\delta^{\beta})$ intersects all large loops crossing the same annulus. In other words, all loop clusters potentially affected by the modification remain connected to the large loops crossing the annulus $B(x,\delta^{\beta'})\setminus B(x,\delta^{\beta})$; in particular, the connectivity relations among the large loops are preserved. In conclusion, the error arising from the fluctuations of the large loops throughout the interpolation vanishes polynomially as $\delta \to 0$.

       We now give a heuristic explanation of the key estimate used in the preceding discussion. Precisely, for $N> n \ge 1$, we define $\mathsf{F}(N,n)$ as the event that there exist a loop and a loop cluster that are disjoint and both cross $B(N)\setminus B(n)$. Then for any $3\le d\le 5$, there exists a constant $c_\dagger>0$ such that  
       \begin{equation}\label{ineq_1.8}
       	\mathbb{P}\big( \mathsf{F}(N,n)  \big) \lesssim  \big( n/N \big)^{d+c_\dagger}, 
       \end{equation}
       where ``$f \lesssim g$'' means that $f\le Cg$ holds for some constant $C>0$ depending only on $d$. To see this, recall that the probability of having a loop $\ell$ (resp. a loop cluster $\mathcal{C}$) crossing $B(N)\setminus B(n)$ is of order $(n/N)^{d-2}$ (resp. $(n/N)^{\frac{d}{2}-1}$). Note that this crossing loop includes two random walk trajectories $\eta_1$ and $\eta_2$ starting from $\partial B(n)$ and stopped upon hitting $\partial B(N)$. In addition, the crossing loop cluster has a uniformly positive probability to block such a random walk at each dyadic scale. As a result (see Lemma \ref{lemma_cluster_dense} below), for some small constant $c_\dagger>0$, with probability $1-O((n/N)^{d+1})$ the cluster is sufficiently dense (we denote this event by $\mathsf{A}$) such that any random walk from $\partial B(n)$ to $\partial B(N)$ avoids it with probability at most $(n/N)^{c_{\dagger}}$. Moreover, the estimates in \cite{cai2025heterochromatic} show that the annealed probability of this avoidance event is $O((n/N)^{3-\frac{d}{2}})$. To sum up, we obtain the desired bound (\ref{ineq_1.8}) (see Lemma \ref{lemma_oneloop_onecluster} for an actual proof):
       \begin{equation}\label{ineq_1.9}
       	\begin{split}
       		\mathbb{P}\big( \mathsf{F}(N,n)  \big) \le  &	\mathbb{P}\big( \mathsf{F}(N,n), \mathsf{A}   \big) + 	\mathbb{P}\big( \mathsf{A}^c \big)  \\
       		\lesssim  &  (\tfrac{n}{N})^{\frac{3d}{2}-3}\cdot \mathbb{P}\big( (\eta_1\cup \eta_2)\cap \mathcal{C}=\emptyset , \mathsf{A}   \big)  +(\tfrac{n}{N})^{d+1}\\
       		\lesssim & (\tfrac{n}{N})^{\frac{3d}{2}-3}\cdot (\tfrac{n}{N})^{c_{\dagger}} \cdot \mathbb{P}\big(   \eta_1  \cap \mathcal{C}=\emptyset   \big)  +(\tfrac{n}{N})^{d+1}  \\
       		\lesssim &    (\tfrac{n}{N})^{\frac{3d}{2}-3}\cdot  (\tfrac{n}{N})^{c_{\dagger}}\cdot (\tfrac{n}{N})^{3-\frac{d}{2}} + (\tfrac{n}{N})^{d+1} \lesssim (\tfrac{n}{N})^{d+c_{\dagger}}. 
       	\end{split}
       \end{equation} 
 A similar argument also shows that the triple-crossing probability satisfies the bound in (\ref{ineq_1.8}). Precisely, for $N>n\ge 1$, let $\mathsf{G}(N,n)$ denote the event that there exist three disjoint loop clusters crossing the annulus $B(N)\setminus B(n)$. In fact (see (\ref{bound_crossing_by_hitting}) below), if a crossing cluster is sufficiently dense such that $\mathsf{A}$ occurs, then the probability of having another crossing cluster within its complement is at most    \begin{equation}
 	 (\tfrac{n}{N})^{c_\dagger} \cdot \mathbb{P}\big( B(n) \xleftrightarrow{} \partial B(N) \big) \asymp (\tfrac{n}{N})^{\frac{d}{2}-1+c_
 	 \dagger},
 \end{equation} 
  where we used the crossing probability estimates in \cite[Theorem 1.2]{cai2024one}.  
  Meanwhile, it was shown in \cite{cai2025heterochromatic} that the probability of having two crossing clusters is proportional to $(n/N)^{\frac{d}{2}+1}$. Thus, similar to (\ref{ineq_1.9}), we have (see Lemma \ref{lemma_threecluster} below)
 \begin{equation}\label{ineq1.11}
 	\mathbb{P}\big( \mathsf{G}(N,n)  \big) \lesssim  (\tfrac{n}{N})^{\frac{d}{2}-1+c_{\dagger}} \cdot (\tfrac{n}{N})^{\frac{d}{2}+1} + (\tfrac{n}{N})^{d+1} \lesssim  (\tfrac{n}{N})^{d+c_{\dagger}}.  
 \end{equation}

   \textbf{Part II: Errors in the connectivity between large loops.} Suppose that we perform the same graph modification inside $B(x,\delta^{\beta})$. The case when a large loop intersects $B(x,\delta^{\beta})$ has been discussed, so we now turn to the remaining case when such a loop is absent. In this case, we explore the sign clusters of the massive GFF containing the large loops, with the exploration restricted to the complement of $B(x,\delta^\beta)$. By the isomorphism theorem, these sign clusters are measurable with respect to the occupation field outside $B(x,\delta^\beta)$. As in Part I, we apply the coupling in Lemma \ref{lemma_coupling_crossingloop} under which with high probability, the occupation field of small loops is unchanged outside $B(x,\delta^\beta)$. As a result, the corresponding sign clusters coincide before and after the graph modification. At this stage, the graph modification can only affect the connecting probabilities among the sign clusters that have not been fully explored, namely, those intersecting $B(x,\delta^\beta)$. In particular, if there is only one such partial cluster, the modification has no effect. On the other hand, by (\ref{ineq1.11}), the probability of having at least three such sign clusters is $O(\delta^{(\beta-\beta')(d+c_{\dagger})})$ and thus, the total error arising from this exceptional event is at most of order $\delta^{a_{\dagger}}$ (recall that $a_{\dagger} = c_{\dagger}(\beta-\beta')-d\beta'>0$). The remaining case (i.e., there are exactly two partial clusters; see Figure \ref{fig1}) can be reduced to the following question: for $N\ge 1$ and connected subsets $D_1,D_2\subset [B(N)]^c$, given the absolute values of $\phi$ on $D_1\cup D_2$ with non-zero values only on $\partial B(N)$, how to estimate the change in the connecting probability $\mathbb{P}(D_1\xleftrightarrow{} D_2)$ under the graph modification inside $B(\frac{N}{2})$?


  \begin{figure}[h]
	\centering
	\includegraphics[width=0.43\textwidth]{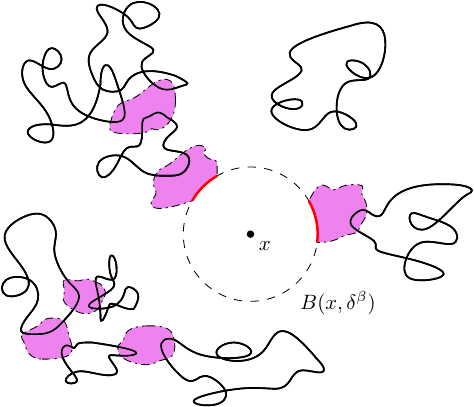}
	\caption{In this figure, the pink areas represent the sign clusters of the massive GFF that connect clusters of large loops to one another or connect such a cluster to $B(x,\delta^\beta)$; for those connecting to the ball, only the portions connected outside $B(x,\delta^\beta)$ are revealed.}\label{fig1}
\end{figure}


     As shown in Lemma \ref{newlemma2.3} below, an analysis similar to that in \cite[Theorem 2]{werner2025switching} yields an explicit relation between the aforementioned connecting probability and its counterpart obtained by forcing $\phi$ to have sign ``$+$'' on both $D_1$ and $D_2$. Using the Lupu--Werner formula in \cite{lupu2018random}, the latter probability can be approximated by 
     \begin{equation}\label{simple_lupu_werner_formula}
     	   \sum\nolimits_{v_1\in   \partial D_1 ,v_2\in \partial D_2  } 
     \mathbb{K}_{D_1\cup D_2}(v_1,v_2)\phi_{v_1}\phi_{v_2}  
     \end{equation} 
where $\mathbb{K}_{D_1\cup D_2}(\cdot, \cdot )$ denotes the  boundary excursion kernel (on the metric graph) for the set $D_1\cup D_2$ (it can be considered as the total mass of excursions from $v_1$ to $v_2$ without hitting $D_1\cup D_2$; see (\ref{def_K}) for its definition). The local modification in $B(\frac{N}{2})$ can only affect the excursions that enter
$B(\frac{N}{2})$. However, thanks to the isotropic interpolation construction
(see \cite[Section 5]{kozma2007scaling}), the total mass of such excursions is changed by at most a factor $1+O(N^{-c_*})$ for some universal constant $c_*>0$. Combining these observations, we conclude that the graph modification inside $B(\frac{N}{2})$ introduces only a polynomially small relative error in $N$ to $\mathbb{P}(D_1\xleftrightarrow{}D_2)$. Returning to the preceding analysis, this implies that the graph modification inside $B(x,\delta^{\beta})$ changes the connecting probability between the two partial sign clusters by $O(\delta^{c_*(1-\beta)})$. The resulting total error is therefore $O(\delta^{c_*(1-\beta)-d\beta})$, where we require $\beta<\frac{c_*}{c_*+d}$ such that the exponent here is positive. In conclusion, the error in the connection between the large loops throughout the interpolation also vanishes polynomially as $\delta \to 0$.

 \textbf{Parts I} and \textbf{II} together establish a limiting law for the connectivity relation among large loops induced by small loops. Meanwhile, the large loops themselves converge to Brownian loops (the entire Brownian loop soup is recovered in the limit, since the diameter threshold defining the ``large loops'' tends to zero with the mesh size). The joint limiting law of the loops and their connectivity relation thus defines the probability space in Theorem \ref{new_thm1}. In fact, since the diameters of clusters formed entirely by small loops vanish, the same limiting law also determines the scaling limit of clusters in any domain $D$ with $C^0$ boundary (as stated in the scaling limit part of Theorem \ref{new_thm1}). The boundary assumption is imposed to rule out the following ambiguity: although the clusters converge, small fluctuations may still cause some of the clusters (especially those near $\partial D$) to switch between being contained in $D$ and extending beyond $D$, which would lead to two distinct subsequential limits for the clusters in $D$.

We now present how the aforementioned convergence comes into play in the proofs of Theorems \ref{thm_asymptotic} and \ref{thm_k_point}.   



 \subsubsection{Asymptotic decay of cluster diameters}

 Let $\delta=N^{-1}$. Note that the event $\{|\mathcal{C}_{\bm{0}}|_{\infty}\ge N\}$ on $\mathbb{Z}^d$ is equivalent (under the rescaling $x\mapsto\delta x$) to $\{|\mathcal{C}_{\bm{0}}^{\delta}|_{\infty}\ge 1\}$ on $\delta\cdot\mathbb{Z}^d$, where $\mathcal{C}^{\delta}_{\bm{0}}$ denotes the cluster containing the origin on $\delta\cdot\mathbb{Z}^d$. By the same graph modification argument as before, replacing the portion of $\delta\cdot\mathbb{Z}^d$ outside $B(\delta^{\beta})$ by a lattice of mesh size $\delta'\in (0,\delta)$ changes the probability of $\{|\mathcal{C}_{\bm{0}}^{\delta}|_{\infty}\ge 1\}$ by at most $O(\delta^{\frac{d}{2}-1+c})$ (we denote the resulting graph by $\widehat{\mathbf{G}}$, and similarly the cluster of the origin thereon as $\widehat{\mathcal{C}}_{\bm{0}}$). In fact, since clusters formed by small loops have small diameter, the event $\{|\mathcal{C}_{\bm{0}}^{\delta}|_{\infty}\ge 1\}$ is essentially determined by the connectivity relation between the origin and the large loops through the small loops. The law of this connectivity relation is only slightly affected by graph modifications away from the origin; however, a graph modification near the origin may substantially alter the variance of the GFF at the origin and consequently have a significant effect on the law of the cluster containing $\bm{0}$. Thus, it remains to show that 
  \begin{equation}\label{addnew_2.6}
 \big| \delta^{-\frac{d}{2}+1} \cdot 	\mathbb{P}\big(  |\widehat{\mathcal{C}}_{\bm{0}}|_{\infty}\ge 1\big) -  (\delta')^{-\frac{d}{2}+1} \cdot  \mathbb{P}\big(  |\mathcal{C}_{\bm{0}}^{\delta'}|_{\infty}\ge 1  \big)  \big| \lesssim \delta^{c}.
 \end{equation} 
We define $\mathsf{D}$ to be the event that there exists a large loop intersecting $B(\delta^{\beta})$ (recall that such a loop must cross the annulus $B(\delta^{\beta'})\setminus B(\delta^{\beta})$, where $\beta>\beta'>0$). It was shown in \cite[(2.37)]{cai2024incipient} that conditioned on a connecting event, the probability of having a loop crossing an annulus decays polynomially in the ratio of the outer to the inner radii. Hence, to prove (\ref{addnew_2.6}), it suffices to show that  
   \begin{equation}\label{adnew_2.7}
 \big| \delta^{-\frac{d}{2}+1} \cdot 	\mathbb{P}\big(   \mathsf{D}^c,  |\widehat{\mathcal{C}}_{\bm{0}}|_{\infty} \ge 1\big) -  (\delta')^{-\frac{d}{2}+1} \cdot  \mathbb{P}\big(  \mathsf{D}^c,  |\mathcal{C}_{\bm{0}}^{\delta'}|_{\infty}\ge 1  \big)  \big| \lesssim \delta^{c}.
 \end{equation}   
The key to achieving this is to estimate the influence of the graph modification inside $B(\delta^\beta)$. As illustrated in Figure \ref{fig1}, we sample all large loops (disjoint from $B(\delta^\beta)$ on $\mathsf{D}^c$) as well as their sign clusters of the massive GFF (associated with the small loops) outside $B(\delta^{\beta})$. By the BKR inequality (see e.g., \cite{arratia2018van}), the conditional probability that at least two such clusters reach $\partial B(\delta^\beta)$ is $O(\delta^c)$. When such cluster is unique, the difference of these two probabilities lies in the quantity (recall (\ref{simple_lupu_werner_formula}))
 \begin{equation}\label{second_lupu_werner}
   \phi_{\bm{0}}     	   \sum\nolimits_{v \in   \partial D } 
     \mathbb{K}_{D\cup \{\bm{0}\}}(v,\bm{0})\phi_{v}, 
 \end{equation}
where $D$ denotes the union of all sign clusters sampled above. In fact, the graph modifications inside $B(\delta^\beta)$ have only a negligible effect on both $ \mathbb{K}_{D\cup \{\bm{0}\}}(v,\bm{0})$ and $\phi_{v}$ for $v\in\partial D$. This follows respectively from the stability of harmonic measure under isotropic interpolation, and from the coupling between the crossing loops before and after the graph modification, as described in \textbf{Part I}. Meanwhile, if the graph modification approximately rescales the lattice near $\bm{0}$ by a factor of $a$, then in order to preserve harmonic measures, the conductance of each edge in this region should be multiplied by $a^{d-2}$ (this scaling can be interpreted via the Delaunay triangulation, or equivalently, the Voronoi tessellation; see Section \ref{section_preliminaries} for more details). Consequently, the variance of $\phi_{\bm{0}}$ is multiplied by $a^{2-d}$, and hence $\phi_{\bm{0}}$ itself is rescaled by a factor of $a^{-\frac{d}{2}+1}$. This gives (\ref{adnew_2.7}) (while also explaining the exponent appearing in the two prefactors), and hence establishes Theorem \ref{thm_asymptotic}.

  \subsubsection{Asymptotics of the $k$-point function}\label{subsubsection_k_point}

The proof of Theorem \ref{thm_k_point} is similar to that of Theorem \ref{thm_asymptotic}. The only difference is that when exploring the sign clusters of large loops, we stop the exploration upon hitting $\cup_{1\le i\le k}B(x_i, \delta^{\beta})$ instead of $B(\delta^{\beta})$ (see Figure \ref{fig_k_point} for an illustration). Typically, there is a unique sign cluster that intersects every ball in $\{B(x_i, \delta^{\beta})\}_{1\le i\le k}$, and the law of this sign cluster is stable under isotropic interpolation. Thus, as in (\ref{adnew_2.7}), if the lattice inside $B(x_i, \delta^{\beta})$ is rescaled by a factor of $a_i$ for each $1\le i\le k$, then the $k$-point function is multiplied by $\prod_{1\le i\le k} a_i^{-\frac{d}{2}+1}$. This explains the exponent appearing in (\ref{def_gamma_k}).


 \begin{figure}[h]
	\centering
	\includegraphics[width=0.34\textwidth]{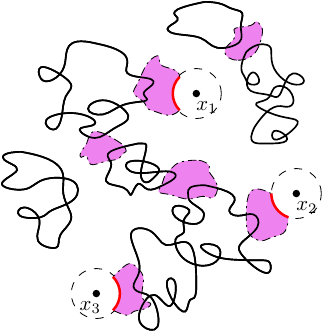}
	\caption{ An illustration of the exploration used to establish the asymptotics of the $k$-point function. As in Figure \ref{fig1}, the pink regions represent the sign clusters of the massive GFF that either connect two clusters of large loops or connect a cluster of large loops to one of the balls in $\{B(x_i, \delta^{\beta})\}_{1\le i\le k}$.   }\label{fig_k_point}
\end{figure}

 \subsection{Conformal invariance}\label{subsection_discuss_conformal}

  For $d\ge 3$, every conformal map on $\mathbb{R}^d$ can be expressed as a finite composition of translations and inversions (see e.g., \cite[Section 3.1]{beardon2012geometry}). Hence, to establish conformal invariance, it suffices to prove the invariance under translation and inversion. In fact, translation invariance of the scaling limit follows from the translation invariance of the lattice together with the stability of clusters under small perturbations (as discussed in Section \ref{subsection_scaling_limit}). The main difficulty therefore lies in establishing inversion invariance. This goes beyond the scope of \cite{kozma2007scaling}, where the interpolation was used to compare lattices related by rotations and dilations. Indeed, unlike these transformations, the inversion does not preserve the law of Brownian motion, even up to a time change. Consequently, the random walk on the inverted lattice no longer resembles that on a lattice, which calls for an interpolation between graphs of a different kind.

  We establish the inversion invariance based on the following two ingredients:   
 \begin{enumerate}

 \item[(\textbf{I.1})]  The scaling limit of loop clusters on the inverted lattice (under inversion) has the same law as the scaling limit on the original lattice after a suitable Doob $h$-transform.

 \item[(\textbf{I.2})] The law of loop clusters on a metric graph is invariant under the Doob $h$-transform, up to a negligible error.

\end{enumerate}
More precisely, we work in the annulus $B(R)\setminus B(\frac{1}{R})$ with $R>1$ (which is preserved under inversion), and apply (\textbf{I.1}) to show that the scaling limit of loop clusters on $I(\delta \cdot \mathbb{Z}^d)$ with uniform conductances (where $I(\cdot)$ denotes inversion) has the same law as that on $\delta\cdot \mathbb{Z}^d$ with conductances obtained from a suitable Doob $h$-transform. We then use (\textbf{I.2}) to conclude that the latter has the same law as the scaling limit on $\delta \cdot  \mathbb{Z}^d$ (see Figure~\ref{figure_cases}), thereby establishing inversion invariance in $B(R)\setminus B(\frac{1}{R})$. Letting $R\to\infty$ extends this to the whole space.

   \begin{figure}[h]
	\centering
	\includegraphics[width=  0.99\textwidth ]{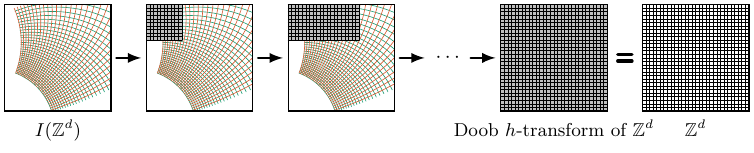}
	\caption{An illustration of the isotropic interpolation in the proof of inversion invariance.} \label{figure_cases}
\end{figure}

Next, we briefly explain the ideas in the proofs of these two ingredients.


\textbf{(I.1): Graph modification.} We begin by recalling the concept of the Doob $h$-transform. Let $X_t$ be a continuous-time Markov chain on a state space $D$ with transition density $q_t(x,y)$, and let $h(\cdot)$ be a positive harmonic function on $D$. The Doob $h$-transform of $X_t$ is the Markov process with transition density 
 \begin{equation}\label{h_2.9}
 	q_t^h(x,y):= \frac{h(y)}{h(x)}\cdot q_t(x,y). 
 \end{equation} 
 Next, we record two examples of the Doob $h$-transform: 
 \begin{enumerate}

 	\item[(i)] Let $D\subset \mathbb{R}^d\setminus\{\bm{0}\}$ (where $d\ge 3$) be a finite domain, and consider Brownian motion $W_t$ up to its exit from $D$. Under a specific continuous time change, the inversion of $W_t$ has the same law as the Doob $h$-transform of $W_t$ with $h(x)=|x|^{2-d}$ (see e.g., \cite[Chapter VIII, Exercise (3.17)]{revuz2013continuous}).

 	\item[(ii)]  Consider a weighted graph $\mathbf{G}$ with vertex set $\mathbf{V}$ and weights (or conductances) $\{w(x,y)\}_{x,y\in \mathbf{V}}$ (we assume $w(x,x)=0$ throughout this paper). Then taking the Doob $h$-transform of a continuous-time random walk on $\mathbf{G}$ is equivalent to replacing the conductances by (see \cite[Exercise 2.32]{MR3616205})
 \begin{equation}\label{h_2.10}
    w^h(x,y):=h(x)h(y)w(x,y). 
\end{equation}

 \end{enumerate} 	
 Item (i) together with the invariance principle suggests that the random walks on the following two graphs behave similarly and hence admit an isotropic interpolation: the inverted lattice $I(\delta\cdot \mathbb{Z}^d)$ with uniform conductances, and the original lattice $\delta \cdot \mathbb{Z}^d$ with the conductances obtained from the Doob $h$-transform with 
 \begin{equation}\label{u2.11}
 	h(x)= \mathbb{E}\big[ |S^x( \tau_{D^c})|^{2-d}  \big], \ \forall x\in D\cap (\delta \cdot \mathbb{Z}^d), 
 \end{equation}
where $S^x(t)$ denotes the simple random walk on $\delta \cdot \mathbb{Z}^d$ starting from $x$, and $\tau_{D^c}$ denotes its first exit time from $D$. Thus, by applying the arguments in Section \ref{subsection_scaling_limit} to this isotropic interpolation, we obtain (\textbf{I.1}).

 
\textbf{(\textbf{I.2}): Invariance under Doob $h$-transform.} We first examine how the law of the GFF changes under the Doob $h$-transform (following \cite{chang2014markov, prevost2023percolation}). Let $h$ be a discrete harmonic function in a finite region $D\subset \mathbf{V}$, i.e., for any $x\in D$, 
 \begin{equation}\label{h_2.12}
 	  \sum\nolimits_{y\in \mathbf{V}} w(x,y)\cdot  [ h(y)-h(x)] =0.
 \end{equation} 
Under the Doob $h$-transform as in (\ref{h_2.10}), the Green's function on $\mathbf{G}$ becomes \begin{equation}
 \begin{split}
 		G^h(x,y)= &  \Big[ \sum\nolimits_{z\in \mathbf{V}} w^h(y,z) \Big]^{-1}\cdot \int_{t>0}q_t^h(x,y) \mathrm{d}t \\
 		 \overset{ (\ref{h_2.9}), (\ref{h_2.10}) }{=}&   \Big[ \sum\nolimits_{z \in \mathbf{V}} h(y)h(z)w(y,z) \Big]^{-1}\cdot \int_{t>0}\frac{h(y)}{h(x)}\cdot q_t(x,y)  \mathrm{d}t  \\
 		\overset{(\ref{h_2.12})}{=} &  [h(x)h(y)]^{-1}\cdot \Big[ \sum\nolimits_{z \in \mathbf{V}} w(y,z) \Big]^{-1}   \int_{t>0}  q_t(x,y)  \mathrm{d}t \\
 		=& [h(x)h(y)]^{-1} \cdot  G(x,y). 
 \end{split}
 \end{equation} 
 This yields that the original GFF $\{\phi_x\}_{x\in D}$ on $\mathbf{G}$ and the transformed GFF $\{\phi_x^h\}_{x\in D}$ can be coupled such that 
 \begin{equation}
 	\phi_x^h= [h(x)]^{-1} \phi_x, \ \forall x\in D. 
 \end{equation} 
Moreover, the metric graph extensions of the two GFFs induce the same law on the connectivity relation among the vertices through sign clusters. This is because conditional on the GFF values at the vertices, the fields in the interiors of distinct intervals are independent, and in addition, for any $x,y\in\mathbf{V}$ with $w(x,y)>0$, the conditional probability that $x$ and $y$ are connected by a sign cluster within the interval joining them is given by (referring to \cite[(3.5)]{lupu2016loop})
\begin{equation}
	\begin{split}
		\mathbbm{1}_{\phi_x\phi_y>0 }\cdot \big(1-e^{-w(x,y)\phi_x\phi_y} \big),
	\end{split}
\end{equation}
 which remains unchanged under Doob $h$-transform since
 \begin{equation}\label{a216}
 	\begin{split}
 	 w^h(x,y)\phi_x^h\phi_y^h =  h(x)h(y)w(x,y)\cdot \frac{\phi_x}{h(x)} \cdot \frac{\phi_y}{h(y)}= w(x,y)\phi_x\phi_y.
 	\end{split}
 \end{equation}
 In short, applying a Doob $h$-transform to a graph does not alter the scaling limit of its loop clusters, as stated in (\textbf{I.2}).

 \textbf{Conformal covariance of the $k$-point function.} The analysis above shows that the law of the sign clusters of large loops sampled as in Figure \ref{fig_k_point} is approximately invariant under inversion. Therefore, as explained in Section \ref{subsubsection_k_point}, if the lattice in the neighborhood of $x_i$ is locally rescaled by a factor of $a_i$ for all $1\le i\le k$, then the $k$-point function is approximately multiplied by $\prod_{1\le i\le k} a_i^{-\frac{d}{2}+1}$. Especially, for inversion we have $a_i=|x_i|^{-2}$, which matches $\lambda_I(x_i)$ defined below (\ref{good_113}). Substituting these factors into the preceding relation yields (\ref{good_113}).

\subsection{Dimension gap}

In this subsection, we explain why the clusters of the loop soup $\mathcal{L}$ on $\widetilde{\mathbb{Z}}^d$ have a strictly larger fractal dimension than those of the Brownian loop soup $\mathcal{L}^{\mathrm{B}}$ on $\mathbb{R}^d$. In fact, in both settings the number of macroscopic clusters is tight. This follows since a cluster crosses the annulus $B(R)\setminus B(r)$ with probability tending to zero as $R/r$ increases (by \cite[Theorem 1.2]{cai2024one}). As a consequence, hyperscaling (a family of quantitative relations among critical exponents) holds, as in critical percolation (see \cite{borgs1999uniform,cai2024quasi}). In particular, the dimension of the clusters is given by $d-\zeta$, where $\zeta$ is the one-arm exponent. Therefore, it suffices to show that the one-arm exponent of $\mathcal{L}^{\mathrm{B}}$ is strictly greater than $\frac{d}{2}-1$, the one-arm exponent of $\mathcal{L}$. The key to proving this gap is to show that as a cluster of $\mathcal{L}$ grows, one can find logarithmically many scales such that at each of these scales, a small pivotal loop (of constant size) appears with uniformly positive probability (see Figure~\ref{figure_multi_scale}). Under the canonical coupling between the macroscopic loops and the Brownian loops, these small loops are present only in $\mathcal{L}$. By the observation above, if a large loop cluster is constructed without using these pivotal loops (which is necessarily the case for Brownian loop clusters), then one loses a constant factor in probability at each of these scales, and hence a polynomial factor in total. This polynomial factor provides a lower bound on the ratio between the one-arm probabilities of the two loop soups, which in turn yields the dimension gap.

     \begin{figure}[h]
	\centering
	\includegraphics[width=0.33\textwidth]{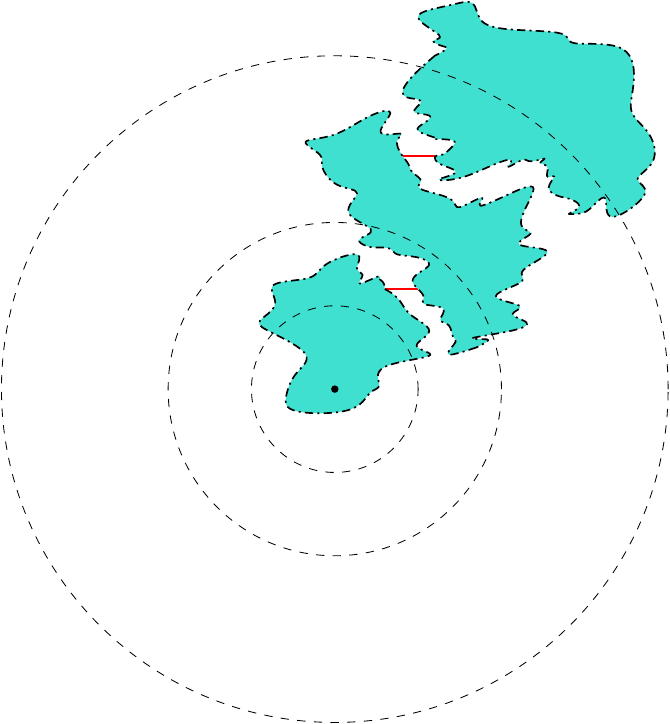}
	\caption{An illustration of the multi-scale structure. The dashed circles indicate successive scales, and the shaded regions represent macroscopic loop clusters of $\mathcal{L}$. The red segments represent small pivotal loops, whose removal would disconnect the cluster.   } \label{figure_multi_scale}
\end{figure}

 The emergence of pivotal loops stems from our estimate on the two-arm probability. Specifically, we proved in \cite{cai2025heterochromatic} that for $3\le d\le 5$, the probability that a vertex and one of its neighbors are contained in two distinct loop clusters of $\mathcal{L}$ with diameters at least $N$ is of order $N^{-\frac{d}{2}-1}$. Consequently, in a box of size $N$, the expected number of points at which two macroscopic clusters touch is of order $N^{d}\cdot N^{-\frac{d}{2}-1}=N^{\frac{d}{2}-1}$. Moreover, using a second moment argument, we showed in \cite{cai2025separation} that with positive probability, such points exist and form a set of dimension $\frac{d}{2}-1$. These touching points correspond to pivotal loops: adding a loop of constant size at a touching point creates a pivotal loop, and conversely, removing a small pivotal loop creates a touching point. As a result, in a cluster connecting two sets at distance $N$, with positive probability there are at least $cN^{\frac{d}{2}-1}$ pivotal loops of constant size. In particular, when one explores such a large cluster from a given point, pivotal loops appear at every scale with uniformly positive probability, which is precisely the property required in the previous paragraph.

 \subsection{Extra randomness}\label{subsection_extra_randomness}

The small pivotal loops discussed above also give rise to the extra randomness of the gluing relation $\sim_{\mathrm{g}}$. Assume for contradiction that $\sim_{\mathrm{g}}$ is measurable with respect to $\mathcal{L}^{\mathrm{B}}$. Recall that the Brownian loops $\mathcal{L}^{\mathrm{B}}$ and the gluing relation $\sim_{\mathrm{g}}$ can be approximated by the macroscopic loops on $\delta\cdot \widetilde{\mathbb{Z}}^d$ (i.e., loops with diameters greater than a small constant $\epsilon>0$) and their connectivity relation induced by the loop clusters respectively. Therefore, our assumption implies that this connectivity relation is asymptotically determined by the macroscopic loops themselves. However, as observed in the previous subsection, adding a small loop at a touching point creates a pivotal loop, which changes this connectivity relation while leaving all macroscopic loops unchanged.

 We now describe the argument in more detail. For two macroscopic loops $\ell$ and $\ell'$, we say that their clusters nearly touch on an edge $e$ if each of the two clusters intersects an endpoint of $e$. Let $\mathsf{T}$ denote the event that $\ell$ and $\ell'$ belong to two distinct clusters which nearly touch on at least $c\delta^{-\frac{d}{2}+1}$ edges. By the estimates on touching points recalled in the previous subsection (with $N$ replaced by $\delta^{-1}$), the event $\mathsf{T}$ occurs with uniformly positive probability. For any configuration in $\mathsf{T}$, we may choose one of these edges and add a small loop (i.e., a loop of diameter $O(\delta)$) traversing it. The added loop is then pivotal for $\{\ell\xleftrightarrow{}\ell'\}$, while all macroscopic loops remain unchanged. Conversely, starting from a configuration in which $\ell$ and $\ell'$ are connected, we can recover the original configuration by removing one of the small pivotal loops. Comparing the multiplicities of these two operations, we obtain that for any event $\mathsf{Q}$ measurable with respect to the macroscopic loops,
\begin{equation}\label{eq_outline_extra}
	\mathbb{P}\big(\mathsf{Q}\cap \mathsf{T}\big)\lesssim  \big(c\delta^{-\frac{d}{2}+1} \big)^{-1 }\mathbb{E}\big[\#\{\text{small pivotal loops}\}\cdot \mathbbm{1}_{\mathsf{Q}\cap \{\ell\xleftrightarrow{}\ell'\}}\big]. 
\end{equation}
 Moreover, the number of small pivotal loops is typically of order $\delta^{-\frac{d}{2}+1}$; more precisely, its second moment is of order $\delta^{-d+2}$ (this follows from the estimates on cut edges in \cite{cai2025separation}). Thus, by the Cauchy--Schwarz inequality, we have 
\begin{equation}\label{eq_outline_extra_CS}
	\mathbb{E}\big[\#\{\text{small pivotal loops}\}\cdot \mathbbm{1}_{\mathsf{Q}\cap \{\ell\xleftrightarrow{}\ell'\}}\big]\lesssim \delta^{-\frac{d}{2}+1}\cdot \mathbb{P}\big(\mathsf{Q}\cap \{\ell\xleftrightarrow{}\ell'\}\big)^{\frac{1}{2}}.
\end{equation}
Combining (\ref{eq_outline_extra}) and (\ref{eq_outline_extra_CS}), we obtain 
\begin{equation}
	\mathbb{P}(\mathsf{Q}\cap \mathsf{T})\lesssim \mathbb{P}(\mathsf{Q}\cap \{\ell\xleftrightarrow{}\ell'\})^{\frac{1}{2}}.
\end{equation}
  Now, if $\sim_{\mathrm{g}}$ were determined by $\mathcal{L}^{\mathrm{B}}$, we could take $\mathsf{Q}$ to be (an approximation of) the event that $\ell$ and $\ell'$ are not glued together.  Then $\mathbb{P}(\mathsf{Q}\cap \{\ell\xleftrightarrow{}\ell'\})$ would tend to zero as $\delta\to 0$, and so would $\mathbb{P}(\mathsf{Q}\cap \mathsf{T})$ by (\ref{eq_outline_extra}) and (\ref{eq_outline_extra_CS}). However, since $\mathsf{T}$ is contained in the event that $\ell$ and $\ell'$ are not connected (which is approximately $\mathsf{Q}$), the probability $\mathbb{P}(\mathsf{Q}\cap \mathsf{T})$ approximates $\mathbb{P}(\mathsf{T})$ and hence is uniformly positive. This yields a contradiction, and thus $\sim_{\mathrm{g}}$ is not measurable with respect to $\mathcal{L}^{\mathrm{B}}$.

 \vspace{0.2cm}

  \textbf{Organization of this paper.} In Section \ref{section_preliminaries}, we review the relevant notations and preliminary results. In Section \ref{section_connecting_prob}, we present a series of estimates for connecting probabilities. In Section \ref{section5_convergence}, we prove the convergence of the connectivity relation between large loops, and then construct the limiting probability measure in Theorem \ref{new_thm1}. In Section \ref{section_proof_theorem_1.3_1.6}, we prove Theorem \ref{thm_asymptotic}, and the existence of the limit (\ref{def_gamma_k}) in Theorem \ref{thm_k_point}. In Section \ref{section3_proof_prop1.2}, we prove the scaling limit part of Theorem \ref{new_thm1}. In Section \ref{section_conformal_invariance}, we establish the conformal invariance part of Theorem \ref{new_thm1} and the conformal covariance (\ref{good_113}) in Theorem \ref{thm_k_point}. In Section \ref{section_extra_randomness} we prove the extra randomness part of Theorem \ref{new_thm1}. Finally, in Section \ref{section_dimension_gap} we establish the dimension gap between Brownian loop clusters and the clusters obtained under the gluing relation $\sim_{\mathrm{g}}$ in Theorem \ref{new_thm1}, as stated in Theorem \ref{theorem_dimension_gap}.

\section{Preliminaries}\label{section_preliminaries}

In this section, we introduce some basic notation and collect some fundamental properties that will be used later.

 \subsection{Basic notions for graphs}

 A weighted graph $\mathbf{G}$ consists of a countable vertex set $\mathbf{V}$, a cemetery state $\mathfrak{c}$, symmetric weights $w: \mathbf{V} \times \mathbf{V}\to [0,\infty)$, and a non-negative killing weight $\kappa: \mathbf{V} \to [0,\infty)$. For convenience, we also write $w(x,\mathfrak{c})=\kappa(x)$ for all $x\in \mathbf{V}$. A continuous-time random walk $\{X_t\}_{t\ge 0}$ on the weighted graph evolves as follows. For any $x\in \mathbf{V}$, $y\in \mathbf{V}\cup \{\mathfrak{c}\}$ and $t\ge 0$, 
  \begin{equation*} 
  \mathbb{P}\big( X_{t+\Delta t}  =y \mid X_t =x \big) = \left\{ \begin{aligned} & w(x,y)\Delta t+ o(\Delta t) &  \text{if}\   x\neq y;   \\ & 1- \lambda_x \Delta t + o(\Delta t)  & \text{if}\ x= y. \end{aligned} \right. 
  \end{equation*} 
Here $\lambda_x:=\kappa(x)+\sum_{y\in \mathbf{V}}w(x,y)$ is the total weight at $x$ (recall that we assume $w(x,x)=0$). Once $X_\cdot$ enters the cemetery state $\mathfrak{c}$, it remains there forever. For any $x,y\in\mathbf V$, we write $x\sim y$ if $w(x,y)>0$. The edge set of $\mathbf{G}$ is denoted by $\mathbf{E}:=\{\{x,y\}:x,y\in \mathbf{V}, x\sim y\}$.

For any $x\in \mathbf{V}$, we denote by $\mathbb{P}_x$ the law of $X_\cdot$ starting from $x$, and by $\mathbb{E}_x$ the expectation under $\mathbb{P}_x$. For any $A\subset \mathbf{V}\cup \{\mathfrak{c}\}$, let 
$$\tau_{A}=\tau_A(X_{\cdot}):= \inf\{t\ge 0: X_t\in A\}$$ denote the first hitting time of $A$ by $X$, with the convention that $\inf \emptyset = +\infty $. Especially, we write $\tau_{\{\mathfrak{c}\}}$ as $\zeta$.

\textbf{Metric graph.} We now describe the construction of a metric graph from a weighted graph. Arbitrarily take a weighted graph $\mathbf{G}$. For each edge $\{x,y\}\in \mathbf{E}$, we assign a compact interval $I_{\{x,y\}}$ of length $|I_{\{x,y\}}|:=[2w(x,y)]^{-1}$ with endpoints identified with $x$ and $y$. In addition, for each $x\in \mathbf{V}$, we assign a ray whose starting point is identified with $x$, and denote by $\mathfrak{c}_x$ the point on this ray such that the length of the interval between $x$ and $\mathfrak{c}_x$ is $[2\kappa(x)]^{-1}$ (when $\kappa(x)=0$, we set $\mathfrak{c}_x=\emptyset$). We then define the metric graph $\widetilde{\mathbf{G}}$ as the union of these intervals and rays, glued at their common endpoints. For $\{x,y\}\in \mathbf{E}$ and $v_1,v_2\in I_{\{x,y\}}$, let $I_{[v_1,v_2]}$ denote the sub-interval of $I_{\{x,y\}}$ with endpoints $v_1$ and $v_2$. Given an embedding $\iota:\mathbf{V}\to \mathbb{R}^d$, each interval $I_{\{x,y\}}$ is identified with the line segment in $\mathbb{R}^d$ connecting $\iota(x)$ and $\iota(y)$; in addition, $v\in I_{\{x,y\}}$ corresponds to $(1-\frac{|I_{[x,v]}|}{|I_{\{x,y\}}|})\cdot \iota(x)+ \frac{|I_{[x,v]}|}{|I_{\{x,y\}}|}\cdot \iota(y)$ in $\mathbb{R}^d$.

    The canonical diffusion on $\widetilde{\mathbf{G}}$, denoted by $\{\widetilde{X}_t\}_{t\ge 0}$, is a Markov process on $\widetilde{\mathbf{G}}$ that behaves as a standard one-dimensional Brownian motion inside each interval. Upon hitting a vertex in $\mathbf{G}$, it chooses one of the incident intervals uniformly and then continues as a Brownian excursion along that interval. In addition, it stops upon hitting $\{\mathfrak{c}_x\}_{x\in \mathbf{V}}$. In fact, the behavior of $\widetilde{X}_{\cdot}$ on $\widetilde{\mathbf{G}}$ is similar to that of the random walk $X_{\cdot}$ on $\mathbf{G}$. Precisely, restricted to $\mathbf{V}$, the transition probabilities of $\widetilde{X}_{\cdot}$ coincide with those of $X_{\cdot}$. Moreover, the holding time of $X_{\cdot}$ at a vertex $x$ before it jumps to one of its neighbors and the analogous quantity for $\widetilde{X}_{\cdot}$---the total local time accumulated at $x$ before it first hits a neighbor of $x$---are both exponential random variables with parameter $\lambda_x$.

(P.S. Compared to the definition of $\widetilde{\mathbb{Z}}^d$ in Section \ref{section_intro}, the present construction attaches an additional infinite ray to each vertex. However, these rays do not affect the law of the loop soup within the intervals joining the vertices. We introduce these rays solely to construct a massive loop soup as a deterministic subcollection of the full loop soup; the construction will be described in detail in Section \ref{subsection_loop_soup}.)

 When $\{\widetilde{X}_t\}_{t\ge 0}$ starts from $v \in \widetilde{\mathbf{G}}$ (i.e., $\widetilde{X}_0=v$), we denote its law by $\widetilde{\mathbb{P}}_v$. The expectation under $\widetilde{\mathbb{P}}_v$ is written as $\widetilde{\mathbb{E}}_v$. For any $D \subset \widetilde{\mathbf{G}}$, we define the first hitting time $\widetilde{\tau}_{D}=\widetilde{\tau}_{D}(\widetilde{X}_{\cdot}):= \inf\{t\ge 0: \widetilde{X}_t\in D\}$. When $D=\{\mathfrak{c}_x\}_{x\in \mathbf{V}}$, we write $\widetilde{\tau}_{D}$ as $\widetilde{\zeta}$.

\subsection{Isotropic interpolation}\label{subsection_isometric_graph}

 In this subsection, we introduce an extension of the isotropic interpolation scheme in \cite{kozma2007scaling}. This extension is inspired by \cite{bou2025random}.

\textbf{Voronoi networks.} Let $a>0$, and assume that $\mathbf{V}\subset \mathbb{R}^d$ (with $d\ge 2$) satisfies the following conditions:
\begin{itemize}

	\item  (uniformly separated) $\mathrm{sep}(\mathbf{V}):=\inf_{x,y\in \mathbf{V}:x\neq y}|x-y|>0$;

    \item  (relatively dense) $\mathrm{den}(\mathbf{V}):=\sup_{x\in \mathbb{R}^d}\mathrm{dist}(x,\mathbf{V})\le a\cdot \mathrm{sep}(\mathbf{V})$.

\end{itemize}
 It induces a Voronoi tessellation of $\mathbb{R}^d$ with respect to the Euclidean distance. More precisely, for each $x \in \mathbf{V}$, its Voronoi cell is defined by 
\begin{equation}
	\mathsf{C}_x:= \big\{ z \in \mathbb{R}^d: |z-x|\le  |z-y| \ \text{for all}\ y\in  \mathbf{V} \setminus \{x\}   \big\}. 
\end{equation} 
For any distinct $x,y\in \mathbf{V}$, the intersection $\mathsf{C}_x \cap \mathsf{C}_y$ is contained in the hyperplane 
  \begin{equation}
  	\{z\in\mathbb{R}^d: |z-x |=|z-y|\},  
  \end{equation}
  which is orthogonal to the vector $x-y$. The two conditions above on $ \mathbf{V}$ then ensure that each $\mathsf{C}_x$ is a polytope. Next, we use this tessellation to construct a weighted graph $\mathbf{G}$ with vertex set $\mathbf{V}$. For any distinct $x,y\in \mathbf{V}$, set
  \begin{equation}\label{add34}
  	w(x,y)= \frac{ \mathrm{vol}_{d-1}(\mathsf{C}_x \cap \mathsf{C}_y) }{ |x-y| }, 
  \end{equation}
where $\mathrm{vol}_k(\cdot)$ denotes the $k$-dimensional Hausdorff measure. Especially, for $\delta\cdot \mathbb{Z}^d$, this definition assigns conductance $\delta^{d-2}$ to each edge. For adjacent vertices $x\sim y$, we write $\mathsf{F}_{x,y}:=\mathsf{C}_x\cap \mathsf{C}_y$ for the common face of $\mathsf{C}_x$ and $\mathsf{C}_y$. Fix $b>0$, and assume that the following condition is satisfied: 
\begin{itemize}

	\item (non-degenerate faces) $\mathrm{vol}_{d-1}(\mathsf{F}_{x,y})\ge b|x-y|^{d-1}$ for all $x\sim y$.  
	
\end{itemize}
We call a weighted graph $\mathbf{G}$ constructed in this way and satisfying the three conditions above a Voronoi network with parameters $(a,b)$. Note that dilation preserves both Voronoi networks and their parameters. A Voronoi network automatically satisfies the following harmonic condition: for any $x\in \mathbf{V}$, 
\begin{equation}
	\sum\nolimits_{y\sim x} w(x,y)\cdot (x-y)=0. 
\end{equation}
 This is because  
 \begin{equation}
 	\begin{split}
 			\sum\nolimits_{y\sim x  } w(x,y)\cdot (x-y) \overset{(\ref{add34})}{=} & \sum\nolimits_{y \sim x}   \mathrm{vol}_{d-1}(\mathsf{F}_{x,y}) \cdot   \frac{x-y }{ |x-y| }  \\
 			=& \int_{z\in \partial \mathsf{C}_x} n_x(z) \mathrm{d} \mathrm{vol}_{d-1}(z) =0,
 	\end{split}
 \end{equation} 
  where $n_x(z)$ denotes the outward unit normal to $\partial \mathsf{C}_x$ at $z$, and the last equality follows from the divergence theorem. We abbreviate $\mathrm{sep}(\mathbf{V})$ by $\mathbf{s}(\mathbf{G})$.

  (P.S. Here we slightly extend the ``Convention for constants'' stated at the end of Section \ref{section_intro}. Namely, whenever the underlying graph is clear from the context, we will not explicitly indicate the dependence of a constant on $d$ and on the parameters $(a,b)$ of the underlying graph. The same convention applies to the implicit constants in the notation ``$\lesssim$'' and ``$\asymp$''.)

We will also need variants of Voronoi networks with killing. Arbitrarily fix $\nu>0$. We then define $\mathbf{G}^{\nu}$ as the graph obtained from $\mathbf{G}$ by assigning to each $x\in\mathbf{V}$ the killing weight 
 \begin{equation}
 	\begin{split}
 		\kappa(x) =  2\nu \cdot \mathrm{vol}_d( \mathsf{C}_x),
 		 	\end{split}
 \end{equation}
while keeping the vertex set and weights unchanged. Equivalently, $\kappa(x)$ can be expressed as 
\begin{equation}
	\kappa(x)= \frac{\nu}{d}\sum\nolimits_{y\sim  x} w(x,y)|x-y|^2. 
\end{equation}
This is because (letting $\mathsf{Q}_{x,y}$ be the convex hull of $\{x,y\}\cup \mathsf{F}_{x,y}$)
\begin{equation}
	 \begin{split}
	 	\mathrm{vol}_d( \mathsf{C}_x) = &  \sum\nolimits_{y\sim x}  \tfrac{1}{2}\mathrm{vol}_d(\mathsf{Q}_{x,y})=\sum\nolimits_{y\sim x}  \tfrac{1}{2d} \mathrm{vol}_{d-1}(\mathsf{C}_x\cap \mathsf{C}_y)  |x-y| \\
 		\overset{(\ref{add34})}{=}&  (2d)^{-1}\sum\nolimits_{y\sim x} w(x,y)|x-y|^2.
	 \end{split}
\end{equation}
Employing the arguments in \cite{kozma2007scaling, bou2025random}, one can show that the random walk on $\mathbf{G}^{\nu}$ behaves similarly to Brownian motion with killing rate $\nu$. To state this result precisely, we first introduce some notation. For any $x\in \mathbf{V}$ and $r>0$, let $A$ be a $(d-1)$-dimensional spherical simplex in $\partial B(x,r)$. Here a spherical simplex is the intersection of the sphere with a polyhedral cone generated by finitely many linearly independent vectors contained in a common open half-space. Define $|A|_{x,r}:=\mathrm{vol}_{d-1}(A)/\mathrm{vol}_{d-1}(\partial B(x,r))$ as the normalized $(d-1)$-dimensional volume of $A$. We then define two discrete approximations of $A$. We denote the discrete Euclidean ball by $\hat{B}(x,r):=\{y \in \mathbf{V}: y \in B(x,r) \}$. For any $D\subset \mathbf{V}$, we denote its external boundary by $\hat{\partial}D:=\{y\in D^c: \exists z\in D\ \text{such that}\ \{y,z\}\in \mathbf{E} \}$. We write $\tau_{x,r}:= \tau_{ \hat{\partial} \hat{B}(x,r)}$. Let $A^+$ be the collection of vertices $y\in \hat{\partial}\hat{B}(x,r)$ such that $y$ is neighboring some vertex in $A$. In addition, let $A^-$ be the subset of $A^+$ consisting of vertices $y$ such that every edge $\{y,z\}$ with $z\in \hat{B}(x,r)$ satisfies $z\in A$.

\begin{lemma}\label{lemma2.1_isometric}

We retain the notation $x, r, A$ as above. For any $\Lambda>0$ and Voronoi network $\mathbf{G}^{\nu}$ with $\nu\in [0,\Lambda r^{-2}]$, there exists $C(\Lambda)>0$ such that for any $\diamond \in \{+,-\}$,   
\begin{equation}\label{u3.9}
	\big|  \mathbb{P}_{x}\big( X_{\tau_{x,r}}  \in A^{\diamond},\tau_{x,r}< \zeta  \big)  -  \mathbf{p}(\nu r^2) |A|_{x,r} \big|\le C\big(\tfrac{\mathbf{s}(\mathbf{G})}{r}\big)^{\frac{1}{5}}, 
    \end{equation} 
    where $\mathbf{p}(u)$ is the probability that a Brownian motion on $\mathbb{R}^d$ starting from the origin exits the unit ball before an independent exponential time with rate $u$.
\end{lemma}

 A detailed proof of Lemma \ref{lemma2.1_isometric} is provided in \cite{technicalpaper}, following \cite{kozma2002scaling, bou2025random}. As in \cite[Section 3]{kozma2007scaling}, this lemma implies that the random walk on $\mathbf{G}^{\nu}$ satisfies the standard estimates for its Green's function, hitting probabilities of balls, escape probabilities from half-spaces, and admits a strong coupling with Brownian motion.

 \textbf{Interpolation between two Voronoi networks.} Let $\mathbf{G}_1=(\mathbf{V}_1,w_1)$ and $\mathbf{G}_2=(\mathbf{V}_2,w_2)$ be two Voronoi networks on $\mathbb{R}^d$. For any $j\in \{1,2\}$ and adjacent vertices $x,y$ of $\mathbf{G}_j$, we denote by $F_{x,y}^j$ the common face of the Voronoi cells of $x$ and $y$. By the definition of Voronoi networks, one can rescale these two graphs by the same factor such that the following holds: there exist $0<s\le 1$, $t\ge 1$ and $u>0$ such that for each $j\in \{1,2\}$, 
 \begin{equation}\label{u310}
 	\mathrm{sep}(\mathbf{V}_j)\ge s,\ \mathrm{den}(\mathbf{V}_j)\le t, 
 \end{equation} 
 \begin{equation}\label{u311}
 	 \mathrm{vol}_{d-1}(F_{x,y}^j)\ge u,\ \forall \ \text{adjacent}\ x,y\in \mathbf{V}_j. \end{equation}
Fix $L,W,Q>1$. For $\xi\in \{1,2\}^{\mathbb{Z}^d}$ and $j\in \{1,2\}$, we say that a box $L\cdot (x+[0,1]^d)$ with $x\in \mathbb{Z}^d$ is Type-$j$ if $\xi(x)=j$; in addition, we define $U_j(\xi)$ to be the union of all Type-$j$ boxes. We then set $S_{\xi}:=U_1(\xi)\cap U_2(\xi)$.

\begin{definition}[Isotropic interpolation]
	We say that a family of graphs $\mathbf{G}_\xi$ for $\xi\in \{1,2\}^{\mathbb{Z}^d}$ is an interpolation between $\mathbf{G}_1$ and $\mathbf{G}_2$ if the following hold:  
\begin{itemize}

	\item   $\{\mathbf{G}_\xi \}_{\xi\in \{1,2\}^{\mathbb{Z}^d}}$ are Voronoi networks with the same parameters;

	\item   For any $\xi\in \{1,2\}^{\mathbb{Z}^d}$ and $j\in \{1,2\}$, $\mathbf{G}_\xi$ coincides with $\mathbf{G}_j$ on the region $\{z\in U_j(\xi) :\mathrm{dist}(z,S_{\xi})\ge W\}$;

	\item  For any $U\subset \mathbb{R}^d$, if $\xi,\xi'\in \{1,2\}^{\mathbb{Z}^d}$ agree at every $x\in \mathbb{Z}^d$ that satisfies $\mathrm{dist}(L\cdot (x+[0,1]^d),U)\le Q$, then $\mathbf{G}_\xi$ and $\mathbf{G}_{\xi'}$ coincide on $U$.

\end{itemize}
\end{definition}

In the case when $\mathbf{G}_1=\mathbb{Z}^2$ and $\mathbf{G}_2=2\cdot \mathbb{Z}^2$, a delicate construction of an isotropic interpolation was presented in \cite{kozma2002scaling} (see \cite[Figure 1]{kozma2002scaling} for an illustration). See also \cite[Section 6]{kozma2007scaling} for  further examples of isotropic interpolations in three dimensions, including the case when $\mathbf{G}_1=\mathbb{Z}^3$ and $\mathbf{G}_2=2\cdot \mathbb{Z}^3$. As shown in the next lemma, isotropic interpolations exist not only for these special lattices, but also between general Voronoi networks.

 \begin{lemma}\label{lemma_exist_interpolation}
 For any $d\ge 2$, $0<s\le 1$, $t\ge 1$ and $u>0$, there exist $L_0,W,Q\ge 1$ and $s',t',u'>0$ depending only on $d,s,t,u$ such that for any $L\ge L_0$ and any Voronoi networks $\mathbf{G}_1,\mathbf{G}_2$ satisfying (\ref{u310}) and (\ref{u311}), there exists an isotropic interpolation $\{\mathbf{G}_\xi \}_{\xi\in \{1,2\}^{\mathbb{Z}^d}}$ with parameters $L,W,Q$  between $\mathbf{G}_1$ and $\mathbf{G}_2$ such that every $\mathbf{G}_\xi$ satisfies (\ref{u310}) and (\ref{u311}) with $s,t,u$ replaced by $s',t',u'$.
 \end{lemma}

  A detailed proof of Lemma \ref{lemma_exist_interpolation} is given in \cite{technicalpaper}, following the method of \cite{10.1007/978-3-662-48350-3_50}. Here we briefly describe the proof. According to the definition of isotropic interpolation, the main issue in the construction lies in the seam region (i.e., a neighborhood of $S_\xi$ of fixed width). In other words, the additional vertices inserted in the seam region must be chosen so that the resulting graph still satisfies the three regularity conditions of a Voronoi network. For the first condition ``uniformly separated'', it suffices to avoid placing the additional vertices too close to one another. The second condition ``relatively dense'' is inherited automatically from the underlying graphs $\mathbf{G}_1$ and $\mathbf{G}_2$. For the third condition ``non-degenerate faces'', we first place a suitably rescaled copy of $\mathbb Z^d$ in the seam region (as illustrated in Figure \ref{fig_interpolate}), and then perturb each of these vertices independently by adding an i.i.d. displacement uniformly distributed in a small ball. Locally, a Voronoi cell is determined by finitely many linear inequalities. Thus, heuristically, for any given Voronoi face, the probability that its $(d-1)$-dimensional volume falls into a given interval of small length vanishes as the length tends to zero. Moreover, each Voronoi cell depends only on a uniformly bounded number of nearby vertices. These two observations, together with the Lov\'asz local lemma (see e.g., \cite[Corollary 5.1.2]{alon2016probabilistic}), yield the existence of a perturbation configuration for which the volumes of all Voronoi faces are bounded below by a uniform positive constant. The resulting graph is therefore still a Voronoi network.

    \begin{figure}[h]
	\centering
	\includegraphics[width=0.49\textwidth]{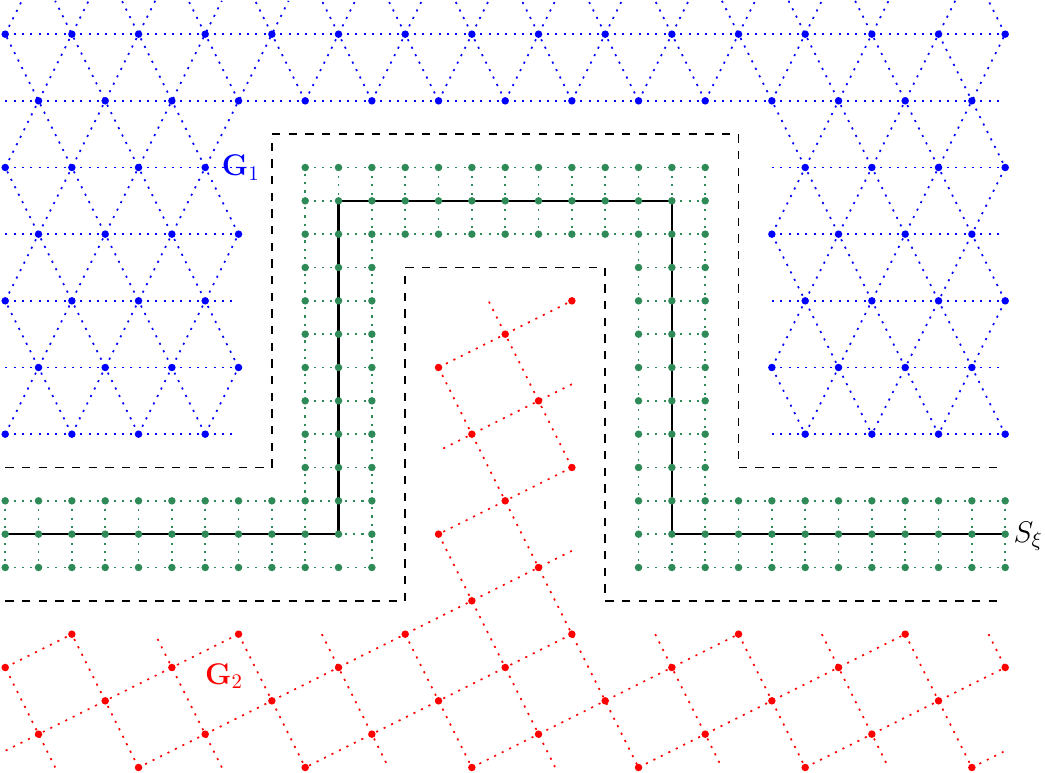}
	\caption{An illustration of the lattice in the seam region. Here the underlying graphs $\mathbf{G}_1$ and $\mathbf{G}_2$ are the triangular lattice and a tilted square lattice. The green lattice points in the neighborhood of $S_\xi$ represent the reference positions of the additional vertices.}\label{fig_interpolate}
\end{figure}

 We record the following analogue of \cite[Lemma 5.3]{kozma2007scaling} (as a straightforward adaptation, its proof will be given in \cite{technicalpaper}).

 \begin{lemma}\label{lemma_hitting_error}
 	  Let $\mathsf{\mathbf{G}}_1^\nu$ and $\mathsf{\mathbf{G}}_2^\nu$ be two Voronoi networks that differ only inside the ball $B(x,r)$, where $x\in \mathbb{R}^d$, $r\ge 1$ and $0\le \nu \le  r^{-2}$. Then there exists $\cl\label{const_hitting}>0$ such that for any $y\in \hat{\partial}\hat{B}(x,2r)$, $A\subset [B(x,4r)]^c$ and $z\in A$, 
 	  \begin{equation}
 	  	 |p_1-p_2|\lesssim  r^{-\cref{const_hitting}}\cdot \max\{p_1,p_2\},
 	  \end{equation}
 where for each $j\in\{1,2\}$, we denote $p_j:=\mathbb{P}_y(\tau_{A}=\tau_{z}<\infty)$ on $\mathsf{\mathbf{G}}_j^\nu$. 
 \end{lemma}

{\color{red}

  }

\subsection{Properties of loop soups} \label{subsection_loop_soup}

 For any metric graph $\widetilde{\mathbf{G}}$, since the canonical diffusion $\widetilde{X}_\cdot$ on $\widetilde{\mathbf{G}}$ has been defined, the loop measure $\mu$ is then given by (\ref{def_1.1}). With a slight abuse of notation, we write $\mathcal{L}$ as the loop soup on $\widetilde{\mathbf{G}}$ with intensity $\frac{1}{2}$, i.e., the Poisson point process with intensity measure $\frac{1}{2}\mu$.

  \textbf{Restriction property.} For a compact set $D\subset \widetilde{\mathbf{G}}$, we denote by $\mu^D$ the loop measure induced by the diffusion on $\widetilde{\mathbf{G}}$ stopped upon hitting $D$, and by $\mathcal{L}^D$ the corresponding loop soup of intensity $1/2$. By the thinning property of Poisson point processes, $\mathcal{L}^D$ has the same distribution as $\mathcal{L}\cdot \mathbbm{1}_{\mathrm{ran}(\ell)\cap D=\emptyset}$. Here $\mathrm{ran}(\ell)$ represents the range of the loop $\ell$, i.e., the set of points visited by $\ell$. We write $A_1 \xleftrightarrow{(D)} A_2$ as the event that $A_1$ and $A_2$ are connected by some cluster in $\mathcal{L}^D$.

    Using this property, we can extract a massive loop soup from $\mathcal{L}$. Assume that $\mathbf{G}$ and $\mathbf{G}'$ are two weighted graphs with the same vertex set $\mathbf{V}$ and edge weights $w$, where $\mathbf{G}$ has no killing, and $\mathbf{G}'$ is equipped with killing weights $\kappa:\mathbf{V}\to[0,\infty)$. Consider their metric graphs 
  $\widetilde{\mathbf{G}}$ and $\widetilde{\mathbf{G}}'$, and denote the corresponding loop soup of intensity $1/2$ by $\mathcal{L}$ and $\mathcal{L}'$ respectively. Recall that $\widetilde{\mathbf{G}}'$ can be obtained from $\widetilde{\mathbf{G}}$ by imposing the absorbing points $\bar{\mathfrak{c}}:=\{\mathfrak{c}_x\}_{x\in \mathbf{V}}$, where each $\mathfrak{c}_x$ lies on the ray attached to $x$ and is at graph distance $[2\kappa(x)]^{-1}$ from $x$. Consequently, we may construct $\mathcal{L}'$ from $\mathcal{L}$ by taking $\mathcal{L}':=\mathcal{L}\cdot \mathbbm{1}_{\mathrm{ran}(\ell)\cap \bar{\mathfrak{c}}=\emptyset}$.

 \textbf{Loop measures of crossing loops.} Let $\mathbf{G}$ be a Voronoi network. Recall that a Brownian motion starting from $x\in \mathbb{R}^d$ hits the unit ball centered at $y\in \mathbb{R}^d$ with probability of order $(|x-y|+1)^{2-d}$. As shown in \cite[Lemma 3.5]{kozma2007scaling}, by the coupling between the random walk on $\mathbf{G}$ and Brownian motion on $\mathbb{R}^d$, this property remains valid for the random walk: for any $r\ge 1$, $x\in \mathbb{R}^d$ and $y\in \mathbf{G}\cap [B(x,2r)]^c$, 
 \begin{equation}\label{ineq_2.15}
 	\widetilde{\mathbb{P}}_y\big( \widetilde{\tau}_{B(x,r)} <\infty \big) \asymp  \big( r/ |x-y| \big)^{d-2}. 
 \end{equation}
  Combining (\ref{ineq_2.15}) with the argument in \cite[Lemma 2.7]{chang2024percolation}, we have the following bound on the total mass of loops crossing an annulus: for $R>r\ge 1$ and $x\in \mathbb{R}^d$, 
 \begin{equation}\label{ineq_crossing_loop}
 	\mu\big( \{ \ell:  \mathrm{ran}(\ell) \cap \partial B(x,R)\neq \emptyset,    \mathrm{ran}(\ell) \cap B(x,r) \neq \emptyset \} \big) \asymp \big( r/ R \big)^{d-2}. 
 \end{equation}
 As a direct corollary of (\ref{ineq_crossing_loop}), one has 
 \begin{equation}\label{bound2.12}
 	\mu\big( \{ \ell:  \mathrm{ran}(\ell)  \subset  B(x,R),   |\ell|\ge   \tfrac{R}{10} \} \big) \asymp 1.   
 \end{equation}
 Here $|\ell|$ denotes the Euclidean diameter of the loop $\ell$.

  Let $\mathcal{L}$ and $\mathcal{L}^{\nu}$ denote the loop soups of intensity $1/2$ on $\widetilde{\mathbf{G}}$ and $\widetilde{\mathbf{G}}^{\nu}$ respectively. For each loop $\ell$ satisfying the conditions in (\ref{bound2.12}), the probability that the duration $T_{\ell}$ of $\ell$ lies in the interval $[tR^2, (t+1)R^2]$ is at most $e^{-ct}$ (using \cite[Proposition 2.4.5]{Lawler2010RandomWA}); in addition, the probability that $\ell$ is killed on $\widetilde{\mathbf{G}}^{\nu}$ is $O(\nu T_{\ell})$. Consequently, 
  \begin{equation}\label{bound26}
  	\begin{split}
  		& \mathbb{P}\big( \exists \ell \in \mathcal{L}-\mathcal{L}^{\nu}\ \text{such that}\ \mathrm{ran}(\ell)\subset B(x,R), |\ell| \ge \tfrac{R}{10}   \big) \\
  		\lesssim  & \sum\nolimits_{t\in \mathbb{N}} e^{-ct}\cdot \nu (t+1)R^2  \lesssim \nu R^{2}. 
  	\end{split}
  \end{equation}

 Next, we estimate the total mass of the loops that are contained in a large ball and come very close to its boundary. Note that the annulus $B(x,R)\setminus B(x,R-r)$ can be covered by $O((R/r)^{d-1})$ balls of radius $r$; we enumerate their centers by $\{y_j\}_{1\le j\le J}$. If a loop $\ell$ with diameter at least $M$ is contained in $B(x,R)$ but not in $B(x,R-r)$, then it must intersect $B(y_j,r)$ for some $1\le j\le J$. Such a loop contains two random walk trajectories: one starts from $B(y_j,r)$ and reaches $\partial B(y_j,M)$ before exiting $B(x,R)$, while the other starts from $\partial B(y_j,M)$ and hits $B(y_j,r)$ before exiting $B(x,R)$. Moreover, the probabilities of the two types of trajectories are $O(r/M)$ and $O((r/M)^{d-1})$ respectively. Consequently, 
 \begin{equation}\label{y3.17}
 	\begin{split}
 	& 	\mu\big(\big\{ \ell :  \mathrm{ran}(\ell)\subset B(x,R),\mathrm{ran}(\ell) \not\subset  B(x,R-r) ,|\ell|\ge M  \big\}\big) \\
 		\lesssim  &\big(\tfrac{R}{r}\big)^{d-1} \cdot \tfrac{r}{M}  \cdot \big(\tfrac{r}{M}\big)^{d-1}  = \tfrac{R^{d-1}r}{M^d}. 
 	\end{split}
 \end{equation}




  \textbf{Convergence of large loops.} As mentioned above, random walk loop soups can be approximated by Brownian loop soups. This approximation has been studied in detail for the lattice $\mathbb{Z}^d$; see \cite{lawler2007random, qian2026coupling, sapozhnikov2018brownian}. The following lemma shows that this approximation extends to Voronoi networks. For any $R,r>0$, we denote 
  \begin{equation}\label{q318}
  	\mathcal{L}_{R,r}:= \mathcal{L}\cdot \mathbbm{1}_{\mathrm{ran}(\ell)\subset B(R), |\ell|>r} \ \text{and}\ 	\mathcal{L}_{R,r}^{\mathbf{B}}:= \mathcal{L}^{\mathbf{B}}\cdot \mathbbm{1}_{\mathrm{ran}(\ell)\subset B(R), |\ell|>r}. 
  	  \end{equation}

%

  \begin{lemma}\label{lemma_coupling}
  	For any Voronoi network $\mathsf{G}$, there exist $\cl\label{const_coupling1},\cl\label{const_coupling2},\Cl\label{const_coupling3}>0$ with $\cref{const_coupling1}<2\cref{const_coupling2}$ such that for any $R>0$, there exists a coupling between $\mathcal{L}$ and the Brownian loop soup $\mathcal{L}^{\mathrm{B}}$ such that 
  	\begin{equation}
  		\mathbb{P}\big( \mathrm{d}_{\mathrm{H}}^*(\mathcal{L}_{R,[\mathbf{s}(\mathbf{G})]^{\cref{const_coupling1}}} , 	\mathcal{L}_{R,[\mathbf{s}(\mathbf{G})]^{\cref{const_coupling1}}}^{\mathbf{B}}) > [\mathbf{s}(\mathbf{G})]^{2\cref{const_coupling2}}\big)  \le \Cref{const_coupling3}(R+1)^{d} [\mathbf{s}(\mathbf{G})]^{ \cref{const_coupling2}}. 
  	\end{equation}

  \end{lemma}

 The proof of this lemma is provided in \cite{technicalpaper}. We emphasize that Lemma~\ref{lemma_coupling} is only meaningful when
$\mathbf{s}(\mathbf{G})\ll 1$; in our applications, it will be applied to
rescaled networks such as $\delta\cdot\mathbb{Z}^d$ with small $\delta>0$.

      {\color{red}



 }

 \textbf{Crossing paths.} Let $\mathbf{G}_1$ and $\mathbf{G}_2$ be two interpolating graphs that differ only within $x+[0,L)^d$. For $i\in \{1,2\}$, we denote by $\mathcal{L}_i$ the loop soup of intensity $1/2$ on $\widetilde{\mathbf{G}}_i$. Note that the loops in $\mathcal{L}_1$ avoiding $x+[0,L)^d$ have the same distribution as their counterpart for $\mathcal{L}_2$. Any loop among the remaining ones that intersects $\partial B(x,10dL)$ must cross the annulus $B(x,5dL)\setminus B(x,2dL)$. As introduced in \cite[Section 2.6.3]{cai2025one}, such a loop can be decomposed into forward and backward crossing paths. Each forward crossing path starts from $\hat{\partial}\hat{B}(x,2dL)$ and stops upon hitting $\hat{\partial}\hat{B}(x,5dL)$, whereas each backward crossing path evolves in the opposite direction. By the spatial Markov property of the loop soup, conditional on the starting and ending points of all these crossing paths, the forward and backward crossing paths are independent. Moreover, a crossing path starting from $y\in\hat{\partial}\hat{B}(x,r)$ and ending at $z\in\hat{\partial}\hat{B}(x,r')$ has distribution  
 \begin{equation}\label{distribution_crossing_path}
 	\widetilde{\mathbb{P}}_y\big( \{\widetilde{X}_t\}_{0\le t\le \widetilde{\tau}_{\hat{\partial}\hat{B}(x,r')}}  \in \cdot  \mid \widetilde{\tau}_{\hat{\partial}\hat{B}(x,r')}=\widetilde{\tau}_{z}<\infty \big). 
 \end{equation}
 Only the forward crossing paths can be affected by the graph modification within $x+[0,L)^d$. Moreover, using Lemma \ref{lemma2.1_isometric} and the arguments in \cite[Section 3.5]{kozma2007scaling}, one can show that the total mass of these forward crossing paths (i.e., the probability of $\{\widetilde{\tau}_{\hat{\partial}\hat{B}(x,r')}=\widetilde{\tau}_{z}<\infty\}$) changes by at most a multiplicative factor of $1+O(L^{-c})$. The same approximation remains valid when the graph has a killing rate $\nu \le  L^{-2}$. Furthermore, one can adapt the approach of \cite[Section 3.4]{kozma2007scaling} to couple the forward crossing paths in $\mathcal{L}_1$ and $\mathcal{L}_2$ such that for some constants $c',c''\in (0,1)$, with probability $1-O(e^{-L^{c'}})$ the Hausdorff distance between corresponding paths is at most $L^{c''}$. Consequently, we have the following lemma.  
  \begin{lemma}\label{lemma_coupling_crossingloop}
   There exist $\Cl\label{const_crossingloop_big}>1$ and $\cl\label{const_crossingloop},\cl\label{const_crossingloop_new}\in (0,1)$ and a coupling between $\mathcal{L}_1$ and $\mathcal{L}_2$ such that with probability at least $1-\Cref{const_crossingloop_big}L^{-\cref{const_crossingloop}}$, the following events happen: 
   \begin{enumerate}

   	\item There is a bijection between the loops in $\mathcal{L}_1$ and $\mathcal{L}_2$ that intersect both $B(x,L)$ and $\partial B(x,10dL)$ such that corresponding loops coincide outside $B(x,5dL)$ and the Hausdorff distance between their remaining parts is at most $L^{\cref{const_crossingloop_new}}$. In particular, $(\cup  \mathcal{L}_1)\cap [ B(x,5dL)]^c  = (\cup  \mathcal{L}_2)\cap [ B(x,5dL)]^c$.

   	\item  After imposing killing at rate $\nu \in [0,L^{-2}]$ on the graphs $\mathbf{G}_1$ and $\mathbf{G}_2$, each pair of corresponding crossing loops in Item (1) is either both killed or both retained.


   \end{enumerate} 
   
 \end{lemma}

 A detailed proof of Lemma \ref{lemma_coupling_crossingloop} can be found in \cite{technicalpaper}.

\subsection{Relation to Gaussian free fields.}
For any $D\subset \widetilde{\mathbf{G}}$, the Green's function on $\widetilde{\mathbf{G}}$ for the set $D$ is defined by 
  \begin{equation}
 	G_D(v,w):=  \widetilde{\mathbb{E}}_v\Big[\int_{0\le t \le  \widetilde{\tau}_D\land \widetilde{\zeta}} 
    \mathbf{1}_{\{\widetilde{X}_t=w\}}\,\mathrm{d}t\Big], \ \forall v,w\in \widetilde{\mathbf{G}}.
 \end{equation}
 In particular, when $D=\emptyset$, we write $G(\cdot,\cdot):=G_{\emptyset}(\cdot,\cdot)$.

The GFF on $\widetilde{\mathbf{G}}$, denoted by $\{\phi_v\}_{v\in \widetilde{\mathbf{G}}}$, is a family of mean-zero Gaussian random variables with covariance given by $G(\cdot, \cdot)$, i.e., 
\begin{equation}
	\mathbb{E}\big[\phi_v \phi_w \big]= G(v,w), \ \forall v,w\in \widetilde{\mathbf{G}}. 
\end{equation}
Conditioned on $\cap_{w\in D}\{\phi_w=0\}$, the field $\{\phi_v\}_{v\in \widetilde{\mathbf{G}}\setminus D}$ is still a mean-zero Gaussian field, whose covariance is given by $G_D(\cdot, \cdot)$. We denote its law by $\mathbb{P}^D$ (we may omit the superscript when $D=\emptyset$). It was shown in \cite[Proposition 5.2]{lupu2016loop} that  
	 \begin{equation}\label{lupu_two_point}
    	\mathbb{P}^D\big(v \xleftrightarrow{\ge 0 }  w \big) = \frac{1}{\pi} \arcsin  \Big(\frac{G_D(v,w)}{\sqrt{G_D(v,v)G_D(w,w)}}   \Big) ,\ \forall v,w \in \widetilde{\mathbf{G}}\setminus D.
    \end{equation}

     \textbf{Isomorphism theorem.} For $v\in \widetilde{\mathbf{G}}\setminus D$, let $\mathbf{L}^D_v$ denote the total local time at $v$ of the loops in $\mathcal{L}^D$. According to \cite[Proposition 2.1]{lupu2016loop}, there exists a coupling between $\mathcal{L}^D$ and $\{\phi_v\}_{v\in \widetilde{\mathbf{G}}\setminus D}\sim \mathbb{P}^D$ such that 
     \begin{itemize}

       \item   $\mathbf{L}^D_v=\frac{1}{2}\phi_v^2$ for all $v\in \widetilde{\mathbf{G}}\setminus D$, which implies that the clusters of $\mathcal{L}^D$ are exactly the sign clusters of $\phi$;

     	\item  Conditioned on the collection of sign clusters, the signs assigned to distinct clusters are independent and uniformly distributed on $\{+,-\}$.

     \end{itemize}

 \textbf{Lupu--Werner formula.} We next review a formula from \cite{lupu2018random} that has played an important role in recent developments on this topic. We first recall the definition of the boundary excursion kernel. For two subsets $D_1,D_2$ of a metric graph $\widetilde{\mathbf{G}}$, their graph distance $\widetilde{\mathrm{d}}(D_1,D_2) $ is defined as the infimum of the lengths of all paths connecting $D_1$ to $D_2$. For any $v,w\in\widetilde{\mathbf{G}}$, we abbreviate $\|v-w\|:=\widetilde{\mathrm{d}}(\{v\},\{w\})$. For a compact set $D\subset \widetilde{\mathbf{G}}$, we denote its boundary by $\widetilde{\partial} D:=\{ x\in D: \widetilde{\mathrm{d}}(\{x\}, D^c)=0\}$. The boundary excursion kernel for $D$ is defined as 
  \begin{equation}\label{def_K}
  	\mathbb{K}_D(v,w):= \lim\limits_{\epsilon \downarrow 0} (2\epsilon)^{-1} \sum_{v'\in \widetilde{\mathbf{G}}: \|v-v'\|=\epsilon  } \widetilde{\mathbb{P}}_{v'}\big( \tau_{D}= \tau_{w}< \infty \big), \ \forall v,w\in  \widetilde{\partial} D. 
  \end{equation}
%
%
%
%
%
%
%
%
%
  For any $D_1,D_2\subset \widetilde{\mathbf{G}}$, we denote by $D_1\xleftrightarrow{\ge 0} D_2$ the event that there exists a path on $\widetilde{\mathbf{G}}$ connecting $D_1$ and $D_2$ along which $\phi$ is non-negative. \cite[Equation (18)]{lupu2018random} shows that given all GFF values on $D_1\cup D_2$, if these values are all non-negative, then the conditional probability of $\{D_1\xleftrightarrow{\ge 0} D_2\}$ equals
  \begin{equation}\label{formula_LW}
  	1- e^{-2\sum_{v_1\in \widetilde{\partial}D_1, v_2\in \widetilde{\partial}D_2 } \mathbb{K}_{D_1\cup D_2}(v_1,v_2)\phi_{v_1} \phi_{v_2} }. 
  \end{equation} 
  This gives an explicit formula for the probability that two sets carrying GFF boundary values of the same sign belong to the same sign cluster. However, in some applications, only the boundary values of the occupation field (i.e., the absolute values of the GFF on the boundary) are given. In this setting, the formula (\ref{formula_LW}) does not directly apply to the corresponding connecting probability. As shown in the following lemma, there is an explicit relation between these two types of conditional connecting probabilities, which provides a way to address this issue. A similar derivation already appears in the proof of \cite[Theorem 2]{werner2025switching}.

 \begin{lemma}\label{newlemma2.3}
 	 For any disjoint $D_1,D_2\subset \widetilde{\mathbf{G}}$, given the values of the occupation field on $D_1\cup D_2$, if for each $j\in \{1,2\}$, all boundary points of $D_j$ with positive values belong to the same sign cluster $\mathcal{C}_j\subset D_j$, then 
  \begin{equation}
 	   \mathbb{P}\big( D_1\xleftrightarrow{} D_2 \mid  \{\mathbf{L}_v \}_{v\in D_1\cup D_2}  \big) = \frac{q}{2-q  },
 	\end{equation}
 	where $q:=\mathbb{P}\big(D_1\xleftrightarrow{} D_2    \mid  \{\mathbf{L}_v \}_{v\in D_1\cup D_2} , \mathcal{C}_1\ \text{and}\ \mathcal{C}_2\ \text{both have sign} +  \big)$. 
 \end{lemma}
 \begin{proof}
  We define $\mathsf{S}$ as the event that $\mathcal{C}_1$ and $\mathcal{C}_2$ have the same sign. Let $\mathcal{F}$ denote the $\sigma$-field generated by $\{\mathbf{L}_v \}_{v\in D_1\cup D_2}$. Since $\{D_1\xleftrightarrow{} D_2\}\subset \mathsf{S}$, we have 
 	\begin{equation}\label{new214}
 		\begin{split}
 			   \mathbb{P}\big( \mathsf{S}  \mid \mathcal{F} \big) 
 			   = 	&	   \mathbb{P}\big( D_1\xleftrightarrow{} D_2 \mid \mathcal{F} \big)  + 		   \mathbb{P}\big(\{ D_1\xleftrightarrow{} D_2\}^c, \mathsf{S}   \mid  \mathcal{F} \big)  \\
 			   =& \mathbb{P}\big( D_1\xleftrightarrow{} D_2 \mid \mathcal{F} \big)  + \tfrac{1}{2}  \mathbb{P}\big( \{ D_1\xleftrightarrow{} D_2\}^c    \mid  \mathcal{F} \big)\\
 			   =& \tfrac{1}{2} \big[ \mathbb{P}\big( D_1\xleftrightarrow{} D_2 \mid \mathcal{F} \big) +1 \big],
 		\end{split}
 	\end{equation}
where the second identity follows from the fact that the sign of each sign cluster is chosen uniformly from $\{+,-\}$. Meanwhile, we also have
\begin{equation}\label{new215}
	\begin{split}
		 \mathbb{P}\big( D_1\xleftrightarrow{} D_2 \mid \mathcal{F} \big) =  \mathbb{P}\big( D_1\xleftrightarrow{} D_2 \mid \mathcal{F}, \mathsf{S} \big) \cdot  \mathbb{P}\big( \mathsf{S} \mid \mathcal{F} \big) =q\cdot \mathbb{P}\big( \mathsf{S} \mid \mathcal{F} \big),
	\end{split}
\end{equation}
where in the last identity we used the symmetry of the GFF. Combining (\ref{new214}) and (\ref{new215}), we complete the proof.  
 \end{proof}


 \textbf{Switching identity.} The following powerful result of \cite{werner2025switching} shows that conditioned on two given points being connected by a loop cluster, the entire loop soup admits an explicit description. For any $D\subset \widetilde{\mathbf{G}}$ and $v,w\in \widetilde{\mathbf{G}}\setminus D$, let $\mathbf{e}_{v,w}^{D}$ denote the Brownian excursion measure from $v$ to $w$ outside $D$, supported on the space of paths in $\widetilde{\mathbf{G}}$ from $v$ to $w$ that avoid $D\cup \{v,w\}$ except at their endpoints. Its precise definition can be found in \cite[Section 2.4]{cai2025separation}.


 \begin{lemma}[{\cite[Theorem 2]{werner2025switching}}]\label{lemma_switching}
 For any metric graph $\widetilde{\mathbf{G}}$, distinct points $v,w\in \widetilde{\mathbf{G}}$, and $a,b>0$, conditioned on $\{v\xleftrightarrow{} w, \mathbf{L}_v=a,\mathbf{L}_w=b \}$, the occupation field $\{\mathbf{L}_z\}_{z\in \widetilde{\mathbf{G}}}$ has the same distribution as the total local time of the following four
independent components:
 
 \begin{enumerate}

 	\item loops in the loop soup $\mathcal{L}^{\{v,w\}}$;

 	\item a Poisson point process with intensity measure $a\cdot \mathbf{e}^{\{w\}}_{v,v}$;

 	\item a Poisson point process with intensity measure $b \cdot \mathbf{e}^{\{v\}}_{w,w}$;

 	\item a Poisson point process with intensity measure $\sqrt{ab}\cdot \mathbf{e} _{v,w}$, where the number of excursions is conditioned to be odd.

 \end{enumerate}

 \end{lemma}

 \subsection{Clusters near the boundary}

  The following estimate is useful for controlling the error caused by graph modifications near the boundary of $\mathbb{B}$.
    \begin{lemma}\label{lemma_cluster_closetoboundary}
 	For any $3\le d\le 5$, $0<a<b<1$, $\epsilon>0$, the clusters $\mathfrak{C}$ of the loop soup on $\mathbb{Z}^d$ satisfy that for all sufficiently small $\delta>0$, 
 	 \begin{equation}\label{350}
 	 	\mathbb{P}\big( \exists \mathcal{C}\in \mathfrak{C}\ \text{such that}\   \delta\cdot \mathcal{C}\subset B(1+ \delta^b),\ \delta\cdot \mathcal{C}\not\subset B(1-\delta^b)\ \text{and} \   | \mathcal{C}|\ge \epsilon \delta^{-1}  \big) \le  \delta^{a}. 
 	 \end{equation}
 \end{lemma}
 \begin{proof}
  Let $\mathfrak{C}_{\epsilon}$ denote the collection of clusters in $\mathfrak{C}$ that are contained in $B(2\delta^{-1})$ and have diameters at least $\epsilon \delta^{-1}$. For each $z\in \mathbb{Z}^d$, we define the event 
  \begin{equation*}
  	 \mathsf{A}_z:= \big\{ \exists \mathcal{C}\in \mathfrak{C}\ \text{such that}\    \mathcal{C}\subset B(z,\delta^{-1}+\delta^{b-1}),\  \mathcal{C}\not\subset B(z,\delta^{-1}-\delta^{b-1})\ \text{and} \   | \mathcal{C}|\ge \epsilon \delta^{-1}  \big\}. 
  \end{equation*}
   Note that $\mathsf{A}_{\bm{0}}$ is the event appearing in (\ref{350}). By the translation invariance of $\mathbb{Z}^d$,   
   \begin{equation}\label{352}
   	 \mathbb{P}( \mathsf{A}_{\bm{0}}) = \big| \hat{B}(\delta^{-1}) \big|^{-1}  \mathbb{E}\Big[  \sum\nolimits_{z\in \hat{B}(\delta^{-1}) } \mathbbm{1}_{\mathsf{A}_z}\Big]. 
   \end{equation}
   Since a cluster $\mathcal{C}$ certifying the event $\mathsf{A}_z$ for some $z\in \hat{B}(\delta^{-1})$ must lie in $\mathfrak{C}_{\epsilon}$,   
 \begin{equation}\label{353}
  	 \sum\nolimits_{z\in \hat{B}(\delta^{-1}) } \mathbbm{1}_{\mathsf{A}_z} \le \sum\nolimits_{\mathcal{C}\in \mathfrak{C}_{\epsilon}}   \#Q(\mathcal{C}),  
  \end{equation}
  where $Q(\mathcal{C}):=\{ z\in \mathbb{Z}^d: \mathcal{C}\subset B(z,\delta^{-1}+\delta^{b-1}), \mathcal{C}\not\subset B(z,\delta^{-1}-\delta^{b-1}) \}$. We claim   
   \begin{equation}\label{354}
      \#Q(\mathcal{C})\lesssim \delta^{-d+b}    , \ \forall \mathcal{C}\in \mathfrak{C}_{\epsilon}. 
   \end{equation}
  Combining (\ref{352}), (\ref{353}) and (\ref{354}), we obtain that for all sufficiently small $\delta>0$, 
  \begin{equation}\label{y333}
  	 \begin{split}
  	 	 \mathbb{P}( \mathsf{A}_{\bm{0}}) \le C \delta^{b}   \mathbb{E}\big[ \#\mathfrak{C}_{\epsilon}   \big]. 
  	 \end{split}
  \end{equation}
  Next, we estimate $\mathbb{E} [  \#\mathfrak{C}_{\epsilon}  ]$. By \cite[Theorem 1.2]{cai2024one}, there exists a sufficiently large constant $C_\dagger>0$ such that 
  	  \begin{equation}\label{y334}
  	  	\mathbb{P}\big(  B(x,n)\xleftrightarrow{} \partial B(x,C_\dagger n ) \big)\le \tfrac{1}{2}, \ \forall x\in \mathbb{Z}^d\ \text{and}\ n\ge 1.
  	   \end{equation}
  Let $R:=(4C_\dagger)^{-1}\epsilon \delta^{-1}$ and define  
  	  \begin{equation}
  	  	\mathbf{Y}:=\{ y\in R \cdot \mathbb{Z}^d: B(y, R)\cap B(2\delta^{-1})\neq \emptyset \}.
  	  \end{equation} 
  Note that $\# \mathbf{Y}\asymp \epsilon^{-d}$. For each $y\in \mathbf{Y}$, let $\mathcal{N}_y$ denote the number of clusters that intersect $B(y,R)$ and have diameter at least $\frac{1}{2}\epsilon \delta^{-1} =2C_\dagger R$. In fact, by (\ref{y334}) and the BKR inequality, $\mathcal{N}_y$ admits an exponential tail: for any $j\in \mathbb{N}^+$, 
  \begin{equation}\label{a336}
  	\mathbb{P}\big(  \mathcal{N}_y\ge j\big)\le \big[ \mathbb{P}\big( B(y,R)\xleftrightarrow{} \partial B(y, C_\dagger R  ) \big) \big]^j\le  2^{-j}.
  \end{equation}
  As a result, we have $\mathbb{E}[\mathcal{N}_y]\lesssim 1$ and thus, 
  \begin{equation}\label{newineq359}
  	\mathbb{E}\big[\#\mathfrak{C}_{\epsilon}   \big]  \le \sum\nolimits_{y\in \mathbf{Y}} \mathbb{E}\big[\mathcal{N}_y \big] \lesssim  \epsilon^{-d}. 
  \end{equation} 
  Combined with (\ref{y333}), it yields the desired bound (\ref{350}).


   The only remaining step is to prove the bound (\ref{354}). For any $R\ge 1$, let 
  \begin{equation}\label{def_UR}
  	U(R):=\{ x \in \mathbb{R}^d: \mathcal{C}\subset B(x,R) \}. 
  \end{equation}
 We claim that $U(R)$ satisfies the following two properties: 
 \begin{enumerate}

 	\item $U(R)= \cap_{y\in \mathcal{C}}B(y,R)$;

 	\item $ B(z,1)\subset U(\delta^{-1}+\delta^{b-1}+1)\setminus U(\delta^{-1}- \delta^{b-1}-1)$ for all $z\in Q(\mathcal{C})$.

 \end{enumerate}
 It follows from Property (1) that $U(R)$ is convex. In addition, Property (2) implies  
 \begin{equation}\label{newadd_368}
 	\begin{split}
 		& \mathrm{vol}_{d}\big( \cup_{z\in Q(\mathcal{C})}B(z,1)  \big)\\
 		\le &\mathrm{vol}_{d}\big(   U(\delta^{-1}+\delta^{b-1}+1)\setminus U(\delta^{-1}- \delta^{b-1}-1) \big) \\
 		= & \int_{\delta^{-1}- \delta^{b-1}-1\le t\le \delta^{-1}+\delta^{b-1}+1}  \mathrm{vol}_{d-1}\big( \partial U(t)  \big) \mathrm{d}t.  
 	\end{split}
 \end{equation} 
  Since the area of the boundary of a convex set is increasing with respect to set inclusion (by Cauchy's surface area formula; see e.g., \cite[Theorem 5.5.2]{klain1997introduction}), one has
   \begin{equation}\label{newadd_369}
   	\mathrm{vol}_{d-1}\big( \partial U(t)  \big) \le \mathrm{vol}_{d-1}\big( \partial B(4\delta^{-1})\big) \lesssim \delta^{-d+1}    \end{equation}
   for all $\delta^{-1}- \delta^{b-1}-1\le t\le \delta^{-1}+\delta^{b-1}+1$. By (\ref{newadd_368}) and (\ref{newadd_369}), we obtain (\ref{354}): 
   \begin{equation}\label{use334}
   	\# Q(\mathcal{C})\lesssim  \mathrm{vol}_{d}(\cup_{z\in Q(\mathcal{C})}B(z,1)) \lesssim  \delta^{-d+b}. 
   \end{equation}

    It remains to prove Properties (1) and (2). For Property (1), by the definition of $U(R)$ in (\ref{def_UR}), a point $x\in \mathbb{R}^d$ belongs to $U(R)$ if and only if $|x-y|<R$ for all $y\in \mathcal{C}$. Since the latter condition is also equivalent to $x\in \cap_{y\in \mathcal{C}}B(y,R)$, we obtain Property (1). For Property (2), for any $z\in Q(\mathcal{C})$, by definition one has $\mathcal{C}\subset B(z,\delta^{-1}+\delta^{b-1})$ and $\mathcal{C}\not\subset B(z,\delta^{-1}-\delta^{b-1}) $, which together with the triangle inequality implies that $\mathcal{C}\subset B(x,\delta^{-1}+\delta^{b-1}+1)$ and $\mathcal{C}\not\subset B(x,\delta^{-1}-\delta^{b-1}-1)$ hold for all $x\in B(z,1)$. This proves Property (2) and thus completes the proof.  
 \end{proof}

To prove the scaling limit part in Theorem \ref{new_thm1}, we need to extend Lemma \ref{lemma_cluster_closetoboundary} from $\mathbb{B}$ to a general bounded open set $D\subset\mathbb{R}^d$ with $C^0$ boundary. We first present the precise definition of a $C^0$ boundary. Roughly speaking, near each boundary point, the boundary is the graph of a continuous function in suitable coordinates, and $D$ lies on one side of this graph.   
   
  \begin{definition}
  	 We say that an open set $D\subset\mathbb{R}^d$ has $C^0$ boundary if for any $x\in \partial D$, there exist a $d\times d$ orthogonal matrix $Q_x$, an open set $O_x\subset \mathbb{R}^{d-1}$ with $\bm{0}\in O_x$, a constant $b_x>0$, and a continuous function $g_x:O_x \to (-b_x, b_x)$ with $g_x(\bm{0})=0$, such that letting $\Psi_x(y):=Q_x(y-x)$ and $V_x:=\Psi_x^{-1}(O_x\times (-b_x,b_x))$, we have 
  	   	 \begin{equation}
  	 	\Psi_x(D\cap V_x ) =  \{ (u,v)\in O_x \times (-b_x,b_x): v< g_x(u) \}.   
  	 \end{equation}
  	  
  
  \end{definition}


For any $D\subset \mathbb{R}^d$ and $h>0$, define 
\begin{equation}\label{a343}
	D^h:=\{x\in \mathbb{R}^d: \mathrm{dist}(x,D)<h \} \ \text{and}\ D^{-h}:=\{x\in D: \mathrm{dist}(x,D^c)>h\}.
\end{equation}
  The event of interest is defined by 
\begin{equation}
	\mathsf{E}_\delta(D;\epsilon,h):= \big\{ \exists \mathcal{C}\in \mathfrak{C}\ \text{such that}\ \delta\cdot  |\mathcal{C}| \ge \epsilon, \delta\cdot  \mathcal{C} \subset D^h, \delta\cdot  \mathcal{C} \not\subset D^{-h} \big\}. 
\end{equation}

 \begin{lemma}\label{lemma_geenral_D}
 	 For any $3\le d\le 5$, $\epsilon>0$, and any $D\subset \mathbb{R}^d$ with $C^0$ boundary, 
 	 \begin{equation}
 	 	\lim_{t\downarrow 0} \sup_{ \delta, h\in (0,t]} \mathbb{P}\big( \mathsf{E}_\delta(D;\epsilon,h) \big) =0. 
 	 \end{equation} 
 \end{lemma}

 The proof of Lemma \ref{lemma_geenral_D} is similar to that of Lemma \ref{lemma_cluster_closetoboundary}. The only difference is that (\ref{use334}) needs to be replaced by a more general argument, which is purely analytic and does not directly concern the properties of the objects studied in this paper. We therefore defer the proof to \cite{newonearm}.

  \section{Connecting probabilities on interpolating graphs}\label{section_connecting_prob}


  In this section, we present several estimates for connecting probabilities. Among them, the following estimate, which gives the exact order of the one-arm probability for Voronoi networks, will play an important role. Throughout this section, we normalize each Voronoi network $\mathbf{G}$ so that $\mathbf{s}(\mathbf{G})=1$; the general case follows by rescaling. When we write $A \xleftrightarrow{ }  B$ below, it means $A$ is connected to $B$ via a loop cluster in the loop soups of intensity $1/2$.


\begin{lemma}\label{lemma_new_one_arm}
For any $3\le d\le 5$, $N\ge 1$, any Voronoi network $\mathbf{G}^\nu$ on $\mathbb{R}^d$ with killing rate $\nu \in [0,N^{-2}]$, and any $x\in \mathbf{V}$,    
      \begin{equation}\label{ineq_new_one_arm}
     	\mathbb{P}\big( x \xleftrightarrow{ } \partial B(x,N) \big) \asymp N^{-\frac{d}{2}+1}.
     \end{equation} 
      \end{lemma}



  The lower bound for $\mathbb{P}(x \xleftrightarrow{} \partial B(x,N))$ in (\ref{ineq_new_one_arm}) follows directly from the arguments in \cite{ding2020percolation}. The upper bound can be obtained by adapting the proof of \cite[Theorem 1.1]{cai2024one} (for $d=3$, this is known thanks to \cite{drewitz2025critical}). This adaptation is not straightforward, as the proof in \cite{cai2024one} uses translation invariance, which does not hold for Voronoi networks in general. See \cite{newonearm} for a detailed proof of Lemma \ref{lemma_new_one_arm}. However, the arguments in \cite{cai2024quasi} do not use translation invariance, and rely only on the sharp estimate for the one-arm probability together with standard properties of random walks. Thus, by Lemma \ref{lemma_new_one_arm}, all the results in \cite{cai2024quasi} remain valid for the interpolating graphs considered here. In the following lemma, we record several properties that will be used later. We assume that $3\le d\le 5$ and that $\mathbf{G}^{\nu}$ is a Voronoi network with killing rate $\nu\in [0,N^{-2}]$. Recall that $A \xleftrightarrow{(D)}  B$ means connection using loops disjoint from $D$. We write $\{A_0\xleftrightarrow{(D)} A_1\xleftrightarrow{(D)} ...\xleftrightarrow{(D)} A_k\}$ for the event $\cap_{0\le i<j\le k}\{A_i\xleftrightarrow{(D)} A_j\}$, and $\{A_0\xleftrightarrow{(D)} A_1,...,A_k\}$ for $\cap_{1\le i\le k}\{A_0\xleftrightarrow{(D)} A_i\}$.

   \begin{lemma}\label{lemma_regularity}
   	 For any $x\in \mathbb{R}^d$, $n\ge 1$ and $N\ge 10d^2n$, the following hold: 
   	 \begin{itemize} 
   	 
   	 	\item (\cite[Theorem 1.1]{cai2024quasi}) For any $A_1,D_1\subset B(x,n)$ and $A_2,D_2\subset [B(x,N)]^c$, 
   	 	 \begin{equation}\label{ineq_quasi}
   	 	 	\mathbb{P}\big( A_1 \xleftrightarrow{(D_1\cup D_2)} A_2   \big) \asymp 	\mathbb{P}\big( A_1 \xleftrightarrow{(D_1)} \partial B(x,\tfrac{N}{d})   \big)\cdot 	\mathbb{P}\big( A_2 \xleftrightarrow{( D_2)} \partial B(x,\tfrac{N}{d})   \big). 
   	 	 \end{equation}

   	 	\item   (\cite[Proposition 1.9]{cai2024quasi}) For any $A,D\subset B(x,n)$ and  $N_1,N_2\ge 10dn$, 
       	 \begin{equation}\label{ineq_crossing_decay}
 	 N_1^{\frac{d}{2}-1} 	\mathbb{P}\big( A \xleftrightarrow{(D)} \partial B(x,N_1 ) \big)\asymp   N_2^{\frac{d}{2}-1} 	\mathbb{P}\big( A \xleftrightarrow{(D)} \partial B(x,N_2 ) \big). 
 	 \end{equation}

   	 	\item  (\cite[(5.2)]{cai2024quasi}) Denote $r_1:=n$, $r_{-1}:=N$, $\mathbf{B}_1:=B(x,n)$ and $\mathbf{B}_{-1}:=[B(x,N)]^c$. Then for $D\subset \widetilde{\mathbf{G}}$, $i\in \{1,-1\}$ and $A\subset \mathbf{B}_{-i}$, 
	\begin{equation}\label{ineq_bound_5.2}
		\mathbb{P}\big( A  \xleftrightarrow{(D)} \mathbf{B}_{i}  \big) \lesssim   r_{i}^{-\frac{d}{2}} \sum\nolimits_{y\in \hat{\partial} \hat{B}(x,d^{i}r_{i})}  \mathbb{P}\big( A  \xleftrightarrow{(D)} y  \big). 
	\end{equation}

   	 	  \item  (\cite[Lemmas 5.1 and 5.2]{cai2024quasi}) For $k \in \mathbb{N}^+$, let $\Upsilon_k$ denote the collection of all partitions of $\{1,2,...,k\}$, where a partition is a collection of disjoint subsets of $\{1,2,...,k\}$ whose union is $\{1,2,...,k\}$. There exist $C(d),C'(d,k)>0$ such that for any $M\ge CN$, $A_0, D\subset B(x,\frac{N}{10C})$, and $A_1,...,A_k,D'\subset [B(x,10CM)]^c$, the following inequalities hold:
   	   \begin{equation}\label{a4.5}
  	\begin{split}
  		 	& \mathbb{P}\big(A_0 \xleftrightarrow{(D\cup D')} A_1,...,A_k \big)\le C'  \mathbb{P}\big( A_0 \xleftrightarrow{(D)} \partial B(x,\tfrac{N}{C})  \big)     \\
   	   	  &\ \ \ \ \ \ \ \ \    \cdot \sum\nolimits_{ u_k \in \Upsilon_k } \prod\nolimits_{\{j_1,...,j_m\}\in u_k} \mathbb{P}\big( A_{j_1} \xleftrightarrow{(D')} ...  \xleftrightarrow{(D')}  A_{j_m}  \xleftrightarrow{(D')}  \partial B(x, CN) \big); 
  	\end{split}
  \end{equation}
   	  \begin{equation}\label{a4.6}
   	   \begin{split}
   	   	& \mathbb{P}\big( A_0 \xleftrightarrow{(D\cup D')} A_1,...,A_k , \mathcal{L}_x^{D\cup D'}[N,M]\neq 0 \big) \\
   	   	\le & C' \big( \tfrac{N}{M} \big)^{d-2} \mathbb{P}\big( A_0 \xleftrightarrow{(D)} \partial B(x, \tfrac{N}{C})  \big)  \\
   	   	& \cdot \sum\nolimits_{ u_k \in \Upsilon_k } \prod\nolimits_{\{j_1,...,j_m\}\in u_k} \mathbb{P}\big( A_{j_1} \xleftrightarrow{(D')} ...  \xleftrightarrow{(D')}  A_{j_m}  \xleftrightarrow{(D')}  \partial B(x, CM)    \big), 
   	   \end{split}
   	  \end{equation}
   	 where we denote $\mathcal{L}_x^{D\cup D'}[N,M]:=\mathcal{L}^{D\cup D'}\cdot \mathbbm{1}_{ \mathrm{ran}(\ell) \cap \partial B(x,N)\neq\emptyset , \mathrm{ran}(\ell) \cap \partial B(x,M)\neq\emptyset }$.


   	 	\item  (\cite[Lemma 5.3]{cai2024quasi}) For any $i\in \{1,-1\}$, $A,D\subset  \mathbf{B}_i$ and $y,z\in \partial  \mathbf{B}_{-i}$,
   	 	   	 	 \begin{equation}\label{ineq_multiset}
   	 	   	 	 	\mathbb{P}\big( A\xleftrightarrow{(D)} y, A\xleftrightarrow{(D)} z  \big) \lesssim  (|y-z|+1)^{-\frac{d}{2}+1}\mathbb{P}\big( A\xleftrightarrow{(D)} y  \big). 
   	 	   	 	 \end{equation}


   	 \end{itemize}

   \end{lemma}

 	As a direct corollary of (\ref{ineq_bound_5.2}), we have 
	\begin{equation}\label{newineq3.3}
		\mathbb{P}\big(B(x,n)  \xleftrightarrow{(D)}  \partial B(x,N) \big)\lesssim  (nN)^{-\frac{d}{2}} \sum\nolimits_{y\in \hat{\partial} \hat{B}(x,dn),z\in \hat{\partial} \hat{B}(x,d^{-1}N)}  \mathbb{P}\big( y  \xleftrightarrow{(D)} z  \big). 
	\end{equation} 
	For $D\subset B(x, \frac{N}{2d})\setminus B(x, 2dn)$, it follows from (\ref{lupu_two_point}) that for any $y\in \hat{\partial} \hat{B}(x,dn)$ and $z\in \hat{\partial} \hat{B}(x,d^{-1}N)$
   \begin{equation}\label{newineq3.8}
   	\begin{split}
   		\mathbb{P}\big(y  \xleftrightarrow{(D)} z \big)\lesssim  N^{2-d}\cdot    \widetilde{\mathbb{P}}_y\big( \widetilde{\tau}_{\partial B(x,N )} < \widetilde{\tau}_{D} \big). 
   	\end{split}
   \end{equation}
  Combining (\ref{newineq3.3}) and (\ref{newineq3.8}), we have 
   \begin{equation}\label{bound_crossing_by_hitting}
   	\mathbb{P}\big( B(x,n)  \xleftrightarrow{(D)}  \partial B(x,N)  \big)\lesssim  (n/N)^{\frac{d}{2}-1} \cdot  \widetilde{\mathbb{P}}_y\big( \widetilde{\tau}_{\partial B(x,N )} < \widetilde{\tau}_{D} \big). 
   \end{equation}

By taking $D=\emptyset$ in (\ref{newineq3.3}) and then applying (\ref{lupu_two_point}), we have
    \begin{equation}
   	 \mathbb{P}\big( B(x,n)\xleftrightarrow{} \partial B(x,N) \big) \lesssim  (n/N)^{\frac{d}{2}-1}.  
   \end{equation}
   Meanwhile, the reverse inequality follows from the second-moment method (see \cite[Section 5.3]{cai2024quasi}, where (\ref{ineq_multiset}) is used to bound the second moment). Combining the two bounds, we obtain
   \begin{equation}\label{crossing_prob}
   		 \mathbb{P}\big( B(x,n)\xleftrightarrow{} \partial B(x,N) \big) \asymp   (n/N)^{\frac{d}{2}-1}.   
   \end{equation}

\subsection{Thickness of loop clusters}
 
 As in Lemma \ref{lemma_regularity}, we consider the loop soup $\mathcal{L}^{\nu}$ of intensity $1/2$ on a Voronoi network $\mathbf{G}^{\nu}$. We assume $N\ge 1$ and $\nu\in [0,N^{-2}]$. The main aim of this subsection is to show that with high probability, a cluster $\mathcal{C}$ in $\mathcal{L}^{\nu}$ crossing an annulus is sufficiently dense that a random walk crossing the same annulus has only a polynomially small probability of avoiding $\mathcal{C}$.


      We first present the following analogue of \cite[Lemma 4.11]{cai2025heterochromatic}. For a compact $D\subset \widetilde{\mathbf{G}}$, its capacity on $\widetilde{\mathbf{G}}$ is given by 
  \begin{equation}
  	\mathrm{cap}(D):= \sum\nolimits_{v\in \widetilde{\partial}D} \mathbb{Q}_D(v),
  \end{equation}
  where $\mathbb{Q}_D(v):= \lim_{N\to \infty}  \sum_{ w\in \hat{\partial} \hat{B}(N) } \mathbb{K}_{D\cup  \hat{\partial} \hat{B}(N) }(v,w)$. It is well-known that 
   \begin{equation}\label{order_cap_ball}
   	 	\mathrm{cap}\big( B(x,R)\cap \widetilde{\mathbf{G}}  \big) \asymp R^{d-2}, \ \forall x\in \mathbb{R}^d\ \text{and} \  R\ge 1.  
   \end{equation}

  For any $y\in \widetilde{\mathbf{G}}$, we denote by $\mathcal{C}_y^{\nu,D}$ the cluster in $\mathcal{L}^{\nu}\cdot \mathbbm{1}_{\mathrm{ran}(\mathrm{\ell})\cap D=\emptyset}$ containing $y$. We also abbreviate $\mathcal{C}_y^{\nu}:=\mathcal{C}_y^{\nu,\emptyset}$. With a slight abuse of notation, for any $A\subset \widetilde{\mathbf{G}}$, we write $\#A:=\#(A\cap \mathbf{V})$. As shown in \cite[(4.114)]{cai2025heterochromatic}, it follows from \cite[Corollary 1.3]{drewitz2023critical} and Lemma \ref{lemma_new_one_arm} that there exist $C,c>0$ such that for any $N\ge 1$ and $y\in \mathbf{G}^{\nu}$ with $\nu\in [0,N^{-2}]$, 
  \begin{equation}\label{good4.15}
  	 \mathbb{P}\big( \mathrm{cap}( \mathcal{C}_y^{\nu})\ge cN^{d-2}, \{ y\xleftrightarrow{} \partial B(y,\tfrac{N}{3}) \}^c  \big) \ge CN^{-\frac{d}{2}+1}.  
  \end{equation}
  By (\ref{order_cap_ball}), there exists $c'>0$ such that 
  \begin{equation}\label{good4.16}
 \big \{ 	\mathrm{cap}( \mathcal{C}_y^{\nu}) \ge cN^{d-2} \big\}\subset \big\{y\xleftrightarrow{} \partial B(y, c'N) \big\}. 
  \end{equation}
   As in \cite{cai2024quasi}, the arguments in \cite[Theorem 1.2]{cai2024incipient} also apply to Voronoi networks. In particular, there exists $c''>0$ such that  
   \begin{equation}\label{good4.17}
   	\mathbb{P}\big(y\xleftrightarrow{} \partial B(y, c'N) ,  \#\mathcal{C}_y^{\nu} < c''N^{\frac{d}{2}+1}  \big) \le \tfrac{1}{2}CN^{-\frac{d}{2}+1}. 
   \end{equation}
   Combining (\ref{good4.15}), (\ref{good4.16}) and (\ref{good4.17}), we get 
   \begin{equation}\label{use4.18}
   	\begin{split}
   	&	\mathbb{P}\big(  \mathrm{cap}( \mathcal{C}_y^{\nu})\ge cN^{d-2}, \#\mathcal{C}_y^{\nu}  \ge c''N^{\frac{d}{2}+1}  , \{ y\xleftrightarrow{} \partial B(y,\tfrac{N}{3}) \}^c  \big)  \\
   		\ge &  \mathbb{P}\big( \mathrm{cap}( \mathcal{C}_y^{\nu})\ge cN^{d-2}, \{ y\xleftrightarrow{} \partial B(y,\tfrac{N}{3}) \}^c  \big) \\
   		&-  	\mathbb{P}\big(\mathrm{cap}( \mathcal{C}_y^{\nu})\ge cN^{d-2} ,  \#\mathcal{C}_y^{\nu} < c''N^{\frac{d}{2}+1}  \big)  \\
   		\ge & CN^{-\frac{d}{2}+1} -  \tfrac{1}{2}CN^{-\frac{d}{2}+1} =  \tfrac{1}{2}CN^{-\frac{d}{2}+1}. 
   	\end{split}
   \end{equation}
   Using this estimate, we can derive the following extension of \cite[Lemma 4.11]{cai2025heterochromatic}. 
      \begin{lemma}\label{lemma_cap_lower}
   There exist $\cl\label{const_cap_lower},\cl\label{const_cap_lower1} >0$ such that for any $x\in \mathbb{R}^d$ and $N\ge 1$, 
      	\begin{equation}\label{ineq_cap_lower}
   		\mathbb{P}\big( \exists y\in \hat{B}(x,N)\ \text{such that}\ \mathsf{V}_y\ \text{occurs}  \big)\asymp 1,  
   	\end{equation}
   	where we denote $\mathsf{V}_y:=\{\mathrm{cap}(\mathcal{C}_y^{\nu,  \partial B(x,N)})\ge \cref{const_cap_lower}N^{d-2}, \#\mathcal{C}_y^{\nu,  \partial B(x,N)}  \ge \cref{const_cap_lower1}N^{\frac{d}{2}+1} \}$. 
   \end{lemma} 
   \begin{proof}
  We use the second moment method. We take a small constant $c_\star>0$, and define $\mathbf{X}:=\sum_{ y\in \hat{B}(x,c_\star N)}\mathbbm{1}_{\mathsf{V}_y}$. For any $y\in \hat{B}(x,c_\star N)$, since $\{ y\xleftrightarrow{} \partial B(y,\tfrac{N}{3}) \}^c\subset \{\mathcal{C}_y^{\nu} =\mathcal{C}_y^{\nu,  \partial B(x,N)}\}$, it follows from (\ref{use4.18}) that 
  \begin{equation}\label{good420}
  	\mathbb{P}\big( \mathsf{V}_y \big)  	 \gtrsim N^{-\frac{d}{2}+1}.
  \end{equation}
   	As a result, we have  
   	\begin{equation}\label{use421}
   	 \mathbb{E}[ \mathbf{X}] = \sum\nolimits_{ y\in \hat{B}(x,c_\star N)} \mathbb{P}\big( \mathsf{V}_y \big) \gtrsim N^{\frac{d}{2}+1}
.    	\end{equation}
   	For the second moment, by (\ref{good4.16}) we have 
   	\begin{equation}
   		\mathsf{V}_y\subset \{y\xleftrightarrow{} \partial B(y, c'N)  \} \subset \{y\xleftrightarrow{} \partial B(x, \tfrac{1}{2}c'N)  \} 
   	\end{equation}
   	for all $y\in \hat{B}(x,c_\star N)$. Here the second inclusion holds for sufficiently small $c_\star>0$. Therefore, using (\ref{ineq_multiset}), we obtain  
   	\begin{equation}\label{use423}
   		\begin{split}
   			 \mathbb{E}\big[ \mathbf{X}^2 \big] = & \sum\nolimits_{ y_1,y_2\in \hat{B}(x,c_\star N)} \mathbb{P}\big( \mathsf{V}_{y_1}\cap \mathsf{V}_{y_2} \big) \\
   			 \le & \sum\nolimits_{ y_1,y_2\in \hat{B}(x,c_\star N)} \mathbb{P}\big(y_1\xleftrightarrow{} \partial B(x, \tfrac{1}{2}c'N),y_2\xleftrightarrow{} \partial B(x, \tfrac{1}{2}c'N)  \big)\\
   			\lesssim & N^{-\frac{d}{2}+1}\sum\nolimits_{y_1,y_2\in \hat{B}(x,c_\star N) } (|y_1-y_2|+1)^{-\frac{d}{2}+1} \asymp N^{d+2}. 
   		\end{split}
   	\end{equation}
  Applying the Paley--Zygmund inequality, (\ref{use421}) and (\ref{use423}), we get (\ref{ineq_cap_lower}). 
   \end{proof}

 The next lemma is the main result of this subsection. Recall that $\nu \in [0,N^{-2}]$. For any $x\in \mathbb{R}^d$, $n\ge 1$, $N\ge 10d^2n$ and $D\subset \widetilde{\mathbf{G}}$ crossing $B(x,N)\setminus B(x,n)$, we say that $D$ is $\lambda$-dense if for any $y \in \hat{\partial}\hat{B}(x,dn)$, 
 \begin{equation}\label{new3.5}
 	\widetilde{\mathbb{P}}_y\big( \widetilde{\tau}_{\partial B(x,N)} < \widetilde{\tau}_{D\cap [B(x,\frac{N}{2d}) \setminus B(x,2dn)]}   \big) \le  (n/N)^{\lambda}. 
 \end{equation}

 \begin{lemma}\label{lemma_cluster_dense}
     	For any $K\ge 1$, there exist $\Cl\label{const_big_cluster_dense}(K),\cl\label{const_small_cluster_dense}(K)>0$ such that for any $x\in \mathbb{R}^d$, $n\ge 1$, $N\ge 10n$, $v\in \hat{\partial }\hat{B}(x,n)$ and $w\in \hat{\partial }\hat{B}(x,N)$,      	\begin{equation}\label{newineq3.6}
     		\mathbb{P}\big( \mathcal{C}_w^{\nu}\ \text{is not}\ \cref{const_small_cluster_dense}\text{-dense}  \mid v\xleftrightarrow{} w    \big) \le \Cref{const_big_cluster_dense} (n/N)^{K}. 
     	\end{equation}
     	Here $v\xleftrightarrow{} w$ denotes the event that $v$ and $w$ belong to the same cluster in $\mathcal{L}^{\nu}$. 
 \end{lemma}

 \begin{proof}

  By Lemma \ref{lemma_switching}, conditioned on the event $\{v\xleftrightarrow{} w \}$, the cluster $\mathcal{C}_w^{\nu}$ contains a Brownian excursion $\eta_{v,w}$ from $v$ to $w$, together with all points connected to $\eta_{v,w}$ by $\mathcal{L}^{\{v,w\}}$. We denote $r_k:=2^kn$, and define $k_\star$ as the maximal integer $k$ such that $r_k \le \frac{N}{10}$. Note that $k_\star \asymp \ln(N/n)$. For each $k\in \mathbb{N}^+$, let 
  \begin{equation}
  	\mathfrak{B}_k^{\delta}:= \big\{ B(z, \delta r_k): z\in (\delta r_k)\cdot \mathbb{Z}^d, B(z, \delta r_k)\cap  [B(x,\tfrac{3}{4}r_{k+1})\setminus B(x,\tfrac{4}{3}r_k)] \neq \emptyset \big\}. 
  \end{equation}
   We define $\mathsf{A}_k^{\delta}$ as the event that there exists $B\in \mathfrak{B}_k^{\delta}$ such that $\{\eta_{v,w}(t):\tau_{\partial B(x, r_{k} )} \le  t\le \tau_{\partial B(x,r_{k+1})} \}$ intersects a cluster $\mathcal{C}$ in $\mathcal{L}^{B^c}$ with $\mathrm{cap}(\mathcal{C})\ge \cref{const_cap_lower}(\delta r_k)^{d-2}$. We claim that for any $\epsilon>0$, there exists $\delta>0$ such that arbitrarily given $\eta_{v,w}(\tau_{\partial B(x, r_{k} )})$,   
  \begin{equation}\label{claim39}
  	\mathbb{P}\big( \mathsf{A}_k^{\delta} \big) \ge 1-  \epsilon , \ \forall 1\le k\le k_\star. 
  \end{equation} 
By (\ref{claim39}) and the strong Markov property of $\eta_{v,w}$, the quantity 
$\mathbf{X}_{\delta}:= \sum\nolimits_{1\le k\le k_\star} \mathbbm{1}_{\mathsf{A}_k^{\delta} } $ stochastically dominates the sum of $k_\star$ i.i.d. Bernoulli random variables with mean $1-\epsilon$. Therefore, the Chernoff bound implies 
\begin{equation}
	\mathbb{P}\big( \mathbf{X}_{\delta} \le \tfrac{1}{2}k_\star \big) \lesssim e^{-\gamma_\epsilon k_\star},  
\end{equation}
 where $\gamma_\epsilon \to \infty$ as $\epsilon\downarrow 0$. Thus, by taking a sufficiently small $\delta=\delta_\dagger(K)$, we have 
 \begin{equation}
 	\mathbb{P}\big( \mathbf{X}_{\delta_\dagger} \le \tfrac{1}{2}k_\star \big) \lesssim (n/N)^{K}.
 	 \end{equation}   
  Meanwhile, on the event $\{ \mathbf{X}_{\delta_\dagger} > \tfrac{1}{2}k_\star \}$, there exist more than $\frac{1}{2}k_\star$ annuli of the form $B(x,r_{k+1})\setminus B(x,r_{k})$ where the cluster $\mathcal{C}_w^{\nu}$ has capacity at least $\cref{const_cap_lower}(\delta_\dagger r_k)^{d-2}$. When a random walk crosses such an annulus, it hits $\mathcal{C}_w^{\nu}$ with uniformly positive probability. This implies that $\mathcal{C}_w^{\nu}$ is $c$-dense for some constant $c=c(\delta_\dagger)>0$.

  It remains to establish the claim in (\ref{claim39}). For any $1\le j\le (100d\delta)^{-1}$, when $\eta_{v,w}$ first intersects $\partial B(x,(\frac{4}{3}+ 5d\delta j )r_k)$, the hitting position is contained in $\frac{1}{2}B$ for some $B\in \mathfrak{B}_k^{\delta}$, where $aB$ denotes the ball concentric with $B$ rescaled by a factor of $a$. Therefore, it follows from Lemma \ref{lemma_cap_lower} that with uniformly positive probability $c_\ddagger$, $\eta_{v,w}$ hits a cluster $\mathcal{C}$ in $\mathcal{L}^{B^c}$ with $\mathrm{cap}(\mathcal{C})\ge \cref{const_cap_lower}(\delta r_k)^{d-2}$ before exiting $2B$. Note that this bound holds uniformly in the hitting position, and that $\mathcal{L}^{B^c}$ is independent of the loops contained in $B(x,(\frac{4}{3}+5d\delta(j-\frac{1}{2}) )r_k)$. Consequently,  
  \begin{equation}
  	\mathbb{P}\big( (\mathsf{A}_k^{\delta})^c \big)\le (1-c_{\ddagger})^{(100d\delta)^{-1}}. 
  \end{equation}  
  This bound implies (\ref{claim39}), and hence completes the proof. 
   \end{proof}

 We also need the following analogue of Lemma \ref{lemma_cluster_dense} concerning cluster volumes.

 \begin{lemma}\label{ulemma_volume}
 For any $a>0$, $N\ge 1$ and $v,w\in \mathbf{V}$ with $|v-w|=N$, 
 	\begin{equation}\label{u431}
 		\mathbb{P}\big( \#\mathcal{C}_v \le N^{(\frac{d}{2}+1)(1-a)}   \mid v\xleftrightarrow{} w \big) \le    e^{-cN^{a}}. 
 	\end{equation}
 \end{lemma}
 \begin{proof}
 As in the preceding proof, Lemma \ref{lemma_switching} implies that conditional on the event $\{v\xleftrightarrow{} w\}$, the cluster $\mathcal{C}_v$ contains a Brownian excursion $\eta_{v,w}$ from $v$ to $w$ together with all clusters in $\mathcal{L}^{\{v,w\}}$ that intersect $\eta_{v,w}$. Note that the excursion $\eta_{v,w}$ hits the boundaries of the balls $B(v,CkN^{1-a})$ for all $1\le k\le cN^{a}$ in order. By Lemma \ref{lemma_cap_lower}, with a uniformly positive probability $c_\star>0$, the portion of the excursion $\eta_{v,w}$ between $\tau_{\partial B(v,CkN^{1-a})}$ and $\tau_{\partial B(v,C(k+1)N^{1-a})}$ intersects a cluster with volume at least $N^{(\frac{d}{2}+1)(1-a)}$, consisting only of loops in $B(v,C(k+1)N^{1-a})\setminus B(v,CkN^{1-a})$ (we denote this event by $\mathsf{A}_k$). Therefore, by the independence of these clusters and the strong Markov property of $\eta_{v,w}$, we have 
 \begin{equation}\label{u4.32}
 	\mathbb{P}\big( \cap_{1\le k\le cN^{a}} \mathsf{A}_k^c \big) \le (1-c_\star )^{cN^{a}} \le e^{-c'N^{a}}. 
 \end{equation}
If $\#\mathcal{C}_v \le N^{(\frac{d}{2}+1)(1-a)}$, none of the events $\mathsf{A}_k$ for $1\le k\le cN^{a}$ occurs. This together with (\ref{u4.32}) yields (\ref{u431}).  
 \end{proof}

 In the next lemma, we establish the exact order of the $k$-point function considered in Theorem \ref{thm_k_point}.

 \begin{lemma}
 
Let $3\le d\le 5$, $k\ge 2$ and $a,b>0$. For any Voronoi network $\mathbf{G}$ on $\mathbb{R}^d$, and any $N\ge 1$ and $x_1,...,x_k\in \mathbf{G}$ satisfying $aN \le  |x_i-x_j| \le bN$ for all $1\le i<j\le k$, we have 
 	\begin{equation}
 		\mathbb{P}\big( x_1 \xleftrightarrow{} x_2 \xleftrightarrow{} ... \xleftrightarrow{} x_k \big) \asymp N^{(-\frac{d}{2}+1)k}, 
 	\end{equation}
 	where the implicit constants may depend on $d$, $k$, $a$, $b$ and the
parameters of $\mathbf{G}$.
 \end{lemma}
 \begin{proof}
 	We first consider the lower bound. For each $1\le i\le k$, by (\ref{good420}), there exist $c_\dagger,c_\diamond>0$ depending on $a$ and $d$ such that 
 	\begin{equation}
 	\mathbb{P}(\mathsf{A}_i):=	\mathbb{P}\big( \mathrm{cap}( \mathcal{C}_{x_i}^{\partial B(x_i,\frac{aN}{10})}  )\ge c_\dagger N^{d-2} \big)  \ge c_\diamond N^{-\frac{d}{2}+1}. 
 	\end{equation}
 	Therefore, since the events $\{\mathsf{A}_i\}_{1\le i\le k}$ are independent, one has 
 	\begin{equation}
 			\mathbb{P}(\bar{\mathsf{A}}):=\mathbb{P}( \cap_{1\le i\le k}\mathsf{A}_i) = \prod\nolimits_{1\le i\le k}  \mathbb{P}( \mathsf{A}_i) \ge c_\diamond^k N^{(-\frac{d}{2}+1)k}.
 	\end{equation} 
   On the event $\bar{\mathsf{A}}$, the lower bounds on the capacities of $\mathcal{C}_{x_i}^{\partial B(x_i,\frac{aN}{10})}$ for $1\le i\le k$ imply that conditional on their configurations, the probability of having a loop intersecting all of them is bounded from below by some constant $c_\ddagger(d,a,k)>0$. If such a loop exists, then $x_1,\ldots,x_k$ belong to the same loop cluster. Consequently, 
   \begin{equation}
   	\mathbb{P}\big( x_1 \xleftrightarrow{} x_2 \xleftrightarrow{} ... \xleftrightarrow{} x_k \big) \ge c_\ddagger  \mathbb{P}(\bar{\mathsf{A}}) \ge c_\ddagger  c_\diamond^k N^{(-\frac{d}{2}+1)k}.
   \end{equation}

   We prove the upper bound by induction on $k$. The base case $k=2$ follows from \eqref{lupu_two_point}. For $k\ge 3$, assume that the upper bound holds for all $k'$ with $2\le k'<k$. Recall that $\Upsilon_{k-1}$ denotes the collection of all partitions of $\{1,...,k-1\}$. By (\ref{a4.5}), 
   \begin{equation*}
   	\begin{split}
   		&	\mathbb{P}\big( x_1 \xleftrightarrow{} x_2 \xleftrightarrow{} ... \xleftrightarrow{} x_k \big) \\
   		  \lesssim  &   \mathbb{P}\big( x_k\xleftrightarrow{ } \partial B(x_k,caN)  \big)       
   	   \cdot 	   \sum\nolimits_{ u_{k-1} \in \Upsilon_{k-1} } \prod\nolimits_{\{j_1,...,j_m\}\in u_{k-1}} \mathbb{P}\big( x_{j_1} \xleftrightarrow{ } ...  \xleftrightarrow{ }  x_{j_m}   \big)\\
   	   	 \overset{(\ref{ineq_new_one_arm})}{\lesssim} &N^{-\frac{d}{2}+1}    \cdot \sum\nolimits_{ u_{k-1} \in \Upsilon_{k-1} } \prod\nolimits_{\{j_1,...,j_m\}\in u_{k-1}} N^{-(\frac{d}{2}-1)m} \lesssim   N^{-(\frac{d}{2}-1)k}. 
   	\end{split}
   \end{equation*}
    This completes the induction. 
    \end{proof}

 \subsection{Multi-arm probabilities}

 As mentioned in Section \ref{subsection2.1_sketch}, we need to estimate probabilities that multiple large loop clusters or loops appear  simultaneously. In this subsection, we collect some estimates of this type that will be used later.

 A series of estimates for two-arm events were established in \cite{cai2025heterochromatic}. The arguments therein essentially rely on the fact that a loop cluster of diameter $R$ typically has capacity of order $R^{d-2}$. By Lemma \ref{lemma_cap_lower}, this property holds for the interpolating graphs considered here. Thus, we have the following analogues of \cite[Theorems 1.1 and 1.4]{cai2025heterochromatic}. In this subsection, we always consider the loop soup $\mathcal{L}^{\nu}$ with killing rate $\nu\in [0,N^{-2}]$. Let $\mathsf{H}^{A_1,A_2}_{A_3,A_4}$ be the event that there exist two disjoint loop clusters, one connecting $A_1$ and $A_2$ and the other connecting $A_3$ and $A_4$.




   \begin{lemma}\label{lemma_twocluster}
   (1) For any $x\in \mathbb{R}^d$ and $N>n\ge 1$,
    \begin{equation}\label{two_arm1}
  	 	\mathbb{P}\big( \mathsf{H}^{B(x,n),\partial B(x,N)}_{B(x,n),\partial B(x,N)} \big) \lesssim \big( n/N \big)^{\frac{d}{2}+1}. 
  	 \end{equation} 
   (2) For any $x\in \mathbb{R}^d$, $N>n\ge 1$, $v_1,v_2\in \hat{B}(x,n)$ and $w_1,w_2\in [\hat{B}(x,N)]^c$,  
  	  \begin{equation}\label{newineq_3.13}
  	 	\mathbb{P}\big( \mathsf{H}^{v_1,w_1}_{v_2,w_2} \big) \lesssim n^{3-\frac{d}{2}} N^{-\frac{3d}{2}+1}.  
  	 \end{equation}
   
  	 
   \end{lemma}

As shown in \cite[Remark 4.6]{cai2025heterochromatic}, one can use (\ref{newineq_3.13}) and the switching identity in \cite{werner2025switching} to compute the probability that two independent random walks are not connected by loop clusters, and the probability that a random walk does not hit a nearby loop cluster. The subsequent lemma follows directly from the arguments in \cite[Section 4.1]{cai2025heterochromatic}, and we therefore omit its proof.


\begin{lemma}\label{lemma_path_avoid_cluster}
	(1) For any $x\in \mathbb{R}^d$, $n\ge 1$, $N>10d^2n$, let $\eta$ be a random walk starting from an arbitrary point in $B(x,dn)$ and stopped upon hitting $\partial B(x,d^{-1}N)$. Then for any $v \in \hat{B}(x,n)$ and $w \in [\hat{B}(x,N)]^c$,  		
	\begin{equation}\label{newineq3.14}
		\mathbb{P}\big( v\xleftrightarrow{} w, \mathcal{C}_v\cap \mathrm{ran}(\eta) =\emptyset   \big) \lesssim  n^{3-\frac{d}{2}}N^{-\frac{d}{2}-1}. 
	\end{equation}
	 (2) Let $\eta_1$ and $\eta_2$ be two independent random walks, each satisfying the assumptions on $\eta$ in Item (1). Then we have 
		\begin{equation}\label{newineq3.15}
		\mathbb{P}\big(  \{ \mathrm{ran}(\eta_1) \xleftrightarrow{} \mathrm{ran}(\eta_2)  \}^c \big) \lesssim  (n/N)^{3-\frac{d}{2}}. 
	\end{equation}
 \end{lemma}

Using Lemma \ref{lemma_path_avoid_cluster}, we may derive the estimate required in (\ref{ineq_1.8}). Let $\mathsf{F}_x(N,n)$ denote the event that there exist a loop and a loop cluster that are disjoint and both cross the annulus $B(x,N)\setminus B(x,n)$.

  \begin{lemma}\label{lemma_oneloop_onecluster}
  	  Let $\cl\label{const_oneloop_onecluster}:=\cref{const_small_cluster_dense}(10d)$.	Then for any $x\in \mathbb{R}^d$ and $N>n\ge 1$, 
  	   	 \begin{equation}
  	 	\mathbb{P}\big( \mathsf{F}_x(N,n) \big) \lesssim \big( n/N \big)^{d+\cref{const_oneloop_onecluster}}. 
  	 \end{equation}
  \end{lemma}
 
   \begin{proof}
   Let $\mathcal{N}$ denote the collection of loops crossing the annulus $B(x,N)\setminus B(x,n)$. We denote by $\mathsf{F}_x^{(1)}(N,n)$ the sub-event of $\mathsf{F}_x(N,n)$ where the involved loop cluster is required to contain a loop in $\mathcal{N}$. Let $\mathsf{F}_x^{(2)}(N,n):=\mathsf{F}_x(N,n)\setminus \mathsf{F}_x^{(1)}(N,n)$.

  On the event $\mathsf{F}_x^{(1)}(N,n)$, one has $\#\mathcal{N}\ge 2$ (where $\# U$ denotes the cardinality of $U$). It follows from (\ref{ineq_crossing_loop}) that 
  \begin{equation}
  	\mathbb{P}\big( \#\mathcal{N} \ge 2 \big) \lesssim (n/N)^{2(d-2)}. 
  \end{equation}
  Moreover, by (\ref{newineq3.15}), the probability that two such loops are not connected by any loop cluster is $O((n/N)^{3-\frac{d}{2}})$. Therefore, 
  \begin{equation}\label{bound3.18}
  	\mathbb{P}\big( \mathsf{F}_x^{(1)}(N,n) \big) \lesssim (n/N)^{2(d-2)}\cdot (n/N)^{3-\frac{d}{2}}= (n/N)^{\frac{3d}{2}-1},
  \end{equation}
 where the exponent $\frac{3d}{2}-1>d$.

 In what follows, we estimate the probability of $\mathsf{F}_x^{(2)}(N,n)$. For any $y\in \hat{\partial}\hat{B}(x,n)$ and $z\in \hat{\partial}\hat{B}(x,N)$, we define the event 
    \begin{equation}\label{def_dense_event}
    	\mathsf{A}_{y,z}:= \{y\xleftrightarrow{}z, \mathcal{C}_y\ \text{is not}\ \cref{const_oneloop_onecluster}\text{-dense}   \}. 
    \end{equation}
  By Lemma \ref{lemma_cluster_dense}, one has
  \begin{equation}\label{bound320}
  	\mathbb{P}\big(\mathsf{A}_{y,z}    \big)   \lesssim (n/N)^{10d}\cdot N^{2-d}. 
  \end{equation} 
  On the event $\mathsf{F}_x^{(2)}(N,n)$, conditioned on all loops in $\mathcal{N}$ and the clusters containing them (we denote by $\mathcal{C}_{\mathcal{N}}$ the union of these clusters), the event $\{B(x,n)\xleftrightarrow{(\mathcal{C}_{\mathcal{N}})} \partial B(x,N)\}$ occurs (recall that this event means connectivity off $\mathcal{C}_{\mathcal{N}}$). In addition, using (\ref{newineq3.3}), one has 
  \begin{equation}\label{bound3.21}
  	\begin{split}
  	&	\mathbb{P}\big(B(x,n)\xleftrightarrow{(\mathcal{C}_{\mathcal{N}})} \partial B(x,N) \mid  \mathcal{C}_{\mathcal{N}}  \big) \\
  		\lesssim & (nN)^{-\frac{d}{2}} \sum\nolimits_{y\in \hat{\partial} \hat{B}(x,dn),z\in \hat{\partial} \hat{B}(x,d^{-1}N)}  \mathbb{P}\big( y  \xleftrightarrow{(\mathcal{C}_{\mathcal{N}})} z  \mid  \mathcal{C}_{\mathcal{N}} \big). 
  	\end{split}
  \end{equation} 
  As a result, we obtain 
  \begin{equation}\label{newbound330}
  		\mathbb{P}\big( \mathsf{F}_x^{(2)}(N,n)  \big) \lesssim (nN)^{-\frac{d}{2}} \sum\nolimits_{y\in \hat{\partial} \hat{B}(x,dn),z\in \hat{\partial} \hat{B}(x,d^{-1}N)}  \mathbb{P}\big( \mathcal{N}\neq \emptyset , y  \xleftrightarrow{( \mathcal{C}_{\mathcal{N}} )} z  \big).  
  \end{equation}
Meanwhile, by the restriction property, one has 
\begin{equation}\label{bound3.22}
	\mathbb{P}\big( \mathcal{N}\neq \emptyset , y  \xleftrightarrow{( \mathcal{C}_{\mathcal{N}} )} z  \big)  \le  \mathbb{P}\big(  y  \xleftrightarrow{} z, \exists \ell \in \mathcal{N}\ \text{such that}\  \mathcal{C}_y\cap \mathrm{ran}(\ell) =\emptyset   \big). 
\end{equation} 
Note that $\mathbb{P}( \mathcal{N}\neq \emptyset )\lesssim (n/N)^{d-2}$. In addition, every loop $\ell$ in $\mathcal{N}$ includes two random walk trajectories $\eta_1$ and $\eta_2$, whose joint distribution is comparable to that of two independent random walks starting from some point in $\partial B(x,dn)$ and stopped upon hitting $\partial B(x,d^{-1}N)$ (see e.g., \cite[Lemma 6.3]{cai2025separation}). On $\mathsf{A}_{y,z}^c$, the conditional probability of $\{\mathcal{C}_y\cap \mathrm{ran}(\eta_1) =\emptyset  \}$ given $\{y\xleftrightarrow{} z\}$ is at most $(n/N)^{\cref{const_oneloop_onecluster}}$. Combined with (\ref{bound320}), it yields  
\begin{equation}\label{bound3.23}
	\begin{split}
		&  \mathbb{P}\big(  y  \xleftrightarrow{} z, \exists \ell \in \mathcal{N}\ \text{such that}\  \mathcal{C}_y\cap \mathrm{ran}(\ell) =\emptyset   \big) \\
		 \lesssim  & (n/N)^{10d} \cdot N^{2-d}+  (n/N)^{d-2+ \cref{const_oneloop_onecluster}} \mathbb{P}\big( y\xleftrightarrow{} z, \mathcal{C}_y\cap \mathrm{ran}(\eta_2)=\emptyset \big)  \\
		 \overset{(\ref{newineq3.14})}{\lesssim } & n^{2-d}(n/N)^{\frac{3d}{2}-1+\cref{const_oneloop_onecluster}}. 
	\end{split}
\end{equation} 
 Putting (\ref{newbound330}), (\ref{bound3.22}) and (\ref{bound3.23}) together, we obtain  
\begin{equation*}
	\mathbb{P}\big( \mathsf{F}_x^{(2)}(N,n)  \big) \lesssim 	(nN)^{-\frac{d}{2}} \cdot \big|\hat{\partial} \hat{B}(x,dn) \big|\cdot \big|\hat{\partial} \hat{B}(x,d^{-1}N)\big|\cdot n^{2-d}(n/N)^{\frac{3d}{2}-1+\cref{const_oneloop_onecluster}} \asymp (n/N)^{d+\cref{const_oneloop_onecluster}}. 
\end{equation*}
 This together with (\ref{bound3.18}) completes the proof of this lemma. 
 \end{proof}

 Next, we establish the estimate needed in (\ref{ineq1.11}). We define $\mathsf{G}_x(N,n)$ as the event that there exist three disjoint loop clusters crossing the annulus $B(x,N)\setminus B(x, n)$.

  \begin{lemma}\label{lemma_threecluster}
  Recall $\cref{const_oneloop_onecluster}$ in Lemma \ref{lemma_oneloop_onecluster}. Then for any $x\in \mathbb{R}^d$ and $N>n\ge 1$, 
    	 \begin{equation}\label{bound324}
  	 	\mathbb{P}\big( \mathsf{G}_x(N,n) \big) \lesssim \big( n/N \big)^{d+\cref{const_oneloop_onecluster}}. 
  	 \end{equation} 
    \end{lemma}

  \begin{proof}
  	 We enumerate the points in $\hat{\partial}\hat{B}(x,n)$ by $\{w_i\}_{1\le i\le l}$. For $1\le i_1<i_2\le l-1$, we define $\mathsf{D}_{i_1,i_2}$ as the event that $w_{i_1}$ and $w_{i_2}$ are the only two points in $\{w_i\}_{1\le i\le i_2}$ that are connected to $\partial B(x,N)$ by loop clusters, and that $\mathcal{C}_{w_{i_1}}\cap \mathcal{C}_{w_{i_2}}=\emptyset$. On the event $\mathsf{G}_x(N,n)$, there exist $1\le i_1<i_2\le l-1$ such that $\mathsf{D}_{i_1,i_2}$ occurs and that a loop cluster in $\mathcal{L}^{\cup_{1\le i\le i_2}\mathcal{C}_{w_{i }}}$ crosses the annulus $B(x,N)\setminus B(x, n)$. As a result, the probability of $\mathsf{G}_x(N,n)$ is at most 
  	\begin{equation}
  		\begin{split}
  	  & \sum_{1\le i_1<i_2\le l-1}   \mathbb{E}\big[ \mathbbm{1}_{\mathsf{D}_{i_1,i_2} } \cdot \mathbb{P}\big(  B(x,n) \xleftrightarrow{(\cup_{1\le i\le i_2}\mathcal{C}_{w_{i }}) }  \partial  B(x,N)  \big)  \big] \\
  			\overset{(\ref{newineq3.3})}{\lesssim } &(nN)^{-\frac{d}{2}} \sum_{y\in \hat{\partial} \hat{B}(x,dn),z\in \hat{\partial} \hat{B}(x,d^{-1}N)} \sum_{1\le i_1<i_2\le l-1}   \mathbb{E}\big[ \mathbbm{1}_{\mathsf{D}_{i_1,i_2} } \cdot \mathbb{P}\big(  y \xleftrightarrow{(\cup_{1\le i\le i_2}\mathcal{C}_{w_{i }}) }  z  \big)  \big] \\
  			\le &  (nN)^{-\frac{d}{2}} \sum_{y\in \hat{\partial} \hat{B}(x,dn),z\in \hat{\partial} \hat{B}(x,d^{-1}N)}  \mathbb{P}\big( \hat{\mathsf{G}}(y,z)  \big),
  		\end{split}
  	\end{equation}
  	where $\hat{\mathsf{G}}(y,z)$ denotes the event that there exist three disjoint loop clusters, one connecting $y$ and $z$, and the other two crossing  $B(x,N)\setminus B(x, n)$. Repeating this argument twice, we obtain that $\mathbb{P} ( \mathsf{G}_x(N,n)  )$ is at most of order 
  	\begin{equation}\label{bound326}
  			 	   (nN)^{-\frac{3d}{2}}\sum_{y_1,y_2,y_3\in \hat{\partial} \hat{B}(x,dn),z_1,z_2,z_3\in \hat{\partial} \hat{B}(x,d^{-1}N)}  \mathbb{P}\big( \bar{\mathsf{G}}(y_1,y_2,y_3;z_1,z_2,z_3)  \big).
  	\end{equation}
  Here $\bar{\mathsf{G}}(y_1,y_2,y_3;z_1,z_2,z_3)$ denotes the event that there exist three disjoint loop clusters connecting $y_i$ to $z_i$ respectively for $i\in \{1,2,3\}$. On this event, if $\mathsf{A}_{y_1,z_1}^c$ also occurs (recall the definition of $\mathsf{A}_{y_1,z_1}$ from (\ref{def_dense_event})), then it follows from (\ref{newineq3.8}) that given the clusters $\mathcal{C}_{y_1}$ and $\mathcal{C}_{y_2}$, the conditional probability of $\{y_3\xleftrightarrow{} z_3\}$ is $O(n^{\cref{const_oneloop_onecluster}}N^{2-d-\cref{const_oneloop_onecluster}})$. Consequently,  
   \begin{equation}\label{bound327}
  	\begin{split}
  		 & \mathbb{P}\big( \bar{\mathsf{G}}(y_1,y_2,y_3;z_1,z_2,z_3), \mathsf{A}_{y_1,z_1}^c   \big) \\  
  		\lesssim    &  n^{\cref{const_oneloop_onecluster}}N^{2-d-\cref{const_oneloop_onecluster}} \mathbb{P}\big( \mathsf{H}^{y_1,z_1}_{y_2,z_2} \big)    \overset{(\ref{newineq_3.13})}{ \lesssim } (n/N)^{3-\frac{d}{2}+\cref{const_oneloop_onecluster}}  N^{6-3d}. 
  	\end{split}
  \end{equation}  
On the other hand, $\bar{\mathsf{G}}(y_1,y_2,y_3;z_1,z_2,z_3)\cap \mathsf{A}_{y_1,z_1}$ implies that $\mathsf{A}_{y_1,z_1}$, $\{y_2\xleftrightarrow{} z_2\}$ and $\{y_3 \xleftrightarrow{} z_3\}$ are certified by three disjoint collections of loops. Thus, by the BKR inequality, we have 
\begin{equation}\label{bound328}
	\begin{split}
		& \mathbb{P}\big( \bar{\mathsf{G}}(y_1,y_2,y_3;z_1,z_2,z_3) , \mathsf{A}_{y_1,z_1}    \big)  \\
		\le & \mathbb{P}\big(  \mathsf{A}_{y_1,z_1}    \big) \cdot \mathbb{P}\big( y_2\xleftrightarrow{} z_2  \big) \cdot \mathbb{P}\big( y_3 \xleftrightarrow{} z_3  \big)  \overset{(\ref{lupu_two_point}),(\ref{bound320})}{ \lesssim}   (n/N)^{10d} N^{6-3d}.  
	\end{split}
\end{equation} 
Plugging (\ref{bound327}) and (\ref{bound328}) into (\ref{bound326}), we obtain that $\mathbb{P}  ( \mathsf{G}_x(N,n) )$ is at most 
  \begin{equation*}
  	\begin{split}
  		C (nN)^{-\frac{3d}{2}}\cdot \big(\big| \hat{\partial} \hat{B}(x,dn) \big| \cdot  \big| \hat{\partial} \hat{B}(x,d^{-1}N) \big| \big)^{3}\cdot  (n/N)^{3-\frac{d}{2}+\cref{const_oneloop_onecluster}}  N^{6-3d}\asymp (n/N)^{d+\cref{const_oneloop_onecluster}}, 
  	\end{split}
  \end{equation*}
 which gives the desired bound (\ref{bound324}).  
  \end{proof}

 \subsection{Crossing probabilities for massive loop soups}\label{section_control_massive}
 
We consider loop soups on $\widetilde{\mathbb{Z}}^d$ with $3\le d\le 5$. We take $n\ll  m  \ll N   \ll M$ (where $a \ll b$ means $a \le cb$), and denote by $\mathcal{L}$ and $\mathcal{L}^{\nu}$ the loop soups of intensity $\frac{1}{2}$ on $\widetilde{\mathbb{Z}}^d$ and its variant with killing rate $\nu\in [0,N^{-2}]$ respectively. To show that the loops in $\mathcal{L}^{\nu}$ cannot form macroscopic clusters by themselves, we need the following lemma:

  
%
%
 
\begin{lemma}\label{lemma_crossing_massiveGFF}
	For any $d\ge 3$, there exists $\cl\label{const_massive_crossing}>0$ such that for any $M\ge 1$, 
	\begin{equation}
		\mathbb{P}\big( B( \tfrac{M}{2})  \xleftrightarrow{\mathcal{L}^{\nu}} \partial B( M) \big) \lesssim  M^{d}e^{-\cref{const_massive_crossing}\nu M^2}.
	\end{equation}
\end{lemma}
\begin{proof}
	Applying the union bound, one has 
	\begin{equation}\label{a330}
	\begin{split}
				  \mathbb{P}\big( B( \tfrac{M}{2})  \xleftrightarrow{\mathcal{L}^{\nu}} \partial B( M)  \big) 
				 \le &M^{2d-2}  \max_{y \in \hat{\partial}\hat{B}( \frac{M}{2}), z\in \hat{\partial}\hat{B}( M) }\mathbb{P}\big( y\xleftrightarrow{\mathcal{L}^{\nu}} z \big). 
	\end{split}
	\end{equation}
	By (\ref{lupu_two_point}) and $G(y,y)\asymp G(z,z)\asymp 1$ (which follows from $\nu\in [0,1]$), we have 
	\begin{equation}
		\mathbb{P}\big( y\xleftrightarrow{\mathcal{L}^{\nu}} z \big) \lesssim \mathbb{P}_y(\tau_z<\infty) \lesssim  M^{2-d}e^{-c\nu M^2}  
	\end{equation}
	for $y \in \hat{\partial}\hat{B}( \frac{M}{2})$ and $z\in \hat{\partial}\hat{B}( M)$. 
	Combined with (\ref{a330}), it completes the proof. 
\end{proof}

For a macroscopic loop cluster, we will use the killed loops it contains to approximate the entire cluster. Although these killed loops are typically large, there are still many of relatively small diameters. One of our tasks is to show that the contribution of these small loops is negligible. To be precise, we define 
\begin{equation}\label{a458}
  \widehat{\mathcal{L}}:=(\mathcal{L}-\mathcal{L}^{\nu}) \cdot \mathbbm{1}_{|\ell|\ge m},\ \widecheck{\mathcal{L}}:=(\mathcal{L}-\mathcal{L}^{\nu}) \cdot \mathbbm{1}_{|\ell|< m}\  \text{and}\   \overline{\mathcal{L}}:=\mathcal{L}-\widehat{\mathcal{L}}.  
\end{equation}
  For $\alpha\ge 0$, we define $\widecheck{\mathcal{L}}_{\alpha}$ as the analogue of $\widecheck{\mathcal{L}}$ obtained by replacing the intensity $1/2$ with $\alpha$. Let $\overline{\mathcal{L}}_{\alpha}:= \mathcal{L}^{\nu}+\widecheck{\mathcal{L}}_{\alpha}$ (the subscript $\alpha$ of $\overline{\mathcal{L}}_{\alpha}$ does not simply refer to the underlying intensity; rather, $\overline{\mathcal{L}}_{\alpha}$ consists of loops from $\mathcal{L}^{\nu}$ with intensity $1/2$ together with loops from $\widecheck{\mathcal{L}}_{\alpha}$ with intensity $\alpha$). Note that $\overline{\mathcal{L}}_{0}= \mathcal{L}^{\nu}$ and $\overline{\mathcal{L}}_{1/2}=\overline{\mathcal{L}}$. The following lemma provides the key estimate to control the influence of $\widecheck{\mathcal{L}}$.




\begin{lemma}\label{lemma3.10_two_arm}
  We retain the notations above. Then there exists $\cl\label{const_check_loop}>0$ such that
  \begin{equation}\label{ineq_two_arm_310}
    \mathbb{P}\big(\mathsf{H}_{\alpha}\big)\lesssim
    \frac{n^{d-2+\cref{const_check_loop}}}{M^{d-2}N^{\cref{const_check_loop}}},\ \ \forall\, 0\le\alpha\le\tfrac{1}{2},
  \end{equation}
  where $\mathsf{H}_{\alpha}$ denotes the event that there exist two disjoint clusters in $\overline{\mathcal{L}}_{\alpha}$ crossing the annulus $B(M)\setminus B(n)$.
\end{lemma}

 \noindent  (P.S. At first sight, it may seem puzzling that the event $\mathsf{H}_{\alpha}$ in (\ref{ineq_two_arm_310}) does not depend on $N$, whereas the bound on the right-hand side does. This arises because the capacity estimates for clusters in $\mathcal{L}^{\nu}$ (which is the main component of $\overline{\mathcal{L}}_{\alpha}$) are effective only up to scale $N$ (since $\nu\in[0,N^{-2}]$; recall Lemma \ref{lemma_cap_lower}). We therefore consider the mutual constraints between the two disjoint clusters of $\overline{\mathcal{L}}_{\alpha}$ only up to scale $N$, which accounts for the appearance of $N$ in the bound.)


Before proving Lemma \ref{lemma3.10_two_arm}, we first present its application.    Using Lemmas \ref{lemma_crossing_massiveGFF} and \ref{lemma3.10_two_arm}, we obtain the following estimate for the crossing probability of $\overline{\mathcal{L}}$. In the subsequent proof, this estimate will be used to bound the probability that $\overline{\mathcal{L}}$ contains a large cluster.  

     \begin{corollary}\label{coro3.11}
        	Under the same conditions as in Lemma \ref{lemma3.10_two_arm}, we have 
        	\begin{equation}\label{bound340}
        		\mathbb{P}\big(    B( \tfrac{M}{2}) \xleftrightarrow{ \overline{ \mathcal{L}} } \partial B( M)  \big) \lesssim M^{d}e^{-\cref{const_massive_crossing}\nu M^2}+\frac{M^2m^{\cref{const_check_loop}}}{N^{2+\cref{const_check_loop}}}. 
        	\end{equation}
        \end{corollary}
        \begin{proof}
        For each $k\in \mathbb{N}^+$ with $2^{k}\le 10m$, let $\{y^j_{k}\}_{1\leq j\leq l_k}$ be the collection of points in $2^k\cdot \mathbb{Z}^d$ such that $B(y^j_{k},2^k)$ intersects $B( M )\setminus B(  \frac{M}{2})$. We define the event  
         \begin{equation}
         	\mathsf{A}_k^j:= \big\{  \exists \ell \in \mathcal{L}-\mathcal{L}^{\nu}\ \text{such that}\ \mathrm{ran}(\ell)\subset B(y^j_{k},2^k) , |\ell| \ge \tfrac{2^k}{10} \big\}.
         \end{equation}
         Here $|\ell|$ denotes the Euclidean diameter of the loop $\ell$.  We also define $\mathsf{H}_k^j$ as the event that there exist two disjoint clusters in $\overline{\mathcal{L}}_{\alpha}$ crossing $B(y^j_{k},cM)\setminus B(y^j_{k},2^k)$.

        By Russo's formula, the difference 
        \begin{equation}
        	\mathbb{P}\big( B( \tfrac{M}{2}) \xleftrightarrow{ \overline{ \mathcal{L}} } \partial B( M)\big) - \mathbb{P}\big( B( \tfrac{M}{2}) \xleftrightarrow{   \mathcal{L}^{\nu}} \partial B( M) \big) 
        \end{equation}
        is bounded by the supremum over $\alpha\in [0,1/2]$ of the expected total mass of the loops in $\widecheck{\mathcal{L}}$ that are pivotal for the event $\{B( \tfrac{M}{2}) \xleftrightarrow{ \overline{ \mathcal{L}}_{\alpha} } \partial B( M)\}$. By the union bound, this total mass is bounded from above by 
        \begin{equation}\label{ineq344}
        \begin{split}
           & \sum\nolimits_{k\ge 1: 2^k\le 10m} \sum\nolimits_{1\le j\le l_k} \mathbb{P}\big( \mathsf{A}_k^j\big) \cdot  \mathbb{P}\big(\mathsf{H}_k^j  \big)   \\
   \overset{(\ref{bound26}), (\ref{ineq_two_arm_310})}{\lesssim } &   \sum\nolimits_{k\ge 1: 2^k\le 10m} \sum\nolimits_{1\le j\le l_k}   2^{2k}N^{-2}\cdot \frac{2^{(d-2+\cref{const_check_loop})k}}{M^{d-2}N^{\cref{const_check_loop}}}  \\
   \overset{l_k\asymp (M/2^k)^d}{\asymp} & \frac{M^2}{N^{2+\cref{const_check_loop}}}  \sum\nolimits_{k\ge 1: 2^k\le 10m}   2^{\cref{const_check_loop}k} \asymp \frac{M^2m^{\cref{const_check_loop}}}{N^{2+\cref{const_check_loop}}}. 
        \end{split}
         \end{equation} 
       Combined with Lemma \ref{lemma_crossing_massiveGFF}, this gives (\ref{bound340}). 
        \end{proof}

We now turn to the proof of Lemma~\ref{lemma3.10_two_arm}. The key is to show that for a cluster $\mathcal{C}$ in $\mathcal{L}$ crossing an annulus, with high probability $\mathcal{C}$ is sufficiently dense that a random walk crossing the same annulus has only polynomially small probability of avoiding all massive loops in $\mathcal{L}^\nu$ contained in $\mathcal{C}$.  To this end, we need the following estimate. Specifically, we take a sufficiently large constant $C_\dagger>0$, and denote $r_i:=C_\dagger^in$ for each $i\in\mathbb{N}$. We define $i_\star:=\min\{i\ge 1:r_{2i+3}\ge N\}$. Note that $i_\star\asymp\log(N/n)$. For any $i\in\mathbb{N}$ and $a>0$, let
$\mathsf{V}^i_a$ be the event that there exists $z\in\hat B(2r_{2i+1})\setminus
\hat B(\frac{1}{2}r_{2i+1})$ such that the cluster $\mathcal{C}^{\nu,\partial B(z,r_{2i})}_z$ satisfies $\mathrm{cap}(\mathcal{C}^{\nu,\partial B(z,r_{2i})}_z)\ge ar_{2i}^{d-2}$ and is pivotal for the event $\{B(n)\leftrightarrow\partial B(M)\}$ (i.e., its removal changes whether $\{B(n)\leftrightarrow\partial B(M)\}$ occurs). Here the notation $\leftrightarrow$ represents connectivity via loops in the loop soup $\mathcal{L}$ on $\widetilde{\mathbb{Z}}^d$; this convention also applies to the remainder of this subsection. We then define the quantity $\mathbf{V}_a:=\sum_{1\le i\le i_\star}\mathbbm{1}_{\mathsf{V}^i_a}$. The following lemma shows that with high probability, $\mathsf{V}^i_a$ occurs for a positive proportion of the scales $1\le i\le i_\star$.

\begin{lemma}\label{newlemma_pivotal}
We retain the notation above. Then there exist $\cl\label{const_pivotal_1},\cl\label{const_pivotal_2},\cl\label{const_pivotal_3}>0$ such that  
	\begin{equation}\label{ineq_lemma_pivotal}
		\mathbb{P}\big(   \mathbf{V}_{\cref{const_pivotal_1}} \le \cref{const_pivotal_2} i_{\star}     \mid  B( n)\xleftrightarrow{} \partial B( M)  \big) \le  (n/N)^{\cref{const_pivotal_3}}. 
	\end{equation}
\end{lemma}

We postpone the proof of Lemma~\ref{newlemma_pivotal} to the end of this subsection. Next, we record a corollary of Lemma~\ref{newlemma_pivotal} as follows. Recall the definition of a $\lambda$-dense set in (\ref{new3.5}); here we take $x=\mathbf{0}$.

  \begin{corollary}\label{coro_density}
  Recall the constant $\cref{const_pivotal_3}$ from Lemma \ref{newlemma_pivotal}. There exists $\cl\label{const_coro_density}>0$ such that for any point process $\mathcal{L}'$ satisfying $\mathcal{L}^{\nu}\le \mathcal{L}'\le \mathcal{L}$,  
  	\begin{equation}\label{coro_ineq_3.16}
  		\mathbb{P}\big( \exists \mathcal{C}' \in \mathfrak{C}'[M,n]\ \text{such that}\ \mathcal{C}' \ \text{is not}\ \cref{const_coro_density}\text{-dense}  \big)\lesssim \frac{n^{\frac{d}{2}-1+\cref{const_pivotal_3}}}{M^{\frac{d}{2}-1}N^{\cref{const_pivotal_3}} }, 
  	\end{equation}
  	where $\mathfrak{C}'[M,n]$ denotes the collection of clusters in $\mathcal{L}'$ crossing $B(M)\setminus B(n)$.  
  	  \end{corollary}
  \begin{proof}
   Referring to Lemma \ref{newlemma_pivotal}, it suffices to show that for some constant $c_\dagger>0$, 
   \begin{equation}\label{inclusion350}
   	\{ \mathbf{V}_{\cref{const_pivotal_1}} > \cref{const_pivotal_2} i_\star \}\subset \big\{ \text{every}\  \mathcal{C}' \in \mathfrak{C}'[M,n] \ \text{is}\ c_\dagger\text{-dense} \big\}. 
   \end{equation}
 Recall that when $\mathsf{V}_{\cref{const_pivotal_1}}^i$ occurs, there exists
$z \in  \hat{B}(2r_{2i+1})\setminus \hat{B}(\frac{1}{2}r_{2i+1})$ such that the cluster $\mathcal{C}_z^{\nu,\partial B(z,r_{2i})}$ satisfies $\mathrm{cap}(\mathcal{C}_z^{\nu,\partial B(z,r_{2i})})\ge \cref{const_pivotal_1} r_{2i}^{d-2}$ and is pivotal for the event $\{B(n)\xleftrightarrow{}\partial B(M)\}$. In fact, $\mathcal{C}_z^{\nu,\partial B(z,r_{2i})}$ must be contained in every $\mathcal{C}'\in \mathfrak{C}'[M,n]$; otherwise, such a cluster $\mathcal{C}'$ would still certify the event $\{B(n)\xleftrightarrow{}\partial B(M)\}$ after the removal of $\mathcal{C}_z^{\nu,\partial B(z,r_{2i})}$, which is contradictory to the pivotality of $\mathcal{C}_z^{\nu,\partial B(z,r_{2i})}$. Therefore, for any $\mathcal{C}'\in \mathfrak{C}'[M,n]$, if $\mathsf{V}_{\cref{const_pivotal_1}}^i$ occurs, then whenever a random walk starting from some $y\in \hat{\partial}\hat{B}(dn)$ crosses the annulus $B(r_{2i+2})\setminus B(r_{2i})$, it intersects $\mathcal{C}'$ with a uniformly positive probability. As a result, $\mathbf{V}_{\cref{const_pivotal_1}} > \cref{const_pivotal_2} i_\star \asymp \log(N/n)$ implies that $\mathcal{C}'$ is $c_\dagger$-dense for some $c_\dagger>0$. This proves (\ref{inclusion350}) and hence completes the proof of the corollary. 
  \end{proof}

With Corollary \ref{coro_density} at hand, we are now ready to establish Lemma \ref{lemma3.10_two_arm}.

         \begin{proof}[Proof of Lemma \ref{lemma3.10_two_arm}]
        We denote by $\mathsf{A}$ the event in (\ref{coro_ineq_3.16}) with $\mathcal{L}'=\overline{\mathcal{L}}_{\alpha}$. Note that $\mathsf{H}_{\alpha}\cap \mathsf{A}$ implies that $\mathsf{A}$ and $B(n)\xleftrightarrow{} \partial B(M)$ occur disjointly. Thus, by the BKR inequality, (\ref{crossing_prob}) and (\ref{coro_ineq_3.16}), we have  
        	\begin{equation}\label{ineq337}
        		\mathbb{P}\big(\mathsf{H}_{\alpha}\cap \mathsf{A} \big)\lesssim   \mathbb{P}\big(\mathsf{A}  \big)  \cdot    \mathbb{P}\big(B(n)\xleftrightarrow{} \partial B(M)  \big)  \lesssim  \frac{n^{d-2+\cref{const_pivotal_3}}}{M^{d-2}N^{\cref{const_pivotal_3}}}.         	\end{equation}
         Next, we estimate the probability of $\mathsf{H}_{\alpha}\cap \mathsf{A}^c$. We enumerate the points in $\hat{\partial } \hat{B}(n)$ as $\{y_i\}_{1\le i\le l}$. We define $\mathsf{D}_i$ as the event that the cluster in $\overline{\mathcal{L}}_{\alpha}$ containing $y_i$ (denoted by $\mathcal{C}_i$) is the unique cluster in $\{\mathcal{C}_j\}_{1\le j\le i}$ that intersects $\partial B(M)$. By the restriction property, the probability $\mathbb{P}( \mathsf{H}_{\alpha}\cap \mathsf{A}^c )$ is at most  
        	 \begin{equation}\label{ineq_AcH}
        	 	\begin{split}
        	 	& 	 \sum\nolimits_{1\le i\le l} \mathbb{E}\big[ \mathbbm{1}_{\mathsf{D}_i\cap \{ \mathcal{C}_i\ \text{is}\ \cref{const_coro_density}\text{-dense} \}} \cdot \mathbb{P}\big( B(n) \xleftrightarrow{ (\cup_{1\le j\le i} \mathcal{C}_j )}  \partial B(M)  \big) \big]\\ 
        	 		 \overset{ (\ref{bound_crossing_by_hitting})}{ \lesssim}  &\sum\nolimits_{1\le i\le l}  \mathbb{E}\Big[ \mathbbm{1}_{\mathsf{D}_i\cap \{ \mathcal{C}_i\ \text{is}\ \cref{const_coro_density}\text{-dense} \}} \cdot \frac{n^{\frac{d}{2}-1+\cref{const_coro_density}}}{M^{\frac{d}{2}-1}N^{\cref{const_coro_density}}} \Big] \\
        \le & \frac{n^{\frac{d}{2}-1+\cref{const_coro_density}}}{M^{\frac{d}{2}-1}N^{\cref{const_coro_density}}}\cdot \mathbb{P}\big( B(n)\xleftrightarrow{} \partial B(M)  \big) 	 	\overset{ (\ref{crossing_prob})}{ \lesssim}   \frac{n^{d-2+\cref{const_coro_density}}}{M^{d-2}N^{\cref{const_coro_density}}}. 
        	 	\end{split}
        	 \end{equation}  
      Combining (\ref{ineq337}) and (\ref{ineq_AcH}), we complete the proof.         \end{proof}

The remainder of this subsection is devoted to the proof of Lemma~\ref{newlemma_pivotal}. Before delving into the details, we first outline the heuristic behind the proof. Conditioned on $\{B(n)\xleftrightarrow{}\partial B(M)\}$, for any annulus $B(R)\setminus B(r)$ with $n\ll r\ll R\ll N$, with uniformly positive probability the cluster containing $B(n)$ has a pivotal cluster of $\mathcal{L}^\nu$ with large capacity in this annulus. Pivotal clusters in different annuli are in fact correlated, with correlations induced by loops crossing multiple annuli. Since the probability of such crossings decays sufficiently fast (see \eqref{eq:increment_tail} below), pivotal clusters in distant annuli are asymptotically independent. Note that there are of order $\log(N/n)$ such annuli. Consequently, the conditional probability that pivotal clusters with large capacity appear in fewer than $c\log(N/n)$ of them decays at least polynomially in $N/n$. To make this rigorous, we explore the cluster containing $B(n)$ scale by scale via partial clusters. Since this argument does not depend on the specific type of pivotal objects, we formulate it for general events satisfying certain conditions (see Proposition~\ref{prop:exploration}), so that it applies both to the pivotal clusters considered here and to the microscopic pivotal loops in Section~\ref{section_dimension_gap}.

 Throughout the proof, we fix $n\ge 1$ and $M\ge N\ge
C_\dagger^{10}n$. In particular, we have $i_\star\ge 2$ and $r_{2i}<N\le M$ for all $1\le i\le i_\star$. We first introduce the partial clusters, which describe the part of the cluster
containing $B(n)$ explored up to a given scale. We recall the partial cluster
$\mathcal{C}_i$ from
\cite[Section 3.2]{cai2025one}. Informally, $\mathcal{C}_i$ is the loop
cluster containing $B(n)$ obtained by allowing loop intersections
only within $B(r_{2i})$. Precisely, for $i\in\mathbb{N}^+$, the
partial cluster $\mathcal{C}_i$ is defined via the following inductive
procedure:
\begin{itemize}
  \item Step $0$: we set $\mathcal{C}_i^{(0)}:=B(n)\cap\widetilde{\mathbb{Z}}^d$;
  \item Step $j$ ($j\ge 1$): given $\mathcal{C}_i^{(j-1)}$, we define
  $\mathcal{C}_i^{(j)}$ as the union of $\mathcal{C}_i^{(j-1)}$ and the ranges
  of all loops in $\mathcal{L}$ intersecting $\mathcal{C}_i^{(j-1)}\cap
  B(r_{2i})$. If $\mathcal{C}_i^{(j)}=\mathcal{C}_i^{(j-1)}$, we
  stop the procedure and set $\mathcal{C}_i:=\mathcal{C}_i^{(j)}$; otherwise,
  we proceed to Step $j+1$.
\end{itemize}
We denote by $\mathcal{F}_{\mathcal{C}_i}$ the $\sigma$-field generated by $\mathcal{C}_i$ and all loops included in $\mathcal{C}_i$. For convenience, we also set $\mathcal{C}_0:=B(n)\cap\widetilde{\mathbb{Z}}^d$, so that $\mathcal{F}_{\mathcal{C}_0}$ is trivial. For a collection $\mathcal{K}$ of loops in $\mathcal{L}$ and an increasing event $\mathsf{E}$, we say that $\mathcal{K}$ is pivotal for $\mathsf{E}$ if $\mathsf{E}$ occurs, but does not occur after all loops in $\mathcal{K}$ are removed from $\mathcal{L}$. The partial clusters satisfy the following three properties:
\begin{enumerate}
  \item[(P1)] For any possible configuration $D$ of $\mathcal{C}_i$,
  conditioned on $\mathcal{F}_{\mathcal{C}_i}$ and on $\{\mathcal{C}_i=D\}$,
  the point measure consisting of all loops that are not included in
  $\mathcal{C}_i$ has the same distribution as $\mathcal{L}\cdot
  \mathbbm{1}_{\mathrm{ran}(\ell)\cap D\cap B(r_{2i})=\emptyset}$.

  \item[(P2)] For any $0\le i\le j$, every loop included in $\mathcal{C}_i$ is
  also included in $\mathcal{C}_j$, and $\mathcal{F}_{\mathcal{C}_i}\subset
  \mathcal{F}_{\mathcal{C}_j}$.

  \item[(P3)] For any $1\le i<j\le i_\star$, on the event
  $\{B(n)\xleftrightarrow{} \partial B(M)\}$, every collection of loops not included
  in $\mathcal{C}_i$ that is pivotal for $\{\mathcal{C}_i\xleftrightarrow{} \partial
  B(r_{2j})\}$ is also pivotal for $\{B(n)\xleftrightarrow{} \partial B(M)\}$.
\end{enumerate}
(P1) is a consequence of the restriction property,
and (P2) follows directly from the construction. For (P3), let $\mathcal{K}$ be such a collection, and assume for contradiction that $\{B(n)\xleftrightarrow{} \partial B(M)\}$ still occurs after removing $\mathcal{K}$ from $\mathcal{L}$. Then some cluster of the remaining loops intersects both $\mathcal{C}_i$ and $\partial B(M)$. Since this cluster is connected and $r_{2j}<M$, it also intersects $\partial B(r_{2j})$. Thus, $\{\mathcal{C}_i\xleftrightarrow{} \partial B(r_{2j})\}$ still occurs after removing $\mathcal{K}$, which contradicts the pivotality of $\mathcal{K}$ for this event.

We now describe the exploration process. For each $i\in\mathbb{N}^+$, we
define $i_+:=\min\{j\ge i+1:\mathcal{C}_i\cap\partial B(r_{2j-1})=\emptyset\}$. We also set $0_+:=1$. Note that $i_+$ is measurable with respect to
$\mathcal{F}_{\mathcal{C}_i}$. For $R'>R>0$, let $\mathcal{L}_i[R,R']$
denote the point measure consisting of the loops in $\mathcal{L}$ that are not included in $\mathcal{C}_i$ and intersect both $B(R)$ and $\partial B(R')$. For any $j\in[i_++2,i_\star]$, if $(i_+)_+\ge j$ (which implies that
$\mathcal{C}_{i_+}$ intersects $\partial B(r_{2j-3})$), then $\mathcal{C}_{i_+}$ must contain a loop intersecting both $B(r_{2i_+})$ and $\partial B(r_{2j-3})$. Since $\mathcal{C}_i$ does not intersect $\partial B(r_{2i_+-1})$, such a loop is not included in $\mathcal{C}_i$, and thus belongs to $\mathcal{L}_i[r_{2i_+},r_{2j-3}]$. Combining these observations, we obtain the following estimate, which shows that the exploration is unlikely to jump over many scales at once:  
\begin{equation}\label{eq:increment_tail}
  \begin{split}
  &\mathbb{P}\big(B(n)\xleftrightarrow{}\partial B(M),\,(i_+)_+\ge j\mid
  \mathcal{F}_{\mathcal{C}_i}\big)\\
  \le\;&\mathbb{P}\big(\mathcal{C}_i\xleftrightarrow{}\partial B(M),\,
  \mathcal{L}_i[r_{2i_+},r_{2j-3}]\neq 0\mid\mathcal{F}_{\mathcal{C}_i}\big)\\
 \overset{ (\ref{a4.6})}{ \lesssim}\;&\big(\tfrac{r_{2i_+}}{r_{2j-3}}\big)^{\frac{d}{2}-1}\cdot\mathbb{P}
  \big(B(n)\xleftrightarrow{}\partial B(M)\mid\mathcal{F}_{\mathcal{C}_i}\big).
  \end{split}
\end{equation}  
The scales visited by the exploration are given by the increasing sequence
$\{i_k\}_{0\le k\le k_\star}$, defined recursively as follows: (a) set $i_0:=0$ and $i_1:=1$; (b) for $k\ge 1$, if $i_k\ge i_\star$,
we stop the construction and set $k_\star:=k$; otherwise, we set
$i_{k+1}:=(i_k)_+$ and proceed to the next step. The next lemma shows that with high probability, the exploration runs for at least $c\log(N/n)$ steps before reaching $i_\star$. We denote $K:=\ln(N/n)$, and let $\mathsf{T}_j:=\{k_\star\le j\}$ for $j\ge 1$.

\begin{lemma}\label{lem:many_steps}
There exist $c_\spadesuit,c_\clubsuit>0$ such that
\begin{equation}\label{eq:many_steps}
  \mathbb{P}\big(\mathsf{T}_{c_\spadesuit K}\mid B(n)\xleftrightarrow{}\partial
  B(M)\big)\lesssim(n/N)^{c_\clubsuit}.
\end{equation}
\end{lemma}

\begin{proof}
 We divide $\mathsf{T}_{c_\spadesuit K}$ into two sub-events
\begin{equation*}
  \mathsf{T}^{(1)}:=\mathsf{T}_{c_\spadesuit K}\cap\big\{i_{k_\star-1}\le
  \tfrac{1}{2}i_\star\big\}\ \ \text{and}\ \
  \mathsf{T}^{(2)}:=\mathsf{T}_{c_\spadesuit K}\cap\big\{i_{k_\star-1}>
  \tfrac{1}{2}i_\star\big\},
\end{equation*}
and estimate their conditional probabilities separately.

For $\mathsf{T}^{(1)}$, note that for any $0\le i<i'\le\frac{1}{2}i_\star$, the
event $\{i_+=i'\}$ is measurable with respect to $\mathcal{F}_{\mathcal{C}_i}$,
and that $\{i_{k_\star-2}=i,i_{k_\star-1}=i'\}\subset\{i_+=i',(i_+)_+\ge
i_\star\}$. Thus, by \eqref{eq:increment_tail} (with $j=i_\star$), we have
\begin{equation*}
  \begin{split}
  &\mathbb{P}\big(B(n)\xleftrightarrow{}\partial B(M),\,\mathsf{T}^{(1)},\,
  i_{k_\star-2}=i,\,i_{k_\star-1}=i'\big)\\
  \le\;&\mathbb{E}\Big[\mathbbm{1}_{i_+=i'}\cdot\mathbb{P}\big(B(n)
  \xleftrightarrow{}\partial B(M),\,(i_+)_+\ge i_\star\mid
  \mathcal{F}_{\mathcal{C}_i}\big)\Big]\\
  \lesssim\;&\big(\tfrac{r_{2i'}}{r_{2i_\star-3}}\big)^{\frac{d}{2}-1}\cdot
  \mathbb{P}\big(B(n)\xleftrightarrow{}\partial B(M)\big)
  \le C_\dagger^{-(\frac{d}{2}-1)(i_\star-3)}\cdot
  \mathbb{P}\big(B(n)\xleftrightarrow{}\partial B(M)\big).
  \end{split}
\end{equation*}
Summing over all $0\le i<i'\le\frac{1}{2}i_\star$, we get
\begin{equation}\label{eq:T1}
  \mathbb{P}\big(\mathsf{T}^{(1)}\mid B(n)\xleftrightarrow{}\partial B(M)\big)
  \lesssim i_\star^2\,C_\dagger^{-(\frac{d}{2}-1)(i_\star-3)}\lesssim e^{-c'i_\star}.
\end{equation}

We now consider $\mathsf{T}^{(2)}$. For $k\ge 1$, let
$D_k:=(i_{k+1}-i_k-1)\cdot\mathbbm{1}_{i_{k+1}<i_\star}$ (with $D_k:=0$ if
$k\ge k_\star$). Note that $D_k\ge 0$, and that $D_k$ is measurable with respect to $\mathcal{F}_{\mathcal{C}_{i_k}}$. For $k\ge 1$ and $l\in\mathbb{N}$, on the event $\{D_k\ge l+1\}$, we have $i_k+2+l\le i_{k+1}<i_\star$. Therefore, applying
\eqref{eq:increment_tail} with $i=i_{k-1}$, one has 
\begin{equation*}
  \mathbb{P}\big(B(n)\xleftrightarrow{}\partial B(M),\,D_k\ge l+1\mid
  \mathcal{F}_{\mathcal{C}_{i_{k-1}}}\big)\lesssim   C_\dagger^{-\frac{1}{2}-l}
  \cdot\mathbb{P}\big(B(n)\xleftrightarrow{}\partial B(M)\mid
  \mathcal{F}_{\mathcal{C}_{i_{k-1}}}\big), 
\end{equation*}
 which further implies that for some $C_\diamond>0$ (assuming $C_\dagger >10$), 
\begin{equation}
  \mathbb{E}\big[\mathbbm{1}_{B(n)\xleftrightarrow{}\partial B(M)}\cdot e^{D_k}
  \mid\mathcal{F}_{\mathcal{C}_{i_{k-1}}}\big]\le e^{C_\diamond
  C_\dagger^{-1/2}} \cdot\mathbb{P}\big(B(n)\xleftrightarrow{}\partial
  B(M)\mid\mathcal{F}_{\mathcal{C}_{i_{k-1}}}\big). 
\end{equation}
 Since $D_1,\dots,D_{k-1}$ are measurable with respect to
$\mathcal{F}_{\mathcal{C}_{i_{k-1}}}$, by applying this bound for
$k=m,m-1,\dots,1$ successively and taking $C_\dagger$ sufficiently
large, we obtain
\begin{equation}\label{eq:exp_moment_D}
  \mathbb{E}\big[\mathbbm{1}_{B(n)\xleftrightarrow{}\partial B(M)}\cdot
  e^{\sum_{1\le k\le m}D_k}\big]\le e^{ m} 
  \cdot\mathbb{P}\big(B(n)\xleftrightarrow{}\partial B(M)\big).
\end{equation}
On the event $\mathsf{T}^{(2)}$, since $i_{k+1}\le i_{k_\star-1}<i_\star$ for
all $1\le k\le k_\star-2$,  \begin{equation*}
  i_{k_\star-1}=i_1+\sum\nolimits_{ 1\le k   \le k_\star-2}(i_{k+1}-i_k)=k_\star-1+
  \sum\nolimits_{1\le k   \le k_\star-2} D_k.
\end{equation*}
This together with $i_{k_\star-1}>\frac{1}{2}i_\star$ and $k_\star\le
c_\spadesuit K$ yields the inclusion 
\begin{equation}\label{l475}
\mathsf{T}^{(2)} \subset \Big\{ 	\sum\nolimits_{1\le k\le \lfloor c_\spadesuit K\rfloor } D_k \ge \tfrac{1}{2}i_\star-c_\spadesuit K \Big\}. 
\end{equation}
 Let $c_\spadesuit:=(32\ln C_\dagger)^{-1}$. Since $r_{2i_\star+3}\ge N$, we have
$i_\star\ge\frac{K}{2\ln C_\dagger}-\frac{3}{2}$, and hence $c_\spadesuit K\le  \tfrac{1}{8}i_\star$. Thus, by Markov's inequality, \eqref{eq:exp_moment_D} and (\ref{l475}), we have 
\begin{equation*}
  \begin{split}
    \mathbb{P}\big(B(n)\xleftrightarrow{}\partial B(M),\,\mathsf{T}^{(2)}\big)  
\overset{(\ref{l475})}{ \le}  & \mathbb{P}\big(B(n)\xleftrightarrow{}\partial B(M),  e^{\sum_{1\le k\le \lfloor c_\spadesuit K\rfloor}D_k}\ge e^{\frac{1}{2}i_\star-c_\spadesuit K } \big)  \\
  \overset{(\ref{eq:exp_moment_D})}{ \le} &   e^{-\frac{1}{2}i_\star+c_\spadesuit K  }\cdot e^{ c_\spadesuit K}\cdot
  \mathbb{P}\big(B(n)\xleftrightarrow{}\partial B(M)\big)\\
 \overset{(c_\spadesuit K\le\frac{1}{8}i_\star)}{ \le}  &e^{-\frac{1}{4}i_\star}\cdot\mathbb{P}\big(B(n)\xleftrightarrow{}\partial
  B(M)\big).
  \end{split}
\end{equation*}
Combined with \eqref{eq:T1}, it completes the proof. 
\end{proof}

The next proposition shows that if at each step of the exploration a certain
event occurs with uniformly positive conditional probability, then with high
probability such events occur in a positive proportion of the steps. 

\begin{proposition}\label{prop:exploration}
Let $\{\mathsf{W}_{i,i'}\}_{1\le i<i'\le i_\star}$ be a family of events
satisfying the following two conditions:
\begin{enumerate}

  \item[(W1)] for any $1\le i<i'\le i_\star$, $\mathsf{W}_{i,i'}$ is measurable
  with respect to $\mathcal{F}_{\mathcal{C}_{i'}}$;

  \item[(W2)] there exists $c_{\rm W}>0$ such that for any $1\le i<i_\star$, on
  the event $\{i_+<i_\star\}$,
  \begin{equation}\label{eq:W2}
    \mathbb{P}\big(B(n)\xleftrightarrow{}\partial B(M),\,\mathsf{W}_{i,i_+}\mid
    \mathcal{F}_{\mathcal{C}_i}\big)\ge c_{\rm W}\,\mathbb{P}\big(B(n)
    \xleftrightarrow{}\partial B(M)\mid\mathcal{F}_{\mathcal{C}_i}\big).
  \end{equation}
\end{enumerate}
Define $\mathbf{W}:=\#\{1\le k\le k_\star-2:\mathsf{W}_{i_k,i_{k+1}}\ \text{occurs}\}$.
Then there exist $c_\heartsuit,c'_\heartsuit>0$ such that
\begin{equation}\label{eq:exploration}
  \mathbb{P}\big(\mathbf{W}\le c_\heartsuit i_\star\mid B(n)\xleftrightarrow{}
  \partial B(M)\big)\lesssim(n/N)^{c'_\heartsuit}.
\end{equation}
\end{proposition}
\begin{proof}
	For $k\ge 1$, let $X_k:=\mathbbm{1}_{\mathsf{W}_{i_k,i_{k+1}}}$ and $U_m:=\sum_{k=1}^mX_k$. We also set $U_0:=0$ and $q:=(1-e^{-1})c_{\rm W}$. We claim that for all $m\ge 0$,
\begin{equation}\label{eq:supermartingale}
  \mathbb{E}\big[\mathbbm{1}_{B(n)\xleftrightarrow{}\partial B(M),i_{m+1}<i_\star } \cdot e^{-U_m}\big]\le e^{-qm}\cdot
  \mathbb{P}\big(B(n)\xleftrightarrow{}\partial B(M)\big).
\end{equation}
We prove \eqref{eq:supermartingale} by induction on $m$. The case $m=0$ is trivial. For $m\ge 1$, note that $\mathbbm{1}_{i_{m+1}<i_\star}$ and $U_{m-1}$ are measurable with respect to $\mathcal{F}_{\mathcal{C}_{i_m}}$. Moreover, on $\{i_{m+1}<i_\star\}$, it follows from \eqref{eq:W2} (with $i=i_m$) that 
\begin{equation*}
  \begin{split}
  \mathbb{E}\big[\mathbbm{1}_{B(n)\xleftrightarrow{}\partial B(M)}\cdot e^{-X_m}
  \mid\mathcal{F}_{\mathcal{C}_{i_m}}\big]
  \le e^{-q}\cdot \mathbb{P}\big(B(n)\xleftrightarrow{}\partial B(M)\mid
  \mathcal{F}_{\mathcal{C}_{i_m}}\big).
  \end{split}
\end{equation*}
Therefore, using $\{i_{m+1}<i_\star\}\subset\{i_m<i_\star\}$, we have
\begin{equation*}
  \begin{split}
  \mathbb{E}\big[\mathbbm{1}_{B(n)\xleftrightarrow{}\partial B(M),i_{m+1}<i_\star}\cdot e^{-U_m}\big]
  = &\mathbb{E}\Big[\mathbbm{1}_{i_{m+1}<i_\star}\cdot e^{-U_{m-1}} 
  \mathbb{E}\big[\mathbbm{1}_{B(n)\xleftrightarrow{}\partial B(M)}e^{-X_m}\mid
  \mathcal{F}_{\mathcal{C}_{i_m}}\big]\Big]\\
  \le &e^{-q} \cdot \mathbb{E}\big[\mathbbm{1}_{B(n)\xleftrightarrow{}\partial B(M),i_m<i_\star}
 \cdot e^{-U_{m-1}}\big],
  \end{split}
\end{equation*}
which together with the induction hypothesis implies
\eqref{eq:supermartingale}.

Let $m:=\lfloor c_\spadesuit K\rfloor-2$. On the event
$\mathsf{T}^c_{c_\spadesuit K}$, we have $k_\star>c_\spadesuit K\ge m+2$, and
hence $i_{m+1}<i_\star$ and $\mathbf{W}\ge U_m$. Thus, by Markov's inequality and 
\eqref{eq:supermartingale}, we have 
\begin{equation*}
  \begin{split}
  \mathbb{P}\big(B(n)\xleftrightarrow{}\partial B(M),
  \mathsf{T}^c_{c_\spadesuit K}, \mathbf{W}\le\tfrac{1}{2}qm\big)
  \le &e^{\frac{1}{2}qm}\cdot e^{-qm}\cdot \mathbb{P}\big(B(n)\xleftrightarrow{}\partial
  B(M)\big)\\
  = & e^{-\frac{1}{2}qm}\cdot \mathbb{P}\big(B(n)\xleftrightarrow{}\partial
  B(M)\big).
  \end{split}
\end{equation*}
Since $m\asymp K\asymp i_\star$, there exists $c_\heartsuit>0$ such that
$c_\heartsuit i_\star\le\frac{1}{2}qm$. Combined with
Lemma~\ref{lem:many_steps}, this yields the desired bound  \eqref{eq:exploration}.
\end{proof}

We are now ready to prove Lemma~\ref{newlemma_pivotal}.

\begin{proof}[Proof of Lemma~\ref{newlemma_pivotal}] 
	For $a>0$ and $1\le i<i'\le i_\star$, let $\mathsf{W}^{a}_{i,i'}$ denote the event that there exists $z\in\hat B(2r_{2i'})\setminus\hat B(\frac{1}{2}r_{2i'})$ such that the cluster $\bar{\mathcal{C}}_z:= \mathcal{C}^{\nu,\partial B(z,r_{2i'-2})}_z$ satisfies $\mathrm{cap}(\bar{\mathcal{C}}_z)\ge
ar_{2i'-2}^{d-2}$ and is pivotal for $\{\mathcal{C}_i\xleftrightarrow{}\partial B(r_{2i'})\}$. Compared with $\mathsf{V}^{i'-1}_a$, the event $\mathsf{W}^a_{i,i'}$ imposes the same requirements on the location and the capacity of the cluster, but requires it to be pivotal for $\{\mathcal{C}_i\xleftrightarrow{}\partial
B(r_{2i'})\}$ instead of $\{B(n)\xleftrightarrow{}\partial B(M)\}$. Thus, it follows from (P3) that for any $a>0$ and $1\le i<i'\le i_\star$,
\begin{equation}\label{eq:W_V}
  \{B(n)\xleftrightarrow{}\partial B(M)\}\cap\mathsf{W}^a_{i,i'}\subset
  \mathsf{V}^{i'-1}_a.
\end{equation} 
Therefore, it suffices to show that $\{\mathsf{W}^a_{i,i'}\}$ satisfies (W1) and
(W2) for some $a>0$. Indeed, assume that this holds, and let $\mathbf{W}$ be defined as in Proposition~\ref{prop:exploration} with $\mathsf{W}_{i,i'}:=\mathsf{W}^a_{i,i'}$. On the event $\{B(n)\xleftrightarrow{}\partial B(M)\}$, it follows from \eqref{eq:W_V} that the occurrence of $\mathsf{W}^a_{i_k,i_{k+1}}$ implies that of $\mathsf{V}^{i_{k+1}-1}_a$. As a result, we have $\mathbf{V}_a\ge\mathbf{W}$. Combined with Proposition~\ref{prop:exploration}, this implies
\begin{equation*}
  \mathbb{P}\big(\mathbf{V}_a\le c_\heartsuit i_\star\mid B(n)
  \xleftrightarrow{}\partial B(M)\big)\le\mathbb{P}\big(\mathbf{W}\le
  c_\heartsuit i_\star\mid B(n)\xleftrightarrow{}\partial B(M)\big)\lesssim
  (n/N)^{c'_\heartsuit}.
\end{equation*}
Thus, we obtain \eqref{ineq_lemma_pivotal} with
$\cref{const_pivotal_1}:=a$, $\cref{const_pivotal_2}:=c_\heartsuit$ and
$\cref{const_pivotal_3}:=\frac{1}{2}c'_\heartsuit$.

We next prove (W1) and (W2). We start with (W1), i.e., the event $\mathsf{W}^a_{i,i'}$ is measurable with
respect to $\mathcal{F}_{\mathcal{C}_{i'}}$. For any chain (of loops) connecting $\mathcal{C}_i$ to $\partial B(r_{2i'})$, stopped at the first loop intersecting $\partial B(r_{2i'})$, all loops except the last one are contained in $B(r_{2i'})$, and hence all loops in this chain are included in $\mathcal{C}_{i'}$. Therefore, $\{\mathcal{C}_i\xleftrightarrow{}\partial B(r_{2i'})\}$ occurs if and only if $\mathcal{C}_i$ is connected to $\partial B(r_{2i'})$ by loops included in $\mathcal{C}_{i'}$; in other words, whether $\{\mathcal{C}_i\xleftrightarrow{}\partial B(r_{2i'})\}$ occurs, before or after removing a given collection of loops, is determined by $\mathcal{F}_{\mathcal{C}_{i'}}$. Thus, $\mathsf{W}^a_{i,i'}$ is determined by $\mathcal{F}_{\mathcal{C}_{i'}}$.

 It remains to verify (W2) for a suitable $a>0$, i.e., there
exist $a,c_{\rm W}>0$ such that for any $1\le i<i_\star$, on the event
$\{i_+<i_\star\}$,
\begin{equation}\label{eq:pivotal_cluster}
  \mathbb{P}\big(B(n)\xleftrightarrow{}\partial B(M),\,
  \mathsf{W}^{a}_{i,i_+}\mid\mathcal{F}_{\mathcal{C}_i}\big)\ge
  c_{\rm W}\,\mathbb{P}\big(B(n)\xleftrightarrow{}\partial B(M)\mid
  \mathcal{F}_{\mathcal{C}_i}\big).
\end{equation}
The proof of this bound consists of the following three steps (see Figure~\ref{figure_strategy}):



 \textbf{Step (i).} We condition on $\mathcal{F}_{\mathcal{C}_i}$ and assume that $i_+=i'$. Let $y$ be an arbitrary point in $\hat{\partial}\hat{B}(\frac{3}{2}r_{2i'-1})$, and let $c_\dagger>0$ be a small constant. Let $\mathsf{A}_1$ be the event that there exist two disjoint clusters $\mathcal{C}^{\mathrm{in}}$ and $\mathcal{C}^{\mathrm{out}}$ such that
$\mathcal{C}^{\mathrm{in}}$ connects $B(y,\frac{1}{4}r_{2i'-2})$ to
$\mathcal{C}_i$, $\mathcal{C}^{\mathrm{out}}$ connects
$B(y,\frac{1}{4}r_{2i'-2})$ to $\partial B(M)$, and both
$\mathcal{C}^{\mathrm{in}}$ and $\mathcal{C}^{\mathrm{out}}$ have capacity at least $cr_{2i'-2}^{d-2}$ within $B(y,\frac{1}{2}r_{2i'-2})$ and are disjoint from $B(y,2c_\dagger r_{2i'-2})$. In fact, following the proof of \cite[Lemma 5.3]{cai2025separation}, the conditional probability of $\mathsf{A}_1$ is of the same order as that of $\{B(n)\xleftrightarrow{}\partial B(M)\}$: 
 \begin{equation}\label{eq:separation}
   \mathbb{P}\big(\mathsf{A}_1\mid\mathcal{F}_{\mathcal{C}_i}\big)\gtrsim 
    \mathbb{P}\big(B(n)\xleftrightarrow{}\partial B(M)\mid
   \mathcal{F}_{\mathcal{C}_i}\big).
 \end{equation}
A detailed proof of this bound will be provided in \cite{newonearm}.



  \textbf{Step (ii).} Since $B(y,r_{2i'-2})\cap B(r_{2i'-1})=\emptyset$, it
 follows from (P1) that the loops not included in $\mathcal{C}_i$ and contained in $B(y,r_{2i'-2})$ have the same law as the loops of $\mathcal{L}$ contained in $B(y,r_{2i'-2})$. Let $\mathsf{A}_2$ be the event that there exists $z\in\hat B(y,c_\dagger r_{2i'-2})$ such that
 $\hat{\mathcal{C}}_z:=\mathcal{C}_z^{\nu,\partial B(y,c_\dagger r_{2i'-2})}$ satisfies $\mathrm{cap}(\hat{\mathcal{C}}_z)\ge \cref{const_cap_lower} (c_\dagger r_{2i'-2})^{d-2}$, and that every loop not included in  $\mathcal{C}_i$ and intersecting $B(y,c_\dagger r_{2i'-2})$ is contained in $B(y,2c_\dagger r_{2i'-2})$. Such $z$ must lie in $\hat{B}(2r_{2i'-1})\setminus\hat{B}(\frac{1}{2}r_{2i'-1})$. By the restriction property, $\mathsf{A}_1$ and $\mathsf{A}_2$ are independent, since they depend on disjoint collections of loops. Moreover, by Lemma~\ref{lemma_cap_lower} and (\ref{ineq_crossing_loop}), $\mathsf{A}_2$ occurs with uniformly positive probability.  Therefore, the probability of $\mathsf{A}_1\cap\mathsf{A}_2$ is also comparable to that of $\{B(n)\xleftrightarrow{}\partial B(M)\}$.

  \textbf{Step (iii).} On the event $\mathsf{A}_1\cap\mathsf{A}_2$, if one adds a
 loop $\ell$ of $\mathcal{L}^\nu$ within $B(y,\frac{1}{2}r_{2i'-2})$
 intersecting $\mathcal{C}^{\mathrm{in}}$, $\mathcal{C}^{\mathrm{out}}$ and
 $\hat{\mathcal{C}}_z$, then $\{B(n)\xleftrightarrow{}\partial B(M)\}$ occurs.
 Moreover, since $B(y,\frac{1}{2}r_{2i'-2})\subset B(z,r_{2i'-2})$, the cluster
 $\bar{\mathcal{C}}_z$ contains both $\ell$ and
 $\hat{\mathcal{C}}_z$. Thus, it has capacity at least
 $\cref{const_cap_lower}c_\dagger^{d-2}r_{2i'-2}^{d-2}$, and it is pivotal for
 $\{\mathcal{C}_i\xleftrightarrow{}\partial B(r_{2i'})\}$ (since after removing
 it, the cluster containing $\mathcal{C}_i$ is contained in
 $\mathcal{C}^{\mathrm{in}}$). Since the capacities of
 $\mathcal{C}^{\mathrm{in}}$, $\mathcal{C}^{\mathrm{out}}$ and
 $\hat{\mathcal{C}}_z$ within $B(y,\frac{1}{2}r_{2i'-2})$ are at least
 $cr_{2i'-2}^{d-2}$, adding such a loop changes the probability by only a constant factor.  
  To sum up, these estimates together imply \eqref{eq:pivotal_cluster} with $a:=\cref{const_cap_lower}c_\dagger^{d-2}$, and thus complete the proof.
\end{proof}

\begin{figure}[h]
	\centering
	\includegraphics[width=0.339\textwidth]{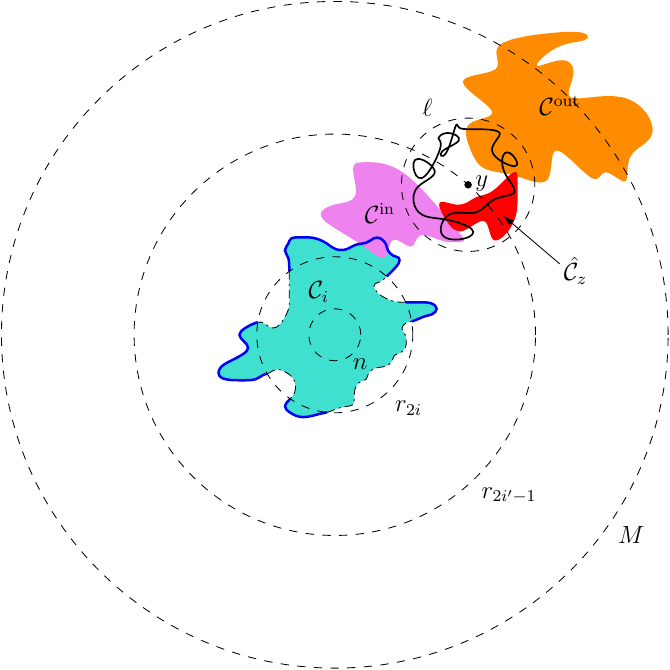}
	\caption{The cyan, pink, red, and orange regions represent $\mathcal{C}_i$, $\mathcal{C}^{\mathrm{in}}$, $\mathcal{C}^{\mathrm{out}}$, and $\hat{\mathcal{C}}_z$ respectively, and the solid curve represents the added loop $\ell$. For the partial cluster $\mathcal{C}_i$, the local time vanishes on the part of its boundary inside $B(r_{2i})$ (dashed), so $\mathcal{C}_i$ cannot be further extended there; in contrast, its part outside $B(r_{2i})$ (blue) can still be extended by loops that do not hit $\mathcal{C}_i\cap B(r_{2i})$.} \label{figure_strategy}
\end{figure}

\section{Convergence of connectivity relations among large loops}\label{section5_convergence}

In this section, we carry out the argument outlined in Section \ref{subsection_scaling_limit}. Roughly speaking, we show that on two different Voronoi networks, the joint distributions of the large loops together with their connectivity relations induced by small loops are close. This comparison constitutes the core of this paper and plays a key role in the proofs of Theorems \ref{new_thm1}, \ref{thm_asymptotic} and \ref{thm_k_point}. For clarity, we begin by fixing some notation. Assume that $3\le d\le 5$, and that $\mathbf{G}_{\mathrm{I}}$ and $\mathbf{G}_{\mathrm{II}}$ are two Voronoi networks. We also take a sufficiently small $\delta>0$. In addition, we choose small parameters $\{\beta_j\}_{1\le j\le 6}$ satisfying the following conditions: 
 \begin{itemize}

 	\item[(i)]  $0<\bl\label{para_2} <  \bl\label{para_error}< \bl\label{para_3}<\bl\label{para_4}< \bl\label{new_6}  < \bl\label{para_1} <(100d)^{-1}\cref{const_hitting}\cref{const_crossingloop}(1-\cref{const_crossingloop_new})$;

 	 \item[(ii)] $\bref{para_3}<(2d)^{-1}\cref{const_check_loop}\bref{para_4}$;

 	\item[(iii)]  $\bref{para_4}< (6d)^{-1}\cref{const_oneloop_onecluster} \bref{para_1}$ and $\bref{para_1}< \frac{d+\cref{const_oneloop_onecluster} }{d+\frac{1}{2} \cref{const_oneloop_onecluster}}\cdot \bref{new_6} $, so that $(\bref{new_6}-\bref{para_4})(d+\cref{const_oneloop_onecluster})>d(\bref{para_2}+\bref{para_1})$.

 \end{itemize} 
Let $L=\delta^{\bref{para_1}-1}$, $M=\delta^{-(\bref{para_2}+\bref{para_1})}$ and $Q_M:=\{-(M-1),-(M-2),...,M-1\}^d$. Let $\{\mathbf{G}_{\xi}\}_{\xi\in\{1,2\}^{\mathbb{Z}^d}}$ be the family of interpolating graphs between $\mathbf{G}_{\mathrm{I}}$ and $\mathbf{G}_{\mathrm{II}}$ constructed in Lemma~\ref{lemma_exist_interpolation}, with $\mathbf{G}_{\mathrm{I}}$ and $\mathbf{G}_{\mathrm{II}}$ playing the roles of $\mathbf{G}_1$ and $\mathbf{G}_2$ there (where $W$ and $Q$ depend only on the parameters of $\mathbf{G}_{\mathrm{I}}$ and $\mathbf{G}_{\mathrm{II}}$, and their precise values are irrelevant). We take a sequence $\{\xi_i\}_{1\le i\le K}$ in $\{\xi \in \{1,2\}^{\mathbb{Z}^d}: \xi(x)=1, \forall x\in Q_M^c\}$ such that $\xi_1\equiv 1$, $\xi_K(x)=2$ for all $x\in Q_M$, and that for every $1\le i\le K-1$, $\xi_i$ and $\xi_{i+1}$ differ at exactly one site. Note that $K\asymp M^d= \delta^{-d(\bref{para_2}+\bref{para_1})}$. For $1\le i\le K$, we abbreviate $\mathbf{G}_i:= \mathbf{G}_{\xi_i}$. Let $\nu=\delta^{2(1-\bref{para_3})}$. We denote the loop soups of intensity $1/2$ on $\mathbf{G}_{i}$ and $\mathbf{G}_{i}^{\nu}$ by $\mathcal{L}_i$ and $\mathcal{L}_i^{\nu}$ respectively (here the subscript $i$ is used to indicate the underlying graph, and the intensity is always $1/2$ unless otherwise specified). We denote by $\mathfrak{C}_i$ the collection of clusters in $\mathcal{L}_i$, and set  
 \begin{equation}\label{notation5.2}
  	\widehat{\mathcal{L}}_i:=(\mathcal{L}_i-\mathcal{L}_i^{\nu} )\cdot \mathbbm{1}_{|\ell|\ge \delta^{\bref{para_4}-1}}. 
 \end{equation} 
These are precisely the ``large loops'' referred to in Section \ref{subsection_scaling_limit}. We consider the loops in $\widehat{\mathcal{L}}_i$ and their connectivity relation induced by the ``small loops'' $\mathcal{L}_i^{\nu}$. (Note that the loop soup also contains loops that are neither large nor small in the aforementioned convention. As mentioned in Section \ref{subsection_scaling_limit}, we will show in the next section that these remaining loops are negligible.) Next, we introduce a distance to measure how close two such configurations are. Let $\mathfrak{T}_\delta$ denote the space of pairs $\mathbf{T}=(\mathcal{T},\sim)$, where $\mathcal{T}$ is a finite collection of compact subsets of $(-LM^{1/2},LM^{1/2})^d$ and $\sim$ is an equivalence relation on $\mathcal{T}$. For any pairs $\mathbf{T}=(\mathcal{T},\sim),\mathbf{T}'=(\mathcal{T}',\sim')\in \mathfrak{T}_\delta$, we say a bijection $\sigma \in \mathrm{bij}(\mathcal{T},\mathcal{T}')$ is admissible if for any $T_1,T_2\in \mathcal{T}$,  
\begin{equation}\label{use5.3}
	T_1 \sim  T_2 \iff  \sigma(T_1) \sim' \sigma(T_2).
\end{equation}
 We define the distance between $\mathbf{T}$ and $\mathbf{T}'$ as follows: if there exists an admissible bijection between $\mathcal{T}$ and $\mathcal{T}'$, then we set  
\begin{equation}\label{newuse5.4}
	\mathrm{d}_{\mathfrak{T}}(\mathbf{T},\mathbf{T}') := \min_{\text{admissible}\   \sigma \in \mathrm{bij}(\mathcal{T},\mathcal{T}') }  \max\limits_{T\in \mathcal{T} }   \mathrm{d}_{\mathrm{H}}\big(T ,\sigma(T) \big);
\end{equation} 
otherwise, $\mathrm{d}_{\mathfrak{T}}(\mathbf{T},\mathbf{T}'):=+\infty$. For each $1\le i\le K$, we define
\begin{equation}\label{def_Ti_5.5}
	\mathbf{T}_i=(\widehat{\mathcal{L}}_i\cdot \mathbbm{1}_{\mathrm{ran}(\ell)\subset (-LM^{1/2},LM^{1/2})^d},\sim_i),
\end{equation}
 where $\ell\sim_i\ell'$ if and only if $\ell$ and $\ell'$ belong to the same cluster of $\widehat{\mathcal{L}}_i+\mathcal{L}_i^{\nu}$.

   The following proposition is the main result in this section.

 \begin{proposition}\label{prop_compare_relation}
  	For each $1\le i\le K-1$, there exist a coupling between $\mathcal{L}_i$ and $\mathcal{L}_{i+1}$ and a constant $\cl\label{const_compare_gluing}>0$ (depending on the parameters) such that 
  	\begin{equation}\label{use5.5}
 		\mathbb{P}\big( \mathrm{d}_{\mathfrak{T}}(\mathbf{T}_i,\mathbf{T}_{i+1}) > \delta^{\cref{const_crossingloop_new}(\bref{para_1}-1)} \big)\le \delta^{d(\bref{para_2}+\bref{para_1})+ \cref{const_compare_gluing}}. 
 	\end{equation}
 \end{proposition}

 To construct such a coupling, we need the following lemma, which is an extension of Strassen's theorem (see e.g., \cite[Theorem 2.5]{leskela2010stochastic}).
 \begin{lemma}\label{lemma_Strassen}
 	Let $S$ be a Polish space equipped with a distance $D:S\times S\to [0,\infty]$, and let $X,Y$ be two random variables on $S$. Fix $r\in [0,\infty)$ and $\epsilon \in [0,1]$. For any $A\subset S$, we define its $r$-neighborhood as $A^{r}:=\{y\in S:\inf\nolimits_{x\in A}D(x,y)\le r \}$. 
 If $\mathbb{P}(X\in A)\le \mathbb{P}(Y\in A^{r}) + \epsilon$ holds for all compact $A\subset S$, then there exists a coupling between $X$ and $Y$ such that 
 	\begin{equation}
 		\mathbb{P}\big(  D(X,Y) >r \big)\le \epsilon.
 	\end{equation}
 \end{lemma}

  To maintain the flow of the exposition, we defer the measure-theoretic proof of Lemma \ref{lemma_Strassen} to the companion paper \cite{newonearm}.

 \subsection{Proof of Proposition \ref{prop_compare_relation}}

 By Lemma \ref{lemma_Strassen}, it suffices to show that 
 \begin{equation}\label{addnew5.9}
 	\mathbb{P}\big(\mathbf{T}_i\in A\big)\le \mathbb{P}\big(\mathbf{T}_{i+1}\in A^{\delta^{\cref{const_crossingloop_new}(\bref{para_1}-1)}}\big) + \delta^{d(\bref{para_2}+\bref{para_1})+ \cref{const_compare_gluing}}  
 \end{equation}
 for all compact $A\subset \mathfrak{T}_\delta$. Suppose that $\mathbf{G}_{i}$ and $\mathbf{G}_{i+1}$ differ only within the box $\mathbf{B}_i:=z_i+[0, \delta^{\bref{para_1}-1})^d$. By Lemma \ref{lemma_coupling_crossingloop}, there exists a coupling between $\mathcal{L}_i$ and $\mathcal{L}_{i+1}$ such that with probability $1-O(\delta^{\cref{const_crossingloop} (1-\bref{para_1})})$, the following events occur:  
           \begin{itemize}

           	\item[(a)]  The loops in $\mathcal{L}_i$ and $\mathcal{L}_{i+1}$ that do not intersect $\mathbf{B}_i$ coincide.

           	\item[(b)] There exists a bijection between loops of $\mathcal{L}_i$ and $\mathcal{L}_{i+1}$ crossing the annulus $B(z_i, 10d\delta^{\bref{para_1}-1})\setminus \mathbf{B}_i$ such that for every pair of corresponding loops $\ell_i\in\mathcal{L}_i$ and $\ell_{i+1}\in\mathcal{L}_{i+1}$, they coincide outside $B(z_i, 5d\delta^{\bref{para_1}-1})$ and their Hausdorff distance is at most $\delta^{\cref{const_crossingloop_new}(\bref{para_1}-1)}$. Moreover, $\ell_i\in\mathcal{L}_i^{\nu}$ if and only if $\ell_{i+1}\in\mathcal{L}_{i+1}^{\nu}$.

           \end{itemize} 
           We denote the intersection of these two events by $\mathsf{A}$. We also denote 
  \begin{equation}
 	\mathsf{T}_i:= \{\mathbf{T}_i\in A \} \ \text{and}\ \mathsf{T}_{i+1}:= \{\mathbf{T}_{i+1}\in A^{\delta^{\cref{const_crossingloop_new}(\bref{para_1}-1)}} \}.
 \end{equation}
 Therefore, by the union bound, we have 
 \begin{equation}\label{addnew511}
 	           	\mathbb{P}(\mathsf{T}_i )\le 	  \mathbb{P}(\mathsf{A}\cap \mathsf{T}_i)+   C\delta^{\cref{const_crossingloop} (1-\bref{para_1})} \le  	  \mathbb{P}(\mathsf{A}\cap \mathsf{T}_i) +\tfrac{1}{3}\delta^{d(\bref{para_2}+\bref{para_1})+\cref{const_compare_gluing}}.   
 \end{equation} 
   Here we used the fact that $\cref{const_crossingloop} (1-\bref{para_1})>d(\bref{para_2}+\bref{para_1})$ (which follows from Condition (i)). Let $\mathring{\mathcal{L}}_i:=\{ \ell \in \widehat{\mathcal{L}}_i: \mathrm{ran}(\ell)\cap B(z_i,\delta^{\bref{new_6}-1})\neq \emptyset \}$. We then divide the event $\mathsf{A}\cap \mathsf{T}_i$ into the following two sub-events: 
                 \begin{equation}
       	\mathsf{D}_1:= \mathsf{A} \cap \mathsf{T}_i  \cap   \{ \mathring{\mathcal{L}}_i \neq \emptyset  \}\ \text{and}\  	\mathsf{D}_2:= \mathsf{A} \cap \mathsf{T}_i  \cap   \{ \mathring{\mathcal{L}}_i = \emptyset  \}. 
       \end{equation} 
      In what follows, we estimate the probabilities of $\mathsf{D}_1$ and $\mathsf{D}_2$ separately.

    \textbf{Estimate for $\mathbb{P}(\mathsf{D}_1)$.} For any $\ell \in \mathring{\mathcal{L}}_i$, we denote by $\ell^{\star}$ its corresponding loop in $\widehat{\mathcal{L}}_{i+1}$ under the aforementioned coupling. On Event (b), if $\ell^{\star}\notin \widehat{\mathcal{L}}_{i+1}$ (i.e., $|\ell^{\star}|<\delta^{\bref{para_4}-1}$), then we have 
     \begin{equation}\label{event3.12}
     	\mathrm{ran}(\ell) \subset  B(z_i, \delta^{\bref{para_4}-1}+ 2\delta^{\cref{const_crossingloop_new}(\bref{para_1}-1)} )\  \text{and}\ 	\mathrm{ran}(\ell) \not\subset  B( z_i, \delta^{\bref{para_4}-1}- 2\delta^{\cref{const_crossingloop_new}(\bref{para_1}-1)} ). 
     \end{equation}  
  For (\ref{event3.12}) to occur, $\ell$ must contain a random walk trajectory that starts from $\hat{\partial}\hat{B}(z_i ,\delta^{\bref{para_4}-1}- 2\delta^{\cref{const_crossingloop_new}(\bref{para_1}-1)} )$ and hits $B(z_i,\delta^{\bref{para_1}-1})$ before exiting $B(z_i, \delta^{\bref{para_4}-1} + 2\delta^{\cref{const_crossingloop_new}(\bref{para_1}-1)})$, whose probability is $O(\delta^{\cref{const_crossingloop_new}(\bref{para_1}-1)-(\bref{para_4}-1)})$. Therefore, 
  \begin{equation}\label{use5.9}
  	\begin{split}
  		\mathbb{P}(\mathsf{E}):= \mathbb{P}\big(\exists \ell \in \mathring{\mathcal{L}}_i\ \text{such that}\  \ell^{\star}\notin \widehat{\mathcal{L}}_{i+1} \big)  
  		\lesssim  \delta^{\cref{const_crossingloop_new}(\bref{para_1}-1)-(\bref{para_4}-1)} \le \delta^{1-\cref{const_crossingloop_new}-\bref{para_4}}. 
  	\end{split}
  \end{equation}
For $j\in \{i,i+1\}$, we define $\mathsf{F}_j$ to be the event that there exist a loop $\ell\in \mathring{\mathcal{L}}_{j}$ and a loop cluster $\mathcal{C}$ in $\widehat{\mathcal{L}}_{j}$ both crossing the annulus $B(z_i, \frac{1}{4}\delta^{\bref{para_4}-1})\setminus B(z_i,\delta^{\bref{new_6}-1})$ such that $\mathrm{ran}(\ell)\cap \mathcal{C}=\emptyset$. It follows from Lemma \ref{lemma_oneloop_onecluster} that for each $j\in \{i,i+1\}$, 
  \begin{equation}\label{use5.10}
    \mathbb{P}(\mathsf{F}_{j})\lesssim \delta^{(\bref{new_6}-\bref{para_4})(d+\cref{const_oneloop_onecluster})}. 
  \end{equation}
Moreover, on the event $\mathsf{A}\cap \{ \mathring{\mathcal{L}}_i \neq \emptyset  \} \cap\mathsf{E}^c\cap \mathsf{F}_{i}^c\cap \mathsf{F}_{i+1}^c$, the cluster $\mathcal{C}_i$ (resp. $\mathcal{C}_{i}^{\star}$) containing $\ell$ (resp. $\ell^{\star}$) is the unique cluster in $\mathfrak{C}_i$ (resp. $\mathfrak{C}_{i+1}$) that crosses the annulus $B(z_i, \frac{1}{4}\delta^{\bref{para_4}-1})\setminus B(z_i,\delta^{\bref{new_6}-1})$. Indeed, if there were another crossing cluster (see Figure~\ref{figure_three_cases}), then $\mathsf{F}_i$ or $\mathsf{F}_{i+1}$ would occur. If a cluster of large loops is affected by the graph modification, then it must hit $\mathbf{B}_i$ and hence cross the annulus $B(z_i, \frac{1}{4}\delta^{\bref{para_4}-1})\setminus B(z_i,\delta^{\bref{new_6}-1})$. Therefore, $\mathcal{C}_i$ and $\mathcal{C}_{i}^{\star}$ are the only affected clusters on their respective graphs. Moreover, the connectivity relation among the large loops within these two clusters remains unchanged---all of them are still connected to $\ell$ (resp. $\ell^{\star}$). Consequently, (\ref{use5.3}) holds for $\mathbf{T}_i$ and $\mathbf{T}_{i+1}$. This together with Event (b) yields that $\mathrm{d}_{\mathfrak{T}}(\mathbf{T}_i,\mathbf{T}_{i+1}) \le  \delta^{\cref{const_crossingloop_new}(\bref{para_1}-1)}$. To sum up, we have
\begin{equation}\label{addnew516}
	\begin{split}
			\mathbb{P}(\mathsf{D}_1)\le  &  \mathbb{P}(\mathsf{T}_{i+1} ,\mathring{\mathcal{L}}_i \neq \emptyset  )+ C  \delta^{ (1-\cref{const_crossingloop_new}-\bref{para_4})\land [(\bref{new_6}-\bref{para_4})(d+\cref{const_oneloop_onecluster})]}   \\
			<&   \mathbb{P}(\mathsf{T}_{i+1},\mathring{\mathcal{L}}_i \neq \emptyset)+ \tfrac{1}{3}\delta^{d(\bref{para_2}+\bref{para_1})+\cref{const_compare_gluing}}.  
	\end{split}
\end{equation}
  Here we used the fact that $1-\cref{const_crossingloop_new}-\bref{para_4}$ and $(\bref{new_6}-\bref{para_4})(d+\cref{const_oneloop_onecluster})$ are both greater than $d(\bref{para_2}+\bref{para_1})$, which can be derived from Conditions (i) and (iii).

   \begin{figure}[h]
	\centering
	\includegraphics[width=0.325 \textwidth]{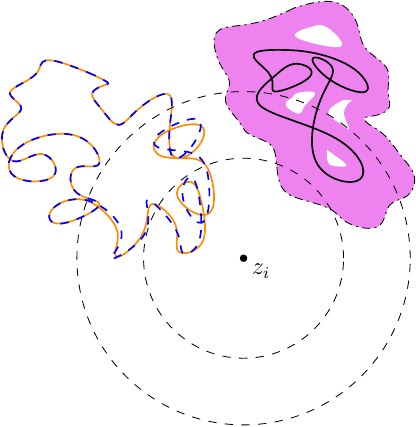}
	\caption{The orange solid curve and the blue dashed curve represent $\ell$ and $\ell^{\star}$ respectively. Under the coupling of Lemma \ref{lemma_coupling_crossingloop}, they typically coincide outside $B(z_i, 5d\delta^{\bref{para_1}-1})$ and remain close to each other inside it. In the situation illustrated here, if there is a crossing cluster other than the cluster $\mathcal{C}_i$ (resp. $\mathcal{C}_i^{\star}$) containing $\ell$ (resp. $\ell^{\star}$), then the event $\mathsf{F}_i$ or $\mathsf{F}_{i+1}$ occurs.}\label{figure_three_cases}
\end{figure}

         \textbf{Estimate for $\mathbb{P}(\mathsf{D}_2)$.} On the event $\mathsf{A}\cap \{\mathring{\mathcal{L}}_i = \emptyset  \}$, $\widehat{\mathcal{L}}_i$ and $\widehat{\mathcal{L}}_{i+1}$ coincide. For each $\ell \in \widehat{\mathcal{L}}_i$, we define $\mathcal{C}_\ell^{\pm}$ as the partial sign cluster outside $B(z_i, \delta^{\bref{new_6}-1})$ containing $\ell$, i.e., the collection of points $v$ that can be connected to $\ell$ by a path disjoint from $B(z_i, \delta^{\bref{new_6}-1})$ along which $\phi$ has a constant sign. Note that $\mathcal{C}_\ell^{\pm}$ is measurable with respect to the occupation field (i.e., it can be determined without knowing the sign of the GFF). Let $\widehat{\mathcal{C}}^{\pm}:= \cup_{\ell \in \widehat{\mathcal{L}}_i}\mathcal{C}_\ell^{\pm}$. On the event $\mathsf{A}$, the occupation fields of $\mathcal{L}_i^\nu$ and $\mathcal{L}_{i+1}^\nu$ coincide outside $B(z_i, 5d\delta^{\bref{para_1}-1})$, and hence $\widehat{\mathcal{C}}^{\pm}$ is the same for $\mathcal{L}_i$ and $\mathcal{L}_{i+1}$ (we denote this event by $\mathsf{U}$). Next, we assume that the event $\mathsf{U}$ occurs and consider the following three cases separately according to the number $\mathcal{N}$ of sign clusters in $\widehat{\mathcal{C}}^{\pm}$ intersecting $\hat{\partial}\hat{B}(z_i, \delta^{\bref{new_6}-1})$: (1) $\mathcal{N}\le 1$; (2) $\mathcal{N}=2$; (3) $\mathcal{N}\ge 3$.

When $\mathcal{N}\le 1$ (see Figure~\ref{figure_case1}), except for the large loops in the partial cluster hitting $B(z_i,\delta^{\bref{new_6}-1})$, the sign clusters of all other large loops are fully sampled. Hence, no matter how the rest of this partial cluster is sampled, the connectivity relations among the loops in $\widehat{\mathcal{L}}_i$ remain unchanged. When $\mathcal{N}\ge 3$, there exist three disjoint sign clusters in $\widehat{\mathcal{C}}^{\pm}$ intersecting $\hat{\partial}\hat{B}(z_i, \delta^{\bref{new_6}-1})$; by Lemma~\ref{lemma_threecluster}, this event has probability at most $C\delta^{(\bref{new_6}- \bref{para_4})(d+\cref{const_oneloop_onecluster})}$. Combining these estimates, we obtain 
                 \begin{equation}\label{newineq420}
      	\begin{split}
      		\mathbb{P}\big( \mathsf{D}_{2}\big) = & \mathbb{P}\big( \mathsf{D}_{2}, \mathsf{U} ,\mathcal{N}\le 1 \big)  + \mathbb{P}\big( \mathsf{D}_{2}, \mathsf{U} ,\mathcal{N}=2\big)+ \mathbb{P}\big( \mathsf{D}_{2}, \mathsf{U} ,\mathcal{N}\ge 3 \big)\\
      		\le  & \mathbb{P}\big( \mathsf{T}_{i+1 } , \mathring{\mathcal{L}}_i= \emptyset  , \mathcal{N}\le 1 \big) + \mathbb{P}\big( \mathsf{D}_{2}, \mathsf{U} ,\mathcal{N}=2\big)+ C \delta^{(\bref{new_6}- \bref{para_4})(d+\cref{const_oneloop_onecluster})}  . 
      	\end{split}
      \end{equation}

  \begin{figure}[h]
	\centering
	\includegraphics[width=0.323 \textwidth]{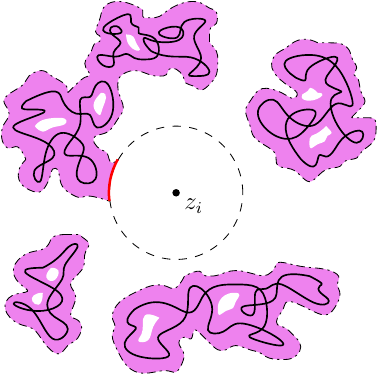}
	\caption{As illustrated here, regardless of how the rest of the partial cluster hitting $B(z_i, \delta^{\bref{new_6}-1})$ (the red arcs) is sampled, it cannot connect to any other large loop. This is because the clusters of those large loops are fully determined and cannot be extended further.}\label{figure_case1}
\end{figure}

      It remains to consider the case when $\mathcal{N}=2$, i.e., there are exactly two clusters $\mathcal{C}^*_1,\mathcal{C}^*_2$ in $\widehat{\mathcal{C}}^{\pm}$ intersecting $\hat{\partial}\hat{B}(z_i, \delta^{\bref{new_6}-1})$.     Note that $\big| \frac{q_1}{2-q_1} -\frac{q_2}{2-q_2}\big|\le 2|q_1-q_2|$ holds for all $q_1,q_2\in [0,1]$. Therefore, by Lemma \ref{newlemma2.3} and (\ref{formula_LW}), given $\widehat{\mathcal{C}}^{\pm}$ and the values of the occupation field on the boundary of $\widehat{\mathcal{C}}^{\pm}$, the difference between the conditional probabilities of $\{\mathcal{C}^*_1\xleftrightarrow{} \mathcal{C}^*_2\}$ on $\widetilde{\mathbf{G}}_i$ and $\widetilde{\mathbf{G}}_{i+1}$ is bounded by 
       \begin{equation}\label{bound420}
       2	\big| e^{-2S^i } - e^{- 2S^{i+1}}\big|,
       \end{equation}
       where $S^j:=\sum_{v_1\in \widetilde{\partial}\mathcal{C}^*_1, v_2\in \widetilde{\partial}\mathcal{C}^*_2 } \mathbb{K}^j_{\widehat{\mathcal{C}}^{\pm}}(v_1,v_2)\phi_{v_1} \phi_{v_2}$ and $\mathbb{K}^{j}_\cdot(\cdot, \cdot)$ is the boundary excursion kernel on $\widetilde{\mathbf{G}}_j$, for $j\in \{i,i+1\}$. Moreover, it follows from Lemma \ref{lemma_hitting_error} that     
        \begin{equation}
     	\big| \mathbb{K}^i_{\widehat{\mathcal{C}}^{\pm}}(v_1,v_2)-  \mathbb{K}^{i+1}_{\widehat{\mathcal{C}}^{\pm}}(v_1,v_2)  \big| \lesssim \delta^{\cref{const_hitting}(1-\bref{para_1})}   \big( \mathbb{K}^i_{\widehat{\mathcal{C}}^{\pm}}(v_1,v_2) + \mathbb{K}^{i+1}_{\widehat{\mathcal{C}}^{\pm}}(v_1,v_2) \big)      \end{equation}
     	for all $v_1\in \widetilde{\partial}\mathcal{C}^*_1$ and $v_2\in \widetilde{\partial}\mathcal{C}^*_2$. As a result,  
     	\begin{equation}\label{bound422}
     		\big|S^i-S^{i+1}\big|\lesssim \delta^{\cref{const_hitting}(1-\bref{para_1})}\big( S^i+ S^{i+1} \big) \lesssim \delta^{\cref{const_hitting}(1-\bref{para_1})} \cdot \max \big\{ S^i, S^{i+1}   \big\}. 
     	\end{equation}
     	Meanwhile, for any $t_1,t_2\ge 0$, one has 
     	\begin{equation}\label{421}
     		\big| e^{-t_1} -e^{-t_2} \big|= \big| \int_{t_1}^{t_2} e^{-s} \mathrm{d}s \big| \le  |t_1-t_2|. 
     	\end{equation}
     	 By (\ref{bound420}), (\ref{bound422}) and (\ref{421}), the difference between the conditional probabilities of $\{\mathcal{C}^*_1\xleftrightarrow{} \mathcal{C}^*_2\}$ on $\widetilde{\mathbf{G}}_i$ and $\widetilde{\mathbf{G}}_{i+1}$ is $O( \delta^{\cref{const_hitting}(1-\bref{para_1})})$. Moreover, if the occurrence of $\{\mathcal{C}^*_1\xleftrightarrow{} \mathcal{C}^*_2\}$ remains unchanged, then the event $\mathsf{T}_i$ implies $\mathsf{T}_{i+1}$. Consequently,       
     	  \begin{equation*}
 	\begin{split}
 		&  \mathbb{P}\big( \mathsf{D}_{2}, \mathsf{U} ,\mathcal{N}=2\big) -  \mathbb{P}\big( \mathsf{T}_{i+1}, \mathring{\mathcal{L}}_i = \emptyset ,\mathcal{N}=2\big) \\
 		\le & \mathbb{E}\big[ \mathbbm{1}_{\mathsf{U} ,\mathcal{N}=2 } \cdot  \big| \mathbb{P}\big(\mathcal{C}^*_1\xleftrightarrow{} \mathcal{C}^*_2\ \text{for}\ \mathcal{L}_i^{\nu} \mid \widehat{\mathcal{C}}^{\pm} \big) - \mathbb{P}\big( \mathcal{C}^*_1\xleftrightarrow{} \mathcal{C}^*_2\ \text{for}\ \mathcal{L}_{i+1}^{\nu}  \mid  \widehat{\mathcal{C}}^{\pm} \big) \big|  \big]  		\lesssim   \delta^{\cref{const_hitting}(1-\bref{para_1})}. 
 	\end{split}
 \end{equation*}
 Combined with (\ref{newineq420}), it implies 
 \begin{equation}\label{427}
 \begin{split}
 	 	\mathbb{P}\big( \mathsf{D}_{2} \big)- \mathbb{P}\big( \mathsf{T}_{i+1} ,\mathring{\mathcal{L}}_i = \emptyset\big)  \le C  \delta^{  [(\bref{new_6}- \bref{para_4})(d+\cref{const_oneloop_onecluster})] \land [\cref{const_hitting}(1-\bref{para_1})] }   
 	 	<    \tfrac{1}{3}\delta^{d(\bref{para_2}+\bref{para_1})+\cref{const_compare_gluing}}. 
 \end{split}
 \end{equation} 
  	As in (\ref{addnew516}), here we used the fact that $(\bref{new_6}- \bref{para_4})(d+\cref{const_oneloop_onecluster})$ and $\cref{const_hitting}(1-\bref{para_1})$ are both greater than $d(\bref{para_2}+\bref{para_1})$ (by Conditions (i) and (iii)). Combining (\ref{addnew511}), (\ref{addnew516}) and (\ref{427}), we obtain (\ref{addnew5.9}) and thus complete the proof.    \qed

\subsection{Extensions of Proposition \ref{prop_compare_relation}}

Fix $k$ distinct points $x_1,...,x_k\in \mathbb{R}^d$. We retain the notation $\mathbf{G}_{\mathrm{I}}$, $\mathbf{G}_{\mathrm{II}}$, $\delta$, $L$, $M$, $Q_M$, $\{\mathbf{G}_{\xi}\}_{\xi\in \{1,2\}^{\mathbb{Z}^d}}$, $\mathfrak{T}_\delta$ and $\mathrm{d}_{\mathfrak{T}}$ introduced above. Let $U_k:=\{x \in \mathbb{Z}^d:  x+[-L,L]^d\subset \delta^{-1}\cdot \cup_{1\le j\le k}B(x_j, \delta^{\bref{new_6}}) \}$, and let $\xi_{\star}:\mathbb{Z}^d\to \{1,2\}$ satisfy $\xi_{\star}(x)=1$ if and only if $x\in Q_M^c\cup U_k$. We define $\Xi$ to be the collection of $\xi \in \{1,2\}^{\mathbb{Z}^d}$ such that $\xi(x)=1$ for all $x\in Q_M^c\cup U_k$. Note that $\xi_{\star}\in \Xi$. We enumerate the elements in $\Xi$ by $\{\xi_i\}_{1\le i\le J}$ (and abbreviate $\mathbf{G}_i:= \mathbf{G}_{\xi_i}$) such that $\mathbf{G}_1=\mathbf{G}_{\mathrm{I}}$, $\mathbf{G}_J=\mathbf{G}_{\xi_{\star}}$ and for every $1\le i\le J-1$, $\xi_i$ and $\xi_{i+1}$ differ at exactly one site. We denote by $\bar{x}_j$ the closest point on $\mathbf{G}_{\mathrm{I}}$ to $\delta^{-1}\cdot x_j$. The loop soups $\mathcal{L}_i$, $\mathcal{L}_i^{\nu}$ and $\widehat{\mathcal{L}}_i$ are defined as in (\ref{notation5.2}).

 
Let $\overline{\mathfrak{T}}_{\delta}$ be the space of triplets $(\mathcal{T},\sim, \mathsf{C})$, where $(\mathcal{T},\sim)\in\mathfrak{T}_\delta$, and $\mathsf{C}$ is an equivalence class of $\sim$ (so $\mathsf{C}$ is a collection of sets in $\mathcal{T}$). For any $\overline{\mathbf{T}}=(\mathcal{T} ,\sim, \mathsf{C}),\overline{\mathbf{T}}'=(\mathcal{T}',\sim',\mathsf{C}')\in \overline{\mathfrak{T}}_\delta$, we say a bijection $\sigma \in \mathrm{bij}(\mathcal{T},\mathcal{T}')$ is admissible if 
\begin{equation}
	 T\in \mathsf{C} \iff \sigma(T)\in \mathsf{C}'
\end{equation}
and for any $T_1,T_2\in \mathcal{T}$, 
\begin{equation}
 T_1 \sim  T_2 \iff \sigma(T_1) \sim' \sigma(T_2). 
\end{equation}
If there exists a bijection in $\mathrm{bij}(\mathcal{T},\mathcal{T}')$
that is admissible with respect to $\overline{\mathbf{T}}$ and
$\overline{\mathbf{T}}'$, then we define their distance by 
\begin{equation}
	\mathrm{d}_{\overline{\mathfrak{T}}}(\overline{\mathbf{T}},\overline{\mathbf{T}}'):= \min_{\text{admissible}\ \sigma\in \mathrm{bij}(\mathcal{T},\mathcal{T}') }  \max\limits_{T\in \mathcal{T} }   \mathrm{d}_{\mathrm{H}}\big(T ,\sigma(T) \big);
\end{equation}
otherwise, $\mathrm{d}_{\overline{\mathfrak{T}}}(\overline{\mathbf{T}},\overline{\mathbf{T}}'):=+\infty$. Recall the definition of $\mathbf{T}_{i}$ in (\ref{def_Ti_5.5}). For $1\le i\le J$, if there exists an equivalence class $\mathsf{C}_i$ of $\sim_i$ such that the corresponding cluster
$\mathcal{C}_i$ in $\widehat{\mathcal{L}}_i+\mathcal{L}_i^{\nu}$ contains all the points $\bar{x}_1,...,\bar{x}_k$, we define  
\begin{equation}
	\overline{\mathbf{T}}_{i}=(\widehat{\mathcal{L}}_i\cdot \mathbbm{1}_{\mathrm{ran}(\ell)\subset (-LM^{1/2},LM^{1/2})^d}  ,\sim_i , \mathsf{C}_i);
\end{equation}
otherwise, we set $\overline{\mathbf{T}}_{i}=\emptyset$. When $\overline{\mathbf{T}}_{i}=\emptyset$ or $\overline{\mathbf{T}}_{i+1}=\emptyset$, we set $\mathrm{d}_{\overline{\mathfrak{T}}}(\overline{\mathbf{T}}_i,\overline{\mathbf{T}}_{i+1})=\infty$. Note that $\overline{\mathbf{T}}_{i}\neq \emptyset$ implies the event $\bar{x}_1\xleftrightarrow{} \bar{x}_2 \xleftrightarrow{} ... \xleftrightarrow{} \bar{x}_k$ (here ``$\xleftrightarrow{}$'' denotes connectivity via loops in  $\widehat{\mathcal{L}}_i+\mathcal{L}_i^{\nu}$; so does in Lemma \ref{lemma_extend_k_point} below).

The following lemma, which extends Proposition \ref{prop_compare_relation}, will play an important role in the proofs of Theorems \ref{thm_asymptotic} and \ref{thm_k_point}.

 \begin{lemma}\label{lemma_extend_k_point}
 	   	For each $1\le i\le J-1$, there exist a coupling between $\mathcal{L}_i$ and $\mathcal{L}_{i+1}$ and a constant $\cl\label{const_for_k_point}>0$ such that the following hold: 
 	   	
 	   	\begin{enumerate}

 	   		\item[(1)] When $k=1$, 	
 	   		  	\begin{equation}\label{use5.28}
 		\mathbb{P}\big( \bar{x}_1\xleftrightarrow{} \partial B(\bar{x}_1, \tfrac{\delta^{-1}}{100}) , \mathrm{d}_{\overline{\mathfrak{T}}}(\overline{\mathbf{T}}_i,\overline{\mathbf{T}}_{i+1}) > \delta^{\cref{const_crossingloop_new}(\bref{para_1}-1)} \big)\le \delta^{\frac{d}{2}-1+d(\bref{para_2}+\bref{para_1})+ \cref{const_for_k_point}}. 
 	         \end{equation}

 	   		\item[(2)]  When $k\ge 2$, 
 	   			\begin{equation}\label{use5.29}
 		\mathbb{P}\big(  \mathrm{d}_{\overline{\mathfrak{T}}}(\overline{\mathbf{T}}_i,\overline{\mathbf{T}}_{i+1}) > \delta^{\cref{const_crossingloop_new}(\bref{para_1}-1)} \big)\le \delta^{(\frac{d}{2}-1)k+ d(\bref{para_2}+\bref{para_1})+ \cref{const_for_k_point}}. 
 	\end{equation}
 	   		
 	   	\end{enumerate}
 	   	 
 \end{lemma}

 The proof of Lemma \ref{lemma_extend_k_point} follows essentially the same lines as that of Proposition \ref{prop_compare_relation}. More precisely, assume that $\mathbf{G}_{i}$ and $\mathbf{G}_{i+1}$ differ only within $z_i+[0, \delta^{\bref{para_1}-1})^d$. We consider separately the cases when a large loop intersects $B(z_i,\delta^{\bref{new_6}-1})$ and when no such loop exists. In the former case, we use an estimate analogous to Lemma \ref{lemma_oneloop_onecluster} to bound the probability that the connectivity relations among large loops change. (For Item (1), the only difference from that estimate is the extra factor $\delta^{\frac{d}{2}-1}$ on the right-hand side of (\ref{use5.28}), which comes from the additional requirement $\{\bar{x}_1\xleftrightarrow{} \partial B(x_1,\delta^{-1})\}$ and the one-arm upper bound in (\ref{one_arm_lowd}); for Item (2), the corresponding extra factor is $\delta^{(\frac{d}{2}-1)k}$, matching the exact order of the $k$-point function.) In the latter case, we explore the large loops and their sign clusters, where the exploration of these sign clusters is stopped upon reaching $B(z_i, \delta^{\bref{new_6}-1})$. We then consider separately three cases according to the number of these sign clusters that intersect $B(z_i,\delta^{\bref{new_6}-1})$. If there is at most one such cluster, the connectivity relations among large loops remain unchanged. The probability that there are at least three such clusters can be bounded using the counterpart of Lemma \ref{lemma_threecluster}. When there are exactly two, we use the Lupu--Werner formula (\ref{formula_LW}) to bound the difference between their connecting probabilities before and after the graph modification. Since the arguments here are entirely analogous, we omit further details.

 \subsection{Construction of the probability measure in Theorem \ref{new_thm1}}\label{subsection5.3_construction}

 We use the following lemma to establish the existence of the limiting distribution.

\begin{lemma}[{\cite[Lemma 5.3.4]{ambrosio2005gradient}}]\label{lemma_gluing}
	For a Polish space $S$, let $\{X_n\}_{n\ge 1}$ be a sequence of random variables on $S$. Suppose that for each $n\ge 1$, there is a coupling $\pi_n$ between $X_n$ and $X_{n+1}$. Then there exists a probability measure $\mathbb{P}$ carrying $\{X_n\}_{n\ge 1}$ such that for any $n\ge 1$, the joint distribution of $X_n$ and $X_{n+1}$ is $\pi_n$.  
\end{lemma}

This lemma allows us to construct a limiting distribution as follows.  

\begin{corollary}\label{coro_limit_exist}
We maintain the notation in Lemma \ref{lemma_gluing}. Assume that  $D:S\times S\to [0,\infty]$ is a distance on $S$; in addition, there exist $a_n,b_n\in \mathbb{R}^+$ for $n\ge 1$ such that $\sum_{n\ge 1}a_n<\infty$, $\sum_{n\ge 1}b_n<\infty$ and for each $n\ge 1$,  
\begin{equation}
	\pi_n\big( D(X_n,X_{n+1})>a_n \big)\le b_n. 
\end{equation}
Then under the probability measure $\mathbb{P}$ constructed in Lemma \ref{lemma_gluing}, $X_n$ is almost surely a Cauchy sequence. Moreover, if $D$ is complete, then there exists a random variable $X_\infty$ such that $X_n$ converges to $X_\infty$ in distribution. 
\end{corollary}
\begin{proof}
	The proof is a direct application of the Borel--Cantelli lemma. 
\end{proof}

In what follows, we construct the probability measure in Theorem \ref{new_thm1}. Let $\mathsf{G}_{\mathrm{I}}=  \mathbb{Z}^d$ and $\mathsf{G}_{\mathrm{II}}=\frac{1}{2} \cdot \mathbb{Z}^d$. For the interpolating graph $\mathbf{G}_i$, we define $\mathcal{L}_i$, $\nu$, $\mathcal{L}_i^{\nu}$, $\widehat{\mathcal{L}}_i$ as in (\ref{notation5.2}), and $\mathbf{T}_i$ as in (\ref{def_Ti_5.5}). For $\diamond\in \{\mathrm{I} ,\mathrm{II} \}$, we denote the analogues of $\mathcal{L}_i$, $\mathcal{L}_i^{\nu}$ and $\mathbf{T}_i$ for $\mathsf{G}_{\diamond}$ by $\mathcal{L}_{\diamond}$, $\mathcal{L}_{\diamond}^{\nu}$ and $\mathbf{T}_{\diamond}$ respectively. By Proposition \ref{prop_compare_relation}, we have 
\begin{equation}\label{y530}
 		\mathbb{P}\big( \mathrm{d}_{\mathfrak{T}}(\mathbf{T}_1,\mathbf{T}_{K}) > \delta^{-\frac{1}{2}} \big)\lesssim  \delta^{  \cref{const_compare_gluing}}. 
 \end{equation}
 Here we used the fact that $\delta^{\cref{const_crossingloop_new}(\bref{para_1}-1)}K\le C\delta^{\cref{const_crossingloop_new}(\bref{para_1}-1)-d(\bref{para_2}+\bref{para_1})}<\delta^{-\frac{1}{2}}$. Moreover, the graph $\mathsf{G}_K$ and $\mathsf{G}_{\mathrm{II}}$ differ only outside $[-LM,LM]^{d}$. Recall that $\mathbf{T}_i$ is measurable with respect to the loop clusters that intersect $[-LM^{1/2},LM^{1/2}]^{d}$. Thus, there exists a natural coupling between $\mathcal{L}_K$ and $\mathcal{L}_{\mathrm{II}}$ such that $\mathbf{T}_{K}=\mathbf{T}_{\mathrm{II}}$ unless some loop cluster crosses the annulus $[-LM,LM]^{d}\setminus [-LM^{1/2},LM^{1/2}]^{d}$ (whose probability,  by (\ref{crossing_prob}), is $O(M^{-\frac{d}{4}+\frac{1}{2}})$). As a result, 
 \begin{equation}\label{y531}
 	\mathbb{P}\big(  \mathbf{T}_K\neq  \mathbf{T}_{\mathrm{II}}  \big)\lesssim M^{-\frac{d}{4}+\frac{1}{2}}=\delta^{ (\frac{d}{4}-\frac{1}{2})(\bref{para_2}+\bref{para_1}) } . 
 \end{equation}
 Note that $\mathbf{T}_1=\mathbf{T}_{\mathrm{I}}$. Therefore, by (\ref{y530}) and (\ref{y531}), we obtain 
 \begin{equation}\label{y5.32}
 	\mathbb{P}\big( \mathrm{d}_{\mathfrak{T}}(\mathbf{T}_{\mathrm{I}},\mathbf{T}_{\mathrm{II}}) > \delta^{-\frac{1}{2}} \big)\lesssim \delta^{ \cref{const_compare_gluing} \land  [(\frac{d}{4}-\frac{1}{2})(\bref{para_2}+\bref{para_1})]} = :  \delta^{c_{\star}}. 
 	 \end{equation}

For any $R,\epsilon>0$ and $\diamond \in \{\mathrm{I}, \mathrm{II}   \}$, 
we denote
\begin{equation}
\widehat{ \mathcal{L}}_{\diamond}^{R,\epsilon}:= \widehat{\mathcal{L}}_{\diamond}\cdot \mathbbm{1}_{\mathrm{ran}(\ell)\subset B(\delta^{-1}\cdot R), |\ell|>\delta^{-1}\epsilon}. 
\end{equation}
Here $|\ell|$ denotes the diameter of the loop $\ell$. Moreover, we define $\mathbf{T}_{\diamond}^{R,\epsilon}$ to be the restriction of $\mathbf{T}_{\diamond}$ to the loops in $\mathcal{L}_{\diamond}^{R,\epsilon}$. Let $\overline{\mathcal{L}}_{\diamond}:=\widehat{\mathcal{L}}_{\diamond}\cdot \mathbbm{1}_{\mathrm{ran}(\ell)\subset (-LM^{1/2},LM^{1/2})^d}$. On the event $\{\mathrm{d}_{\mathfrak{T}}(\mathbf{T}_{\mathrm{I}},\mathbf{T}_{\mathrm{II}}) \le  \delta^{-\frac{1}{2}}\}$, there exists a bijection $\sigma$ between the loops in $\overline{\mathcal{L}}_{\mathrm{I}}$ and $\overline{\mathcal{L}}_{\mathrm{II}}$ such that the equivalence relation is preserved and for any $\ell \in \overline{\mathcal{L}}_{\mathrm{I}}$, 
\begin{equation}
 \mathrm{d}_{\mathrm{H}}\big(  \mathrm{ran}(\ell), \mathrm{ran}(\sigma(\ell))  \big)\le    \delta^{-\frac{1}{2}}. 
\end{equation}
Thus, if $\widecheck{\mathcal{L}}_{\mathrm{I}} :=\{\ell \in  \mathcal{L}_{\mathrm{I}}: \mathrm{ran}(\ell)\subset B(\delta^{-1} R+ 2\delta^{-\frac{1}{2}}) ,\mathrm{ran}(\ell)\not\subset B(\delta^{-1} R- 2\delta^{-\frac{1}{2}}), |\ell|>\delta^{-1}\epsilon  \} =\emptyset$, then $\sigma$ also induces a bijection between the loops in $\widehat{ \mathcal{L}}_{\mathrm{I}}^{R,\epsilon}$ and $\widehat{ \mathcal{L}}_{\mathrm{II}}^{R,\epsilon}$. Consequently, 
\begin{equation}\label{y535}
	\begin{split}
			  \mathbb{P}\big( \mathrm{d}_{\mathfrak{T}}(\mathbf{T}_{\mathrm{I}}^{R,\epsilon}, \mathbf{T}_{\mathrm{II}}^{R,\epsilon})  > \delta^{-\frac{1}{2}}\big) 
			\le &\mathbb{P}\big( \mathrm{d}_{\mathfrak{T}}(\mathbf{T}_{\mathrm{I}}, \mathbf{T}_{\mathrm{II}} )  > \delta^{-\frac{1}{2}}\big)  + \mathbb{P}\big(\widecheck{\mathcal{L}}_{\mathrm{I}} \neq \emptyset \big) \\
			\overset{(\ref{y5.32}), (\ref{y3.17})}{\lesssim }  &\delta^{c_{\star}} + R^{d-1}\epsilon^{-d}\delta^{d-\frac{1}{2}} \lesssim \delta^{ c_{\star} }. 
	\end{split}
\end{equation}

For any $n\in \mathbb{N}^+$, we define
\begin{equation}\label{def_Ti_5.5new}
	\mathbf{T}_{(n)}^{R,\epsilon}=\big( (2^{-n}\cdot \widehat{\mathcal{L}}_{\mathrm{I}})\cdot \mathbbm{1}_{\mathrm{ran}(\ell)\subset B(R),|\ell|>\epsilon}, \sim_{(n)}^{R,\epsilon} \big),
\end{equation}
 where $\ell \sim_{(n)}^{R,\epsilon} \ell'$ if and only if $\ell$ and $\ell'$ belong to the same cluster of $(2^{-n}\cdot (\widehat{\mathcal{L}}_{\mathrm{I}}+\mathcal{L}_{\mathrm{I}}^{\nu})$. By taking $\delta = 2^{-n}$ in (\ref{y535}), we obtain 
 \begin{equation}
 	  \mathbb{P}\big( \mathrm{d}_{\mathfrak{T}}(\mathbf{T}_{(n)}^{R,\epsilon}, \mathbf{T}_{(n+1)}^{R,\epsilon})  > 2^{-\frac{1}{2}n}\big) \lesssim 2^{-c_{\star}n}. 
 \end{equation}
 This together with Corollary \ref{coro_limit_exist} yields that $\mathbf{T}_{(n)}^{R,\epsilon}$ converges in distribution to 
 \begin{equation}
 	\mathbf{T}_{(\infty)}^{R,\epsilon}= \big( \mathcal{L}_{(\infty)}^{R,\epsilon}, \sim_{(\infty)}^{R,\epsilon}  \big). 
 \end{equation}
It follows from Lemma \ref{lemma_coupling} that $\mathcal{L}_{(\infty)}^{R,\epsilon}\overset{\mathrm{d}}{=}\mathcal{L}^{\mathrm{B}}\cdot \mathbbm{1}_{\mathrm{ran}(\ell)\subset B(R), |\ell|>\epsilon}$. In fact, the family
$\big\{\mathcal{L}_{(\infty)}^{R,\epsilon} \big\}_{R,\epsilon>0}$ satisfies the consistency conditions required by the Kolmogorov extension theorem. To see this, for any $n\in \mathbb{N}^+$, $\epsilon>0$ and $R_2>R_1>0$, since the equivalence relations $\sim_{(n)}^{R_1,\epsilon}$ and $\sim_{(n)}^{R_2,\epsilon}$ are induced by the same loop clusters, the restriction of $\mathbf{T}_{(n)}^{R_2,\epsilon}$ to the loops contained in $B(R_1)$ coincides with $\mathbf{T}_{(n)}^{R_1,\epsilon}$; the same holds when $\epsilon$ varies. The limiting family $\{\mathbf{T}_{(\infty)}^{R,\epsilon}\}_{R,\epsilon>0}$ therefore inherits these consistency properties. Thus, we may apply the Kolmogorov extension theorem to obtain a probability space $\mathbb{P}$ carrying $\mathcal{L}^{\mathrm{B}}$ together with an equivalence relation $\sim_{\mathrm{g}}$ on its clusters, as stated in Theorem \ref{new_thm1}.

   \begin{remark}[Universality of the limit]\label{rmk:universality}

 	 In this section, we compared the large loops in the two graphs, together with their connectivity via the remaining loops. The proof is not restricted to the regular lattice: it only requires that both graphs are Voronoi networks or, more generally, that the random walks on the two graphs converge to Brownian motion at a uniform rate. Since the shapes of large loops and their connectivity already determine the configuration of all macroscopic clusters, we have shown that for a general Voronoi network $\mathbf{G}$, the loop clusters converge to the limit constructed here (by interpolating between $\mathbf{G}$ and $\mathbb{Z}^d$). This shows that the limit in Theorem~\ref{new_thm1} is, in a sense, universal.

         We expect that the interpolation between Voronoi networks and $\mathbb{Z}^d$ presented in this section can also be extended to the LERW, albeit with substantial technical effort. Notably, a novel approach to the scaling limit of the LERW in three dimensions was recently introduced in \cite{cai2026characterization}; it provides an elegant route to the universality of the LERW scaling limit.

 \end{remark}

\section{Asymptotics of diameter tail probability and $k$-point function}\label{section_proof_theorem_1.3_1.6}

 This section is devoted to proving Theorem \ref{thm_asymptotic} and the convergence of (\ref{def_gamma_k}) in Theorem \ref{thm_k_point}. We assume $3\le d\le 5$ and maintain the parameters $\beta$'s employed in the last section. We denote by $\mathcal{L}$ the loop soup on $\widetilde{\mathbb{Z}}^d$ of intensity $1/2$. Let $\mathcal{L}^{\nu}$ be its counterpart on $\widetilde{\mathbb{Z}}^d$ with killing rate $\nu=\delta^{2(1-\bref{para_3})}$, where $\delta>0$ is sufficiently small. Echoing (\ref{notation5.2}), we define
  \begin{equation}\label{u6.1}
  		\widehat{\mathcal{L}}:=(\mathcal{L}-\mathcal{L}^{\nu} )\cdot \mathbbm{1}_{|\ell|\ge \delta^{\bref{para_4}-1}}\  \text{and} \ \widecheck{\mathcal{L}}:=(\mathcal{L}-\mathcal{L}^{\nu} )\cdot \mathbbm{1}_{|\ell| < \delta^{\bref{para_4}-1}},
  \end{equation}
 and let $\overline{\mathcal{L}}:=\mathcal{L}^{\nu} + \widecheck{\mathcal{L}}$ and $\mathcal{L}^{\circ}:=\mathcal{L}-\widecheck{\mathcal{L}}$. For any $v\in \widetilde{\mathbb{Z}}^d$, we denote by $\mathcal{C}_{v}$ (resp. $\mathcal{C}_{v}^{\circ}$) the cluster of $\mathcal{L}$ (resp. $\mathcal{L}^{\circ}$) containing the point $v\in \widetilde{\mathbb{Z}}^d$. In addition, we define $\widehat{\mathcal{C}}_{v}^{\circ}$ to be the union of $\{v\}$ and the ranges of loops in $\widehat{\mathcal{L}}$ that are contained in $\mathcal{C}_{v}^{\circ}$.

 We first prove the scaling relation for the tail probability of the $\ell^\infty$-diameter of $\widehat{\mathcal{C}}_{v}^{\circ}$. Recall that $\hat{x}^\delta$ denotes the closest point to $\delta^{-1}\cdot x$ in $\mathbb{Z}^d$.   
 
 \begin{lemma}

(1) There exist $\cl\label{const_compare_lemma1},\cl\label{const_compare_lemma2}>0$ such that for $\delta_*\in (\frac{\delta}{10}, 10\delta)$ and $a>0$, 
 	\begin{equation}\label{a6.2}
 	\delta^{-\frac{d}{2}+1}	\mathbb{P}\big(  | \delta \cdot   \widehat{\mathcal{C}}_{\bm{0}}^{\circ}|_{\infty} \ge a \big)   \le 	\delta_*^{-\frac{d}{2}+1}\mathbb{P}\big(  | \delta_* \cdot   \widehat{\mathcal{C}}_{\bm{0}}^{\circ}|_{\infty} \ge a-\delta^{\cref{const_compare_lemma1}} \big)  + \delta^{\cref{const_compare_lemma2}}. 
 	\end{equation}

(2) For any distinct points $\{x_i\}_{1\le i\le k}$ in $\mathbb{R}^d$, there exists $\cl\label{const_compare1_kpoint}>0$ such that for all sufficiently small $\delta>0$ and $\delta_*\in (\frac{\delta}{10}, 10\delta)$,   
 \begin{equation}\label{h6.3}
 \begin{split}
 	& \delta^{-(\frac{d}{2}-1)k} 	\mathbb{P}(\hat{x}_1^{\delta}\xleftrightarrow{ \circ } \hat{x}_2^{\delta} \xleftrightarrow{\circ} ... \xleftrightarrow{\circ} \hat{x}_k^{\delta}) \\
 	  \le &  \delta_*^{-(\frac{d}{2}-1)k} 	\mathbb{P}(\hat{x}_1^{\delta_*}\xleftrightarrow{ \circ } \hat{x}_2^{\delta_*} \xleftrightarrow{\circ} ... \xleftrightarrow{\circ} \hat{x}_k^{\delta_*}) +  \delta^{ \cref{const_compare1_kpoint}}. 
 \end{split}
 \end{equation}
 	Here ``$\xleftrightarrow{ \circ }$'' denotes connectivity via loops in $\mathcal{L}^{\circ}$.

 \end{lemma}
 \begin{proof}
 	(1) We employ the interpolation argument in Section \ref{prop_compare_relation} with $\mathbf{G}_{\mathrm{I}}=\mathbb{Z}^d$ and $\mathbf{G}_{\mathrm{II}}=\frac{\delta_*}{\delta}\cdot \mathbb{Z}^d$. We also maintain the notations $\{\mathbf{G}_j\}_{1\le j\le J}$, $\{\overline{\mathbf{T}}_{j}\}_{1\le j\le J}$, $L$ and $M$ therein. In this case, $\mathbf{G}_J$ and $\mathbf{G}_{\mathrm{II}}$ differ only within $B(\delta^{\bref{new_6}-1})\cup [(-LM,LM)^d]^c$. Let $\mathbf{G}_{\mathrm{III}}$ be the graph that coincides with $\mathbf{G}_{J}$ (or equivalently $\mathbf{G}_{\mathrm{I}}$) inside $B(\delta^{\bref{new_6}-1})$ and with $\mathbf{G}_{\mathrm{II}}$ outside it. For $\diamond \in \{\mathrm{II},\mathrm{III} \}$, let $\mathcal{L}_{\diamond}$ (resp. $\widehat{\mathcal{L}}_{\diamond}$ and $\widehat{\mathcal{C}}_{\bm{0},\diamond}^{\circ}$) be the analogue of $\mathcal{L}$ (resp. $\widehat{\mathcal{L}}$ and $\widehat{\mathcal{C}}_{\bm{0}}^{\circ}$) for $\mathbf{G}_{\diamond}$. We also define $\overline{\mathbf{T}}_{\diamond}$ to be the counterpart of $\overline{\mathbf{T}}_i$ for $\mathbf{G}_\diamond$.

 	 Applying Item (1) in Lemma \ref{lemma_extend_k_point} for $x_1=\bm{0}$, we have 
 	 	 \begin{equation}\label{u63}
 	 	  	\mathbb{P}\big( \bm{0} \xleftrightarrow{} \partial B(\tfrac{\delta^{-1}}{100} ) , \mathrm{d}_{\overline{\mathfrak{T}}}(\overline{\mathbf{T}}_1,\overline{\mathbf{T}}_{J}) > \delta^{\cref{const_crossingloop_new}(\bref{para_1}-1)} \big)\lesssim  \delta^{\frac{d}{2}-1 + \cref{const_for_k_point}}.  
 	 \end{equation}
 	 Moreover, similar to (\ref{y531}), one has 
 	 \begin{equation}\label{u64}
 		\mathbb{P}\big( \bm{0} \xleftrightarrow{} \partial B(\tfrac{\delta^{-1}}{100} ) ,  \overline{\mathbf{T}}_J\neq \overline{\mathbf{T}}_{\mathrm{III}}    \big) \lesssim  	\delta^{ \frac{d}{2}-1 + (\frac{d}{4}-\frac{1}{2})(\bref{para_2}+\bref{para_1}) }. 
 	 \end{equation}
 Note that $\mathrm{d}_{\overline{\mathfrak{T}}}(\overline{\mathbf{T}}_1,\overline{\mathbf{T}}_{\mathrm{III}})\le \delta^{\cref{const_crossingloop_new}(\bref{para_1}-1)}$ implies $\mathrm{d}_{\mathrm{H}}(\widehat{\mathcal{C}}_{\bm{0}}^{\circ}, \widehat{\mathcal{C}}_{\bm{0},\mathrm{III}}^{\circ})\le \delta^{\cref{const_crossingloop_new}(\bref{para_1}-1)}$. Thus, by (\ref{u63}) and (\ref{u64}), we obtain 
 \begin{equation}\label{a66}
 	\mathbb{P}\big(  | \delta \cdot   \widehat{\mathcal{C}}_{\bm{0}}^{\circ}|_{\infty} \ge a \big) \le 	\mathbb{P}\big(  | \delta \cdot   \widehat{\mathcal{C}}_{\bm{0},\mathrm{III}}^{\circ}|_{\infty} \ge a - \delta^{1+\cref{const_crossingloop_new}(\bref{para_1}-1)} \big)   + \delta^{ \frac{d}{2}-1 + \cref{const_for_k_point} \land  [(\frac{d}{4}-\frac{1}{2})(\bref{para_2}+\bref{para_1})]}. 
 \end{equation}

 It remains to compare $\widehat{\mathcal{C}}_{\bm{0},\mathrm{III}}^{\circ}$ with its counterpart on $\mathbf{G}_{\mathrm{II}}$ (which has the same distribution as $\frac{\delta_*}{\delta}\cdot  \widehat{\mathcal{C}}_{\bm{0}}^{\circ}$). Let $\mathsf{R}$ denote the event that some loop in $\mathcal{L}_{\mathrm{III}}$ crosses $B(\delta^{\frac{3}{4}\bref{new_6}-1})\setminus B(\delta^{\bref{new_6}-1})$ or $B(\delta^{\frac{1}{4}\bref{new_6}-1})\setminus B(\delta^{\frac{1}{2}\bref{new_6}-1})$. It follows from (\ref{a4.6}) that 
\begin{equation}\label{a67}
		\mathbb{P}\big( | \delta \cdot   \widehat{\mathcal{C}}_{\bm{0},\mathrm{III}}^{\circ}|_{\infty} \ge a - \delta^{1+\cref{const_crossingloop_new}(\bref{para_1}-1)}  , \mathsf{R}  \big)  \lesssim \delta^{ (\frac{d}{2}-1)(1+ \frac{1}{4}\bref{new_6}  )}. 
\end{equation}
   Meanwhile, we couple $\mathcal{L}_{\mathrm{II}}$ and $\mathcal{L}_{\mathrm{III}}$ by identifying the loops that avoid $B(\delta^{\bref{new_6}-1})$. Under this coupling, on the event $\mathsf{R}^c$, the following events occur: 
   \begin{itemize}

   	\item[(a)] Every loop in $\mathcal{L}_{\mathrm{III}}$ intersecting $B(\delta^{\bref{new_6}-1})$ is contained in $B(\delta^{\frac{3}{4}\bref{new_6}-1})$ (as a result, every loop in $\mathcal{L}_{\mathrm{III}}$ intersecting $[B(\delta^{\frac{3}{4}\bref{new_6}-1})]^c$ cannot reach $B(\delta^{\bref{new_6}-1})$, and hence must be contained in $\mathcal{L}_{\mathrm{II}}$);

   	\item[(b)]   No loop in $\widehat{\mathcal{L}}_{\mathrm{III}}$ intersects $B(\delta^{\frac{1}{2}\bref{new_6}-1})$ (consequently, all loops in $\widehat{\mathcal{L}}_{\mathrm{III}}$ must be contained in $\widehat{\mathcal{L}}_{\mathrm{II}}$), since by (\ref{u6.1}) any such loop would also have to intersect $\partial B(\delta^{\frac{1}{4}\bref{new_6}-1})$ and thus does not exist on $\mathsf R^c$.

   \end{itemize}
    These two events yield  
    \begin{equation}\label{a6.8}
  \widehat{\mathcal{L}}_{\mathrm{II}}^*:=  	\widehat{\mathcal{L}}_{\mathrm{II}}\cdot \mathbbm{1}_{\mathrm{ran}(\ell)\cap B(\delta^{\bref{new_6}-1})= \emptyset } =  \widehat{\mathcal{L}}_{\mathrm{III}}  \ \text{and}\  \mathcal{L}_{\mathrm{II}}^{\nu,*}:= \mathcal{L}_{\mathrm{II}}^{\nu} \cdot  \mathbbm{1}_{\mathrm{ran}(\ell)\cap B(\delta^{\bref{new_6}-1})= \emptyset } =   \mathcal{L}_{\mathrm{III}}^{\nu}  
    \end{equation}
   respectively. As in the proof of Proposition \ref{prop_compare_relation}, for each $\ell \in \widehat{\mathcal{L}}_{\mathrm{III}}$, we define $\mathcal{C}_\ell^{\pm}$ to be the partial sign cluster of $\mathcal{L}_{\mathrm{III}}^{\nu}$ outside $B(\delta^{\frac{3}{4}\bref{new_6}-1})$ containing $\ell$. Let $\widehat{\mathcal{C}}^{\pm}_{\mathrm{III}}:=\cup_{\ell\in \widehat{\mathcal{L}}_{\mathrm{III}} } \mathcal{C}_\ell^{\pm}$, and define $\widehat{\mathcal{C}}^{\pm,*}_{\mathrm{II}}$ to be its analogue obtained by replacing $\widehat{\mathcal{L}}_{\mathrm{III}}$ and $\mathcal{L}_{\mathrm{III}}^{\nu}$ with $\widehat{\mathcal{L}}_{\mathrm{II}}^{*}$ and $\mathcal{L}_{\mathrm{II}}^{\nu,*}$. On $\mathsf{R}^c$, it follows from (\ref{a6.8}) that $\widehat{\mathcal{C}}^{\pm}_{\mathrm{III}}=\widehat{\mathcal{C}}^{\pm,*}_{\mathrm{II}}$. We denote by $\mathcal{N}$ the number of clusters in $\widehat{\mathcal{C}}^{\pm}_{\mathrm{III}}$ intersecting $B(\delta^{\frac{3}{4}\bref{new_6}-1})$. Note that $| \delta \cdot   \widehat{\mathcal{C}}_{\bm{0},\mathrm{III}}^{\circ}|_{\infty} \ge a - \delta^{1+\cref{const_crossingloop_new}(\bref{para_1}-1)}$ ensures $\mathcal{N}\ge 1$. When $\mathcal{N}\ge 2$, there exists an extra loop cluster crossing the annulus $B(\delta^{\frac{1}{2}\bref{new_6}-1})  \setminus B(\delta^{\frac{3}{4}\bref{new_6}-1})$. Thus, by the BKR inequality, we have 
   \begin{equation}\label{a6.9}
   	 \mathbb{P}\big( | \delta \cdot   \widehat{\mathcal{C}}_{\bm{0},\mathrm{III}}^{\circ}|_{\infty} \ge a - \delta^{1+\cref{const_crossingloop_new}(\bref{para_1}-1)}  , \mathsf{R}^c , \mathcal{N}\ge 2 \big)  \lesssim \delta^{ (\frac{d}{2}-1)(1+ \frac{1}{4}\bref{new_6}  )}. 
   \end{equation}
   When $\mathcal{N}=1$, we denote by $\mathcal{U}^{\pm}_{\mathrm{III}}$ the unique cluster in
$\widehat{\mathcal{C}}^{\pm}_{\mathrm{III}}$ intersecting $B(\delta^{\frac{3}{4}\bref{new_6}-1})$, and by $\widehat{\mathcal{U}}^{\pm}_{\mathrm{III}}$ the union of the ranges of the loops in $\widehat{\mathcal{L}}_{\mathrm{III}}$ that are contained in $\mathcal{U}^{\pm}_{\mathrm{III}}$. On $\mathsf{R}^c\cap\{\mathcal{N}=1\}$, one has $\widehat{\mathcal{C}}^{\circ}_{\bm{0},\mathrm{III}}=\{\bm{0}\}\cup\widehat{\mathcal{U}}^{\pm}_{\mathrm{III}}$ if $\bm{0}\xleftrightarrow{}\mathcal{U}^{\pm}_{\mathrm{III}}$, and $\widehat{\mathcal{C}}^{\circ}_{\bm{0},\mathrm{III}}=\{\bm{0}\}$ otherwise. In particular, the event $\{|\delta\cdot\widehat{\mathcal{C}}^{\circ}_{\bm{0},\mathrm{III}}|_\infty\ge a-\delta^{1+\cref{const_crossingloop_new}(\bref{para_1}-1)}\}$ occurs if and only if $|\delta\cdot(\widehat{\mathcal{U}}^{\pm}_{\mathrm{III}}\cup\{\bm{0}\})|_\infty\ge a-\delta^{1+\cref{const_crossingloop_new}(\bref{para_1}-1)}$ and $\bm{0}\xleftrightarrow{}\mathcal{U}^{\pm}_{\mathrm{III}}$. Therefore, 
 \begin{equation}\label{b610}
 	\begin{split}
 			& \mathbb{P}\big( | \delta \cdot   \widehat{\mathcal{C}}_{\bm{0},\mathrm{III}}^{\circ}|_{\infty} \ge a - \delta^{1+\cref{const_crossingloop_new}(\bref{para_1}-1)}  , \mathsf{R}^c , \mathcal{N}=1  \big) \\
 			=  & \mathbb{E}\Big[  \mathbbm{1}_{\delta\cdot |\widehat{\mathcal{U}}^{\pm}_{\mathrm{III}}\cup \{\bm{0}\}|_{\infty}  \ge a - \delta^{1+\cref{const_crossingloop_new}(\bref{para_1}-1)},  \mathcal{N}=1 ,\mathsf{R}^c }\cdot \mathbb{P}\big(\bm{0}\xleftrightarrow{} \mathcal{U}^{\pm}_{\mathrm{III}}  \mid \widehat{\mathcal{C}}^{\pm}_{\mathrm{III}}  \big)    \Big]. 
 	\end{split}
 \end{equation}
 By the Lupu--Werner formula (\ref{formula_LW}), the conditional probability $\mathbb{P} (\bm{0}\xleftrightarrow{} \mathcal{U}^{\pm}_{\mathrm{III}} \mid \widehat{\mathcal{C}}^{\pm}_{\mathrm{III}} ) $ is asymptotically proportional to the product of the harmonic average at $\bm{0}$, i.e., $\sum_{v\in  \partial B(\delta^{\frac{3}{4}\bref{new_6}-1})} \widetilde{\mathbb{P}}_{\bm{0}}( \tau_{\mathcal{U}^{\pm}_{\mathrm{III}}} =\tau_{v} <\infty ) \phi_v$, and the standard deviation of $\phi_{\bm{0}}$, up to a multiplicative error $1+O(\delta^{\frac{1}{5}(1- \frac{3}{4}\bref{new_6})} )$ arising from the effective conductance between $\bm{0}$ and $\partial B(\delta^{\frac{3}{4}\bref{new_6}-1})$ (by Lemma \ref{lemma2.1_isometric}). Note that the harmonic average coincides with the counterpart of $\widehat{\mathcal{C}}^{\pm,*}_{\mathrm{II}}$ under the aforementioned coupling, and the standard deviation equals $(\frac{\delta}{\delta*})^{\frac{d}{2}-1}$ times that in $\mathbf{G}_{\mathrm{II}}$. (The factor $(\frac{\delta}{\delta*})^{\frac{d}{2}-1}$ arises because by (\ref{add34}), the conductance of each edge in $\mathbb{Z}^d$ is $(\frac{\delta}{\delta_*})^{d-2}$ times that in $\frac{\delta_*}{\delta}\cdot \mathbb{Z}^d$.) This observation together with (\ref{b610}) implies 
 \begin{equation}\label{a6.11}
 	\begin{split}
 		& \mathbb{P}\big( | \delta \cdot   \widehat{\mathcal{C}}_{\bm{0},\mathrm{III}}^{\circ}|_{\infty} \ge a - \delta^{1+\cref{const_crossingloop_new}(\bref{para_1}-1)}  , \mathsf{R}^c , \mathcal{N}=1  \big) \\ 
 		\le & (\tfrac{\delta}{\delta_*})^{\frac{d}{2}-1}  \mathbb{P}\big( | \delta \cdot   \widehat{\mathcal{C}}_{\bm{0},\mathrm{II}}^{\circ}|_{\infty} \ge a - \delta^{1+\cref{const_crossingloop_new}(\bref{para_1}-1)}  \big) \cdot \big(1+ C \delta^{\frac{1}{5}(1- \frac{3}{4}\bref{new_6})} \big) \\
 		\le & (\tfrac{\delta}{\delta_*})^{\frac{d}{2}-1}  \mathbb{P}\big( | \delta \cdot   \widehat{\mathcal{C}}_{\bm{0},\mathrm{II}}^{\circ}|_{\infty} \ge a - \delta^{1+\cref{const_crossingloop_new}(\bref{para_1}-1)}  \big)  + C'\delta^{ \frac{d}{2}-1+\frac{1}{5}(1- \frac{3}{4}\bref{new_6}) }. 
 	\end{split}
 \end{equation}
 Combining (\ref{a66}), (\ref{a67}), (\ref{a6.9}) and (\ref{a6.11}), we obtain (\ref{a6.2}).

(2) Likewise, we employ the interpolation argument in Section \ref{prop_compare_relation} with $\mathbf{G}_{\mathrm{I}}=\mathbb{Z}^d$ and $\mathbf{G}_{\mathrm{II}}=\frac{\delta_*}{\delta}\cdot \mathbb{Z}^d$ with $x_j=\hat{x}_j^{\delta}$ for $1\le j\le k$. As above, we define $\mathbf{G}_{\mathrm{III}}$ to be the graph that coincides with $\mathbf{G}_{J}$ inside $\cup_{1\le j\le k} B(x_j,\delta^{\bref{new_6}-1})$ and with $\mathbf{G}_{\mathrm{II}}$ outside it. Following the same argument as in (\ref{a66}) (while using Item (2) in Lemma \ref{lemma_extend_k_point} instead of Item (1)), we have 
\begin{equation}\label{d6.12}
	\begin{split}
		   	\mathbb{P}(\hat{x}_1^{\delta}\xleftrightarrow{ \circ } \hat{x}_2^{\delta} \xleftrightarrow{\circ} ... \xleftrightarrow{\circ} \hat{x}_k^{\delta})  
			\le   \mathbb{P}\big(\hat{x}_1^{\delta}\xleftrightarrow[\mathrm{III}]{} \hat{x}_2^{\delta} \xleftrightarrow[\mathrm{III}]{} ...  \xleftrightarrow[\mathrm{III}]{} \hat{x}_k^{\delta} \big)  + \delta^{(\frac{d}{2}-1)k+ c_{\dagger}} 
	\end{split}
\end{equation}
for some constant $c_{\dagger}>0$. Here $\xleftrightarrow[\mathrm{III}]{}$ is the analogue of $\xleftrightarrow{ \circ }$ for the graph $\mathbf{G}_{\mathrm{III}}$. To relate the connection on $\mathbf{G}_{\mathrm{III}}$ to $\mathbf{G}_{\mathrm{II}}$, it suffices to change the lattice within $B(x_j,\delta^{\bref{new_6}-1})$ from $\mathbb{Z}^d$ to $\frac{\delta_*}{\delta}\cdot \mathbb{Z}^d$ for all $1\le j\le k$. Let $\mathbf{G}_{\mathrm{III}}'$ be the graph obtained from $\mathbf{G}_{\mathrm{III}}$ by making this modification inside $B(x_1,\delta^{\bref{new_6}-1})$ only. We then define $\mathsf{R}$ to be the event that some loop in $\mathcal{L}_{\mathrm{III}}$ crosses $B(x_1, \delta^{\frac{3}{4}\bref{new_6}-1})\setminus B(x_1,\delta^{\bref{new_6}-1})$ or $B(x_1,\delta^{\frac{1}{4}\bref{new_6}-1})\setminus B(x_1,\delta^{\frac{1}{2}\bref{new_6}-1})$. We also define $\widehat{\mathcal{C}}^{\pm}_{\mathrm{III}}$ and $\mathcal{N}$ as above. For the same reason as in (\ref{a67}) and (\ref{a6.9}), we have   
  \begin{equation}\label{i613}
  	\begin{split}
  		\mathbb{P}\big(\hat{x}_1^{\delta}\xleftrightarrow[\mathrm{III}]{} \hat{x}_2^{\delta} \xleftrightarrow[\mathrm{III}]{} ... \xleftrightarrow[\mathrm{III}]{} \hat{x}_k^{\delta} , \mathsf{R} \cup \{\mathsf{R}^c, \mathcal{N}\ge 2  \}  \big) \lesssim \delta^{(\frac{d}{2}-1)(k+ \frac{1}{4}\bref{new_6})}. 
  	\end{split}
  \end{equation}
  Moreover, similar to (\ref{a6.11}), it follows from the Lupu--Werner formula (\ref{formula_LW}) and the relation between the conductances of edges on $\mathbb{Z}^d$ and $\frac{\delta_*}{\delta}\cdot \mathbb{Z}^d$ that  
  \begin{equation}\label{i614}
  	\begin{split}
  	& 	\mathbb{P}\big(\hat{x}_1^{\delta}\xleftrightarrow[\mathrm{III}]{} \hat{x}_2^{\delta} \xleftrightarrow[\mathrm{III}]{} ... \xleftrightarrow[\mathrm{III}]{} \hat{x}_k^{\delta} , \mathsf{R}^c ,  \mathcal{N}=1   \big)  \\
	\le & (\tfrac{\delta}{\delta_*})^{\frac{d}{2}-1}  \mathbb{P}\big( \tfrac{\delta}{\delta_* }\cdot  \hat{x}_1^{\delta_*} \xleftrightarrow[\mathrm{III}']{}  \hat{x}_2^{\delta} \xleftrightarrow[\mathrm{III}']{} ... \xleftrightarrow{(\mathrm{III}')} \hat{x}_k^{\delta} \big)  + C\delta^{ (\frac{d}{2}-1)k+\frac{1}{5}(1- \frac{3}{4}\bref{new_6}) },  
  	\end{split}
  \end{equation}
  where $\xleftrightarrow[\mathrm{III}']{}$ denotes the analogue of $\xleftrightarrow{ \circ }$ for the graph $\mathbf{G}_{\mathrm{III}}'$. By (\ref{i613}) and (\ref{i614}), 
   \begin{equation}
   \begin{split}
   	&	\mathbb{P}\big(\hat{x}_1^{\delta}\xleftrightarrow[\mathrm{III}]{} \hat{x}_2^{\delta} \xleftrightarrow[\mathrm{III}]{} ... \xleftrightarrow[\mathrm{III}]{} \hat{x}_k^{\delta} \big)  \\
   		\le & (\tfrac{\delta}{\delta_*})^{\frac{d}{2}-1}  \mathbb{P}\big( \tfrac{\delta}{\delta_* }\cdot  \hat{x}_1^{\delta_*}\xleftrightarrow[\mathrm{III}']{}  \hat{x}_2^{\delta} \xleftrightarrow{(\mathrm{III}')} ... \xleftrightarrow{(\mathrm{III}')} \hat{x}_k^{\delta} \big)  + C'\delta^{ (\frac{d}{2}-1)k+c_\ddagger },
   \end{split}
 \end{equation}
 where $c_\ddagger=[\frac{1}{4}(\frac{d}{2}-1)\bref{new_6}]\land  [\frac{1}{5}(1- \frac{3}{4}\bref{new_6})]$. Repeating this for each $x_j$, we obtain 
 \begin{equation}
 \begin{split}
  & 	\mathbb{P}\big(\hat{x}_1^{\delta}\xleftrightarrow[\mathrm{III}]{} \hat{x}_2^{\delta} \xleftrightarrow[\mathrm{III}]{} ... \xleftrightarrow[\mathrm{III}]{} \hat{x}_k^{\delta} \big) \\
  	\le &  	 (\tfrac{\delta}{\delta_*})^{(\frac{d}{2}-1)k}  \mathbb{P} \big( \tfrac{\delta}{\delta_* }\cdot  \hat{x}_1^{\delta_*}\xleftrightarrow{ (\mathrm{II}) } \tfrac{\delta}{\delta_* }\cdot  \hat{x}_2^{\delta_*}  \xleftrightarrow{(\mathrm{II})} ... \xleftrightarrow{(\mathrm{II})} \tfrac{\delta}{\delta_* }\cdot  \hat{x}_k^{\delta_*} \big) + C'k\delta^{ (\frac{d}{2}-1)k+c_\ddagger }\\
  	=&   (\tfrac{\delta}{\delta_*})^{(\frac{d}{2}-1)k}  \mathbb{P}\big(  \hat{x}_1^{\delta_*}\xleftrightarrow{   }   \hat{x}_2^{\delta_*}  \xleftrightarrow{ } ... \xleftrightarrow{ }    \hat{x}_k^{\delta_*} \big) + C'k\delta^{ (\frac{d}{2}-1)k+c_\ddagger }. 
 \end{split}
 \end{equation}
 This together with (\ref{d6.12}) yields the desired bound (\ref{h6.3}). 
  \end{proof}

 	The next step is to show that $|\delta\cdot \mathcal{C}_{\bm{0}}|_{\infty}$ in Theorem \ref{thm_asymptotic} can be approximated by $|\delta\cdot \widehat{\mathcal{C}}_{\bm{0}}^{\circ}|_{\infty}$, and the $k$-point function for $\mathcal{L}$ in Theorem \ref{thm_k_point} can be approximated by its counterpart for $\mathcal{L}^{\circ}$.

  \begin{lemma}\label{lemma_apply_russo}

  	(1) There exist $\cl\label{lemma_apply_russo1},\cl\label{lemma_apply_russo2}>0$ such that for any sufficiently small $\delta>0$,  		\begin{equation}\label{z6.2}
  	  	\mathbb{P}\big(   | \delta\cdot  \mathcal{C}_{\bm{0}}|_{\infty} \ge 1 \big) \le  \mathbb{P}\big(  |\delta\cdot  \widehat{\mathcal{C}}_{\bm{0}}^{\circ}|_{\infty} \ge 1 - \delta^{\cref{lemma_apply_russo1}}  \big)  +  \delta^{\frac{d}{2}-1+ \cref{lemma_apply_russo2} }. 
  	\end{equation}

(2) For $a,b>0$, there exists $\cl\label{lemma_apply_russo4}(d,a,b)>0$ such that for any sufficiently small $\delta >0$ and $x_1,...,x_k\in \mathbb{Z}^d$ satisfying $a\delta^{-1} \le  |x_i-x_j| \le b\delta^{-1}$ for all $1\le i<j\le k$, 
\begin{equation}\label{r618}
\begin{split}
	\mathbb{P}(x_1\xleftrightarrow{} x_2 \xleftrightarrow{} ... \xleftrightarrow{} x_k)  
	 \le  \mathbb{P}(x_1 \xleftrightarrow{\circ } x_2 \xleftrightarrow{ \circ } ... \xleftrightarrow{\circ } x_k)  + \delta^{(\frac{d}{2}-1)k+\cref{lemma_apply_russo4}}. 
\end{split}
\end{equation} 

  \end{lemma}
 

 	 
   \begin{proof}
 	(1) We claim that there exists $c_\dagger>0$ such that 
 	 \begin{equation}\label{claim619}
 	 \int_{\frac{1}{2}\le t\le 2}	\big[  \mathbb{P}\big(   | \delta\cdot  \mathcal{C}_{\bm{0}}|_{\infty} \ge t \big)  - \mathbb{P}\big(  |\delta\cdot   \mathcal{C}_{\bm{0}}^{\circ}|_{\infty} \ge t  \big) \big] \mathrm{d}t \lesssim   \delta^{\frac{d}{2}-1+c_\dagger}. 
 	 \end{equation}
 	 In fact, by the monotonicity of $\mathbb{P} (   | \delta\cdot  \mathcal{C}_{\bm{0}}|_{\infty} \ge t  )$ and $\mathbb{P} (  |\delta\cdot   \mathcal{C}_{\bm{0}}^{\circ}|_{\infty} \ge t  )$, we have 
 	 \begin{equation}\label{r620}
 	 	\begin{split}
 	 	& 	\mathbb{P}\big(   | \delta\cdot  \mathcal{C}_{\bm{0}}|_{\infty} \ge 1  \big)  - \mathbb{P}\big(  |\delta\cdot   \mathcal{C}_{\bm{0}}^{\circ}|_{\infty} \ge 1-\delta^{\frac{1}{2}c_\dagger}  \big) \\ 
 	 		\le & \delta^{-\frac{1}{2}c_\dagger}\int_{ 1-\delta^{\frac{1}{2}c_\dagger}  \le t\le 1 }  \big[  \mathbb{P}\big(   | \delta\cdot  \mathcal{C}_{\bm{0}}|_{\infty} \ge t \big)  - \mathbb{P}\big(  |\delta\cdot   \mathcal{C}_{\bm{0}}^{\circ}|_{\infty} \ge t  \big) \big]  \mathrm{d}t    \overset{(\ref{claim619})}{\lesssim } \delta^{\frac{d}{2}-1+\frac{1}{2}c_\dagger}. 
 	 	\end{split}
 	 \end{equation}
 	 Moreover, on the event $\{ |\delta\cdot   \mathcal{C}_{\bm{0}}^{\circ}|_{\infty} \ge 1-\delta^{\frac{1}{2}c_\dagger} ,  |\delta\cdot  \widehat{ \mathcal{C}}_{\bm{0}}^{\circ}|_{\infty} < 1 -   \delta^{\frac{1}{2}c_\dagger}  - \delta^{\frac{1}{2}\bref{para_3}}    \} $, there must exist a cluster in $\mathcal{L}^{\nu}$ that intersects $B(2\delta^{-1})$ and has diameter at least $\delta^{\frac{1}{2}\bref{para_3}-1}$. However, by Lemma \ref{lemma_crossing_massiveGFF}, the probability of this event is $O(e^{-c\delta^{-\frac{1}{2}\bref{para_3}}})$, which together with (\ref{r620}) implies the desired bound (\ref{z6.2}).

 	  In what follows, we prove (\ref{claim619}). We assign to each loop $\ell\in\widecheck{\mathcal{L}}$ (recall (\ref{u6.1})) an independent label $U_\ell$ uniformly distributed on $[0,1]$. For any $\alpha\in[0,\frac12]$, we define
$$\widecheck{\mathcal{L}}^{\alpha}:=\{\ell\in\widecheck{\mathcal{L}}: U_\ell\le 2\alpha\}.$$
By the thinning property of Poisson point processes, $\widecheck{\mathcal{L}}^{\alpha}$ is distributed as $\widecheck{\mathcal{L}}$ with intensity $1/2$ replaced by $\alpha$. Moreover, $\widecheck{\mathcal{L}}^{\alpha}$ is increasing in $\alpha$, and $\widecheck{\mathcal{L}}^{1/2}=\widecheck{\mathcal{L}}$. For any loop $\ell$, let $\mathcal{C}_{\bm{0},\ell}^{\alpha}$ (resp. $\mathcal{C}_{\bm{0} }^{\alpha }$) denote the cluster of $\mathcal{L}-\widecheck{\mathcal{L}}^{\alpha}+\mathbbm{1}_{\ell}$ (resp. $\mathcal{L}-\widecheck{\mathcal{L}}^{\alpha}$) containing $\bm{0}$. We also take the function $g(u) := \int_{\frac{1}{2}\le t\le 2 } \mathbbm{1}_{ u\ge \delta^{-1}t  } \mathrm{d}t$, and let $J^{\alpha}_{\ell}:=g(|\mathcal{C}_{\bm{0},\ell}^{\alpha} |_\infty) - g(|\mathcal{C}_{\bm{0} }^{\alpha } |_\infty)$. By Russo's formula, one has 
 	   \begin{equation}
 	   	\begin{split}
 	   		& \mathbb{P}\big(   | \delta\cdot  \mathcal{C}_{\bm{0}}|_{\infty} \ge t \big)  - \mathbb{P}\big(  |\delta\cdot   \mathcal{C}_{\bm{0}}^{\circ}|_{\infty} \ge t  \big) \\
 	   		 =& \int_{0\le \alpha\le \frac{1}{2}}  \int_{\ell}   \mathbb{P}\big(   | \delta\cdot \mathcal{C}_{\bm{0},\ell}^{\alpha}|_{\infty} \ge t , |\delta\cdot  \mathcal{C}_{\bm{0} }^{\alpha }|_{\infty} < t  \big) \hat{\mu}(\mathrm{d} \ell)  \mathrm{d}\alpha   \\
 	   		 =&   \int_{0\le \alpha\le \frac{1}{2}}  \int_{\ell}  \mathbb{E}\big[ \mathbbm{1}_{ |\mathcal{C}_{\bm{0},\ell}^{\alpha} |_\infty \ge \delta^{-1}t}  - \mathbbm{1}_{ |\mathcal{C}_{\bm{0} }^{\alpha }|_\infty \ge \delta^{-1} t} \big]  \hat{\mu}(\mathrm{d} \ell)  \mathrm{d}\alpha ,
 	   	\end{split}
 	   \end{equation}	 
 	   where $\hat{\mu}$ denotes the intensity measure of $\widecheck{\mathcal{L}}$. Thus, by taking the integral over $t\in [\frac{1}{2},2]$ and using Tonelli's theorem, we obtain  
 	   \begin{equation}\label{c6.22}
 	   	\begin{split}
 	   		 &  \int_{\frac{1}{2}\le t\le 2}	\big[  \mathbb{P}\big(   | \delta\cdot  \mathcal{C}_{\bm{0}}|_{\infty} \ge t \big)  - \mathbb{P}\big(  |\delta\cdot   \mathcal{C}_{\bm{0}}^{\circ}|_{\infty} \ge t  \big) \big] \mathrm{d}t  \\
 	   		 =& \int_{\frac{1}{2}\le t\le 2}  \int_{0\le \alpha\le \frac{1}{2}}  \int_{\ell}  \mathbb{E}\big[ \mathbbm{1}_{|\mathcal{C}_{\bm{0},\ell}^{\alpha}|_\infty \ge \delta^{-1}t}  - \mathbbm{1}_{|\mathcal{C}_{\bm{0} }^{\alpha }|_{\infty} \ge \delta^{-1} t} \big] \mathrm{d}\hat{\mu}(\mathrm{d} \ell)  \mathrm{d}\alpha  \mathrm{d}t  \\
 	   		 =&   \int_{0\le \alpha\le \frac{1}{2}}   \int_{\ell}    \mathbb{E}\big[J^{\alpha}_{\ell} \big] \hat{\mu}(\mathrm{d} \ell)   \mathrm{d}\alpha.  
 	   	\end{split}
 	   \end{equation}
For any $z\in \mathbb{Z}^d$ and $l\in \mathbb{N}^+$, let
\begin{equation*}
	\mathcal{H}_{z,l}:=2^l\vee \sup_{w\in\mathcal{C}_{\bm{0},\ell}^{\alpha} \setminus \mathcal{C}_{\bm{0}}^{\alpha}}|z-w|.
\end{equation*}
 For $m\in\mathbb{N}$, we define $\mathcal{I}_{z,l,m}$ to be the integral in (\ref{c6.22}) restricted to the case when the pivotal loops are contained in $B(z,2^l)$, have diameter at least $\frac{2^{l}}{10}$, and satisfy the following: $\mathcal{H}_{z,l}\ge \frac{1}{2}\delta^{-1}$ if $m=0$, and $\mathcal{H}_{z,l}\in [2^{m+l-1},2^{m+l}]$ if $m\ge 1$. Thus, 
  	   \begin{equation}\label{h623}
 	 	\begin{split}
 	 	   \int_{0\le \alpha\le \frac{1}{2}}   \int_{\ell}    \mathbb{E}\big[J^{\alpha}_{\ell} \big] \hat{\mu}(\mathrm{d} \ell)   \mathrm{d}\alpha  \le \sum_{l\in \mathbb{N}^+:2^l\le 10\delta^{\bref{para_4}-1}}\sum_{ m:0\le 2^m\le 2^{-l-1}\delta^{-1}} \sum_{z\in \mathcal{Z}_l}\mathcal{I}_{z,l,m},
 	 	\end{split}
 	 \end{equation}
where $\mathcal{Z}_l:= \{z\in 2^{l-1}\cdot \mathbb{Z}^d: B(z,2^{l})\cap B(\delta^{-1})\neq \emptyset \}$.

 	   Next, we estimate $\mathcal{I}_{z,l,m}$ for $m\ge 0$ separately.

 	  	 \textbf{For $m=0$.} In this case, there exists a loop cluster that is disjoint from $\mathcal{C}_{\bm{0},\ell}^{\alpha}$ and connects $B(z,2^l)$ to $\partial B(z,\frac{1}{10}\delta^{-1})$. Let $n:= \lfloor 2^{-l}|z| \rfloor+1$ . For the same reason as in the proof of Lemma \ref{lemma3.10_two_arm}, the probability of this event is at most 
  \begin{equation}\label{q66}
  \begin{split}
    	& \mathbb{P}\big( B(z,2^l) \xleftrightarrow{} \bm{0} \big) \cdot   \mathbb{P}\big( B(z,2^l) \xleftrightarrow{}\partial B(z,\tfrac{1}{10}\delta^{-1})  \big) \cdot C \big( \tfrac{|z|\land \delta^{\bref{para_3}-1} }{2^l}  \big)^{-\cref{const_check_loop} }  \\
  \lesssim 	& n^{2-d}2^{-l(\frac{d}{2}-1)}\cdot 2^{l(\frac{d}{2}-1)} \delta^{\frac{d}{2}-1}\cdot  (n\land 2^{-l} \delta^{\bref{para_3}-1})^{-\cref{const_check_loop} }\\
  =&  n^{2-d} \delta^{\frac{d}{2}-1} \cdot   (n +  2^{-l} \delta^{\bref{para_3}-1})^{-\cref{const_check_loop} }. 
  \end{split}
  \end{equation}
Combined with (\ref{bound26}), it implies 
 \begin{equation}\label{h625}
 	\begin{split}
 		 \sum\nolimits_{z\in \mathcal{Z}_l} 	\mathcal{I}_{z,l,0} 
 		 \lesssim & \sum\nolimits_{1\le n\le  \frac{1}{10}\cdot 2^{-l}\delta^{-1}  } 2^{2l}\delta^{2(1-\bref{para_3})} \cdot n  \delta^{\frac{d}{2}-1} \cdot   (n +  2^{-l} \delta^{\bref{para_3}-1})^{-\cref{const_check_loop} }  \\
 		 \lesssim  & 2^{\cref{const_check_loop}l}\delta^{\frac{d}{2}-1+\cref{const_check_loop}-2\bref{para_3}}. 
 	\end{split}
 \end{equation}

 	\textbf{For $1\le m\le 2^{-l-1}\delta^{-1}$.} Note that $g(u_1)-g(u_2)\le \delta|u_1-u_2|$ for all $u_1,u_2>0$, and that $|\mathcal{C}_{\bm{0},\ell}^{\alpha} |_\infty- |\mathcal{C}_{\bm{0} }^{\alpha } |_\infty\lesssim \mathcal{H}_{z,l}\lesssim 2^{m+l}$ in this case. Therefore,
 	\begin{equation}
 		J^{\alpha}_{\ell} \lesssim \delta 2^{m+l}. 
 	\end{equation}
 	Next, we estimate the probability of this case. Firstly, there exists a cluster $\mathcal{C}_\star$ that is disjoint from $\mathcal{C}_{\bm{0},\ell}^{\alpha}$ and connects $B(z,2^l)$ to $B(z,2^{m+l})$. Secondly, since $\mathcal{H}_{z,l} \in [ 2^{m+l-1},2^{m+l}]$, there exists a hyperplane $\mathcal{P}$ orthogonal to one of the coordinate axes such that $\operatorname{dist}(z,\mathcal{P})=2^{m+l}$ and $\mathcal{C}_{\bm{0},\ell}^{\alpha} \cap \mathcal{P}=\emptyset$. Thus, similar to (\ref{q66}), the probability of the intersection of these connecting events is at most 
  	  \begin{equation}
 	  	\begin{split}
 	  		& \mathbb{P}\big( B(z,2^l) \xleftrightarrow{(\mathcal{P} )} \bm{0} \big) \cdot   \mathbb{P}\big( B(z,2^l) \xleftrightarrow{}\partial B(z,2^{m+l})  \big) \cdot C  2^{-\cref{const_check_loop}m}  \\
 	  	\overset{(\ref{ineq_bound_5.2}),(\ref{crossing_prob})}{\lesssim}	  &  2^{l(\frac{d}{2}-1)}\mathbb{P}\big( z \xleftrightarrow{(\mathcal{P} )} \bm{0} \big) \cdot 2^{-(\frac{d}{2}-1+\cref{const_check_loop})m} \lesssim  2^{\frac{1}{2}dl-(\frac{d}{2}-2+\cref{const_check_loop} )m }\delta^{d-1}. 
 	  		  	  	\end{split}
  	  \end{equation}
 	Consequently, we have 
 	\begin{equation}\label{h628}
 		\begin{split}
 		& \sum\nolimits_{m: 0\le 2^m\le 2^{-l-1}\delta^{-1}}  	 \sum\nolimits_{z\in \mathcal{Z}_l} \mathcal{I}_{z,l,m}  \\
 		 \lesssim &\sum\nolimits_{ m: 0\le 2^m\le 2^{-l-1}\delta^{-1}}  	 \sum\nolimits_{z\in \mathcal{Z}_l}  \delta 2^{m+l} \cdot 2^{2l}\delta^{2(1-\bref{para_3})} \cdot  2^{\frac{1}{2}dl-(\frac{d}{2}-2+\cref{const_check_loop} )m }\delta^{d-1} \\
 			  \lesssim & \sum\nolimits_{ m: 0\le 2^m\le 2^{-l-1}\delta^{-1}} \delta^{2(1-\bref{para_3})}2^{(3-\frac{d}{2})l+(3-\frac{d}{2}-\cref{const_check_loop})m}\lesssim  2^{\cref{const_check_loop}l}\delta^{\frac{d}{2}-1+\cref{const_check_loop}-2\bref{para_3}}.
 		\end{split}
 	\end{equation}

 Plugging (\ref{h625}) and (\ref{h628}) into (\ref{h623}), we obtain 
 	 \begin{equation}
 	 	\begin{split}
 	 			  \int_{0\le \alpha\le \frac{1}{2}} J^{\alpha}_{\ell} \hat{\mu}(\mathrm{d} \ell) \lesssim & \sum\nolimits_{l\in \mathbb{N}^+:2^l\le 10\delta^{\bref{para_4}-1}}2^{\cref{const_check_loop}l}\delta^{\frac{d}{2}-1+\cref{const_check_loop}-2\bref{para_3}} \lesssim \delta^{\frac{d}{2}-1+ \cref{const_check_loop} \bref{para_4}- 2 \bref{para_3}  }. 
 	 	\end{split}
 	 \end{equation}
  This together with (\ref{c6.22}) yields the claim (\ref{claim619}).


    (2) By Russo's formula, the difference $\mathbb{P}(x_1\xleftrightarrow{} x_2 \xleftrightarrow{} ... \xleftrightarrow{} x_k) - \mathbb{P}(x_1 \xleftrightarrow{\circ } x_2 \xleftrightarrow{ \circ } ... \xleftrightarrow{\circ } x_k)$ is bounded by the supremum over $\alpha\in [0,1/2]$ of the expected total mass of the loops in $\widecheck{\mathcal{L}}$ that are pivotal for the event $\{x_1\xleftrightarrow{(\alpha)} x_2 \xleftrightarrow{} ... \xleftrightarrow{(\alpha)} x_k\}$. Here ``$\xleftrightarrow{(\alpha)}$'' denotes the connection via the loops in $\mathcal{L}-\widecheck{\mathcal{L}}^{\alpha}$. When the pivotal loop is contained in $B(z,2^l)$ and has diameter at least $\frac{2^{l}}{10}$, there exists a division $\Upsilon_1,\Upsilon_2$ of $\{1,2,...,k\}$ and two disjoint clusters $\mathcal{C}_{1},\mathcal{C}_{2}$ such that $\mathcal{C}_{1}$ and $\mathcal{C}_{2}$ both intersect $B(z,2^l)$ and for each $j\in \{1,2\}$, $\mathcal{C}_{j}$ contains every $x_i$ with $i\in \Upsilon_j$. Similar to (\ref{ineq344}), the probability of this event is at most proportional to 
    \begin{equation}
    	\begin{split}
    		 \delta^{(\frac{d}{2}-1)\cdot \#\Upsilon_1} \cdot  \delta^{(\frac{d}{2}-1)\cdot \#\Upsilon_2} \cdot \big( \tfrac{2^l}{\delta^{-1} }\big)^{d-2+\cref{const_check_loop} } =   \delta^{(\frac{d}{2}-1)k +\cref{const_check_loop} +d-2 } 2^{ (\cref{const_check_loop} +d-2)l}. 
    	\end{split}
    \end{equation}
   Combined with (\ref{bound26}), it implies that the total mass is at most of order
    \begin{equation}
   \begin{split}
   & \sum\nolimits_{l\in \mathbb{N}^+:2^l\le 10\delta^{\bref{para_4}-1}} \sum\nolimits_{z\in \mathcal{Z}_l} 2^{2l}\delta^{2(1-\bref{para_3})} \cdot \delta^{(\frac{d}{2}-1)k +\cref{const_check_loop} +d-2 } 2^{ (\cref{const_check_loop} +d-2)l}\\
   \lesssim & \delta^{(\frac{d}{2}-1)k +\cref{const_check_loop} -2 \bref{para_3}  }  \sum\nolimits_{l\in \mathbb{N}^+:2^l\le 10\delta^{\bref{para_4}-1}}   2^{  \cref{const_check_loop} l} \lesssim   \delta^{(\frac{d}{2}-1)k +\cref{const_check_loop} \bref{para_4} -2 \bref{para_3}  }. 
   \end{split}
    \end{equation}
 This completes the proof of Item (2). 
  \end{proof}

 With these two lemmas at hand, we are now ready to prove
Theorem~\ref{thm_asymptotic}. 
 
 \begin{proof}[Proof of Theorem \ref{thm_asymptotic}]
 	    We claim that there exist $c_\dagger,c_{\ddagger},C_{\star}>0$ such that for any $N_1\ge 10$ and $N_2\in (\frac{N_1}{10}, 10N_1)$, 
     \begin{equation}\label{new6.7}
     	\begin{split}
     	N_1^{\frac{d}{2}-1}	\mathbb{P}\big(  |  \mathcal{C}_{\bm{0}}  |_{\infty} \ge  N_1  \big)  \le  	N_2^{\frac{d}{2}-1}	 \mathbb{P}\big(  |  \mathcal{C}_{\bm{0}}  |_{\infty} \ge  N_2(1-N_1^{-c_\dagger})  \big)+ C_{\star} N_1^{-c_{\ddagger}}. 
     	\end{split}
     \end{equation} 
     To see this, by taking $\delta=N_1^{-1}$ in (\ref{z6.2}), we have 
     \begin{equation}\label{d6.28}
     	\begin{split}
     		\mathbb{P}\big(   | N_1^{-1}\cdot   \mathcal{C}_{\bm{0}}  |_{\infty} \ge 1  \big)  \le  \mathbb{P}\big(   | N_1^{-1}\cdot   \mathcal{C}_{\bm{0}}^{\circ }  |_{\infty} \ge 1 - N_1^{-\cref{lemma_apply_russo1}}  \big)    + N_1^{- \cref{lemma_apply_russo2} }. 
     	\end{split}
     \end{equation}
     Meanwhile, applying (\ref{a6.2}) with $\delta=N_1^{-1}$ and $\delta_*= N_2^{-1}$ gives 
    \begin{equation}\label{d6.29}
    	\mathbb{P}\big(   | N_1^{-1}\cdot   \mathcal{C}_{\bm{0}}^{\circ }  |_{\infty} \ge 1 - N_1^{-\cref{lemma_apply_russo1}}  \big)  \le \mathbb{P}\big(   | N_2^{-1}\cdot   \mathcal{C}_{\bm{0}}^{\circ }  |_{\infty} \ge 1 - N_1^{-\cref{lemma_apply_russo1}} - N_1^{-\cref{const_compare_lemma1}}  \big) + N_1^{-\cref{const_compare_lemma2}}.
    \end{equation}
     Combining (\ref{d6.28}) and (\ref{d6.29}), we conclude the claim (\ref{new6.7}).

     Next, we prove (\ref{ineq_thm_asymptotic1}) and (\ref{ineq_thm_asymptotic2}) using (\ref{new6.7}). For any $n_1\ge 10$ and $n_2\in (\frac{n_1}{10},10n_1)$, by taking $N_1=n_2(1-n_1^{-c_\dagger})$ and $N_2=n_2(1+10n_1^{-c_\dagger})$ in (\ref{new6.7}), we have 
      \begin{equation}
      \begin{split}
             	\mathbb{P}\big(  |  \mathcal{C}_{\bm{0}}  |_{\infty} \ge  n_2(1-n_1^{-c_\dagger})  \big)  
 	\le & \Big( \tfrac{1+10n_1^{-c_\dagger} }{1-n_1^{-c_\dagger}} \Big)^{\frac{d}{2}-1}  \mathbb{P}\big(  |  \mathcal{C}_{\bm{0}}  |_{\infty} \ge  n_2 \big)  +   C_\star n_2^{-\frac{d}{2}+1}N_1^{-c_\ddagger}\\
 	\overset{(\ref{one_arm_lowd})}{\le} &  \mathbb{P}\big(  |  \mathcal{C}_{\bm{0}}  |_{\infty} \ge  n_2 \big) +    C n_1^{-\frac{d}{2}+1-c_\dagger \land c_\ddagger}, 
      \end{split}
  \end{equation}
      where in the second inequality we used the fact that 
      \begin{equation}
      	 N_2(1-N_1^{-c_\dagger}) = n_2(1+10n_1^{-c_\dagger})\big[ 1-(n_2(1-n_1^{-c_\dagger}))^{-c_\dagger} \big]   >n_2.
      \end{equation}
      Combined with (\ref{new6.7}) for $N_1=n_1$ and $N_2=n_2$, we obtain 
      \begin{equation}
      	\begin{split}
      			n_1^{\frac{d}{2}-1}	\mathbb{P}\big(  |  \mathcal{C}_{\bm{0}}  |_{\infty} \ge  n_1  \big)  \le  n_2^{\frac{d}{2}-1}	\mathbb{P}\big(  |  \mathcal{C}_{\bm{0}}  |_{\infty} \ge  n_2  \big)  + C'n_1^{-c_\dagger \land c_\ddagger}. 
      	\end{split}
      \end{equation}
      Swapping the roles of $n_1$ and $n_2$, one has
      \begin{equation}\label{u6.11}
      	\big| n_1^{\frac{d}{2}-1}	\mathbb{P}\big(  |  \mathcal{C}_{\bm{0}}  |_{\infty} \ge  n_1  \big) - n_2^{\frac{d}{2}-1}	\mathbb{P}\big(  |  \mathcal{C}_{\bm{0}}  |_{\infty} \ge  n_2  \big)  \big| \lesssim n_1^{-c_\dagger \land c_\ddagger}.
      \end{equation}

      \textbf{Proof of (\ref{ineq_thm_asymptotic1}).} Arbitrarily take $M_2>M_1\ge 10$. For each $i\in \{1,2\}$, assume that $M_i\in [2^{m_i}, 2^{m_i+1})$. By applying (\ref{u6.11}) for $n_1=M_i$ and $n_2=2^{m_i}$, one has 
      \begin{equation}\label{u_6.21}
      		\big| M_i^{\frac{d}{2}-1}	\mathbb{P}\big(  |  \mathcal{C}_{\bm{0}}  |_{\infty} \ge  M_i  \big) - 2^{(\frac{d}{2}-1)m_i} \mathbb{P}\big(  |  \mathcal{C}_{\bm{0}}  |_{\infty} \ge  2^{m_i}  \big)  \big| \lesssim 2^{-(c_\dagger \land c_\ddagger)m_i}. 
      \end{equation}
 Likewise, for each $m_1\le j\le m_2-1$, we have 
 \begin{equation}\label{u_6.22}
 	     		\big| 2^{(\frac{d}{2}-1)j} \mathbb{P}\big(  |  \mathcal{C}_{\bm{0}}  |_{\infty} \ge  2^{j}  \big) - 2^{(\frac{d}{2}-1)(j+1)} \mathbb{P}\big(  |  \mathcal{C}_{\bm{0}}  |_{\infty} \ge  2^{j+1}  \big) \big| \lesssim 2^{-(c_\dagger \land c_\ddagger)j}. 
 \end{equation}
      Combining (\ref{u_6.21}) and (\ref{u_6.22}), we obtain 
      \begin{equation}
      		\big| M_1^{\frac{d}{2}-1}	\mathbb{P}\big(  |  \mathcal{C}_{\bm{0}}  |_{\infty} \ge  M_1 \big) - M_2^{\frac{d}{2}-1}	\mathbb{P}\big(  |  \mathcal{C}_{\bm{0}}  |_{\infty} \ge  M_2 \big)  \big| \lesssim M_1^{-c_\dagger \land c_\ddagger}.  
      \end{equation}
   This shows that $\{N^{\frac{d}{2}-1}	\mathbb{P} (  |  \mathcal{C}_{\bm{0}}  |_{\infty} \ge  N  )\}_{N\ge 100}$ is a Cauchy sequence, and thus the limit $\xi_{d,\infty}$ in (\ref{ineq_thm_asymptotic1}) exists.

      \textbf{Proof of (\ref{ineq_thm_asymptotic2}).} For any $N\ge 10$ and $k\in \mathbb{N}^+$, we have 
      \begin{equation}
      	\begin{split}
      		 & 	\big| N^{\frac{d}{2}-1}	\mathbb{P}\big(  |  \mathcal{C}_{\bm{0}}  |_{\infty} \ge   N \big) - (2^{k }N)^{\frac{d}{2}-1}	\mathbb{P}\big(  |  \mathcal{C}_{\bm{0}}  |_{\infty} \ge  2^{k }N \big)  \big| \\
      		 	\le &  \sum_{0\le j\le k-1}\big| (2^jN)^{\frac{d}{2}-1}	\mathbb{P}\big(  |  \mathcal{C}_{\bm{0}}  |_{\infty} \ge   2^j N \big) - (2^{j+1 }N)^{\frac{d}{2}-1}	\mathbb{P}\big(  |  \mathcal{C}_{\bm{0}}  |_{\infty} \ge  2^{j+1}N \big)  \big| \\
      		 \overset{(\ref{u6.11})}{ \lesssim } & \sum_{0\le j\le k-1}  (2^j N)^{-c_\dagger \land c_\ddagger } \asymp N^{-c_\dagger \land c_\ddagger}. 
      	\end{split}
      \end{equation}
      Letting $k\to \infty$, this gives the desired bound (\ref{ineq_thm_asymptotic2}).   
 \end{proof}

 We now turn to the convergence of the $k$-point function.

     \begin{proof}[Proof of the convergence of (\ref{def_gamma_k}) in Theorem \ref{thm_k_point}]
     	  We define   
 \begin{equation}
 	q(\delta):=\delta^{-(\frac{d}{2}-1)k}  \cdot  \mathbb{P}(\hat{x}_1^{\delta}\xleftrightarrow{} \hat{x}_2^{\delta} \xleftrightarrow{} ... \xleftrightarrow{} \hat{x}_k^{\delta}), 
 \end{equation}
   \begin{equation}
 	q^{\circ}(\delta):=\delta^{-(\frac{d}{2}-1)k}  \cdot  \mathbb{P}(\hat{x}_1^{\delta}\xleftrightarrow{\circ} \hat{x}_2^{\delta} \xleftrightarrow{\circ} ... \xleftrightarrow{\circ} \hat{x}_k^{\delta}). 
 \end{equation} 
   Arbitrarily fix $\delta_1>\delta_2>0$. For each $j\in \{1,2\}$, assume that $\delta_j\in [10^{-m_j-1},10^{-m_j}]$ for $m_j\in \mathbb{N}$. For each integer $l\in [0,m_2-m_1]$, denote $\lambda_l:=10^{-m_1-l-1}$. For convenience, we also set $\lambda_{-1}:=\delta_1$ and $\lambda_{m_2-m_1+1}:=\delta_2$. For $-1\le l\le m_2-m_1$, by taking $\delta=\lambda_{l}$ and $\delta_*=\lambda_{l+1}$ in (\ref{h6.3}) and swapping the roles of $\delta$ and $\delta_*$, we have 
   \begin{equation}
   	|q^{\circ}( \lambda_{l})-q^{\circ}( \lambda_{l+1})|\lesssim  \lambda_{l}^{ \cref{const_compare1_kpoint}}.  
   \end{equation}
     Combined with (\ref{r618}), it implies 
     \begin{equation}
     		|q( \lambda_{l})-q( \lambda_{l+1})|\lesssim  \lambda_{l}^{ \cref{const_compare1_kpoint} \land \cref{lemma_apply_russo4}}.  
     \end{equation}
Therefore, by the triangle inequality, one has      
     \begin{equation}\label{d648}
     	| q(\delta_1)  -q(\delta_2)  | \le \sum\nolimits_{-1\le l\le m_2-m_1 }|q( \lambda_{l})-q( \lambda_{l+1})| \lesssim  \delta_1^{\cref{const_compare1_kpoint} \land \cref{lemma_apply_russo4}}.     \end{equation}
     	 As a result, $q(\delta)$ satisfies the Cauchy criterion as $\delta \to 0$, and hence the limit $\lim_{\delta\to 0} q(\delta)$ in (\ref{def_gamma_k}) exists.  
     \end{proof}

 \section{Scaling limit}\label{section3_proof_prop1.2}

The aim of this section is to establish the scaling limit part of Theorem \ref{new_thm1}. We first establish the following property: with high probability, there is no loop cluster (after rescaling) within $B(2)$ whose diameter is close to $\epsilon$. Therefore, a local modification of a cluster inside a small box will not change whether it falls within the range under consideration. Recall $\Cref{const_asymptotic2}, \cref{const_asymptotic1}$ from Theorem \ref{thm_asymptotic}.

    \begin{lemma}\label{lemma_continuous_diameter}
 	For any $3\le d\le 5$, $0<a<b<\cref{const_asymptotic1}$, $\epsilon>0$, the clusters $\mathfrak{C}$ of the loop soup on $\widetilde{\mathbb{Z}}^d$ satisfy that for all sufficiently small $\delta>0$, 
 	\begin{equation}\label{345}
 	\mathbb{P}\big( \exists \mathcal{C}\in \mathfrak{C}\ \text{such that}\  \delta\cdot    \mathcal{C}\subset B(2)\ \text{and} \   \delta\cdot  | \mathcal{C}|_\infty \in [\epsilon   -\delta^{b} , \epsilon +\delta^{b}]   \big) \le  \delta^{a}.  
 	\end{equation}
 \end{lemma}
\begin{proof}
 For any $x\in \mathbb{Z}^d$, let $\mathcal{C}_x$ denote the loop cluster containing $x$. We then define 
 \begin{equation}
 	\mathsf{A}_x:= \big\{  \delta\cdot  |\mathcal{C}_x|_\infty \in  [\epsilon   -\delta^{b} , \epsilon +\delta^{b}]    \big\}. 
 \end{equation}
 It follows from (\ref{ineq_thm_asymptotic2}) that 
 \begin{equation}\label{y7.3}
 \begin{split}
 		\mathbb{P}(\mathsf{A}_x)= &  \mathbb{P}\big(\delta\cdot  |\mathcal{C}_x|_\infty  \ge \epsilon   -\delta^{b} \big) - \mathbb{P}\big(\delta\cdot  |\mathcal{C}_x|_\infty  \ge \epsilon   + \delta^{b} \big)  \\
 	\le &  \xi_{d,\infty}\delta^{\frac{d}{2}-1}\big[ (\epsilon-\delta^b)^{-\frac{d}{2}+1} -(\epsilon+\delta^b)^{-\frac{d}{2}+1}  \big]  \\
 	 &+ \Cref{const_asymptotic2} \delta^{\frac{d}{2}-1+\cref{const_asymptotic1} }  \big[(\epsilon-\delta^b)^{-\frac{d}{2}+1- \cref{const_asymptotic1} } + (\epsilon+\delta^b)^{-\frac{d}{2}+1-\cref{const_asymptotic1} }   \big]\\
 	 \le  & C\big(  \delta^{\frac{d}{2}-1+b}\epsilon^{-\frac{d}{2}} + \delta^{\frac{d}{2}-1+\cref{const_asymptotic1} }\epsilon^{-\frac{d}{2}+1-\cref{const_asymptotic1} }    \big)  \le \delta^{\frac{d}{2}-1+\frac{1}{2}(a+b)}.  
 \end{split}
 \end{equation}
 We denote by $\mathsf{D}_1$ the event on the left-hand side of (\ref{345}). We also define the event 
 \begin{equation}
 	\mathsf{D}_2:= \cup_{x\in \hat{B}(2\delta^{-1})} \big\{ |\mathcal{C}_x|_\infty\ge \tfrac{1}{2}\epsilon \delta^{-1},   \# \mathcal{C}_x \le \delta^{-\frac{d}{2}-1+\frac{1}{4}(b-a)}  \big\}. 
 \end{equation}
 Applying Lemma \ref{ulemma_volume} and the union bound, one has 
 \begin{equation}\label{r7.5}
 	\mathbb{P}( \mathsf{D}_2) \lesssim e^{-c\delta^{\frac{1}{5}(b-a)}}. 
 \end{equation}
 On the event $\mathsf{D}_1\cap \mathsf{D}_2^c$, there exists a cluster $\mathcal{C}$ satisfying the conditions in (\ref{345}) whose volume is at least $\delta^{-\frac{d}{2}-1+\frac{1}{4}(b-a)}$; hence, $\mathbf{X}:=  \sum\nolimits_{x\in \hat{B}(2\delta^{-1})} \mathbbm{1}_{ \mathsf{A}_x} \ge \delta^{-\frac{d}{2}-1+\frac{1}{4}(b-a)}$. Therefore, by Markov's inequality and (\ref{y7.3}), we have  
  \begin{equation}\label{r7.6}
 	 \begin{split}
 	 \mathbb{P}\big( \mathsf{D}_1 , \mathsf{D}_2^c  \big) \le  	 \mathbb{P}\big(  \mathbf{X} \ge \delta^{-\frac{d}{2}-1+\frac{1}{4}(b-a)}  \big) \le     \frac{ \mathbf{E}[\mathbf{X} ] }{ \delta^{ -\frac{d}{2}-1+ \frac{1}{4}(b-a)}  } \lesssim  \delta^{a+\frac{1}{4}(b-a) }. 
 	 \end{split}
 \end{equation} 
 Combining (\ref{r7.5}) and (\ref{r7.6}), we complete the proof. 
\end{proof}

       For any $D\subset \mathbb{R}^d$ and $n\in \mathbb{N}^+$, we define $\mathfrak{X}^n_D$ as the space of finite collections of at most $n$ open subsets of $D$, equipped with the distance $\mathrm{d}_{\mathrm{H}}^*$. Let $\mathfrak{X}_D:=\cup_{n\in \mathbb{N}}\mathfrak{X}^n_D$. For any $\mathcal{A}\in \mathfrak{X}^n_D$ and $r>0$, we define $\mathcal{A}+B(r):= \{B(A,r)\}_{A\in \mathcal{A}}$. For any $\mathcal{A}_1, \mathcal{A}_2 \in \mathfrak{X}_n$,  we write $\mathcal{A}_1\sqsubseteq\mathcal{A}_2$ if there exists a bijection $\sigma\in \mathrm{bij}(\mathcal{A}_1,\mathcal{A}_2)$ such that $A\subset \sigma(A)$ holds for all $A\in \mathcal{A}_1$. Recall the notation $\mathfrak{D}$ and $\mathfrak{C}$ from Theorem \ref{new_thm1}.

The key ingredient of the proof in this section is the following lemma.

\begin{lemma}\label{lemmab72}
	For $3 \le d\le 5$, there exist $\cl\label{const_lemmab72_1},\cl\label{const_lemmab72_2}>0$ such that for any $\epsilon>0$, $R>1$, $\mathcal{A}\in \mathfrak{X}^n_{B(R)}$, and any sufficiently small $\delta>0$ and $\delta'\in (\frac{\delta}{10},10\delta)$,
	\begin{equation}\label{b77}
		 	\mathbb{P}\big(  ( \overline{ \delta\cdot  \mathfrak{C}})^{>\epsilon}_{B(R)}\sqsubseteq  \mathcal{A}  \big) \le  \mathbb{P}\big(  ( \overline{   \delta' \cdot  \mathfrak{C} })^{>\epsilon}_{B(R)}\sqsubseteq\mathcal{A} + B(\delta^{\cref{const_lemmab72_1}})  \big)  +  \delta^{\cref{const_lemmab72_2}}.
	\end{equation}
\end{lemma} 
 \begin{proof}
 	 We employ the notation $\widehat{\mathcal{L}}$ and $\widecheck{\mathcal{L}}$ from (\ref{u6.1}), and recall that $\overline{\mathcal{L}}=\mathcal{L}^{\nu} + \widecheck{\mathcal{L}}$. Let $\mathfrak{C}^{\circ}$ denote the collection of clusters in $\mathcal{L}-\widecheck{\mathcal{L}}$. For each $\mathcal{C}^{\circ}\in\mathfrak{C}^{\circ}$, we define $\widehat{\mathcal{C}}^{\circ}$ as the union of loops $\ell\in \widehat{\mathcal{L}}$ contained in $\mathcal{C}^{\circ}$ (we set $\widehat{\mathcal{C}}^{\circ}:=\emptyset$ if $\mathcal{C}^{\circ}$ does not contain any loop in $\widehat{\mathcal{L}}$), and set $\widehat{\mathfrak{C}}^{\circ}:=\{\widehat{\mathcal{C}}^{\circ}:\mathcal{C}^{\circ}\in \mathfrak{C}^{\circ}\}$.

 For $i\in \{0,1\}$, we define $\mathsf{A}_i$ as the intersection of the following events: (for convenience, we set $\widehat{\mathfrak{C}}^{\circ}_0:=\mathfrak{C}$ and $\widehat{\mathfrak{C}}^{\circ}_1:=\widehat{\mathfrak{C}}^{\circ}$)

  \begin{itemize}

	\item  $\mathsf{A}_i^1$: $(\delta\cdot \widehat{ \mathfrak{C}}^{\circ}_i)^{>\epsilon}_{B(R)} \sqsubseteq  \mathcal{A}$;

	\item  $\mathsf{A}_i^2$: There does not exist $\mathcal{C} \in  \widehat{ \mathfrak{C}}^{\circ}_i$ such that $\mathcal{C}\subset B(2R\delta^{-1}-i\delta^{ \bref{para_error}-1 })$ and
  $$|\mathcal{C}|_\infty \in [\epsilon \delta^{-1} -  \delta^{ \bref{para_error}-1 }  ,\  \epsilon \delta^{-1}+(2-i)  \delta^{ \bref{para_error}-1 }   ];$$

	\item  $\mathsf{A}_i^3$: There does not exist $\mathcal{C}\in  \widehat{ \mathfrak{C}}^{\circ}_i$ such that $\mathcal{C}\subset B(R\delta^{-1} +(2-i)\delta^{ \bref{para_error}-1 })$, $\mathcal{C} \not\subset B(R\delta^{-1} - \delta^{ \bref{para_error}-1 })$ and $|\mathcal{C}| \ge \frac{1}{2}\epsilon \delta^{-1}$.

\end{itemize}
By Lemmas \ref{lemma_cluster_closetoboundary} and \ref{lemma_continuous_diameter}, we have
 \begin{equation}\label{a78}
 		\mathbb{P}\big(  ( \overline{ \delta\cdot  \mathfrak{C}})^{>\epsilon}_{B(R)}\sqsubseteq  \mathcal{A}  \big) \le  \mathbb{P} (\mathsf{A}_0)  +   2\delta^{\frac{1}{2}\bref{para_error}}.
 \end{equation}

Next, we aim to show that
	 \begin{equation}\label{ineq_i0_lemma3.1}
	\mathbb{P} (\mathsf{A}_0 ) \le \mathbb{P} (\mathsf{A}_{1} )   +  3\delta^{\cref{const_check_loop}\bref{para_4}-2d\bref{para_3}} + \delta^{\frac{1}{4}\cref{const_check_loop}\bref{para_error}},
 \end{equation}
  where the exponent $\cref{const_check_loop}\bref{para_4}-2d\bref{para_3}$ is positive by Condition (ii) in Section \ref{section5_convergence}. Before diving into the details, we briefly describe the proof strategy as follows. We begin with the event $\mathsf{A}_0$, concerning the clusters $\mathfrak{C}$ of the complete loop soup $\mathcal{L}$. Each cluster in $\mathfrak{C}$ is the union of some clusters in $\mathfrak{C}^{\circ}$ and some loops in $\widecheck{\mathcal{L}}$. Let $\eta:=\frac{1}{8}\delta^{\bref{para_error}-1}$. We will show that with high probability, every macroscopic cluster $\mathcal{C}\in\mathfrak{C}$ contains exactly one cluster $\mathcal{C}^{\circ}\in\mathfrak{C}^{\circ}$ with $|\mathcal{C}^{\circ}|\ge 3\eta$, and that
  \begin{equation}\label{incl_skeleton}
  	\widehat{\mathcal{C}}^{\circ}\subset \mathcal{C}\subset B\big(\widehat{\mathcal{C}}^{\circ},\tfrac{1}{2}\delta^{\bref{para_error}-1}\big).
  \end{equation}
  To this end, we need to control two effects:
  \begin{itemize}
  	\item[-] The loops in $\overline{\mathcal{L}}$ (which consists of the massive loop soup $\mathcal{L}^{\nu}$ and the remaining loops $\widecheck{\mathcal{L}}$ of diameter less than $\delta^{\bref{para_4}-1}$) attached to the loops in $\widehat{\mathcal{L}}$. We utilize Corollary \ref{coro3.11} to show that with high probability, the clusters in $\overline{\mathcal{L}}$ have diameters at most $\frac{1}{4}\delta^{\bref{para_error}-1}$.

  	\item[-] The loops in $\widecheck{\mathcal{L}}$ gluing different clusters in $\mathfrak{C}^{\circ}$. We show that with high probability, for every $\mathcal{C}^{\circ}\in\mathfrak{C}^{\circ}$ with $|\mathcal{C}^{\circ}|\ge 3\eta$, the cluster in $\mathfrak{C}$ containing $\mathcal{C}^{\circ}$ lies within distance $2\eta$ of $\mathcal{C}^{\circ}$ and does not contain any other cluster in $\mathfrak{C}^{\circ}$ of diameter at least $3\eta$. In particular, the loops in $\widecheck{\mathcal{L}}$ never glue two such clusters together.
  \end{itemize}

 We now turn to the detailed proof. For each loop $\ell$, we choose a point $x_\ell\in\mathrm{ran}(\ell)$ in a predetermined (measurable) way. We define the following events:
 \begin{itemize}
 
 	\item[$\mathsf{K}_1$:] there exists a cluster in $\overline{\mathcal{L}}$ that intersects $B(3R\delta^{-1})$ and has diameter greater than $\frac{1}{4}\delta^{\bref{para_error}-1}$;

 	\item[$\mathsf{K}_2$:] there exist $\mathcal{C}^{\circ}\in\mathfrak{C}^{\circ}$ with $|\mathcal{C}^{\circ}|\ge 3\eta$ and $\ell\in\widecheck{\mathcal{L}}$ with $\mathrm{ran}(\ell)\cap\mathcal{C}^{\circ}\neq\emptyset$ and $\mathrm{ran}(\ell)\cap B(3R\delta^{-1})\neq\emptyset$, such that $\mathrm{ran}(\ell)$ is connected to $\partial B(x_\ell,\eta)$ by loops in $\mathcal{L}-\mathbbm{1}_\ell$ that are disjoint from $\mathcal{C}^{\circ}$;

 	\item[$\mathsf{K}_3$:] there exists a cluster of $\mathcal{L}\cdot\mathbbm{1}_{|\ell|<3\eta}$ that intersects $B(3R\delta^{-1})$ and has diameter at least $\frac{1}{4}\epsilon\delta^{-1}$;

 	\item[$\mathsf{K}_4$:] there exists $\ell\in\mathcal{L}^{\nu}$ with $\mathrm{ran}(\ell)\cap B(3R\delta^{-1})\neq\emptyset$ and $|\ell|\ge 3\eta$.
 \end{itemize}
Let $\mathsf{K}:=\cup_{1\le j\le 4}\mathsf{K}_j$. In what follows, we first bound the probability of each $\mathsf{K}_j$, and then show that on $\mathsf{K}^c$, every macroscopic cluster in $\mathfrak{C}$ satisfies (\ref{incl_skeleton}).

 \textbf{For $\mathsf{K}_1$.} Note that $B(3R\delta^{-1})$ can be covered by $C\delta^{-d\bref{para_error}}$ balls of radius $c\delta^{\bref{para_error}-1}$, and that on the event $\mathsf{K}_1$, there exists one of these balls $B(z,c\delta^{\bref{para_error}-1})$ such that $B(z,c\delta^{\bref{para_error}-1})\xleftrightarrow{\overline{\mathcal{L}}} B(z,2c\delta^{\bref{para_error}-1})$ occurs. Therefore, by Corollary \ref{coro3.11}, we have
\begin{equation}\label{boundK1}
	 \begin{split}
	\mathbb{P}\big( \mathsf{K}_1 \big) \le   C \delta^{-d\bref{para_error}}\cdot    \frac{\delta^{ 2(  \bref{para_error}-1 ) } \delta^{ \cref{const_check_loop} (  \bref{para_4} -1 ) } }{\delta^{ (2+\cref{const_check_loop} ) (   \bref{para_3}-1 ) }}  <  \delta^{ \cref{const_check_loop}\bref{para_4}- 2d\bref{para_3}  }.
	 \end{split}
\end{equation}

  \textbf{For $\mathsf{K}_2$.} For $k\in\mathbb{N}^+$ with $2^k\le 10\delta^{\bref{para_4}-1}$, let $\mathcal{Z}_k$ be the set of $z\in 2^{k-1}\cdot\mathbb{Z}^d$ with $B(z,2^k)\cap B(3R\delta^{-1})\neq\emptyset$. Note that $|\mathcal{Z}_k|\lesssim (R\delta^{-1}2^{-k})^d$, and that every $\ell\in\widecheck{\mathcal{L}}$ intersecting $B(3R\delta^{-1})$ satisfies $\mathrm{ran}(\ell)\subset B(z,2^k)$ and $|\ell|\ge 2^k/10$ for some $z\in\mathcal{Z}_k$ and $k\in\mathbb{N}^+$ with $2^k\le 10\delta^{\bref{para_4}-1}$. For such $\ell$, if $\ell$ certifies $\mathsf{K}_2$ together with $\mathcal{C}^{\circ}$, then $\mathcal{C}^{\circ}$ crosses $B(z,\eta)\setminus B(z,2^k)$ (since $|\mathcal{C}^{\circ}|\ge3\eta$), and $B(z,2^k)$ is connected to $\partial B(z,\eta)$ by loops in $\mathcal{L}$ disjoint from $\mathcal{C}^{\circ}$. We denote the latter event (which does not involve $\ell$) by $\mathsf{H}_{z,k}$. By the BKR inequality and (\ref{bound26}), we have 
\begin{equation}\label{lemmab72_mecke}
\begin{split}
\mathbb{P}(\mathsf{K}_2)\le & \sum_{k: 2^k\le 10\delta^{\bref{para_4}-1}}\sum_{z\in\mathcal{Z}_k}
\check\mu\big(\{\ell:\mathrm{ran}(\ell)\subset B(z,2^k),\,|\ell|\ge 2^k/10\}\big)\cdot
\mathbb{P}(\mathsf{H}_{z,k})\\
\lesssim & \sum_{k:2^k\le 10\delta^{\bref{para_4}-1}}\sum_{z\in\mathcal{Z}_k} 2^{2k}\delta^{2(1-\bref{para_3})} 
\mathbb{P}(\mathsf{H}_{z,k}),
\end{split}
\end{equation}
where $\check\mu$ is the intensity measure of $\widecheck{\mathcal{L}}$.

To bound $\mathbb{P}(\mathsf{H}_{z,k})$, note that conditioned on a cluster $\mathcal{C}^{\circ}\in\mathfrak{C}^{\circ}$, by the restriction property and the independence between $\widecheck{\mathcal{L}}$ and $\mathcal{L}-\widecheck{\mathcal{L}}$, the loops in $\mathcal{L}$ disjoint from $\mathcal{C}^{\circ}$ have the same law as $\mathcal{L}^{\mathcal{C}^{\circ}}$. Therefore, the same argument as in the proof of Lemma \ref{lemma3.10_two_arm} applies: Corollary \ref{coro_density} (with $\mathcal{L}'=\mathcal{L}-\widecheck{\mathcal{L}}$) shows that $\mathcal{C}^{\circ}$ satisfies the dense property in (\ref{new3.5}) except on an event of negligible probability, and on this event (\ref{bound_crossing_by_hitting}) (with $D=\mathcal{C}^{\circ}$) bounds the conditional probability of the connection, exactly as in (\ref{ineq337})--(\ref{ineq_AcH}). These observations yield
\begin{equation}\label{lemmab72_Hzk}
\mathbb{P}(\mathsf{H}_{z,k})\lesssim
\frac{2^{k(d-2+\cref{const_check_loop})}}{\eta^{d-2}\,\delta^{\cref{const_check_loop}(\bref{para_3}-1)}}.
\end{equation} 
 	Plugging (\ref{lemmab72_Hzk}) into (\ref{lemmab72_mecke}), we get
\begin{equation}\label{lemmab72_K2}
\begin{split}
\mathbb{P}(\mathsf{K}_2)\lesssim{}& (R\delta^{-1})^d\,\delta^{(2+\cref{const_check_loop})(1-\bref{para_3})}\,
\eta^{2-d}\sum\nolimits_{k:\,2^k\le10\delta^{\bref{para_4}-1}}2^{\cref{const_check_loop}k}\\
\lesssim{}& R^d\delta^{\,\cref{const_check_loop}\bref{para_4}-(2+\cref{const_check_loop})\bref{para_3}-(d-2)\bref{para_error}}
< \delta^{ \cref{const_check_loop}\bref{para_4}- 2d\bref{para_3}  }.
\end{split}
\end{equation}

  \textbf{For $\mathsf{K}_3$.} Let $M:=\frac{1}{16}\epsilon\delta^{-1}$. The ball $B(3R\delta^{-1})$ can be covered by $O((R/\epsilon)^d)$ balls $B(y,\frac{1}{4}M)$ with $y\in\frac{1}{4}M\cdot\mathbb{Z}^d$. On $\mathsf{K}_3$, there exists a cluster of $\mathcal{L}\cdot\mathbbm{1}_{|\ell|<3\eta}$ crossing $B(y,M)\setminus B(y,\frac{1}{2}M)$. Let $a=\frac{2+\frac{1}{2}\cref{const_check_loop}}{2+\cref{const_check_loop}}\in(0,1) $, $N':=(3\eta)^{1-a}M^a$ and $\nu':=(N')^{-2}$. Since every loop in $\mathcal{L}\cdot\mathbbm{1}_{|\ell|<3\eta}$ belongs either to $\mathcal{L}^{\nu'}$ or to $(\mathcal{L}-\mathcal{L}^{\nu'})\cdot\mathbbm{1}_{|\ell|<3\eta}$, it follows from   Corollary \ref{coro3.11} (with $(m,N)$ replaced by $(3\eta,N')$) that 
\begin{equation}\label{lemmab72_K3}
\mathbb{P}(\mathsf{K}_3)\le C(\tfrac{R}{\epsilon})^d\Big[M^de^{-c(M/3\eta)^{2(1-a)}}
+\big(\tfrac{M}{3\eta}\big)^{-\frac{1}{2}\cref{const_check_loop}}\Big]
< \delta^{\frac{1}{4}\cref{const_check_loop}\bref{para_error}},
\end{equation}
where we used the fact that $\tfrac{M}{3\eta}=\frac{1}{6}\epsilon\delta^{-\bref{para_error}}$.

 \textbf{For $\mathsf{K}_4$.} The total mass of loops intersecting $B(3R\delta^{-1})$ with diameter at least $3\eta$ is $O((R\delta^{-1}/\eta)^d)$ (see (\ref{bound2.12})). Moreover, the duration of such a loop is at least $\eta\delta^{\bref{para_3}-1}$ except with probability $e^{-c\eta\delta^{1-\bref{para_3}}}$ (see e.g., \cite[Proposition 2.4.5]{Lawler2010RandomWA}), and hence it survives the killing at rate $\nu=\delta^{2(1-\bref{para_3})}$ with probability at most $2e^{-c\eta\delta^{1-\bref{para_3}}}$. Since $\eta\delta^{1-\bref{para_3}}=\frac{1}{8}\delta^{\bref{para_error}-\bref{para_3}}$, we obtain \begin{equation}\label{lemmab72_K4}
	\mathbb{P}(\mathsf{K}_4)<\delta^{ \cref{const_check_loop}\bref{para_4}- 2d\bref{para_3}  }.
\end{equation}

Combining (\ref{boundK1}), (\ref{lemmab72_K2}), (\ref{lemmab72_K3}) and (\ref{lemmab72_K4}), we have 
  	 \begin{equation}\label{q717}
  	 	\mathbb{P}(\mathsf{K})  \le 3\delta^{ \cref{const_check_loop}\bref{para_4}- 2d\bref{para_3}  }+ \delta^{\frac{1}{4}\cref{const_check_loop}\bref{para_error}}. 
  	 \end{equation}
 Next, we show that on $\mathsf{K}^c$, for every $\mathcal{C}^{\circ}\in\mathfrak{C}^{\circ}$ with $|\mathcal{C}^{\circ}|\ge 3\eta$ and $\widehat{\mathcal{C}}^{\circ}\subset B(2R\delta^{-1})$, there exists a cluster $\mathcal{C}\in\mathfrak{C}$ such that $\mathcal{C}$ does not contain any other cluster in $\mathfrak{C}^{\circ}$ of diameter at least $3\eta$ and that \eqref{incl_skeleton} holds (and such $\mathcal{C}$ is clearly unique). First, it follows from $\mathsf{K}_1^c$ that  
 \begin{equation}\label{lemmab72_S1}
 	\mathcal{C}^{\circ}\subset B\big(\widehat{\mathcal{C}}^{\circ},\tfrac{1}{4}\delta^{\bref{para_error}-1}\big)\subset B\big(\tfrac{5}{2}R\delta^{-1}\big).
 \end{equation}
 Second, any loop $\ell$ in $\mathcal{C}$ that intersects $\mathcal{C}^{\circ}$ but is not contained in $\mathcal{C}^{\circ}$ must belong to $\widecheck{\mathcal{L}}$ and intersect $B(3R\delta^{-1})$ (by (\ref{lemmab72_S1})). On the event $\mathsf{K}_2^c$, every loop that is connected to $\ell$ by loops in $\mathcal{L}-\mathbbm{1}_\ell$ disjoint from $\mathcal{C}^{\circ}$ is contained in $B(x_\ell,\eta)$. Hence, $\mathcal{C}$ is contained in $B(\mathcal{C}^{\circ},\eta+\delta^{\bref{para_4}-1})$ and does not include any other cluster in $\mathfrak{C}^{\circ}$ of diameter at least $3\eta$. Combined with (\ref{lemmab72_S1}) and $\eta+\delta^{\bref{para_4}-1}+\frac{1}{4}\delta^{\bref{para_error}-1}<\frac{1}{2}\delta^{\bref{para_error}-1}$, this yields (\ref{incl_skeleton}). In what follows, we use these properties to prove that
 \begin{equation}\label{inclusionA1}
 	\mathsf{K}^c\cap \mathsf{A}_0 \subset \mathsf{A}_1.
 \end{equation} 
 We first show that 
 \begin{equation}\label{1720}
 	\mathsf{K}^c \cap  \mathsf{A}_0^2 \subset \mathsf{A}_1^{2}.  
 \end{equation}
 To see this, on $\mathsf{K}^c\cap (\mathsf{A}_1^{2})^c$, there exists $\mathcal{C}^{\circ}\in\mathfrak{C}^{\circ}$ such that $\widehat{\mathcal{C}}^{\circ}\subset B(2R\delta^{-1}-\delta^{ \bref{para_error}-1 })$ and $|\widehat{\mathcal{C}}^{\circ}|_\infty\in [\epsilon \delta^{-1} - \delta^{ \bref{para_error}-1 }, \epsilon \delta^{-1}+  \delta^{ \bref{para_error}-1 } ]$. In particular, $|\mathcal{C}^{\circ}|\ge|\widehat{\mathcal{C}}^{\circ}|\ge3\eta$. Therefore, it follows from (\ref{incl_skeleton}) that $\mathcal{C} \subset B(2R\delta^{-1} )$ and $| \mathcal{C} |_\infty\in [\epsilon \delta^{-1} - \delta^{ \bref{para_error}-1 }, \epsilon \delta^{-1}+  2\delta^{ \bref{para_error}-1 } ]$. This contradicts $\mathsf{A}_0^2$, and thus implies (\ref{1720}). By the same arguments, we also have 
 \begin{equation}\label{1721}
 	\mathsf{K}^c \cap  \mathsf{A}_0^3 \subset \mathsf{A}_1^{3}.  
 \end{equation}
  Finally, we prove (\ref{inclusionA1}) using (\ref{1720}) and (\ref{1721}). Assume that $\mathsf{K}^c\cap\mathsf{A}_0$ occurs. For any $\mathcal{C}$ with $\delta\cdot\mathcal{C}\in( \delta\cdot  \mathfrak{C})^{>\epsilon}_{B(R)}$, on $\mathsf{K}_3^c$, $\mathcal{C}$ must contain a loop with diameter at least $3\eta$, which belongs to $\widehat{\mathcal{L}}$ by $\mathsf{K}_4^c$. Let $\mathcal{C}^{\circ}$ be the cluster in $\mathfrak{C}^{\circ}$ containing this loop. Since $|\mathcal{C}^{\circ}|\ge 3\eta$ and $\widehat{\mathcal{C}}^{\circ}\subset\mathcal{C}\subset B(R\delta^{-1})$, as shown above, $\mathcal{C}^{\circ}$ is the unique cluster in $\mathfrak{C}^{\circ}$ of diameter at least $3\eta$ contained in $\mathcal{C}$, and (\ref{incl_skeleton}) holds. If $\delta\cdot\widehat{\mathcal{C}}^{\circ}$ does not belong to $( \delta\cdot \widehat{ \mathfrak{C}}^{\circ})^{>\epsilon}_{B(R)}$, then since $\widehat{\mathcal{C}}^{\circ}\subset\mathcal{C}\subset B(\widehat{\mathcal{C}}^{\circ}, \frac{1}{2}\delta^{ \bref{para_error}-1 })$, the cluster $\mathcal{C}$ must satisfy $|\mathcal{C}|\in [\epsilon \delta^{-1}  , \epsilon \delta^{-1} + \delta^{ \bref{para_error}-1 } ]$, which contradicts $\mathsf{A}_0^2$. Meanwhile, if for some $\mathcal{C}^{\circ}$ with $\delta\cdot \widehat{\mathcal{C}}^{\circ}\in ( \delta\cdot \widehat{ \mathfrak{C}}^{\circ})^{>\epsilon}_{B(R)}$, the cluster $\mathcal{C}\in\mathfrak{C}$ containing $\mathcal{C}^{\circ}$ does not belong to $( \delta\cdot  \mathfrak{C})^{>\epsilon}_{B(R)}$, then $\mathcal{C}\not\subset B( R \delta^{-1} )$ since $|\mathcal{C}|\ge|\widehat{\mathcal{C}}^{\circ}|>\epsilon\delta^{-1}$. Moreover, since $|\mathcal{C}^{\circ}|\ge|\widehat{\mathcal{C}}^{\circ}|\ge 3\eta$, (\ref{incl_skeleton}) implies $\mathcal{C}\subset B(R \delta^{-1}+\frac{1}{2}\delta^{ \bref{para_error}-1 })$, which is incompatible with $\mathsf{A}_0^{3}$. In conclusion, the map $\delta\cdot\mathcal{C}\mapsto\delta\cdot\widehat{\mathcal{C}}^{\circ}$ is a bijection from $( \delta\cdot  \mathfrak{C})^{>\epsilon}_{B(R)}$ to $( \delta\cdot \widehat{ \mathfrak{C}}^{\circ})^{>\epsilon}_{B(R)}$ with $\widehat{\mathcal{C}}^{\circ}\subset\mathcal{C}$, and thus $\mathsf{K}^c\cap \mathsf{A}_0 \subset \mathsf{A}_1^{1}$. This together with (\ref{1720}) and (\ref{1721}) yields (\ref{inclusionA1}). Putting (\ref{q717}) and (\ref{inclusionA1}) together, we obtain (\ref{ineq_i0_lemma3.1}).

      We define $\mathbf{T}_{\mathrm{I}}$ and $\mathbf{T}_{\mathrm{II}}$ as in Section \ref{subsection5.3_construction}, with $\mathbf{G}_{\mathrm{I}}=\mathbb{Z}^d$ and $\mathbf{G}_{\mathrm{II}}=\frac{\delta'}{\delta}\cdot \mathbb{Z}^d$. By interpolating between the graphs $\mathbb{Z}^d$ and $\frac{\delta'}{\delta}\cdot \mathbb{Z}^d$, it was shown in (\ref{y5.32}) that
 \begin{equation}\label{n717}
 	\begin{split}
 		 	\mathbb{P}\big( \mathrm{d}_{\mathfrak{T}}(\mathbf{T}_{\mathrm{I}},\mathbf{T}_{\mathrm{II}}) > \delta^{-\frac{1}{2}} \big)\lesssim  \delta^{c_{\star}},
 	\end{split}
 \end{equation}
 where $c_{\star}= \cref{const_compare_gluing} \land  [(\frac{d}{4}-\frac{1}{2})(\bref{para_2}+\bref{para_1})]$. When $\mathrm{d}_{\mathfrak{T}}(\mathbf{T}_{\mathrm{I}},\mathbf{T}_{\mathrm{II}})\le \delta^{-\frac{1}{2}}$, there exists a bijection $\sigma$ between the loops in $\mathcal{L}$ and $\frac{\delta'}{\delta}\cdot \mathcal{L}$ contained in $[-\delta^{\frac{1}{2}(\bref{para_1}-\bref{para_2})-1},\delta^{\frac{1}{2}(\bref{para_1}-\bref{para_2})-1}]
 ^d$ such that (\ref{use5.3}) holds and for any $\ell\in \mathcal{L}$,
 \begin{equation}
 	\mathrm{d}_{\mathrm{H}}\big(  \mathrm{ran}(\ell), \mathrm{ran}(\sigma(\ell))  \big) \le  \delta^{-\frac{1}{2}}.
 \end{equation}
This together with $\mathsf{A}_1$ implies that the following counterparts of $\{\mathsf{A}_1^j\}_{j=1,2,3}$ for the graph $\frac{\delta'}{\delta}\cdot \mathbb{Z}^d$  (where the loop clusters have the same law as $\frac{\delta'}{\delta}\cdot \mathfrak{C}$) all occur
  \begin{itemize}

	\item  $(\delta' \cdot \widehat{ \mathfrak{C}}^{\circ} )^{>\epsilon}_{B(R)} \sqsubseteq  \mathcal{A} + B( \delta^{\frac{1}{2}})$;

	\item  There does not exist $\mathcal{C} \in  \frac{\delta'}{\delta}\cdot \widehat{ \mathfrak{C}}^{\circ}$ such that $\mathcal{C}\subset B(2R\delta^{-1}-\delta^{ \bref{para_error}-1 } + \delta^{-\frac{1}{2}})$ and
  $$|\mathcal{C}|_\infty \in [\epsilon \delta^{-1} -  \delta^{ \bref{para_error}-1 } - \delta^{-\frac{1}{2}} ,\  \epsilon \delta^{-1}+  \delta^{ \bref{para_error}-1 }  +\delta^{-\frac{1}{2}}   ];$$

	\item  There does not exist $\mathcal{C}\in  \frac{\delta'}{\delta}\cdot  \widehat{ \mathfrak{C}}^{\circ}$ such that $\mathcal{C}\subset B(R\delta^{-1} + \delta^{ \bref{para_error}-1 })$, $\mathcal{C} \not\subset B(R\delta^{-1} - \delta^{ \bref{para_error}-1 })$ and $|\mathcal{C}| \ge \frac{1}{2}\epsilon \delta^{-1}$.

\end{itemize}
 Let $\mathsf{A}_1'$ denote the intersection of these three events. Thus, by (\ref{n717}) we have
 \begin{equation}\label{a719}
 	\mathbb{P} \big(  \mathsf{A}_1 \big) \le \mathbb{P} \big(  \mathsf{A}_1' \big) +  C\delta^{c_{\star}}.
 \end{equation}
The same reasoning as in the proof of (\ref{ineq_i0_lemma3.1}), applied to the graph $\frac{\delta'}{\delta}\cdot \mathbb{Z}^d$, also gives  
\begin{equation}\label{a720}
	\mathbb{P} \big(  \mathsf{A}_1' \big)  \le \mathbb{P} \big(  (\delta' \cdot \mathfrak{C} )^{>\epsilon}_{B(R)} \sqsubseteq  \mathcal{A} +  B(\delta^{\frac{1}{2}} +
	\tfrac{1}{2}\delta^{ \bref{para_error} } ) \big) + 3\delta^{\cref{const_check_loop}\bref{para_4}-2d\bref{para_3}} + \delta^{\frac{1}{4}\cref{const_check_loop}\bref{para_error}}.
\end{equation}
 By (\ref{a78}), (\ref{ineq_i0_lemma3.1}), (\ref{a719}) and (\ref{a720}), we obtain
   \begin{equation*}
   	\begin{split}
   		    	\mathbb{P}\big( (\delta\cdot   \mathfrak{C} )^{>\epsilon}_{B(R)} \sqsubseteq  \mathcal{A}   \big) 
   		   	\le   	\mathbb{P} \big(  (\delta' \cdot \mathfrak{C} )^{>\epsilon}_{B(R)} \sqsubseteq  \mathcal{A} +  B(\delta^{\frac{1}{2}} +
	\tfrac{1}{2}\delta^{ \bref{para_error} } ) \big)
   		   	 + C \delta^{(\frac{1}{4}\cref{const_check_loop}\bref{para_error})\land (\cref{const_check_loop}\bref{para_4}- 2d\bref{para_3})  \land c_{\star}   }.
   	\end{split}
   \end{equation*}
    Combined with $\delta^{\frac{1}{2}} +
	\tfrac{1}{2}\delta^{ \bref{para_error}} \le  \delta^{\bref{para_error}}$, it completes the proof.
  \end{proof}

We are now ready to prove the main result of this section.

 \begin{proof}[Proof of the scaling limit part of Theorem \ref{new_thm1}]

 	By Lemma~\ref{lemma_Strassen}, Lemma~\ref{lemmab72}, and Corollary~\ref{coro_limit_exist}, we obtain that for any $R>1$ and $\epsilon>0$, there exists a coupling of $\{(\overline{2^{-n}\cdot \mathfrak{C}})^{>\epsilon}_{B(R)}\}_{n\in \mathbb{N}^+}$ such that for each $n \in \mathbb{N}^+$,
\begin{equation}
	\mathbb{P}\Big( \mathrm{d}_{\mathrm{H}}^*\big( ( \overline{2^{-n}\cdot  \mathfrak{C}})^{>\epsilon}_{B(R)}, ( \overline{2^{-n-1}\cdot  \mathfrak{C}})^{>\epsilon}_{B(R)}  \big)\ge  2^{-\cref{const_lemmab72_1}n}  \Big)\le 2^{-\cref{const_lemmab72_2}n},
\end{equation}
and therefore $(\overline{2^{-n}\cdot \mathfrak{C}})^{>\epsilon}_{B(R)}$ converges almost surely as $n\to \infty$. Moreover, as shown in (\ref{ineq_i0_lemma3.1}), each cluster in $(\overline{\delta \cdot\mathfrak{C}})^{>\epsilon}_{B(R)}$ can be approximated by the loops in $\widehat{\mathcal{L}}$ together with their connectivity relation via the remaining loops, which converges to the tuple $(\mathfrak{C}^{\mathrm{B}},\sim_{\mathrm{g}})$ constructed in Section \ref{subsection5.3_construction}. Thus, the limit of $(\overline{2^{-n}\cdot \mathfrak{C}})^{>\epsilon}_{B(R)}$ has the same law as $(\mathfrak{C}^{\mathrm{B}}_{\sim_{\mathrm{g}}})^{>\epsilon}_{B(R)}$, and hence we may assume that $(\overline{2^{-n}\cdot \mathfrak{C}})^{>\epsilon}_{B(R)}$ converges almost surely to $(\mathfrak{C}^{\mathrm{B}}_{\sim_{\mathrm{g}}})^{>\epsilon}_{B(R)}$. For each $n\in\mathbb{N}^+$, let
\begin{equation}
	\eta_n:=\min_{\sigma \in \mathrm{bij}( (\mathfrak{C}^{\mathrm{B}}_{\sim_{\mathrm{g}}})^{>\epsilon}_{B(R)} , (\overline{2^{-n}\cdot \mathfrak{C}})^{>\epsilon}_{B(R)} )}\max_{\mathcal{C} \in (\mathfrak{C}^{\mathrm{B}}_{\sim_{\mathrm{g}}})^{>\epsilon}_{B(R)} }\,\mathrm{d}_{\mathrm{H}}\big(\mathcal{C},\sigma(\mathcal{C})\big),
\end{equation}
 and define $\sigma_n$ to be a bijection achieving this minimum. Thus, the convergence above implies that $\eta_n\to 0$ almost surely as $n\to \infty$.


Next, we extend this convergence to an arbitrary set $D\in\mathfrak{D}$ and to the full limit $\delta\to 0$. For any $h>0$, it follows from Lemma~\ref{lemma_geenral_D} that
\begin{equation}
	\lim_{n\to \infty}	\mathbb{P}\big( \exists \mathcal{C} \in (\overline{2^{-n}\cdot \mathfrak{C}})^{>\epsilon}_{B(R)}\ \text{such that}\ \mathcal{C} \subset D^h,\ \mathcal{C} \not\subset D^{-h}  \big) =0.
\end{equation}
This together with Fatou's lemma implies
\begin{equation}
	\mathbb{P}\big( \exists \mathcal{C} \in(\mathfrak{C}^{\mathrm{B}}_{\sim_{\mathrm{g}}})^{>\epsilon}_{B(R)}\ \text{such that}\ \mathcal{C} \subset D^h,\ \mathcal{C} \not\subset D^{-h}  \big) =0.
\end{equation}
In other words, there almost surely exists $h_\star>0$ such that every $\mathcal{C}\in(\mathfrak{C}^{\mathrm{B}}_{\sim_{\mathrm{g}}})^{>\epsilon}_{B(R)}$ is either contained in $D^{-h_\star}$ or not contained in $D^{h_\star}$. In particular, for such $\mathcal{C}$, one has $\mathcal{C}\subset D$ if and only if $\mathcal{C}\subset D^{-h_\star}$. We record the following observation: if $\sigma$ is a bijection from $(\mathfrak{C}^{\mathrm{B}}_{\sim_{\mathrm{g}}})^{>\epsilon}_{B(R)}$ to $(\overline{\delta\cdot \mathfrak{C}})^{>\epsilon}_{B(R)}$ such that the Hausdorff distance between each pair of corresponding clusters is at most $\frac{h_\star}{2}$, then $\sigma$ maps $(\mathfrak{C}^{\mathrm{B}}_{\sim_{\mathrm{g}}})^{>\epsilon}_{D}$ bijectively onto $(\overline{\delta\cdot \mathfrak{C}})^{>\epsilon}_{D}$. This is because if $\mathcal{C}\subset D^{-h_\star}$, then $\sigma(\mathcal{C})\subset D^{-\frac{h_\star}{2}}\subset D$; if $\mathcal{C}\not\subset D^{h_\star}$, then $\sigma(\mathcal{C})\not\subset D^{\frac{h_\star}{2}}$, and hence $\sigma(\mathcal{C})\not\subset D$.

For $\delta>0$, assume that $\delta\in [2^{-m},2^{-m+1})$ for $m\in \mathbb{N}^+$. Meanwhile, by Lemmas~\ref{lemma_Strassen} and~\ref{lemmab72}, there is a coupling between $(\overline{\delta\cdot \mathfrak{C}})^{>\epsilon}_{B(R)}$ and $(\overline{2^{-m}\cdot \mathfrak{C}})^{>\epsilon}_{B(R)}$ such that
\begin{equation}\label{a723}
	\mathbb{P}\Big( \mathrm{d}_{\mathrm{H}}^*\big( ( \overline{\delta\cdot  \mathfrak{C}})^{>\epsilon}_{B(R)}, ( \overline{2^{-m}\cdot  \mathfrak{C}})^{>\epsilon}_{B(R)}  \big)\ge  2^{-\cref{const_lemmab72_1}m}  \Big) \le  2^{-\cref{const_lemmab72_2}m}.
\end{equation}
Referring to Lemma \ref{lemma_gluing}, we may assume that this coupling and the coupling of $\{(\overline{2^{-n}\cdot \mathfrak{C}})^{>\epsilon}_{B(R)}\}_{n\in \mathbb{N}^+}$ constructed above are defined on the same probability space. Fix an arbitrary $\eta>0$, and let $\mathsf{G}_m:=\{\eta_m+2^{-\cref{const_lemmab72_1}m}< \frac{h_\star}{2}\land \eta \}$. Since $h_\star>0$ and $\eta_n\to 0$ almost surely, we have $\lim_{m\to\infty}\mathbb{P}(\mathsf{G}_m^c)=0$. On the event
\begin{equation}
	\mathsf{G}_m\cap\big\{\mathrm{d}_{\mathrm{H}}^*\big( ( \overline{\delta\cdot  \mathfrak{C}})^{>\epsilon}_{B(R)}, ( \overline{2^{-m}\cdot  \mathfrak{C}})^{>\epsilon}_{B(R)}  \big)<  2^{-\cref{const_lemmab72_1}m}\big\},
\end{equation}
there is a bijection $\sigma'_\delta$ from $(\overline{2^{-m}\cdot \mathfrak{C}})^{>\epsilon}_{B(R)}$ to $(\overline{\delta\cdot \mathfrak{C}})^{>\epsilon}_{B(R)}$ such that the Hausdorff distance between each pair of corresponding clusters is less than $2^{-\cref{const_lemmab72_1}m}$. Thus, $\tau_\delta:=\sigma'_\delta\circ\sigma_m$ is a bijection from $(\mathfrak{C}^{\mathrm{B}}_{\sim_{\mathrm{g}}})^{>\epsilon}_{B(R)}$ to $(\overline{\delta\cdot \mathfrak{C}})^{>\epsilon}_{B(R)}$, where the Hausdorff distance between each pair of corresponding clusters is less than $\eta_m+2^{-\cref{const_lemmab72_1}m}<\min\{\frac{h_\star}{2},\eta\}$. By the observation above, $\tau_\delta$ maps $(\mathfrak{C}^{\mathrm{B}}_{\sim_{\mathrm{g}}})^{>\epsilon}_{D}$ bijectively onto $(\overline{\delta\cdot \mathfrak{C}})^{>\epsilon}_{D}$, and thus $\mathrm{d}_{\mathrm{H}}^*\big( ( \overline{\delta\cdot  \mathfrak{C}})^{>\epsilon}_{D}, (\mathfrak{C}^{\mathrm{B}}_{\sim_{\mathrm{g}}})^{>\epsilon}_{D}  \big)<\eta$. Combined with \eqref{a723}, it implies
\begin{equation}
	\mathbb{P}\Big( \mathrm{d}_{\mathrm{H}}^*\big( ( \overline{\delta\cdot  \mathfrak{C}})^{>\epsilon}_{D}, (\mathfrak{C}^{\mathrm{B}}_{\sim_{\mathrm{g}}})^{>\epsilon}_{D}  \big)\ge \eta \Big) \le 2^{-\cref{const_lemmab72_2}m}+\mathbb{P}(\mathsf{G}_m^c).
\end{equation}
Let $f$ be a bounded Lipschitz function on $\mathfrak{X}_D$. Under the coupling above, one has
\begin{equation}
	\begin{split}
		&\Big|\mathbb{E}\big[f\big((\overline{\delta\cdot \mathfrak{C}})^{>\epsilon}_{D}\big)\big]-\mathbb{E}\big[f\big((\mathfrak{C}^{\mathrm{B}}_{\sim_{\mathrm{g}}})^{>\epsilon}_{D}\big)\big]\Big|\\
		\le   &C\eta + 2\|f\|_\infty\cdot \mathbb{P}\Big( \mathrm{d}_{\mathrm{H}}^*\big( ( \overline{\delta\cdot  \mathfrak{C}})^{>\epsilon}_{D}, (\mathfrak{C}^{\mathrm{B}}_{\sim_{\mathrm{g}}})^{>\epsilon}_{D}  \big)\ge \eta \Big)\\
		\le &C\eta + 2\|f\|_\infty\big(2^{-\cref{const_lemmab72_2}m}+\mathbb{P}(\mathsf{G}_m^c)\big).
	\end{split}
\end{equation}
Taking the limit superior as $\delta\to 0$ (i.e., $m\to\infty$) on both sides and then using $\lim_{m\to\infty}\mathbb{P}(\mathsf{G}_m^c)=0$, we obtain 
\begin{equation}
	\limsup_{\delta\to 0}\Big|\mathbb{E}\big[f\big((\overline{\delta\cdot \mathfrak{C}})^{>\epsilon}_{D}\big)\big]-\mathbb{E}\big[f\big((\mathfrak{C}^{\mathrm{B}}_{\sim_{\mathrm{g}}})^{>\epsilon}_{D}\big)\big]\Big| \le C\eta. 
\end{equation}
Since $\eta>0$ is arbitrary, this yields 
 \begin{equation}
	\lim_{\delta\to 0}\mathbb{E}\big[f\big((\overline{\delta\cdot \mathfrak{C}})^{>\epsilon}_{D}\big)\big]=\mathbb{E}\big[f\big((\mathfrak{C}^{\mathrm{B}}_{\sim_{\mathrm{g}}})^{>\epsilon}_{D}\big)\big].
\end{equation} 
Since $f$ is an arbitrary bounded Lipschitz function, it follows from the Portmanteau theorem that $(\overline{\delta\cdot \mathfrak{C}})^{>\epsilon}_{D}$ converges weakly to $(\mathfrak{C}^{\mathrm{B}}_{\sim_{\mathrm{g}}})^{>\epsilon}_{D}$ as $\delta\to 0$. 
  \end{proof}

     	\section{Conformal invariance}\label{section_conformal_invariance}

       In this section, we present the proofs of the conformal invariance part of Theorem \ref{new_thm1} and the conformal covariance in (\ref{good_113}). Since the proofs here follow the same general strategy as those in the preceding sections, we omit some technical details and present only the key steps. As discussed in Section \ref{subsection_discuss_conformal}, it suffices to consider translations and the inversion. The invariance under translations follows readily from the previous sections (see the beginning of Section \ref{asection8.2}), so the main task is to establish the invariance under the inversion, which is the purpose of this section.

       We first show that the image under inversion of the random walk on a Voronoi network can be approximated by a Doob $h$-transform of the random walk on another Voronoi network. We begin with a basic property of Brownian motion. Let $D=B(R)\setminus B(\frac{1}{R})$ with $R>1$. Note that $h(x)=|x|^{2-d}$ is a harmonic function on $D$. For a ball $B(y,r) \subset D$, its image under inversion is still a ball, which we denote by $B(y',r')$. Let $\{W_t\}_{0\le t\le \tau}$ be a Brownian motion on $\mathbb{R}^d$ starting from $y$, and let $\tau$ denote the first time when $W_t$ reaches $\partial B(y,r)$. For any $t\ge 0$, we define  
\begin{equation}
\varrho_t:= \inf \Big\{ s \ge 0: \int_0^s |W_u|^{-4}\,\mathrm{d}u>t   \Big\}. 
\end{equation}
Recall that $I(x)= \frac{x}{|x|^2}$. According to \cite[Chapter VIII, Exercise (3.17)]{revuz2013continuous}, the process $\widehat{W}_t:=I(W_{\varrho_t})$ has the same law as the Doob $h$-transform of Brownian motion starting from $y'$ with $h(x)=|x|^{2-d}$. We denote this $h$-transformed process by $W^h_t$. Since $\widehat{W}_{\cdot}$ and $W^h_{\cdot}$ have the same law, we use $\tau'$ to denote the first hitting time of $\partial B(y',r')$ by either process. Therefore, it follows from (\ref{h_2.9}) that for any $(d-1)$-dimensional spherical simplex $A$ in $\partial B(y',r')$,  
\begin{equation}
	\mathbb{P}\big( \widehat{W}_{\tau'}  \in A \big)  = \mathbb{P}\big( W^h_{\tau'}  \in A \big)= |A|_{y',r'}^h := |y'|^{d-2} \int_{z\in A}  |z|^{2-d} \omega_{y',r'}(\mathrm{d}z), 
\end{equation}
where $\omega_{y',r'}(\cdot)$ denotes the uniform measure on $\partial B(y',r')$, normalized so that its total mass is $1$. Therefore, since the continuous time change $\varrho_t$ does not alter the hitting distribution of $I(W_t)$, we have 
\begin{equation}\label{u3.15}
		\mathbb{P}\big( I(W_\tau)   \in A \big) =  \mathbb{P}\big( W^h_{\tau'}  \in A \big) = |A|_{y',r'}^h.  
 \end{equation} 
In what follows, we compare the discrete counterparts of $I(W_t)$ and $W^h_{t'}$. Here we restrict our discussion to the domain $D=B(R)\setminus B(\frac{1}{R})$. For a weighted graph $\mathbf{G}$, the Doob $h$-transform refers to the change of edge conductances in (\ref{h_2.10}), together with the following transformation of the killing weights:
\begin{equation}
    \kappa^h(x):=[h(x)]^2\kappa(x),\ \forall x\in\mathbf{V}\cap D.
\end{equation}
Here we take $h$ to be the discrete harmonic function on $\mathbf{V}\cap D$ with the Dirichlet boundary condition $h(x)=|x|^{2-d}$ for $x\in\mathbf{V}\cap D^c$. Recall that $B(y,r)$ and $B(y',r')=I(B(y,r))$ are both contained in $D$. Also recall the notation $\mathbf{p}(\cdot)$ introduced below \eqref{u3.9}. Arbitrarily fix $\Lambda>0$ and $\nu \in [0,\Lambda r^{-2}]$. Let $X_\cdot$ be the random walk on $\mathbf{G}^\nu$ starting from the vertex in $\mathbf{V}$ closest to $y$, and let $X_\cdot^h$ be the random walk on the Doob $h$-transform of $\mathbf{G}^\nu$ starting from the vertex in $\mathbf{V}$ closest to $y'$ (with ties broken in a predetermined manner). Recall $A^\pm$ from Lemma \ref{lemma2.1_isometric}.

\begin{lemma}
We retain the notation introduced above. Then we have

(1) There exists $C(R,\Lambda)>0$ such that for any $\diamond\in \{+,-\}$, 
	\begin{equation}\label{u3.17}
		\big| \mathbb{P}\big( I(X_{\tau_{y,r}})\in A^{\diamond}, \tau_{y,r}<\zeta  \big) - \mathbf{p}(\nu r^2)|A|_{y',r'}^h \big| \le C\big(\tfrac{\mathbf{s}(\mathbf{G})}{r}\big)^{\frac{1}{5}}. 
	\end{equation} 
	
(2) There exists $C'(R,\Lambda)>0$ such that for any $\diamond\in \{+,-\}$, 
 	\begin{equation}\label{u3.18}
		\big| \mathbb{P}\big(  X^h_{\tau_{y',r'}} \in A^{\diamond}, \tau_{y',r'}<\zeta  \big) - \mathbf{p}(\nu r^2)|A|_{y',r'}^h \big| \le C'\big(\tfrac{\mathbf{s}(\mathbf{G})}{r}\big)^{\frac{1}{5}}. 
	\end{equation} 
	
\end{lemma}

In this lemma, Item (1) is a direct consequence of Lemma \ref{lemma2.1_isometric} and (\ref{u3.15}). Item (2) follows from an adaptation of the proof of Lemma \ref{lemma2.1_isometric} to the present setting; full details are provided in \cite{technicalpaper}.

\subsection{Interpolation under inversion}\label{section_inversion_interpolation}

According to (\ref{u3.17}) and (\ref{u3.18}), the random walks on the inversion of a Voronoi network and on a Doob $h$-transform of a Voronoi network behave very similarly. Therefore, as in Lemma \ref{lemma_exist_interpolation}, we aim to construct an isotropic interpolation between these two graphs. We begin by introducing some notation. Arbitrarily fix two Voronoi networks $\mathbf{G}_1=(\mathbf{V}_1,w_1)$ and $\mathbf{G}_2=(\mathbf{V}_2,w_2)$. For $j\in \{1,2\}$ and $\delta\in(0,1)$, let $\mathbf{G}_{j,\delta}=(\mathbf{V}_{j,\delta},w_{j,\delta})$ be the Voronoi network induced by $\mathbf{V}_{j,\delta}:=\delta\cdot\mathbf{V}_j$. For $\varrho>1$, define the annulus $\mathbb{A}(\varrho):=\{x\in \mathbb{R}^d:\ \varrho^{-1}<|x|<\varrho\}$. Fix an exponent $c_{\star}\in\big(0,\frac{1}{2d-1}\big)$. For $k\in\{-1,0,\dots,4\}$, we define
\begin{equation}
	\mathbb{A}_k(\delta):=\mathbb{A}(2^{k+1}\delta^{-c_{\star}}).
\end{equation}
For convenience, we write $\Omega_\delta:=\mathbb{A}_{-1}(\delta)$. Recall that $I(x)=\frac{x}{|x|^2}$ and $h(x)=|x|^{2-d}$. Fix a smooth function $\vartheta: \mathbb{R}\to [0,1]$ with $\vartheta=0$ on $(-\infty,\frac{5}{4}]$ and $\vartheta=1$ on $[\frac{7}{4},\infty)$. We define $\psi_\delta(0):=1$ and for $x\neq 0$,
\begin{equation}
	\psi_\delta(x):=1+\vartheta\big(32\delta^{-c_{\star}}|x|\big)\,\vartheta\big(\tfrac{32\delta^{-c_{\star}}}{|x|}\big)\,\big(h(x)-1\big).
\end{equation}
Thus, the function $\psi_\delta(x)$ is smooth, positive, coincides with $h$ on $\mathbb{A}_3(\delta)$ and equals $1$ outside $\mathbb{A}_4(\delta)$. For a random walk $(X_n)_{n\ge 0}$ on a graph with vertex set $\mathbf{V}$, let $\tau_{\Omega_\delta}:=\inf\{n\ge 0:\ X_n\notin\Omega_\delta\}$ be its first exit time from $\Omega_\delta$. Define the discrete harmonic function $H_\delta:\mathbf{V}_{1,\delta}\to(0,\infty)$ by
\begin{equation}
	H_\delta(x):=\left\{\begin{aligned}
		& \mathbb{E}^{\mathbf{G}_{1,\delta}}_x\big[\psi_\delta\big(X_{\tau_{\Omega_\delta}}\big)\big] &\ \ \  & \text{if}\ x\in\mathbf{V}_{1,\delta}\cap\Omega_\delta;\\
		& \psi_\delta(x) & & \text{if}\ x\in\mathbf{V}_{1,\delta}\setminus\Omega_\delta.
	\end{aligned}\right.
\end{equation}
Here $\mathbb{E}^{\mathbf{G}}_x$ denotes the expectation under $\mathbb{P}^{\mathbf{G}}_x$, the law of the random walk on $\mathbf{G}$ starting from $x$.

We aim to interpolate between the following two weighted graphs: 
\begin{equation*}
	\begin{split}
		&\mathbf{G}^{(1)}_\delta=(\mathbf{V}^{(1)}_\delta,w^{(1)}_\delta)\ \text{with}\ \mathbf{V}^{(1)}_\delta:=\mathbf{V}_{1,\delta} \ \text{and} \  w^{(1)}_\delta(x,y):=w_{1,\delta}(x,y)\,H_\delta(x)\,H_\delta(y);\\
		&\mathbf{G}^{(2)}_\delta=(\mathbf{V}^{(2)}_\delta,w^{(2)}_\delta)\ \text{with}\   \mathbf{V}^{(2)}_\delta:=\mathcal{I}(\mathbf{V}_{2,\delta}\setminus\{0\})  \ \text{and} \   w^{(2)}_\delta(x,y):=w_{2,\delta}\big(\mathcal{I}(x),\mathcal{I}(y)\big).
	\end{split}
\end{equation*}
Intuitively, $\mathbf{G}^{(1)}_\delta$ is the Doob $H_\delta$-transform of $\mathbf{G}_{1,\delta}$, and $\mathbf{G}^{(2)}_\delta$ is the image of $\mathbf{G}_{2,\delta}$ under inversion. Fix $L>0$. We say that $\xi\in\{1,2\}^{\mathbb{Z}^d}$ is admissible if $\xi(z)=2$ only when $L\delta\cdot(z+[0,1)^d)\subset\mathbb{A}_1(\delta)$. For $j\in\{1,2\}$, we say that the box $L\delta\cdot(z+[0,1]^d)$ with $z\in\mathbb{Z}^d$ is Type-$j$ if $\xi(z)=j$. Let $U_j(\xi)$ be the union of all Type-$j$ boxes, and let $S_\xi:=U_1(\xi)\cap U_2(\xi)$ be the seam. Let $\xi^{(1)}\equiv1$, and let $\xi^{(2)}$ be the admissible configuration with $\xi^{(2)}(z)=2$ whenever $L\delta\cdot(z+[0,1)^d)\subset\mathbb{A}_1(\delta)$.

Analogous to Lemma \ref{lemma_exist_interpolation}, we have the following interpolation result.

\begin{lemma}\label{thm_inversion_interpolation}
	There exist $\delta_{\diamond},C_{\diamond}>0$, depending only on $d,s,t,u,\vartheta,c_{\star}$, such that for any $\delta\in(0,\delta_{\diamond})$, $L\in [C_{\diamond}\delta^{-2c_{\star}}, C_{\diamond}^{-1}\delta^{c_{\star}-1} ]$, and any Voronoi networks $\mathbf{G}_1,\mathbf{G}_2$ satisfying (\ref{u310}) and (\ref{u311}), there exists a family of graphs $\mathbf{G}^*_\xi=(\mathbf{V}_\xi,w^*_\xi)$ indexed by admissible $\xi\in\{1,2\}^{\mathbb{Z}^d}$, with the following properties. Let $\mathbf{G}^{\mathrm{V}}_\xi=(\mathbf{V}_\xi,w^{\mathrm{V}}_\xi)$ be the Voronoi network induced by $\mathbf{V}_\xi$, and write 
\begin{equation}\label{def_Crho}
	C_\varrho:=C_\diamond\,\varrho^{(d-1)(12d+14)},   \ \forall  \varrho>1.
\end{equation}
 Then for all admissible $\xi,\xi'\in \{1,2\}^{\mathbb{Z}^d}$:
 	\begin{enumerate}

		\item[(i)] The weights of $\mathbf{G}^*_\xi$ satisfy that for any $x,y\in\mathbf{V}_\xi$,
		\begin{equation}
			\big|w^*_\xi(x,y)-\psi_\delta(x)\psi_\delta(y)\,w^{\mathrm{V}}_\xi(x,y)\big|\le C_{\diamond} \delta^{1-(2d-1)c_\star}\,\psi_\delta(x)\psi_\delta(y)\,w^{\mathrm{V}}_\xi(x,y).
		\end{equation}

		\item[(ii)] Each $\mathbf{G}^{\mathrm{V}}_\xi$ is a Voronoi network with parameters $(a_\delta,b_\delta)$, where $a_\delta=C_{\diamond}\delta^{-4c_\star}$ and $b_\delta=C_{\diamond}^{-1}\delta^{c_\star(d-1)(12d+14)}$. In addition, for every $\varrho\in ( 1, \frac{1}{4}\delta^{-c_{\star}})$,
		\begin{equation}\label{local_sep_cov}
			\min_{ \text{distinct}\ x, y\in \mathbf{V}_\xi\cap\mathbb{A}(\varrho)}	|x-y|\ge C_\varrho^{-1}\delta \ \ \text{and}\ \ \max_{z\in\mathbb{A}(\varrho)}\mathrm{dist}(z,\mathbf{V}_\xi)\le C_\varrho\delta.
		\end{equation}
		Moreover, for $x\in\mathbf{V}_\xi\cap\mathbb{A}(\varrho)$ and $y\sim x$, one has $|x-y|\in [ C_\varrho^{-1}\delta, C_\varrho \delta] $ and
		\begin{equation}\label{local_face_weight}
			\mathrm{vol}_{d-1}(\mathsf{F}_{x,y})\ge C_\varrho^{-1}\delta^{d-1},\   \ w^{\mathrm{V}}_\xi(x,y), w^*_\xi(x,y)\in [ C_\varrho^{-1}\delta^{d-2}, C_\varrho \delta^{d-2} ].
		\end{equation}

\item[(iii)] For any $j\in\{1,2\}$ and any $x\in\mathbf{V}_\xi\cap U_j(\xi)$, if $\mathrm{dist}(x,S_\xi)>C_\diamond\delta^{1-2c_\star}$, then $x\in\mathbf{V}^{(j)}_\delta$ and its weights satisfy $w^*_\xi(x,\cdot)=w^{(j)}_\delta(x,\cdot)$.

		\item[(iv)] For $E\subset\mathbb{R}^d$, if $\xi(k)=\xi'(k)$ whenever $\mathrm{dist}(L\delta\cdot(k+[0,1]^d),E)\le C_\diamond\delta^{1-2c_\star}$, then $\mathbf{V}_\xi\cap E=\mathbf{V}_{\xi'}\cap E$, and for every $x$ in this set, $w^{\mathrm{V}}_\xi(x,\cdot)=w^{\mathrm{V}}_{\xi'}(x,\cdot)$ and $w^*_\xi(x,\cdot)=w^*_{\xi'}(x,\cdot)$.

		\item[(v)] Let $z\in\mathbf{V}_\xi$, $r\ge C_\diamond\delta^{1-2c_\star}$, and let $A\subset\partial B(z,r)$ be a spherical simplex. Define $\tau_{z,r}$ and $A^{\pm}$ as in Lemma \ref{lemma2.1_isometric}. Then for any $\diamond\in\{+,-\}$,
		\begin{equation}\label{e813}
			\big|\,\mathbb{P}^{\mathbf{G}^{\mathrm{V}}_\xi}_z\big(X_{\tau_{z,r}}\in A^{\diamond}\big)-|A|_{z,r}\big|\le C_\diamond\big(\tfrac{\delta^{1-2c_\star}}{r}\big)^{\frac15}.
		\end{equation}
		In addition, if $B(z,2r)\subset\Omega_\delta$, then
		\begin{equation}\label{e814}
			\big|\,\mathbb{P}^{\mathbf{G}^*_\xi}_z\big(X_{\tau_{z,r}}\in A^{\diamond}\big)-|A|_{z,r}^h\big|\le C_\diamond \big[  \big(\tfrac{\delta^{1-2c_\star}}{r}\big)^{\frac15}+  \delta^{\frac{1-(2d-1)c_\star}{3}} \big].
		\end{equation}
		Moreover, for any $\varrho\in (1,\frac{1}{4}\delta^{-c_{\star}})$ and $r\ge C_\diamond\varrho^{2}\delta$, if $B(z,2r)\subset\mathbb{A}(\varrho)$, then (\ref{e813}) and (\ref{e814}) hold with $\delta^{1-2c_\star}$ replaced by $\varrho^{2}\delta$.

	\end{enumerate}
\end{lemma}

See Figure \ref{figure_inverse_inter} below for an illustration of the construction; a detailed proof of Lemma \ref{thm_inversion_interpolation} will be given in \cite{technicalpaper}.

  \begin{figure}[H]
	\centering
	\includegraphics[width= 0.59\textwidth]{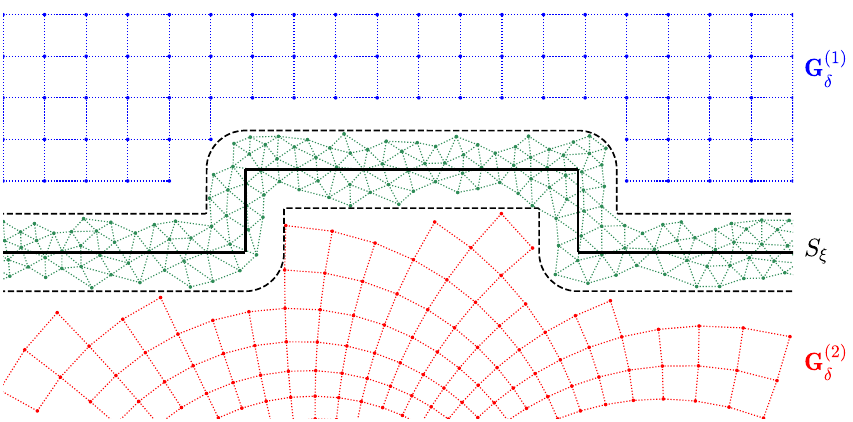}
	\caption{An illustration of the interpolation between $\delta\cdot\mathbb{Z}^d$ and $\mathcal{I}(\delta\cdot\mathbb{Z}^d)$. Away from the seam $S_\xi$, the graph coincides with $\mathbf{G}^{(1)}_\delta$ (blue) or $\mathbf{G}^{(2)}_\delta$ (red). The green points near $S_\xi$ are the additional vertices inserted to glue the two graphs together.} \label{figure_inverse_inter}
\end{figure}

\subsection{Proof of the conformal invariance part of Theorem \ref{new_thm1}}\label{asection8.2}

 We first prove translation invariance. Let $\tau_x$ denote the translation by $x\in \mathbb{R}^d$. Recall the notation $\mathbf{T}_{(n)}^{R,\epsilon}$ from (\ref{def_Ti_5.5new}) (likewise, here we also take $\mathbf{G}_{\mathrm{I}}=\mathbb{Z}^d$). We define $\mathbf{T}_{x,(n)}^{R,\epsilon}$ to be its analogue with $2^{-n}\cdot \mathbb{Z}^d$ replaced by $\tau_x(2^{-n}\cdot\mathbb{Z}^d)$. It was established in Section~\ref{subsection5.3_construction} that there is a coupling of $\{\mathbf{T}_{(n)}^{R,\epsilon}\}_{n\in \mathbb{N}^+}$ such that $\mathbf{T}_{(n)}^{R,\epsilon}$ almost surely converges to a limit $\mathbf{T}_{(\infty)}^{R,\epsilon}$ as $n\to \infty$. As a consequence, $\mathbf{T}_{x,(n)}^{R,\epsilon}$ converges in distribution to $\tau_{x}(\mathbf{T}_{(\infty)}^{R,\epsilon})$ as $n\to \infty$. Therefore, for the translation invariance, it suffices to prove that for any $x\in \mathbb{R}^d$,
\begin{equation}\label{eq_translation_invariance}
	\tau_x\big(\mathbf{T}_{(\infty)}^{R,\epsilon}\big) \overset{\mathrm{d}}{=} \mathbf{T}_{(\infty)}^{R,\epsilon}.
\end{equation} 
Fix $x\in\mathbb{R}^d$. For each $n\ge 1$, we can choose $z_n\in\mathbb{R}^d$ with $|z_n|\le\sqrt{d}\,2^{-n}$ such that $x-z_n\in 2^{-n}\cdot\mathbb{Z}^d$, and hence $\tau_x(2^{-n}\cdot\mathbb{Z}^d)=\tau_{z_n}(2^{-n}\cdot\mathbb{Z}^d)$. we may realize $\mathbf{T}_{x,(n)}^{R,\epsilon}$ as the analogue of $\mathbf{T}_{(n)}^{R,\epsilon}$ constructed from $\tau_{z_n}(2^{-n}\cdot\widehat{\mathcal{L}}_{\mathrm{I}})$ and $\tau_{z_n}(2^{-n}\cdot\mathcal{L}^\nu_{\mathrm{I}})$. This gives a coupling of $\mathbf{T}_{(n)}^{R,\epsilon}$ and $\mathbf{T}_{x,(n)}^{R,\epsilon}$ in which each loop $\ell$ corresponds to its translate $\tau_{z_n}(\ell)$, so that $\mathrm{d}_{\rm H}(\mathrm{ran}(\ell),\mathrm{ran}(\tau_{z_n}(\ell)))\le\sqrt{d}\,2^{-n}$. Moreover, this correspondence preserves the diameters of loops and the equivalence relations (since the clusters of the two loop soups are mapped onto each other by $\tau_{z_n}$). Let $\mathsf{B}_n$ denote the event that there exists a loop $\ell$ in $2^{-n}\cdot\widehat{\mathcal{L}}_{\mathrm{I}}$ with $|\ell|>\epsilon$ such that $\mathrm{ran}(\ell)\subset B(R+\sqrt{d}\,2^{-n})$ and $\mathrm{ran}(\ell)\not\subset B(R-\sqrt{d}\,2^{-n})$. On $\mathsf{B}_n^c$, a loop $\ell$ satisfies $\mathrm{ran}(\ell)\subset B(R)$ if and only if $\mathrm{ran}(\tau_{z_n}(\ell))\subset B(R)$, and thus the correspondence restricts to an admissible bijection between the loops involved in $\mathbf{T}_{(n)}^{R,\epsilon}$ and those involved in $\mathbf{T}_{x,(n)}^{R,\epsilon}$. Consequently, on the event $\mathsf{B}_n^c$, we have $\mathrm{d}_{\mathfrak{T}}\big(\mathbf{T}_{(n)}^{R,\epsilon},\mathbf{T}_{x,(n)}^{R,\epsilon}\big)\le\sqrt{d}\,2^{-n}$. 
Rescaling by $2^n$ and applying (\ref{y3.17}) (with $r=2\sqrt{d}$ and $M=2^n\epsilon$), we get $\mathbb{P}(\mathsf{B}_n)\lesssim R^{d-1}\epsilon^{-d}2^{-n}\to 0$, and hence $\mathrm{d}_{\mathfrak{T}}\big(\mathbf{T}_{(n)}^{R,\epsilon},\mathbf{T}_{x,(n)}^{R,\epsilon}\big)\to 0$ in probability. This together with the convergence in distribution of $\mathbf{T}_{(n)}^{R,\epsilon}$ to $\mathbf{T}_{(\infty)}^{R,\epsilon}$ implies that $\mathbf{T}_{x,(n)}^{R,\epsilon}$ also converges to $\mathbf{T}_{(\infty)}^{R,\epsilon}$ in distribution (see e.g., \cite[Theorem~3.1]{Billingsley1999}). Combined with the convergence in distribution of $\mathbf{T}_{x,(n)}^{R,\epsilon}$ to $\tau_x(\mathbf{T}_{(\infty)}^{R,\epsilon})$, it yields \eqref{eq_translation_invariance}, and hence the translation invariance.





 {\color{red}




 
 }

Now we prove the inversion invariance. We define $\widehat{\mathbf{T}}_{(n)}^{R,\epsilon}$ (resp. $\widehat{\mathbf{T}}_{(\infty)}^{R,\epsilon}$) as the restriction of $\mathbf{T}_{(n)}^{R,\epsilon}$  (resp. $ \mathbf{T}_{(\infty)}^{R,\epsilon}$) to the loops contained in the annulus $\mathbb{A}(R)$. Note that $\mathbb{A}(R)$ is invariant under $I$, i.e., $I(\mathbb{A}(R)) = \mathbb{A}(R)$. Let $\widehat{\mathbf{T}}_{I,(n)}^{R,\epsilon}$ (resp. $\widehat{\mathbf{T}}_{H,(n)}^{R,\epsilon}$) denote the analogue of $\widehat{\mathbf{T}}_{(n)}^{R,\epsilon}$ with $2^{-n}\cdot \mathbb{Z}^d$ replaced by $I(2^{-n}\cdot\mathbb{Z}^d)$ (resp. its Doob $H_{2^{-n}}$-transform). Under the coupling from Section \ref{subsection5.3_construction}, $\widehat{\mathbf{T}}_{(n)}^{R,\epsilon}$ and $\widehat{\mathbf{T}}_{I,(n)}^{R,\epsilon}$ almost surely converge to $\widehat{\mathbf{T}}_{(\infty)}^{R,\epsilon}$ and $I(\widehat{\mathbf{T}}_{(\infty)}^{R,\epsilon})$ respectively. Meanwhile, by the invariance under the Doob $H_{2^{-n}}$-transform (recall (\ref{a216})), we have $\widehat{\mathbf{T}}_{(n)}^{R,\epsilon} \overset{\mathrm{d}}{=} \widehat{\mathbf{T}}_{H, (n)}^{R,\epsilon}$. Thus, to achieve $\widehat{\mathbf{T}}_{(\infty)}^{R,\epsilon}\overset{\mathrm{d}}{=}I(\widehat{\mathbf{T}}_{(\infty)}^{R,\epsilon})$, it remains to show that there exist $c_{\ddagger}>0$ and a coupling between $\widehat{\mathbf{T}}_{I, (n)}^{R,\epsilon}$ and $\widehat{\mathbf{T}}_{H, (n)}^{R,\epsilon}$ such that   
 	  \begin{equation}\label{q821}
	\mathbb{P}\big( \mathrm{d}_{\mathfrak{T}}(\widehat{\mathbf{T}}_{I,(n)}^{R,\epsilon}, \widehat{\mathbf{T}}_{H,(n)}^{R,\epsilon}) > 2^{-\frac{1}{2}n} \big)\lesssim 2^{-c_\ddagger n}.  
 \end{equation} 
 This bound follows from the interpolation argument in Section~\ref{section5_convergence} between $I(2^{-n}\cdot\mathbb{Z}^d)$ and the Doob $H_{2^{-n}}$-transform of $2^{-n}\cdot\mathbb{Z}^d$, based on Lemma~\ref{thm_inversion_interpolation}. Specifically, as illustrated in Figure~\ref{figure_cases}, we start from $I(2^{-n}\cdot\mathbb{Z}^d)$ and modify the graph inside one small box at a time, replacing it by $2^{-n}\cdot\mathbb{Z}^d$ with edge conductances given by the Doob $H_{2^{-n}}$-transform. Applying the estimates in Section \ref{section5_convergence} to each pair of successive graphs (when coupling the crossing loops on the two graphs, we use (\ref{e814}) instead of (\ref{u3.9})), we can bound the error produced by each modification. Unlike in Section \ref{section5_convergence}, this error is no longer determined solely by the Euclidean size of the box, but also depends on the distance from the box to the origin, since the number of lattice spacings across the box varies with the location. To handle this, we write $\delta=2^{-n}$ and perform the modifications only within the annulus $\mathbb{A}(\delta^{-a})=B(\delta^{-a})\setminus\overline{B(\delta^{a})}$ for some sufficiently small $a\in(0,c_\star)$. The graphs outside this annulus can affect $\widehat{\mathbf{T}}^{R,\epsilon}_{I,(n)}$ and $\widehat{\mathbf{T}}^{R,\epsilon}_{H,(n)}$ only through a loop cluster crossing $B(\delta^{-a})\setminus B(R)$ or $B(R^{-1})\setminus B(\delta^{a})$; by (\ref{crossing_prob}), such a crossing occurs with probability $O\big((R\delta^{a})^{\frac d2-1}\big)$. Within $\mathbb{A}(\delta^{-a})$, we apply Properties (ii) and (v) of Lemma \ref{thm_inversion_interpolation} with $\varrho=\delta^{-a}$. Compared with the setting of Section \ref{section5_convergence}, the local mesh size is enlarged by at most a factor of order $\delta^{-2a}$, and hence the error bound for each modification acquires an extra factor of at most order $\delta^{-2a}$. Recall from (\ref{use5.5}) that each modification in Section \ref{section5_convergence} produces an error of order $\delta^{d(\beta_1+\beta_6)+c_{17}}$, while the number of modifications is $K\asymp\delta^{-d(\beta_1+\beta_6)}$. Therefore, if $a$ is chosen so small that $a\le\beta_1$ (so that the number of boxes in $\mathbb{A}(\delta^{-a})$ is $O(K)$) and $2a<c_{17}$, then the total error is $O\big(\delta^{c_{17}-2a}+(R\delta^{a})^{\frac d2-1}\big)$, which is bounded by a power of the mesh size $2^{-n}$. Combining these estimates, we obtain (\ref{q821}).

  

 
 
 Since the group of conformal maps is generated by translations and the inversion, the results above yield the desired conformal invariance. \qed

\subsection{Proof of (\ref{good_113}) in Theorem \ref{thm_k_point}}

  As in the preceding proof, it suffices to consider the case when $f = \tau_y$ with $y \in \mathbb{R}^d$, and the case when $f = I$.

\textbf{When $f = \tau_y$.} We define the event $\mathsf{A}:= \{ \hat{x}_1^{\delta}\xleftrightarrow{} \hat{x}_2^{\delta} \xleftrightarrow{} ... \xleftrightarrow{} \hat{x}_k^{\delta}\}$, and denote by $\mathsf{A}_y$ its analogue with $x_j$ replaced by $x_j+y$ for all $1\le j\le k$. To obtain $\gamma_k (x_1+y, x_2+y,...,x_k+y)=  \gamma_k(x_1,x_2,...,x_k)$, it is sufficient to prove that 
  \begin{equation}
   \big| 	\mathbb{P}(\mathsf{A} ) - \mathbb{P}(\mathsf{A}_y)  \big| \lesssim  \delta^{(\frac{d}{2}-1)k+c_\dagger}
  \end{equation}
  for some constant $c_\dagger>0$. This directly follows from the proof of (\ref{h6.3}), except that we take $\mathbf{G}_{\mathrm{II}} = \tau_y(\mathbb{Z}^d)$ instead of $\frac{\delta_*}{\delta}\cdot\mathbb{Z}^d$.

 \textbf{When $f = I$.} Let $\mathsf{A}_{I}$ denote the analogue of $\mathsf{A}$ with $x_j$ replaced by $I(x_j)$ for all $1\le j\le k$. Fix $\varrho > 1$ such that $x_j \in \mathbb{A}(\varrho/2)$ for all $1 \le j \le k$. Recall the conformal scale factor $\lambda_f(\cdot)$ defined below (\ref{good_113}). In particular, we have $\lambda_I(x) = |x|^{-2}$. To prove the equality in (\ref{good_113}) for $f = I$, it suffices to show that 
 \begin{equation}
  \Big| \mathbb{P}(\mathsf{A}_{I}) 	-\mathbb{P}(\mathsf{A})  \cdot \prod\nolimits_{1\le j\le k}  |x_j|^{d-2} \Big|\le \delta^{(\frac{d}{2}-1)k+c_\ddagger}
 \end{equation}
for some $c_\ddagger>0$. We also denote by $\mathsf{A}_{H}$ the analogue of $\mathsf{A}$ after taking the Doob $H_\delta$-transform. By the invariance in (\ref{a216}), we have $\mathbb{P}(\mathsf{A})=\mathbb{P}(\mathsf{A}_{H})$. Therefore, it remains to establish 
\begin{equation}\label{p_inv_final}
	  \Big| 	\mathbb{P}(\mathsf{A}_{I}) - \mathbb{P}(\mathsf{A}_{H})\cdot \prod\nolimits_{1\le j\le k}  |x_j|^{d-2} \Big|\le \delta^{(\frac{d}{2}-1)k+c_\ddagger}. 
\end{equation}
The bound \eqref{p_inv_final} can be derived by the same argument as in the proof of \eqref{h6.3}, namely by interpolating between the two underlying graphs away from the points $I(x_1), ..., I(x_k)$, and then modifying the graphs near each $I(x_j)$. The only difference in the first step is that the family $\{\mathbf{G}_\xi\}_{\xi}$ from Lemma~\ref{lemma_exist_interpolation} is replaced by the family $\{\mathbf{G}^*_\xi\}_{\xi}$ from Lemma~\ref{thm_inversion_interpolation}, which interpolates between $I(\delta \cdot \mathbb{Z}^d)$ and the Doob $H_\delta$-transform of $\delta \cdot \mathbb{Z}^d$. This lemma provides all the inputs required by the argument. Indeed, by Properties~(iii) and (iv), each $\mathbf{G}^*_\xi$ coincides with one of the two original graphs away from the interface $S_\xi$, and two graphs $\mathbf{G}^*_\xi$, $\mathbf{G}^*_{\xi'}$ with $\xi, \xi'$ differing at a single site differ only near the corresponding box. By Property~(i), the weights of $\mathbf{G}^*_\xi$ are, up to negligible relative error, those of the Doob transform of the Voronoi network $\mathbf{G}^{\mathrm{V}}_\xi$, and by Property~(v), the exit distributions of the random walks on $\mathbf{G}^{\mathrm{V}}_\xi$ and $\mathbf{G}^*_\xi$ are close to their continuum counterparts, which replaces the quantitative convergence to Brownian motion used in Section~\ref{section5_convergence}. Compared with Section~\ref{section5_convergence}, the error of each modification now also depends on the distance to the origin, but this can be handled exactly as in the proof of \eqref{q821}, namely by performing the modifications only within the annulus $\mathbb{A}(\delta^{-a})$ for a sufficiently small $a>0$. Hence, the first step can be carried out as before with $\{\mathbf{G}_\xi\}_{\xi}$ replaced by $\{\mathbf{G}^*_\xi\}_{\xi}$. For the second step, by the last statement of Property~(v), within $\mathbb{A}(\varrho)$ the exit distributions of the random walk on $\mathbf{G}^*_\xi$ approximate their continuum counterparts with an error of the same order as for $\mathbf{G}_\xi$ in the proof of \eqref{h6.3}, uniformly in $\xi$ up to constants depending only on $\varrho$. Since $\mathbb{A}(\varrho)$ contains the neighborhoods of $I(x_1),..., I(x_k)$, the error in the Lupu--Werner formula arising from the effective conductance is of the same order as in the proof of \eqref{h6.3}. Hence, the second step can also be carried out as before.

   The factor $\prod_{1 \le j \le k} |x_j|^{d-2}$ in \eqref{p_inv_final} arises from the second step. As in \eqref{i614}, when we modify the graph near $I(x_j)$, the Lupu--Werner formula expresses the conditional probability that $I(x_j)$ is connected to the partial sign cluster explored outside a small ball around $I(x_j)$ as asymptotically proportional to the product of the harmonic average of the boundary values and the standard deviation of $\phi$ at $I(x_j)$. Under the coupling, the harmonic average is unchanged up to a negligible error, so the only change comes from the variance, which is determined by the weights near $I(x_j)$. By Property~(i), all the graphs $\mathbf{G}^*_\xi$ carry the same factor $\psi_\delta(x)\psi_\delta(y)$, and hence the ratio between the weights of the two graphs near $I(x_j)$ coincides with that of the corresponding Voronoi networks. Near $I(x_j)$, the graph $I(\delta \cdot \mathbb{Z}^d)$ is approximately the rescaling of $\delta \cdot \mathbb{Z}^d$ by the factor $\lambda_I(x_j)=|x_j|^{-2}$. Since the weights of a Voronoi network scale as the $(d-2)$-th power of the mesh size, this is exactly the situation in \eqref{i614} with $\frac{\delta_*}{\delta}$ replaced by $|x_j|^{-2}$. Consequently, the factor $(\frac{\delta}{\delta_*})^{\frac{d}{2}-1}$ in \eqref{i614} is now replaced by $|x_j|^{2(\frac{d}{2}-1)} = |x_j|^{d-2}$. Repeating this for each $1 \le j \le k$ yields the factor $\prod_{1 \le j \le k} |x_j|^{d-2}$. Since the remaining steps are identical to those in the proof of \eqref{h6.3}, we omit further details. \qed

\section{Extra randomness}\label{section_extra_randomness}

In this section, we prove the extra randomness part of Theorem~\ref{new_thm1}, i.e., that the gluing relation $\sim_{\mathrm{g}}$ is not measurable with respect to $\mathcal{L}^{\mathrm{B}}$. As outlined in Section~\ref{subsection_extra_randomness}, the proof mainly relies on the following two estimates:
\begin{itemize}
	\item[(1)] With uniformly positive probability, there exist two macroscopic loops (at scale $N$) that belong to distinct clusters, and the number of touching points between these two clusters is at least of order $N^{\frac{d}{2}-1}$;
	\item[(2)] The second moment of the number of touching points between two macroscopic clusters is at most of order $N^{d-2}$.
\end{itemize}
We establish these two estimates in Lemmas~\ref{lemma_rich_touching} and \ref{lemma_pivotal_second_moment} respectively.

For ease of exposition, we first fix some notation. For $n\in \mathbb{N}^+$, let $\delta:=2^{-n}$, and let $\mathcal{L}^{\delta}:=\delta\cdot \mathcal{L}$ denote the rescaled loop soup on $\delta\cdot \widetilde{\mathbb{Z}}^d$, where $\mathcal{L}$ is the loop soup on $\widetilde{\mathbb{Z}}^d$ of intensity $\frac{1}{2}$. Note that the clusters of $\mathcal{L}^{\delta}$ are given by $\delta\cdot \mathfrak{C}$. As in (\ref{q318}), for any $R,\varsigma>0$, we denote 
\begin{equation}\label{df9.1}
	\mathcal{L}^{\delta}_{R,\varsigma}:=\mathcal{L}^{\delta}\cdot \mathbbm{1}_{\mathrm{ran}(\ell)\subset B(R),|\ell|>\varsigma} \  \ \text{and}\ \  \mathcal{L}^{\mathrm{B}}_{R,\varsigma}:=\mathcal{L}^{\mathrm{B}}\cdot \mathbbm{1}_{\mathrm{ran}(\ell)\subset B(R),|\ell|>\varsigma}. 
\end{equation}
For any $\ell\in \mathcal{L}^{\delta}$, let $\mathcal{C}_{\ell}$ denote the cluster of $\mathcal{L}^{\delta}$ containing $\ell$. For $\ell,\ell'\in \mathcal{L}^{\delta}$, we write $\ell\xleftrightarrow{}\ell'$ if $\mathcal{C}_{\ell}=\mathcal{C}_{\ell'}$. As in Section~\ref{subsection_extra_randomness}, we say that two disjoint clusters nearly touch on an edge $e$ of $\delta\cdot \mathbb{Z}^d$ if each of them intersects an endpoint of $e$. We define the touching number $\mathcal{T}(\ell,\ell')$ as the number of edges on which $\mathcal{C}_{\ell}$ and $\mathcal{C}_{\ell'}$ nearly touch if $\mathcal{C}_{\ell}\neq \mathcal{C}_{\ell'}$, and set $\mathcal{T}(\ell,\ell'):=0$ otherwise.

 \begin{lemma}\label{lemma_rich_touching}
	There exist $R_\star,\varsigma_\star,\epsilon_\star,c_\star>0$ such that for all sufficiently large $n$,
	\begin{equation}\label{9.2}
		\mathbb{P}\big(\exists\, \ell,\ell'\in \mathcal{L}^{\delta}_{R_\star,\varsigma_\star}\ \text{such that}\ \mathcal{T}(\ell,\ell')\ge \epsilon_\star\delta^{-\frac{d}{2}+1}\big)\ge c_\star.
	\end{equation}
\end{lemma}
 \begin{proof}
 	Let $\mathfrak{C}^{\delta}_{\mathrm{mac}}$ denote the collection of clusters of $\mathcal{L}^{\delta}$ crossing $B(2)\setminus B(1)$ (where ``mac'' stands for ``macroscopic''). We first show that with uniformly positive probability, there exist two clusters in $\mathfrak{C}^{\delta}_{\mathrm{mac}}$ nearly touching on many edges. Let $\mathbf{X}$ denote the number of edges $\{u,v\}$ of $\delta\cdot \mathbb{Z}^d$ contained in $B(\frac{1}{2})$ such that $u$ and $v$ belong to two distinct clusters in $\mathfrak{C}^{\delta}_{\mathrm{mac}}$. By the two-arm estimate in \cite[Theorem 1.1]{cai2025heterochromatic}, each edge in $B(\frac{1}{2})$ is counted in $\mathbf{X}$ with probability of order $\delta^{\frac{d}{2}+1}$, and thus 
 	\begin{equation}
 		\mathbb{E}[\mathbf{X}]\asymp \delta^{-d}\cdot  \delta^{\frac{d}{2}+1} = \delta^{-\frac{d}{2}+1}. 
 	\end{equation}
 	Moreover, referring to \cite[(3.19)]{cai2025separation}, we have $\mathbb{E}[\mathbf{X}^2]\lesssim \delta^{2-d}$. Therefore, by the Paley--Zygmund inequality, there exist $c_\diamond,c_\dagger>0$ such that for all sufficiently large $n$,
	\begin{equation}\label{9.3}
		\mathbb{P}\big(\mathbf{X}\ge c_\diamond\delta^{-\frac{d}{2}+1}\big)\ge c_\dagger.
	\end{equation} 
Meanwhile, by (\ref{a336}) and (\ref{newineq359}), there exists a sufficiently large $k_{\dagger}\in \mathbb{N}^+$ such that 
  \begin{equation}
 	 \mathbb{P}\big(\#\mathfrak{C}^{\delta}_{\mathrm{mac}}>k_{\dagger} \big)\le \tfrac{1}{4}c_{\dagger}. 
 \end{equation}
 	On the event $\{\mathbf{X}\ge c_{\diamond} \delta^{-\frac{d}{2}+1}\}\cap \{\#\mathfrak{C}^{\delta}_{\mathrm{mac}}\le k_\dagger\}$, since there are at most $k_\dagger^2$ pairs of clusters in $\mathfrak{C}^{\delta}_{\mathrm{mac}}$, there exist two distinct clusters in $\mathfrak{C}^{\delta}_{\mathrm{mac}}$ that nearly touch on at least $c_\diamond k_\dagger^{-2}\delta^{-\frac{d}{2}+1}$ edges. Combined with (\ref{9.3}), this implies
	\begin{equation}\label{9.4}
		\mathbb{P}\big(\exists\, \mathcal{C}_1\neq\mathcal{C}_2\in\mathfrak{C}^{\delta}_{\mathrm{mac}}\ \text{nearly touching on at least}\ c_{\diamond}k_0^{-2}\delta^{-\frac{d}{2}+1}\ \text{edges}\big)\ge \tfrac{3}{4}c_\dagger.
	\end{equation}

 It remains to show that with high probability, every cluster in $\mathfrak{C}^{\delta}_{\mathrm{mac}}$ contains a loop in $\mathcal{L}^{\delta}_{R,\varsigma}$ for suitable $R$ and $\varsigma$. On the one hand, by (\ref{crossing_prob}), the probability that some cluster in $\mathfrak{C}^{\delta}_{\mathrm{mac}}$ is not contained in $B(R)$ is $O(R^{-\frac{d}{2}+1})$, uniformly in $n$. On the other hand, let $\mathfrak{C}^{\mathrm{B}}_{\mathrm{mac}}(R)$ denote the collection of clusters in $\mathfrak{C}^{\mathrm{B}}_{\sim_{\mathrm{g}}}$ that are contained in $B(R)$ and cross $B(2) \setminus B(1)$, and let $\mathfrak{C}^{\delta}_{\mathrm{mac}}(R)$ denote the collection of clusters in $\mathfrak{C}^{\delta}_{\mathrm{mac}}$ contained in $B(R)$. By the convergence of $\mathbf{T}^{R,r}_{(n)}$ to $\mathbf{T}^{R,r}_{(\infty)}$ in Section~\ref{subsection5.3_construction} and the proof of the scaling limit part of Theorem~\ref{new_thm1}, for any $r>0$, there is a coupling with the following property: almost surely for all sufficiently large $n$, there exist bijections $\sigma_n:\mathcal{L}^{\mathrm{B}}_{R,r}\to \mathcal{L}^{\delta}_{R,r}$ and $\tau_n:\mathfrak{C}^{\mathrm{B}}_{\mathrm{mac}}(R)\to \mathfrak{C}^{\delta}_{\mathrm{mac}}(R)$ such that
\begin{equation*}
	\ell\subset \mathcal{C}  \iff    \sigma_n(\ell)\subset \tau_n(\mathcal{C}) 
	\end{equation*} 
 for all $\ell\in \mathcal{L}^{\mathrm{B}}_{R,r}$ and $\mathcal{C}\in \mathfrak{C}^{\mathrm{B}}_{\mathrm{mac}}(R)$. Note that each cluster in $\mathfrak{C}^{\mathrm{B}}_{\sim_{\mathrm{g}}}$ is the closure of the union of the loops it contains and has positive diameter, and that almost surely only finitely many clusters in $\mathfrak{C}^{\mathrm{B}}_{\sim_{\mathrm{g}}}$ cross $B(2)\setminus B(1)$. Thus, there exists a (random) $r_0>0$ such that each of these clusters contains a loop of diameter greater than $r_0$. Combining these observations, we obtain 
	\begin{equation*}
		\limsup_{n\to \infty}\mathbb{P}\big(\exists\, \mathcal{C}\in \mathfrak{C}^{\delta}_{\mathrm{mac}}\ \text{containing no loop in}\ \mathcal{L}^{\delta}_{R,\varsigma}\big)\le \mathbb{P}(r_0\le \varsigma)+CR^{-\frac{d}{2}+1}.
	\end{equation*}
 By taking a sufficiently large $R$ and then a sufficiently small $\varsigma$, the right-hand side is at most $\frac{1}{4}c_\dagger$. Combined with (\ref{9.4}), it yields the desired bound (\ref{9.2}).
   \end{proof}

We next turn to the second estimate, which concerns the pivotal loops for the connection between two macroscopic loops. To this end, we first enumerate the macroscopic loops in a way determined by these loops themselves. For any loop $\ell$, let $\mathfrak{h}(\ell):=\min\{x^{(1)}:x\in \mathrm{ran}(\ell)\}$, where $x^{(1)}$ denotes the first coordinate of $x\in \mathbb{R}^d$. For any $R,\varsigma>0$, we enumerate the loops in $\mathcal{L}^{\delta}_{R,\varsigma}$ (resp.\ $\mathcal{L}^{\mathrm{B}}_{R,\varsigma}$) as $\ell^{\delta}_1,...,\ell^{\delta}_{\mathfrak{m}_\delta}$ (resp.\ $\ell_1,...,\ell_\mathfrak{m}$) in increasing order of $\mathfrak{h}(\cdot)$ (with ties broken in an arbitrary measurable way). In particular, this enumeration is determined by $\mathcal{L}^{\delta}_{R,\varsigma}$ (resp.\ $\mathcal{L}^{\mathrm{B}}_{R,\varsigma}$). By (\ref{bound2.12}) and the scaling invariance, $\mathbb{E}[\mathfrak{m}_\delta]$ is bounded from above by a constant depending only on $d$, $R$ and $\varsigma$.

We then introduce the small loops used in the local modification. For each edge $e=\{v,w\}$ of $\delta\cdot \mathbb{Z}^d$, let $\mathcal{S}_e$ denote the collection of loops $\ell$ on $\delta\cdot \widetilde{\mathbb{Z}}^d$ with duration in $[\delta^2,2\delta^2]$ such that $\{v,w\}\subset \mathrm{ran}(\ell)\subset B(v,2\delta)$, and let $\mathcal{S}:=\cup_{e}\mathcal{S}_e$. Note that every loop in $\mathcal{S}$ has diameter at most $4\delta$. By the scaling invariance of the loop measure, $\mu(\mathcal{S}_e)$ is a positive constant independent of $e$ and $\delta$. Moreover, every loop in $\mathcal{S}$ belongs to $\mathcal{S}_e$ for at most $C$ edges $e$. Consequently, there exists $c>0$ such that for any collection $\mathcal{E}$ of edges,
\begin{equation}\label{9.1}
	\mu\big(\cup_{e\in \mathcal{E}}\mathcal{S}_e\big)\ge c\cdot \#\mathcal{E}.
\end{equation}
For $\ell,\ell'\in \mathcal{L}^{\delta}$ with $\ell\xleftrightarrow{}\ell'$, we say that a loop $\ell''\in \mathcal{L}^{\delta}\cap \mathcal{S}$ is pivotal (for $\{\ell\xleftrightarrow{}\ell'\}$) if $\ell$ and $\ell'$ are not connected by the loops in $\mathcal{L}^{\delta}-\delta_{\ell''}$. We denote by $\mathcal{P}(\ell,\ell')$ the number of pivotal loops, and set $\mathcal{P}(\ell,\ell'):=0$ if $\ell$ and $\ell'$ are not connected.

\begin{lemma}\label{lemma_pivotal_second_moment}
	For any $R,\varsigma>0$, there exists $C_\star>0$ such that for all $n\in \mathbb{N}^+$ and $i,j\in \mathbb{N}^+$,
	\begin{equation}\label{9.5}
		\mathbb{E}\big[\mathcal{P}(\ell^{\delta}_i,\ell^{\delta}_j)^2\cdot \mathbbm{1}_{i,j\le \mathfrak{m}_\delta}\big]\le C_\star\delta^{2-d}.
	\end{equation}
\end{lemma}
\begin{proof}
	If a loop $\ell\in \mathcal{S}$ is pivotal for $\{\ell^{\delta}_i\xleftrightarrow{}\ell^{\delta}_j\}$, then for any vertex $v\in \mathrm{ran}(\ell)$, after removing all loops in $\mathcal{S}$ intersecting $B(v,4\delta)$, the clusters containing $\ell^{\delta}_i$ and $\ell^{\delta}_j$ are disjoint and both intersect $B(v,8\delta)$. Since $|\ell^{\delta}_i|,|\ell^{\delta}_j|>\varsigma$, both of these clusters reach distance $\frac{1}{2}\varsigma$ from $v$. In other words, two disjoint clusters with diameters at least $\frac{1}{2}\varsigma$ nearly touch within $B(v,8\delta)$. Moreover, the number of loops in $\mathcal{S}$ intersecting $B(v,4\delta)$ is a Poisson random variable with bounded mean, which is independent of the loops not intersecting $B(v,4\delta)$ (for two vertices $v$ and $v'$, we remove the loops in $\mathcal{S}$ intersecting $B(v,4\delta)\cup B(v',4\delta)$ instead). Therefore, (\ref{9.5}) follows from the second moment estimate on the number of touching points in \cite[(3.19)]{cai2025separation}.\end{proof}

We now make the bounds in (\ref{eq_outline_extra}) and (\ref{eq_outline_extra_CS}) rigorous. For $\varsigma>0$, let $\mathcal{F}^{\delta}_{>\varsigma}$ denote the $\sigma$-field generated by the loops in $\mathcal{L}^{\delta}$ with diameter greater than $\varsigma$. For $i,j\in \mathbb{N}^+$ and $\epsilon>0$, we define the events
\begin{equation*}
	\mathsf{D}^{\delta}_{ij}:=\big\{i,j\le \mathfrak{m}_\delta, \ell^{\delta}_i\xleftrightarrow{}\ell^{\delta}_j\big\}\ \ \text{and}\ \ \mathsf{T}^{\delta}_{ij}(\epsilon):=\big\{i,j\le \mathfrak{m}_\delta,  \mathcal{T}(\ell^{\delta}_i,\ell^{\delta}_j)\ge \epsilon\delta^{-\frac{d}{2}+1}\big\}.
\end{equation*}
Note that $\mathsf{T}^{\delta}_{ij}(\epsilon)$ and $\mathsf{D}^{\delta}_{ij}$ are disjoint.


\begin{lemma}\label{lemma_single_link}
	For any $R,\varsigma,\epsilon>0$, there exists $K>0$ such that for any $\varsigma'\in (0,\varsigma]$, any $n$ with $4\delta<\varsigma'$, any $i,j\in \mathbb{N}^+$ and any event $\mathsf{Q}\in \mathcal{F}^{\delta}_{>\varsigma'}$,
	\begin{equation}\label{9.6}
		\mathbb{P}\big(\mathsf{Q}\cap \mathsf{T}^{\delta}_{ij}(\epsilon)\big)\le K \cdot \mathbb{P}\big(\mathsf{Q}\cap \mathsf{D}^{\delta}_{ij}\big)^{\frac{1}{2}}.
	\end{equation}
\end{lemma}

\begin{proof}
	For a configuration $\omega$ of loops and a loop $\ell$, let $f(\omega,\ell)$ be the indicator of the event that $\omega\in \mathsf{Q}\cap \mathsf{T}^{\delta}_{ij}(\epsilon)$, and $\ell\in \mathcal{S}_e$ for some edge $e$ on which the clusters $\mathcal{C}_{\ell^{\delta}_i}$ and $\mathcal{C}_{\ell^{\delta}_j}$ nearly touch. Since $\mathcal{L}^{\delta}$ is a Poisson point process with intensity measure $\frac{1}{2}\mu$, it follows from the Mecke formula (see e.g., \cite[Theorem 4.1]{last2017lectures}) and (\ref{9.1}) that
	\begin{equation}\label{9.7}
		\mathbb{E}\Big[\sum\nolimits_{\ell\in \mathcal{L}^{\delta}}f\big(\mathcal{L}^{\delta}-\delta_{\ell},\ell\big)\Big]=\tfrac{1}{2}\int \mathbb{E}\big[f(\mathcal{L}^{\delta},\ell)\big]\,\mu(\mathrm{d}\ell)\ge \tfrac{1}{2}c\epsilon\delta^{-\frac{d}{2}+1} \mathbb{P}\big(\mathsf{Q}\cap \mathsf{T}^{\delta}_{ij}(\epsilon)\big),	\end{equation}
		where $c$ is the constant in (\ref{9.1}). We now give an upper bound on the left-hand side of (\ref{9.7}). Suppose that $f(\omega'-\delta_{\ell},\ell)=1$ for some configuration $\omega'$ and some loop $\ell$ in $\omega'$. Let $\omega:=\omega'-\delta_{\ell}$. Since $|\ell|\le 4\delta<\varsigma'$, the loops with diameters greater than $\varsigma'$ in $\omega$ and $\omega'$ coincide. In particular, we have $\omega'\in \mathsf{Q}$, and the loops $\ell^{\delta}_1,\dots,\ell^{\delta}_{\mathfrak{m}_\delta}$ are the same for $\omega$ and $\omega'$. Moreover, since $\ell\in \mathcal{S}_e$ for some edge $e$ on which $\mathcal{C}_{\ell^{\delta}_i}$ and $\mathcal{C}_{\ell^{\delta}_j}$ (of $\omega$) nearly touch, $\mathrm{ran}(\ell)$ intersects both of these clusters; as a result, $\omega'\in \mathsf{D}^{\delta}_{ij}$. Since $\ell^{\delta}_i$ and $\ell^{\delta}_j$ are not connected in $\omega$, the loop $\ell$ is pivotal for $\{\ell^{\delta}_i\xleftrightarrow{}\ell^{\delta}_j\}$ in $\omega'$. Consequently, by the Cauchy--Schwarz inequality and Lemma~\ref{lemma_pivotal_second_moment}, the left-hand side of (\ref{9.7}) is at most
	\begin{equation*}
		\mathbb{E}\big[\mathcal{P}(\ell^{\delta}_i,\ell^{\delta}_j)\cdot \mathbbm{1}_{\mathsf{Q}\cap \mathsf{D}^{\delta}_{ij}}\big]\le \mathbb{E}\big[\mathcal{P}(\ell^{\delta}_i,\ell^{\delta}_j)^2\cdot \mathbbm{1}_{i,j\le \mathfrak{m}_\delta}\big]^{\frac{1}{2}}\cdot\mathbb{P}\big(\mathsf{Q}\cap \mathsf{D}^{\delta}_{ij}\big)^{\frac{1}{2}}\le C_\star^{\frac{1}{2}}\delta^{-\frac{d}{2}+1}\cdot \mathbb{P}\big(\mathsf{Q}\cap \mathsf{D}^{\delta}_{ij}\big)^{\frac{1}{2}}.
	\end{equation*}
	Combined with (\ref{9.7}), it yields the desired bound (\ref{9.6}).  
\end{proof}

We now conclude via the argument by contradiction outlined in Section~\ref{subsection_extra_randomness}.

\begin{proof}[Proof of the extra randomness part of Theorem~\ref{new_thm1}]
	Assume for contradiction that $\sim_{\mathrm{g}}$ is measurable with respect to $\mathcal{L}^{\mathrm{B}}$. Recall that $\delta=2^{-n}$. Let $R_\star,\varsigma_\star,\epsilon_\star,c_\star$ be as in Lemma~\ref{lemma_rich_touching}, and let $K$ be the constant in Lemma~\ref{lemma_single_link} with $(R,\varsigma,\epsilon)=(R_\star,\varsigma_\star,\epsilon_\star)$. For $i,j\in \mathbb{N}^+$, we define the event $\mathsf{D}_{ij}:=\{i,j\le \mathfrak{m},\ \ell_i\sim_{\mathrm{g}}\ell_j\}$, where $\ell_1,...,\ell_\mathfrak{m}$ (resp. $\ell^{\delta}_1,...,\ell^{\delta}_{\mathfrak{m}_\delta}$) are the loops in $\mathcal{L}^{\mathrm{B}}_{R_\star,\varsigma_\star}$ (resp. $\mathcal{L}^{\delta}_{R_\star,\varsigma_\star}$) enumerated in increasing order of $\mathfrak{h}(\cdot)$ as above. By our assumption, $\mathsf{D}_{ij}$ is measurable with respect to $\mathcal{L}^{\mathrm{B}}$.


	Arbitrarily fix $i,j\in \mathbb{N}^+$ and $\eta>0$. We first approximate $(\mathsf{D}_{ij})^c$ by an event depending only on finitely many macroscopic loops. Since the $\sigma$-fields generated by $\mathcal{L}^{\mathrm{B}}_{R',\varsigma'}$ increase to that generated by $\mathcal{L}^{\mathrm{B}}$ as $R'\uparrow \infty$ and $\varsigma'\downarrow 0$, there exist $R'\ge R_\star$, $\varsigma'\in (0,\varsigma_\star]$ and a Borel set $U$ (in the space of finite collections of loops, equipped with $\mathrm{d}^*_{\mathrm{H}}$) such that 
	\begin{equation}\label{q1}
		\mathbb{P}(\{\mathcal{L}^{\mathrm{B}}_{R',\varsigma'}\in U\}\triangle (\mathsf{D}_{ij})^c)<\eta.
	\end{equation}
	Referring to \cite[Theorem 1.1]{Billingsley1999}, there exists an open set $V\supset U$ such that 
	\begin{equation}\label{q2}
		\mathbb{P}(\mathcal{L}^{\mathrm{B}}_{R',\varsigma'}\in V\setminus U)<\eta. 
	\end{equation}
	For $t>0$, let $O_t:=\{x:\mathrm{d}^*_{\mathrm{H}}(x,V^c)>t\}$. Since the sets $\{x:\mathrm{d}^*_{\mathrm{H}}(x,V^c)=t\}$ are disjoint for different $t$ and $O_t\uparrow V$ as $t\downarrow 0$, we can take $t>0$ such that $\mathbb{P}(\mathcal{L}^{\mathrm{B}}_{R',\varsigma'}\in \partial O_t)=0$ and
	\begin{equation}\label{q3}
		\mathbb{P}(\mathcal{L}^{\mathrm{B}}_{R',\varsigma'}\in V\setminus O_t)<\eta.
	\end{equation}
	 We then define $\mathsf{Q} =\mathsf Q(i,j):=\{\mathcal{L}^{\mathrm{B}}_{R',\varsigma'}\in O_t\}$. Since $O_t\subset V$ and $U\subset V$, we have $O_t\triangle U\subset (V\setminus U)\cup (V\setminus O_t)$. By (\ref{q1}), (\ref{q2}) and (\ref{q3}), $\mathbb{P}\big(\mathsf{Q}\triangle (\mathsf{D}_{ij})^c\big)$ is at most \begin{equation}\label{9.8}
  \mathbb{P}\big(\{\mathcal{L}^{\mathrm{B}}_{R',\varsigma'}\in U\}\triangle (\mathsf{D}_{ij})^c\big)+\mathbb{P}\big(\mathcal{L}^{\mathrm{B}}_{R',\varsigma'}\in V\setminus U\big)+\mathbb{P}\big(\mathcal{L}^{\mathrm{B}}_{R',\varsigma'}\in V\setminus O_t\big)<3\eta.
\end{equation} 
	We define the discrete analogue of $\mathsf{Q}$ by $\mathsf{Q}^{\delta}=\{\mathcal{L}^{\delta}_{R',\varsigma'}\in O_t\}$, which belongs to $\mathcal{F}^{\delta}_{>\varsigma'}$.


Next, we couple the discrete and continuum loop soups. Recall from Section~\ref{subsection5.3_construction} that there is a coupling such that $\mathbf{T}^{R',\varsigma'}_{(n)}$ converges to $\mathbf{T}^{R',\varsigma'}_{(\infty)}$ almost surely, and that the restriction of $\sim_{\mathrm{g}}$ to $\mathcal{L}^{\mathrm{B}}_{R',\varsigma'}$ is given by $\sim^{R',\varsigma'}_{(\infty)}$. Since almost surely no loop in $\mathcal{L}^{\mathrm{B}}$ has diameter exactly $\varsigma'$ or is contained in $B(R')$ while intersecting $\partial B(R')$, this convergence yields that almost surely for all sufficiently large $n$, there exists a bijection $\sigma_n:\mathcal{L}^{\mathrm{B}}_{R',\varsigma'}\to \mathcal{L}^{\delta}_{R',\varsigma'}$ which preserves the equivalence relations $\sim^{R',\varsigma'}_{(\infty)}$ and $\sim^{R',\varsigma'}_{(n)}$. Moreover, these bijections satisfy
	\begin{equation}\label{9.12}
		\lim\nolimits_{n\to \infty}\max\nolimits_{\ell\in \mathcal{L}^{\mathrm{B}}_{R',\varsigma'}}\mathrm{d}_{\mathrm{H}}\big(\mathrm{ran}(\ell),\mathrm{ran}(\sigma_n(\ell))\big)=0.
	\end{equation}
	Recall that $\sim^{R',\varsigma'}_{(n)}$ is induced by the clusters of $\widehat{\mathcal{L}}+\mathcal{L}^{\nu}$ rather than those of $\mathcal{L}$. However, by the argument for (\ref{ineq_i0_lemma3.1}) in the proof of Lemma~\ref{lemmab72}, the probability that these two relations differ on $\mathcal{L}^{\delta}_{R',\varsigma'}$ is $O(\delta^{c})$ for some $c>0$. Since $\sum_{n\ge 1}2^{-cn}<\infty$, the Borel--Cantelli lemma implies that almost surely for all sufficiently large $n$,
	\begin{equation}\label{9.13}
		\ell\sim_{\mathrm{g}}\ell' \iff \sigma_n(\ell)\xleftrightarrow{}\sigma_n(\ell'),\ \ \forall \ell,\ell'\in \mathcal{L}^{\mathrm{B}}_{R',\varsigma'}.
	\end{equation}

	We now compare the events in the two settings. Since almost surely no loop in $\mathcal{L}^{\mathrm{B}}$ has diameter exactly $\varsigma_\star$ or is contained in $B(R_\star)$ while intersecting $\partial B(R_\star)$, by (\ref{9.12}), for all sufficiently large $n$ the bijection $\sigma_n$ maps $\mathcal{L}^{\mathrm{B}}_{R_\star,\varsigma_\star}$ onto $\mathcal{L}^{\delta}_{R_\star,\varsigma_\star}$. Moreover, since $\mathfrak{h}(\cdot)$ is continuous with respect to the Hausdorff distance, and almost surely the values $\{\mathfrak{h}(\ell)\}_{\ell\in \mathcal{L}^{\mathrm{B}}}$ are distinct, (\ref{9.12}) implies that $\mathfrak{m}_\delta=\mathfrak{m}$ and $\ell^{\delta}_k=\sigma_n(\ell_k)$ for all $1\le k\le \mathfrak{m}$. In addition, (\ref{9.12}) also implies that $\mathcal{L}^{\delta}_{R',\varsigma'}$ converges to $\mathcal{L}^{\mathrm{B}}_{R',\varsigma'}$ with respect to $\mathrm{d}^*_{\mathrm{H}}$. Therefore, by (\ref{9.13}) and $\mathbb{P}(\mathcal{L}^{\mathrm{B}}_{R',\varsigma'}\in \partial O_t)=0$, almost surely for all sufficiently large $n$, one has $\mathbbm{1}_{\mathsf{Q}^{\delta}}=\mathbbm{1}_{\mathsf{Q}}$ and $\mathbbm{1}_{\mathsf{D}^{\delta}_{ij}}=\mathbbm{1}_{\mathsf{D}_{ij}}$. By the bounded convergence theorem and (\ref{9.8}), we obtain
	\begin{equation}\label{9.9}
		\lim\nolimits_{n\to \infty}\mathbb{P}\big(\mathsf{Q}^{\delta}\cap \mathsf{D}^{\delta}_{ij}\big)=\mathbb{P}\big(\mathsf{Q}\cap \mathsf{D}_{ij}\big)<3\eta.
	\end{equation}
	Similarly, since $\mathsf{T}^{\delta}_{ij}(\epsilon_\star)\subset \{i,j\le \mathfrak{m}_\delta\}\setminus \mathsf{D}^{\delta}_{ij}$, we have
	\begin{equation}\label{9.10}
	\begin{split}
				\limsup\nolimits_{n\to \infty}\mathbb{P}\big(\mathsf{T}^{\delta}_{ij}(\epsilon_\star)\setminus \mathsf{Q}^{\delta}\big)\le & \lim\nolimits_{n\to \infty}\mathbb{P}\big(\{i,j\le \mathfrak{m}_\delta\}\setminus (\mathsf{D}^{\delta}_{ij}\cup \mathsf{Q}^{\delta})\big) \\
				=&\mathbb{P}\big(\{i,j\le \mathfrak{m}\}\setminus (\mathsf{D}_{ij}\cup \mathsf{Q})\big)<3\eta.
	\end{split}
	\end{equation}
	Applying Lemma~\ref{lemma_single_link} to $\mathsf{Q}^{\delta}$, and using (\ref{9.9}) and (\ref{9.10}), we have
	\begin{equation*}
		\limsup_{n\to \infty}\mathbb{P}\big(\mathsf{T}^{\delta}_{ij}(\epsilon_\star)\big)\le \limsup_{n\to \infty}\Big[K\cdot \mathbb{P}\big(\mathsf{Q}^{\delta}\cap \mathsf{D}^{\delta}_{ij}\big)^{\frac{1}{2}}+\mathbb{P}\big(\mathsf{T}^{\delta}_{ij}(\epsilon_\star)\setminus \mathsf{Q}^{\delta}\big)\Big]\le K(3\eta)^{\frac{1}{2}}+3\eta.
	\end{equation*}
	Since $\eta>0$ is arbitrary, we conclude that for any $i,j\in \mathbb{N}^+$, 
	\begin{equation}\label{9.11}
		\lim\nolimits_{n\to \infty}\mathbb{P}\big(\mathsf{T}^{\delta}_{ij}(\epsilon_\star)\big)=0.
	\end{equation}

	We now derive a contradiction. Recall that $\mathbb{E}[\mathfrak{m}_\delta]$ is bounded uniformly in $n$. Thus, there exists $m_\diamond \in \mathbb{N}^+$ such that $\mathbb{P}(\mathfrak{m}_\delta>m_\diamond)\le \frac{1}{2}c_\star$ for all $n$. Therefore, by Lemma~\ref{lemma_rich_touching}, for all sufficiently large $n$, we have 
	\begin{equation}
		c_\star\le \mathbb{P}\Big(\bigcup\nolimits_{i,j\in \mathbb{N}^+}\mathsf{T}^{\delta}_{ij}(\epsilon_\star)\Big)\le \sum\nolimits_{1\le i,j\le m_\diamond}\mathbb{P}\big(\mathsf{T}^{\delta}_{ij}(\epsilon_\star)\big)+\tfrac{1}{2}c_\star.
	\end{equation}
	By (\ref{9.11}), the sum on the right-hand side tends to zero as $n\to \infty$. This yields $c_\star\le \frac{1}{2}c_\star$, a contradiction. Thus, $\sim_{\mathrm{g}}$ is not measurable with respect to $\mathcal{L}^{\mathrm{B}}$, which completes the proof of the extra randomness part of Theorem~\ref{new_thm1}. 
\end{proof}

  \section{Dimension gap between loop clusters}\label{section_dimension_gap}

This section is devoted to proving Theorem~\ref{theorem_dimension_gap}. The proof of Item~(i) consists of three steps. In Section~\ref{subsec:remove_small_loops}, we show that removing the microscopic loops of $\mathcal{L}$ strictly increases the exponent of the crossing probability (see Proposition~\ref{prop:remove_small_loops}). In Section~\ref{subsec:BLS_one_arm}, we combine this estimate with the coupling in Lemma~\ref{lemma_coupling} to show that in three dimensions, the one-arm exponent of the Brownian loop soup $\mathcal{L}^{\rm B}$ is strictly larger than $\frac{1}{2}$, the one-arm exponent of $\mathcal{L}$ (see Proposition~\ref{prop:BLS_one_arm}). In Section~\ref{subsec:proof_item_i}, we convert this estimate into an upper bound on the box-counting dimension via a first moment argument. Finally, in Section~\ref{subsec:proof_item_ii}, we prove Item~(ii): the upper bound follows from (\ref{crossing_prob}) together with the convergence in Theorem~\ref{new_thm1}, while the lower bound relies on the fact that a macroscopic loop must pass through many intermediate scale regions, each containing a thick cluster with positive probability.

 \subsection{Crossing probabilities after the removal of small loops}
\label{subsec:remove_small_loops}

We denote by $\mathcal{L}^{\rm tri}$ the point measure consisting of the loops $\ell$ in $\mathcal{L}$ such that $\mathrm{ran}(\ell)\subset I_e$ for some $e\in\mathbb{L}^d$ and $\mathrm{ran}(\ell)$ contains both trisection points of $I_e$.

 \begin{proposition}\label{prop:remove_small_loops}
For any $3\le d\le 5$, there exists $\cl\label{const_remove_small}>0$ such that for all $N>n\ge 1$,
\begin{equation}\label{eq:remove_small_loops}
  \mathbb{P}\big(B(n)\xleftrightarrow{ \mathcal{L}-\mathcal{L}^{\rm tri}  }
  \partial B(N)\big)\lesssim (n/N)^{\frac{d}{2}-1+\cref{const_remove_small}}.
\end{equation}
\end{proposition} 
 
 \begin{proof}
 	We employ the exploration process introduced in the proof of Lemma~\ref{newlemma_pivotal}. We take $M=N$ and retain the notation there: the constant $C_\dagger$, the scales $r_i=C_\dagger^in$, the index $i_\star$, the partial clusters $\mathcal{C}_i$ and the $\sigma$-fields $\mathcal{F}_{\mathcal{C}_i}$, (P1)--(P3), and the sequence $\{i_k\}_{0\le k\le k_\star}$ from Section~\ref{section_control_massive}. Let $\mathsf{F}$ denote the event that $B(n)\xleftrightarrow{} \partial B(N)$ and that no loop in $\mathcal{L}^{\rm tri}$ is pivotal for $\{B(n)\xleftrightarrow{}\partial B(N)\}$. In particular, $\{B(n)\xleftrightarrow{\ \mathcal{L}-\mathcal{L}^{\rm tri}\ }
\partial B(N)\}\subset\mathsf{F}$. Combined with (\ref{crossing_prob}), this implies
\begin{equation*}
  \mathbb{P}\big(B(n)\xleftrightarrow{\ \mathcal{L}-\mathcal{L}^{\rm tri}\ }
  \partial B(N)\big)\lesssim(n/N)^{\frac{d}{2}-1}\cdot\mathbb{P}\big(\mathsf{F}
  \mid B(n)\xleftrightarrow{}\partial B(N)\big).
\end{equation*}
Therefore, it suffices to show that for some constant $c>0$,
\begin{equation}\label{eq:cond_F}
  \mathbb{P}\big(\mathsf{F} \mid B(n)\xleftrightarrow{}\partial B(N)\big)\lesssim
  (n/N)^{c}.
\end{equation}

 	For $1\le i<i'\le i_\star$, let $\mathsf{W}^{\rm tri}_{i,i'}$ denote the event that some loop in $\mathcal{L}^{\rm tri}$ that is not included in $\mathcal{C}_i$ is pivotal for $\{\mathcal{C}_i\xleftrightarrow{}\partial B(r_{2i'})\}$. We claim that $\{\mathsf{W}^{\rm tri}_{i,i'}\}_{1\le i<i'\le i_\star}$ satisfies (W1) and (W2) in Proposition~\ref{prop:exploration}. For (W1), recall from the proof
of Lemma~\ref{newlemma_pivotal} that every chain of loops connecting
$\mathcal{C}_i$ to $\partial B(r_{2i'})$, stopped at the first loop intersecting $\partial B(r_{2i'})$, consists of loops included in $\mathcal{C}_{i'}$, and recall that whether $\{\mathcal{C}_i\xleftrightarrow{}\partial B(r_{2i'})\}$ still occurs after removing a given loop is determined by $\mathcal{F}_{\mathcal{C}_{i'}}$. Since a loop that is pivotal for $\{\mathcal{C}_i\xleftrightarrow{}\partial B(r_{2i'})\}$ belongs to every such chain, it is included in $\mathcal{C}_{i'}$. Hence $\mathsf{W}^{\rm tri}_{i,i'}$ is measurable with respect to $\mathcal{F}_{\mathcal{C}_{i'}}$. Moreover, Condition (W2) follows from \cite[Lemma 5.3]{cai2025separation}. Thus, by Proposition~\ref{prop:exploration}, we have 
\begin{equation}\label{eq:W_tri}
  \mathbb{P}\big(\mathbf{W}^{\rm tri}\le c_\heartsuit i_\star\mid
  B(n)\xleftrightarrow{}\partial B(N)\big)\lesssim(n/N)^{c'_\heartsuit},
\end{equation}
where $\mathbf{W}^{\rm tri}:=\sum_{1\le k\le k_\star-2} \mathbbm{1}_{\mathsf{W}^{\rm tri}_{i_k,i_{k+1}}}$. On $\{B(n)\xleftrightarrow{}\partial B(N)\}\cap\mathsf{W}^{\rm tri}_{i,i'}$, it follows from (P3) that there exists a loop in $\mathcal{L}^{\rm tri}$ that is pivotal for
$\{B(n)\xleftrightarrow{}\partial B(N)\}$. As a result, $\mathsf{F}\subset
\{\mathbf{W}^{\rm tri}=0\}$. Thus, the desired bound \eqref{eq:cond_F} follows from \eqref{eq:W_tri}, which completes the proof. 
 \end{proof}

 \subsection{One-arm probability of the Brownian loop soup} \label{subsec:BLS_one_arm}

 Throughout this subsection, we fix $d=3$. For any $A,A'\subset\mathbb{R}^3$, we write $A\xleftrightarrow{\mathcal{L}^{\rm B}}A'$ if there exists a cluster in $\mathfrak{C}^{\rm B}$ intersecting both $A$ and $A'$. For $\epsilon\in(0,1)$, we denote
\begin{equation*}
  \mathcal{Q}(\epsilon):=\mathbb{P}\big(\partial B(\epsilon)
  \xleftrightarrow{\mathcal{L}^{\rm B}}\partial B(1)\big).
\end{equation*}
For any $x\in\mathbb{R}^3$ and $R>r>0$, using the scaling and translation invariance of $\mathcal{L}^{\rm B}$, we have 
\begin{equation}\label{eq:BLS_scaling}
  \mathbb{P}\big(\partial B(x,r)\xleftrightarrow{\mathcal{L}^{\rm B}}
  \partial B(x,R)\big)=\mathcal{Q}(r/R).
\end{equation}

\begin{proposition}\label{prop:BLS_one_arm}
There exists $\cl\label{const_BLS_onearm}\in(0,\frac{1}{2}]$ such that for all
$\epsilon\in(0,1)$,
\begin{equation}\label{eq:BLS_one_arm}
  \mathcal{Q}(\epsilon)\lesssim \epsilon^{\frac{1}{2}+\cref{const_BLS_onearm}}.
\end{equation}
\end{proposition}

Before proving Proposition~\ref{prop:BLS_one_arm}, we record two facts as follows:
\begin{enumerate}

	\item For any $d\ge 3$, $x\in\mathbb{R}^d$ and $R>r>0$,
\begin{equation}\label{eq:BLS_crossing}
  \mathbb{P}\big(\exists\ell\in\mathcal{L}^{\rm B}\ \text{crossing}\ B(x,R) \setminus B(x,r) \big)\lesssim (r/R)^{d-2}.
\end{equation}
This can be derived by repeating the argument for (\ref{ineq_crossing_loop}), with the hitting probabilities of random walks replaced by those of Brownian motion.

	\item Let $\mathcal{L}^{\mathbb{Z}^3}:=\mathcal{L}^{\mathbb{Z}^3}_{1/2}$ be the random walk loop soup on $\mathbb{Z}^3$ of intensity $\frac{1}{2}$, and denote by $\mu^{\mathbb{Z}^3}$ its loop measure. For a loop $\ell$ on $\mathbb{Z}^3$, we denote by $\mathrm{ran}(\ell)\subset\mathbb{Z}^3$ the set of lattice points visited by $\ell$. As shown in \cite[Section 2]{lupu2016loop}, $\mathcal{L}$ (on $\widetilde{\mathbb{Z}}^3$) and $\mathcal{L}^{\mathbb{Z}^3}$ can be coupled so that there is a bijection $\varpi$ between the loops in $\mathcal{L}^{\mathbb{Z}^3}$ and in $\mathcal{L}$ that visit at least two lattice points. Moreover, for every such loop $\ell\in\mathcal{L}^{\mathbb{Z}^3}$,
\begin{equation}\label{eq:projection}
\mathrm{ran}(\varpi(\ell))\cap\mathbb{Z}^3=\mathrm{ran}(\ell)\ \ \text{and}\ \ \sup_{v\in\mathrm{ran}(\varpi(\ell))}\mathrm{d}\big({v},\mathrm{ran}(\ell)\big)\le 1.
\end{equation} 
 A straightforward computation (similar to that for (\ref{ineq_crossing_loop})) implies that for any $x \in\mathbb{R}^3$ and $R>r>0$,
\begin{equation}\label{eq:RW_crossing}
  \mu^{\mathbb{Z}^3}\big(\{\ell: \ell \ \text{crosses}\  B(x,R) \setminus B(x,r)\}\big)\lesssim r/R.
\end{equation}

\end{enumerate}

 \begin{proof}[Proof of Proposition~\ref{prop:BLS_one_arm}]
For any $\delta>0$ and $\epsilon\in(0,1)$, let $\mathsf{F}^{\ge\delta}_\epsilon$ denote the event that there exists a finite sequence $\{\ell_i\}_{1\le i\le k}$ of loops in $\mathcal{L}^{\rm B}$ with diameters at least $\delta$ such that $\mathrm{ran}(\ell_1)\cap\partial B(\epsilon)\neq\emptyset$, $\mathrm{ran}(\ell_k)\cap\partial B(1)\neq\emptyset$, and $\mathrm{ran}(\ell_i)\cap\mathrm{ran}(\ell_{i+1})\neq\emptyset$ for all $1\le i\le k-1$. Since clusters in $\mathfrak{C}^{\rm B}$ are formed by finite chains of loops, $\mathbbm{1}_{\mathsf{F}^{\ge\delta}_\epsilon}\to\mathbbm{1}_{\{\partial B(\epsilon)\xleftrightarrow{\mathcal{L}^{\rm B}}\partial B(1)\}}$ almost surely as $\delta\downarrow 0$. Moreover, $\mathsf{F}^{\ge\delta}_\epsilon$ is decreasing in $\delta$. Hence, by the monotone convergence theorem, we have 
\begin{equation}\label{eq:monotone_F}
  \mathcal{Q}(\epsilon)=\lim_{\delta\downarrow 0}\mathbb{P}\big(
  \mathsf{F}^{\ge\delta}_\epsilon\big).
\end{equation}
Therefore, it suffices to show that for some constant $c_\diamond>0$,
\begin{equation}\label{eq:F_delta}
  \mathbb{P}\big(\mathsf{F}^{\ge\delta}_\epsilon\big)\lesssim\epsilon^{\frac{1}{2}
  +c_\diamond},\ \ \forall\epsilon\in(0,1)\ \text{and}\ \delta\in(0,\epsilon).
\end{equation}
The remainder of this proof is devoted to proving \eqref{eq:F_delta} with
$c_\diamond:=\cref{const_remove_small}\wedge\frac{1}{2}$, which yields
\eqref{eq:BLS_one_arm} with $\cref{const_BLS_onearm}:=c_\diamond$.

  Without loss of generality, we assume that $\epsilon\in(0,\frac{1}{100})$. Arbitrarily fix $\delta\in(0,\epsilon)$ and let $R:=2\epsilon^{-1}$. Next, we approximate the event $\mathsf{F}^{\ge\delta}_\epsilon$ step by step, and control the error at each step. We define the event
\begin{equation}\label{eq:G_def}
  \mathsf{G}:=\big\{\exists\ell\in\mathcal{L}^{\rm B}\ \text{intersecting both}\
  B(2)\ \text{and}\ \partial B(\epsilon^{-1})\big\},
\end{equation}
and let $\mathsf{F}^{(1)}:=\mathsf{F}^{\ge\delta}_\epsilon\cap\mathsf{G}^c$. By the union bound and \eqref{eq:BLS_crossing},
\begin{equation}\label{eq:F1}
  \mathbb{P}\big(\mathsf{F}^{\ge\delta}_\epsilon\big)\le\mathbb{P}\big(
  \mathsf{F}^{(1)}\big)+\mathbb{P}(\mathsf{G})\le\mathbb{P}\big(
  \mathsf{F}^{(1)}\big)+C\epsilon.
\end{equation}
On the event $\mathsf{F}^{\ge\delta}_\epsilon$, we take a chain $\{\ell_i\}_{1\le i\le k}$ certifying it with $k$ minimal. By the minimality, $\ell_1,\dots,\ell_{k-1}$ do not intersect $\partial B(1)$; since $\ell_1$ intersects $\partial B(\epsilon)$ and each $\ell_{i+1}$ intersects $\ell_i$, we have $\mathrm{ran}(\ell_i)\subset B(1)$ for all $1\le i\le k-1$. Moreover, on $\mathsf{G}^c$, since $\ell_k$ intersects $B(2)$, we have $\mathrm{ran}(\ell_k)\subset B(\epsilon^{-1})$. Hence, on $\mathsf{F}^{(1)}$ all loops $\ell_i$ belong to the collection $\mathcal{A}^*$ of the loops in $\mathcal{L}^{\rm B}$ that are contained in $B(R)$ and have diameter at least $\delta$.

Subsequently, we approximate the loops in $\mathcal{A}^*$ by the loops in $N^{-1}\cdot\mathcal{L}$. Take $N\in\mathbb{N}$ sufficiently large such that $N^{-\cref{const_coupling1}}<\delta$ and $2N^{-2\cref{const_coupling2}}<\frac{1}{4}\delta$. It follows from Lemma \ref{lemma_coupling} that \begin{equation}\label{eq:P_A}
  \mathbb{P}(\mathsf{A}^c)\le\Cref{const_coupling3}(R+1)^3N^{-\cref{const_coupling2}},
\end{equation}
where $\mathsf{A}$ denotes the event that there exists a bijection $\sigma$ from the loops in $\mathcal{L}^{\rm B}_{R,N^{-\cref{const_coupling1}}}$ to those in $(N^{-1}\cdot\mathcal{L})_{R, N^{-\cref{const_coupling1}}}$ such that
$\mathrm{d}_{\rm H}\big(\mathrm{ran}(\ell),\mathrm{ran}(\sigma(\ell))\big)\le N^{-2\cref{const_coupling2}}$ for all $\ell$. We define the event $\mathsf{F}^{(2)}:=\mathsf{F}^{(1)}\cap\mathsf{A}$. By the union bound and \eqref{eq:P_A}, we have 
\begin{equation}\label{eq:F2}
  \mathbb{P}\big(\mathsf{F}^{(1)}\big)\le\mathbb{P}\big(\mathsf{F}^{(2)}\big)
  +\mathbb{P}(\mathsf{A}^c)\le\mathbb{P}\big(\mathsf{F}^{(2)}\big)+C\epsilon.
\end{equation}

We now estimate the probability of $\mathsf{F}^{(2)}$. On this event, every $\ell_i$ belongs to $\mathcal{L}^{\rm B}_{R,N^{-\cref{const_coupling1}}}$, and thus $\ell_i':=\sigma(\ell_i)$ is well-defined. Each $\ell_i'$ is contained in $B(R)$ and has diameter at least $\delta-2N^{-2\cref{const_coupling2}}\ge\frac{1}{2}\delta$. For any two curves $\eta,\eta'$ in $\mathbb{R}^3$, we denote the Euclidean distance between their ranges by $\|\eta-\eta'\|$. For each $1\le i\le k-1$, since $\mathrm{ran}(\ell_i)\cap\mathrm{ran}(\ell_{i+1})\neq\emptyset$ (i.e., $\|\ell_i-\ell_{i+1}\|=0$),
\begin{equation}\label{eq:distance}
  \|\ell_i'-\ell_{i+1}'\|\le\|\ell_i-\ell_i'\|+\|\ell_{i+1}-\ell_{i+1}'\|\le
  2N^{-2\cref{const_coupling2}}.
\end{equation}
 In addition, $\ell_1'$ intersects $B(\epsilon+N^{-2\cref{const_coupling2}})\subset B(2\epsilon)$, and $\ell_k'$ intersects $[B(1-N^{-2\cref{const_coupling2}})]^c \subset[B(\frac{1}{2})]^c$. Thus, either the union of $\{\ell_i'\}_{1\le i\le k}$ is connected, in which case $B(2\epsilon)$ and $\partial B(\frac{1}{2})$ are connected by the loops in $N^{-1}\cdot\mathcal{L}$ with diameter at least $\frac{1}{2}\delta$; or there exists $1\le i\le k-1$ such that $\ell_i'$ and $\ell_{i+1}'$ are disjoint. By \eqref{eq:distance}, we obtain
\begin{equation}\label{eq:F2_decomp}
  \mathbb{P}\big(\mathsf{F}^{(2)}\big)\le\mathbb{P}\big(B(2\epsilon)
  \xleftrightarrow{ (N^{-1}\cdot\mathcal{L})\cdot \mathbbm{1}_{|\ell|\ge \delta/2 } }
  \partial B(\tfrac{1}{2})\big)+\mathbb{P}(\mathsf{H}),
\end{equation}
where $\mathsf{H}$ denotes the event that there exist two disjoint loops in $N^{-1}\cdot\mathcal{L}$ such that they are contained in $B(R)$, have diameters at least $\frac{1}{2}\delta$, and their Euclidean distance is at most $2N^{-2\cref{const_coupling2}}$.

 	For the first term on the right-hand side of \eqref{eq:F2_decomp}, by scaling, it equals the probability that $B(2\epsilon N)$ and $\partial B(\frac{1}{2}N)$ are connected by the loops in $\mathcal{L}$ with diameters at least $\frac{1}{2}\delta N$. Note that every loop in $\mathcal{L}^{\rm tri}$ has diameter at most $1$. Thus, applying Proposition~\ref{prop:remove_small_loops} (with $n=2\epsilon N$ and
$\frac{1}{2}N$ in place of $N$), we have
\begin{equation}\label{eq:F2_conn}
\begin{split}
  \mathbb{P}\big(B(2\epsilon)\xleftrightarrow{ (N^{-1}\cdot\mathcal{L})\cdot \mathbbm{1}_{|\ell|\ge \delta/2 } }\partial B(\tfrac{1}{2})\big) 
  \le   \mathbb{P}\big(
  B(2\epsilon N)\xleftrightarrow{  \mathcal{L}-\mathcal{L}^{\rm tri} }\partial
  B(\tfrac{1}{2}N)\big)\lesssim\epsilon^{\frac{1}{2}+\cref{const_remove_small}}.
\end{split}
\end{equation}

Next, we estimate $\mathbb{P}(\mathsf{H})$. For $z\in\mathbb{R}^3$ and $\varrho>\chi\ge 1$, let $\mathcal{L}_z=\mathcal{L}_z(\chi,\varrho)$ denote the collection of loops $\ell\in\mathcal{L}^{\mathbb{Z}^3}$ such that $\mathrm{ran}(\ell)\cap B(z,\chi)\neq\emptyset$ and $\mathrm{ran}(\ell)\not\subset B(z,\varrho)$, and let $\mathsf{H}_z(\chi,\varrho)$ denote the event that $\mathcal{L}_z$ contains two loops with disjoint ranges. Let $\chi:=4N^{-2\cref{const_coupling2}}$. By a standard covering argument, there exists a collection $\mathcal{D}\subset\frac{\chi}{4}\cdot \mathbb{Z}^3$ of $O\big((R/\chi)^3\big)$ points such that $B(R)\subset \bigcup_{z\in\mathcal{D}}B(z,\frac{\chi}{4})$. On the event $\mathsf{H}$, there exist two loops $\ell,\ell'$ in $N^{-1}\cdot\mathcal{L}$ certifying $\mathsf{H}$, and points $x\in\mathrm{ran}(\ell)$, $x'\in\mathrm{ran}(\ell')$ with $|x-x'|\le\frac{1}{2}\chi$. Take $z\in\mathcal{D}$ such that $|x-z|<\frac{1}{4}\chi$, and hence $x,x'\in B(z,\chi)$. Since $\ell$ and $\ell'$ have diameters at least $\frac{1}{2}\delta$ and their ranges are connected, both of them intersect $\partial B(z,\frac{1}{8}\delta)$. We now rescale by $N$ and project onto $\mathbb{Z}^3$ as in \eqref{eq:projection}. Note that $N\cdot\ell$ and $N\cdot\ell'$ are loops in $\mathcal{L}$ that visit at least two lattice points, and that their projections onto $\mathbb{Z}^3$ are disjoint. By \eqref{eq:projection}, these projections intersect $B(Nz,N\chi+1)\subset B(Nz,2N\chi)$ and are not contained in $B(Nz,\frac{1}{16}N\delta)$. In other words, $\mathsf{H}_{Nz}(2N\chi,\frac{1}{16}N\delta)$ occurs. Therefore, by the union bound, we have 
\begin{equation}\label{eq:H_union}
  \mathbb{P}(\mathsf{H})\le\sum\nolimits_{z\in\mathcal{D}}\mathbb{P}\big(\mathsf{H}_{Nz}(
  2N\chi,\tfrac{1}{16}N\delta)\big)\lesssim(R/\chi)^3\cdot\max_{z\in\mathcal{D}}
  \mathbb{P}\big(\mathsf{H}_{Nz}(2N\chi,\tfrac{1}{16}N\delta)\big).
\end{equation}
To bound the right-hand side of \eqref{eq:H_union}, we claim that there exists
$\cl\label{const_two_loops_1}>0$ such that for all sufficiently large $N$ and all
$z\in\mathbb{R}^3$,
\begin{equation}\label{eq:two_loops}
  \mathbb{P}\big(\mathsf{H}_{Nz}(2N\chi,\tfrac{1}{16}N\delta)\big)\lesssim(\chi/\delta)^{3+\cref{const_two_loops_1}}.
\end{equation}
 We first complete the proof assuming \eqref{eq:two_loops}. By plugging
\eqref{eq:two_loops} into \eqref{eq:H_union}, we obtain
\begin{equation}\label{eq:P_H}
  \mathbb{P}(\mathsf{H})\lesssim(R/\chi)^3\cdot(\chi/\delta)^{3+\cref{const_two_loops_1}}
  \lesssim R^3\delta^{-3-\cref{const_two_loops_1}}\chi^{\cref{const_two_loops_1}}\lesssim  \epsilon.
\end{equation}
 Putting \eqref{eq:F1}, \eqref{eq:F2},
\eqref{eq:F2_decomp}, \eqref{eq:F2_conn} and \eqref{eq:P_H} together, we obtain \eqref{eq:F_delta}: 
\begin{equation}
	\mathbb{P}(\mathsf{F}^{\ge\delta}_\epsilon)\lesssim\epsilon^{\frac{1}{2}+
\cref{const_remove_small}}+\epsilon\lesssim\epsilon^{\frac{1}{2}+c_\diamond}. 
\end{equation}
 In conclusion, we have confirmed Proposition~\ref{prop:BLS_one_arm} assuming \eqref{eq:two_loops}.

  It remains to prove \eqref{eq:two_loops}. We fix $z\in\mathbb{R}^3$, and write $w:=Nz$, $\bar\chi:=2N\chi$ and $\bar\varrho:=\frac{1}{16}N\delta$. We abbreviate $\mathcal{L}_w:=\mathcal{L}_w(\bar\chi,
\bar\varrho)$ and $\mathsf{H}_w:=\mathsf{H}_w(\bar\chi,\bar\varrho)$ (in particular, $\mathsf{H}_w$ is the event in \eqref{eq:two_loops}). It follows from \eqref{eq:RW_crossing} that for all sufficiently small $\chi$, 
  \begin{equation}
  	\mathbb{P}(\mathcal{L}_w\neq\emptyset) \le C\bar\chi/\bar\varrho   \le \tfrac{1}{2}. 
  \end{equation}
  On $\mathsf{H}_w\cap\{|\mathcal{L}_w|\ge 3\}$, there exist two loops in $\mathcal{L}_w$ that certify $\mathsf{H}_w$, and another loop that certifies $\{\mathcal{L}_w\neq \emptyset\}$. Thus, by the BKR inequality we have 
\begin{equation}\label{eq:BKR}
  \mathbb{P}\big(\mathsf{H}_w\cap\{|\mathcal{L}_w|\ge 3\}\big)\le
  \mathbb{P}(\mathsf{H}_w)\cdot\mathbb{P}(\mathcal{L}_w\neq\emptyset)\le \tfrac{1}{2}\mathbb{P}(\mathsf{H}_w).
\end{equation}
This together with the inclusion $\mathsf{H}_w\subset\{|\mathcal{L}_w|\ge 2\}$
implies
\begin{equation}\label{eq:reduce_to_two}
  \mathbb{P}(\mathsf{H}_w)\le  2 \mathbb{P}(\widehat{\mathsf{H}}_w):=2\mathbb{P}(\mathsf{H}_w\cap\{|\mathcal{L}_w|=2\}).\end{equation}

  In what follows, we estimate the probability of $\widehat{\mathsf{H}}_w$. When $|\mathcal{L}_w|=2$, we denote the loops in $\mathcal{L}_w$ by $\ell_1$ and $\ell_2$. We denote by $\mathbb{P}_y$ the law of the simple random walk on $\mathbb{Z}^3$ starting from $y$. By \cite[Lemma 2.5]{lawler1996cut}, there exists $c_\star>0$ such that for any $x_0\in\hat B(w,\bar\chi)$, with probability $O((\bar\chi/\bar\varrho)^6)$ a simple random walk $\eta_0$ starting from $x_0$ and stopped at $\tau_{w,\bar\varrho/2}:=\tau_{\hat{\partial}\hat{B}(w,\bar\varrho/2)}$ satisfies  
\begin{equation}\label{eq:cut_time}
  \max_{x\in\hat B(w,\bar\chi)}\mathbb{P}_x\big(\tau_{w,\bar\varrho/2}<
  \tau_{\mathrm{ran}(\eta^0)}\big)\ge (\bar\chi/\bar\varrho)^{c_\star}.
\end{equation}
Since $\ell_1$ stochastically dominates (up to constant factors in the law) the range of such a random walk, this implies that
\begin{equation}\label{e1026}
	\mathbb{P}\big(\widehat{\mathsf{H}}_w,  	\mathsf{R}_\star(\ell_1)\big) \lesssim (\bar\chi/\bar\varrho)^6,
\end{equation}
where the event $\mathsf{R}_\star(\ell_1)$ is defined by
  \begin{equation}
  	\mathsf{R}_\star(\ell_1):= \Big\{ 	\max_{x\in\hat B(w,\bar\chi)}\mathbb{P}_x\big(\tau_{w,\bar\varrho/2}<
\tau_{\mathrm{ran}(\ell_1)}\big)\ge( \bar\chi/\bar\varrho)^{c_\star}  \Big\}. 
  \end{equation}
 Next, we fix $\ell_1$ such that $\mathsf{R}_\star(\ell_1)$ does not occur. By the spatial Markov property, the loop measure of $\ell_2$ such that $\mathrm{ran}(\ell_1)\cap \mathrm{ran}(\ell_2)=\emptyset $ is at most of order 
\begin{equation*}
  \max_{y_1\in\hat B(w,\bar\chi)}\mathbb{P}_{y_1}\big(\tau_{w,\bar\varrho}<\tau_{\mathrm{ran}(\ell_1)}\big)  \cdot\max_{y_2\in\mathbb{Z}^3\setminus B(w,\bar\varrho)}\mathbb{P}_{y_2} \big(\tau_{\hat B(w,\bar\chi)}<\tau_{\mathrm{ran}(\ell_1)}\big).
\end{equation*}
Since a random walk must hit $\hat\partial\hat B(w,\bar\varrho/2)$ before $\hat\partial\hat B(w,\bar\varrho)$, and $\mathsf{R}_\star(\ell_1)$ does not occur, the first term is at most $(\bar\chi/\bar\varrho)^{c_\star}$. Meanwhile, by the last-exit decomposition (see e.g., \cite[Proposition 4.6.4]{Lawler2010RandomWA}), the second term is bounded by 
\begin{equation}
	C\bar\chi^{-1}\sum\nolimits_{y_3\in\hat\partial\hat B(w,10\bar\chi)}
G_{\mathrm{ran}(\ell_1)}(y_2,y_3).
\end{equation}
 By the strong Markov property and (\ref{ineq_2.15}), one has 
 \begin{equation*}
  G_{\mathrm{ran}(\ell_1)}(y_2,y_3)\lesssim\bar\varrho^{-1}\,\mathbb{P}_{y_3}
  \big(\tau_{w,\bar\varrho/2}<\tau_{\mathrm{ran}(\ell_1)}\big).
\end{equation*}
Putting these estimates together, on the event $[\mathsf{R}_\star(\ell_1)]^c$, the loop measure of $\ell_2$ with $\mathrm{ran}(\ell_1)\cap\mathrm{ran}(\ell_2)=\emptyset$ is at most of order
\begin{equation}\label{eq:mass_avoid_2}
  \bar\varrho^{-1-c_\star}\bar\chi^{-1+c_\star}\sum\nolimits_{y_3\in\hat\partial
  \hat B(w,10\bar\chi)}\mathbb{P}_{y_3}\big(\tau_{w,\bar\varrho/2}<
  \tau_{\mathrm{ran}(\ell_1)}\big).
\end{equation}
It was shown in \cite[Remark 6.5]{cai2025separation} that for any
$y_3\in\hat\partial\hat B(w,10\bar\chi)$, the integral of
$\mathbb{P}_{y_3}\big(\tau_{w,\bar\varrho/2}<\tau_{\mathrm{ran}(\ell_1)}\big)$ over the loops $\ell_1$ crossing $B(w,\bar\varrho)\setminus B(w,\bar\chi) $, with respect to the loop measure, is $O((\bar\chi/\bar\varrho)^{2-\frac{1}{2}c_\star})$. To sum up, we have 
\begin{equation}\label{a1030}
	\begin{split}
		  \mathbb{P}\big(\widehat{\mathsf{H}}_w,[\mathsf{R}_\star(\ell_1)]^c\big)\lesssim \bar\varrho^{-1-c_\star}\bar\chi^{-1+c_\star}  \#[\hat\partial
  \hat B(w,10\bar\chi)](\bar\chi/\bar\varrho)^{2-\frac{1}{2}c_\star} \asymp   (\bar\chi/\bar\varrho)^{3+\frac{1}{2}c_\star}. 
	\end{split}
\end{equation}

Combining (\ref{eq:reduce_to_two}), (\ref{e1026}) and (\ref{a1030}), we obtain the claim \eqref{eq:two_loops}, thereby completing the proof. 
 \end{proof}

\subsection{Proof of Item (i) of Theorem \ref{theorem_dimension_gap}}\label{subsec:proof_item_i}

%

We first recall the definition of the box-counting dimension. For any $A\subset\mathbb{R}^d$ and
$r>0$, let $\mathcal{N}_r(A)$ denote the number of boxes $z+[0,r)^d$ with $z\in
r\cdot\mathbb{Z}^d$ that intersect $A$. For a bounded set $A\subset\mathbb{R}^d$,
its upper and lower box-counting dimensions are given by
\begin{equation}\label{eq:box_dim}
  \begin{split}
  \overline{\dim}_{\rm B}(A)
  &= \limsup_{r\to 0}\frac{\ln \mathcal{N}_r(A)}{\ln (1/r)}
   = \limsup_{m\to\infty}\frac{\ln \mathcal{N}_{2^{-m}}(A)}{m\ln 2},\\
  \underline{\dim}_{\rm B}(A)
  &= \liminf_{r\to 0}\frac{\ln \mathcal{N}_r(A)}{\ln (1/r)}
   = \liminf_{m\to\infty}\frac{\ln \mathcal{N}_{2^{-m}}(A)}{m\ln 2}.
  \end{split}
\end{equation}
Here the second equality in each line follows from the fact that for $r\in[2^{-m-1},2^{-m}]$, 
\begin{equation}
	2^{-d}\mathcal{N}_{2^{-m}}(A)\le \mathcal{N}_{r}(A)\le 3^d\mathcal{N}_{2^{-m}}(A). 
\end{equation}
This is because every box of side length $2^{-m}$
intersects at most $3^d$ boxes of size $r$, and every box of side length $r$ intersects at most $2^d$ boxes of size $2^{-m}$.

Since every cluster in $\mathfrak{C}^{\rm B}$ contains a loop and thus has
positive diameter, it suffices to show that for all $\epsilon\in\mathbb{Q}^+$ and
$R\in\mathbb{N}^+$, almost surely
\begin{equation}\label{eq:item_i_goal}
  \overline{\dim}_{\rm B}(\mathcal{C})\le\tfrac{5}{2}-\zeta,\ \ \forall
  \mathcal{C}\in(\mathfrak{C}^{\rm B})^{>\epsilon}_{B(R)}.
\end{equation}
In what follows, we fix $\epsilon\in\mathbb{Q}^+$ and $R\in\mathbb{N}^+$. For
$m\ge 1$, let $\mathcal{Z}_m$ denote the collection of $z\in 2^{-m}\cdot
\mathbb{Z}^3$ such that the box $Q_z:=z+[0,2^{-m})^3$ intersects $B(R)$. We
define
\begin{equation*}
  \mathbf{Y}_m:=\sum\nolimits_{z\in\mathcal{Z}_m}\mathbbm{1}_{\{\partial
  B(z,\sqrt{3}\cdot 2^{-m})\xleftrightarrow{\mathcal{L}^{\rm B}}\partial
  B(z,\frac{1}{2}\epsilon) \}}.
\end{equation*}
We claim that for all sufficiently large $m$ (such that $\sqrt{3}\cdot 2^{-m}<
\frac{1}{4}\epsilon$),
\begin{equation}\label{eq:N_le_Y}
  \mathcal{N}_{2^{-m}}(\mathcal{C})\le\mathbf{Y}_m,\ \ \forall\mathcal{C}\in
  (\mathfrak{C}^{\rm B})^{>\epsilon}_{B(R)}.
\end{equation}
To see this, arbitrarily take $\mathcal{C}\in(\mathfrak{C}^{\rm
B})^{>\epsilon}_{B(R)}$ and $z\in 2^{-m}\cdot\mathbb{Z}^3$ such that $Q_z\cap
\mathcal{C}\neq\emptyset$. Since $\mathcal{C}\subset B(R)$, we have
$z\in\mathcal{Z}_m$. Moreover, $\mathcal{C}$ intersects $Q_z\subset
B(z,\sqrt{3}\cdot 2^{-m})$, and $\mathcal{C}\not\subset B(z,\frac{1}{2}
\epsilon)$ (otherwise, $|\mathcal{C}|_\infty\le\epsilon$). Since $\mathcal{C}$ is
connected, it intersects both $\partial B(z,\sqrt{3}\cdot 2^{-m})$ and $\partial
B(z,\frac{1}{2}\epsilon)$. This implies \eqref{eq:N_le_Y}.

By \eqref{eq:BLS_scaling} and Proposition~\ref{prop:BLS_one_arm}, we have
\begin{equation*}
  \mathbb{E}[\mathbf{Y}_m]=|\mathcal{Z}_m|\cdot\mathcal{Q}\big(2\sqrt{3}\cdot
  2^{-m}\epsilon^{-1}\big)\lesssim R^3 2^{3m}\cdot\big(2^{-m}\epsilon^{-1}
  \big)^{\frac{1}{2}+\cref{const_BLS_onearm}}\lesssim2^{m(\frac{5}{2}-\cref{const_BLS_onearm})}.
\end{equation*}
Thus, by Markov's inequality,
\begin{equation*}
  \mathbb{P}\big(\mathbf{Y}_m\ge 2^{m(\frac{5}{2}-\zeta)}\big)\lesssim
  2^{-m(\cref{const_BLS_onearm}-\zeta)}=2^{-\frac{1}{2}\cref{const_BLS_onearm}m}.
\end{equation*}
Since the right-hand side is summable in $m$, by the Borel--Cantelli lemma, almost
surely $\mathbf{Y}_m<2^{m(\frac{5}{2}-\zeta)}$ for all sufficiently large $m$.
Combined with \eqref{eq:N_le_Y}, this yields
\eqref{eq:item_i_goal}, and thus completes the proof of Item (i) of
Theorem~\ref{theorem_dimension_gap}. \qed

\subsection{Proof of Item (ii) of Theorem \ref{theorem_dimension_gap}}\label{subsec:proof_item_ii}

Throughout this subsection, we assume that $3\le d\le 5$. It
suffices to show that for all $\epsilon\in\mathbb{Q}^+$ and $R\in\mathbb{N}^+$,
almost surely for every $\mathcal{C}\in(\mathfrak{C}^{\rm B}_{\sim_{\rm
g}})^{>\epsilon}_{B(R)}$,
\begin{align}
  \overline{\dim}_{\rm B}(\mathcal{C})&\le\tfrac{d}{2}+1,\label{eq:upper_dim}\\
  \underline{\dim}_{\rm B}(\mathcal{C})&\ge\tfrac{d}{2}+1.\label{eq:lower_dim}
\end{align}
Next, we fix $\epsilon\in\mathbb{Q}^+$ and $R\in\mathbb{N}^+$, and prove \eqref{eq:upper_dim} and \eqref{eq:lower_dim} separately.


\textbf{Proof of \eqref{eq:upper_dim}.} Since $(2^{-n}\cdot\mathfrak{C})^{>\epsilon}_{B(R)}$ converges weakly to
$(\mathfrak{C}^{\rm B}_{\sim_{\rm g}})^{>\epsilon}_{B(R)}$, by
the Skorokhod representation theorem (see e.g., \cite[Theorem
6.7]{Billingsley1999}), we may assume that $\{(2^{-n}\cdot\mathfrak{C})^{>\epsilon}_{B(R)}
\}_{n\ge 1}$ and $(\mathfrak{C}^{\rm B}_{\sim_{\rm g}})^{>\epsilon}_{B(R)}$ are coupled such that
$(2^{-n}\cdot\mathfrak{C})^{>\epsilon}_{B(R)}$ converges to $(\mathfrak{C}^{\rm B}_{\sim_{\rm g}})^{>\epsilon}_{B(R)}$ almost surely. In other words, almost surely for all sufficiently large $n$, there  is a bijection $\sigma_n:(\mathfrak{C}^{\rm B}_{\sim_{\rm g}})^{>\epsilon}_{B(R)}\to (2^{-n}\cdot\mathfrak{C})^{>\epsilon}_{B(R)}$ such that
\begin{equation}\label{eq:skorokhod}
  \lim_{n\to\infty}\max_{\mathcal{C}\in(\mathfrak{C}^{\rm B}_{\sim_{\rm g}})^{>
  \epsilon}_{B(R)}}\mathrm{d}_{\rm H}\big(\mathcal{C},\sigma_n(\mathcal{C})
  \big)=0.
\end{equation}
We fix $a>0$. For $m,n\ge 1$, let $\mathcal{Z}_m$ denote the collection of
$z\in 2^{-m}\cdot\mathbb{Z}^d$ such that $Q_z:=z+[0,2^{-m})^d$ intersects
$B(R)$, and define
\begin{equation*}
  \begin{split}
  \mathbf{Y}_{m,n}:=\sum\nolimits_{z\in\mathcal{Z}_m}\mathbbm{1}_{ \{ 
  \exists \  \text{cluster in}\ 2^{-n}\cdot\mathfrak{C}\ \text{crossing}\  B(z,\frac{1}{2}\epsilon)\setminus B(z,\sqrt{d}\cdot 2^{-m})  \}}.
  \end{split}
\end{equation*}
By the same argument as in \eqref{eq:N_le_Y}, for all sufficiently large $m$, we have 
\begin{equation}\label{eq:N_le_Y_discrete}
  \mathcal{N}_{2^{-m}}(\mathcal{C}')\le\mathbf{Y}_{m,n},\ \ \forall
  \mathcal{C}'\in(2^{-n}\cdot\mathfrak{C})^{>\epsilon}_{B(R)}.
\end{equation}
Moreover, after rescaling by $2^n$, it follows from (\ref{crossing_prob}) that for any $n\ge m$, 
\begin{equation}\label{eq:E_Y_discrete}
  \mathbb{E}[\mathbf{Y}_{m,n}]\lesssim|\mathcal{Z}_m|\cdot\big(2^{-m}\epsilon^{-1}
  \big)^{\frac{d}{2}-1}\lesssim 2^{m(\frac{d}{2}+1)}.
\end{equation}

For $m\ge 1$, we define the event
\begin{equation*}
  \mathsf{U}_m:=\big\{\exists\mathcal{C}\in(\mathfrak{C}^{\rm B}_{\sim_{\rm
  g}})^{>\epsilon}_{B(R)}\ \text{such that}\ \mathcal{N}_{2^{-m}}(\mathcal{C})
  \ge 3^d\cdot 2^{m(\frac{d}{2}+1+a)}\big\}.
\end{equation*}
We use the following elementary fact: for any $A,A'\subset\mathbb{R}^d$ and
$r>0$,
\begin{equation}\label{eq:N_Hausdorff}
  \mathrm{d}_{\rm H}(A,A')\le r\ \Longrightarrow\ \mathcal{N}_r(A')\le
  3^d\mathcal{N}_r(A),
\end{equation}
since every point of $A'$ lies in one of the $3^d$ boxes adjacent or equal to a
box intersecting $A$. By \eqref{eq:skorokhod}, almost surely for all sufficiently large $n$, one has
\begin{equation}
	\mathrm{d}_{\rm H}(\mathcal{C},\sigma_n(\mathcal{C}))\le
2^{-m}, \ \ \forall \mathcal{C}\in(\mathfrak{C}^{\rm B}_{\sim_{\rm
g}})^{>\epsilon}_{B(R)}. 
\end{equation}
This together with \eqref{eq:N_Hausdorff} and
\eqref{eq:N_le_Y_discrete} implies 
\begin{equation}\label{1042}
	\mathcal{N}_{2^{-m}}(\mathcal{C})\le 3^d\mathcal{N}_{2^{-m}}(\sigma_n(\mathcal{C}))\le 3^d\mathbf{Y}_{m,n}. 
\end{equation}
 As a result, $\mathsf{U}_m\subset\cup_{b\ge 1}\cap_{n>b}\big\{\mathbf{Y}_{m,n}\ge 2^{m(\frac{d}{2}+1+a)}\big\}$.
Since $\cap_{n>b}\{\mathbf{Y}_{m,n}\ge 2^{m(\frac{d}{2}+1+a)}\}$ is decreasing in $b$, by Markov's inequality and \eqref{eq:E_Y_discrete}, we have  
\begin{equation*}
  \mathbb{P}(\mathsf{U}_m)\le\liminf_{n\to\infty}\mathbb{P}\big(\mathbf{Y}_{m,n}
  \ge 2^{m(\frac{d}{2}+1+a)}\big)\lesssim 2^{-am}.
\end{equation*}
Since the right-hand side is summable in $m$, by the Borel--Cantelli lemma, almost
surely $\mathsf{U}_m$ occurs for only finitely many $m$. Combined with
\eqref{1042}, this shows that almost surely $\overline{\dim}_{\rm B}
(\mathcal{C})\le\frac{d}{2}+1+a$ for all $\mathcal{C}\in(\mathfrak{C}^{\rm
B}_{\sim_{\rm g}})^{>\epsilon}_{B(R)}$. Taking the intersection over $a\in
\mathbb{Q}^+$, we obtain \eqref{eq:upper_dim}. \qed

\vspace{0.1cm}

 \textbf{Proof of \eqref{eq:lower_dim}.} For $\delta>0$, let $\mathcal{L}^\delta:=\delta\cdot\mathcal{L}$, and recall $\mathcal{L}^\delta_{R,\varsigma}$ from (\ref{df9.1}). We claim that for any $\eta>0$, there exist $c_\triangle=c_\triangle(d,\eta)>0$ and $m_0=m_0(d,\eta)\in\mathbb{N}^+$ such that the following holds. For $m\ge 1$, let $r_0:=2^{-m}$, $\theta:=\frac{\eta}{d+2}$, $r_1:=r_0^\theta$ and $r_2:=r_0^{\theta/2}$. Then for any $R>1$, $m\ge m_0$ and $n\ge m$, letting $\delta:=2^{-n}$, we have
\begin{equation}\label{eq:volume_lower}
  \mathbb{P}\Big(\exists\ell\in\mathcal{L}^\delta_{R,r_2}\ \text{such that}\
  \mathcal{N}_{2^{-m}}(\mathcal{C}^\delta_\ell)\le 2^{m(\frac{d}{2}+1-\eta)}
  \Big)\lesssim R^d e^{-c_\triangle 2^{\frac{\eta m}{2(d+2)}}}.
\end{equation}
We first complete the proof of \eqref{eq:lower_dim} assuming
\eqref{eq:volume_lower}. By the convergence in
Section~\ref{subsection5.3_construction} and the proof of the scaling limit part of Theorem~\ref{new_thm1}, for any $r>0$, there exists a coupling with the following property: almost surely for all
sufficiently large $n$ (with $\delta_n:=2^{-n}$), there exist bijections
$\sigma_n:\mathcal{L}^{\rm B}_{R,r}\to\mathcal{L}^{\delta_n}_{R,r}$ and
$\tau_n:(\mathfrak{C}^{\rm B}_{\sim_{\rm g}})^{>\epsilon}_{B(R)}\to
(\delta_n\cdot\mathfrak{C})^{>\epsilon}_{B(R)}$ such that
\begin{equation}\label{eq:joint_coupling}
  \ell\subset\mathcal{C}\iff\sigma_n(\ell)\subset\tau_n(\mathcal{C}),\ \ \forall
  \ell\in\mathcal{L}^{\rm B}_{R,r},\ \mathcal{C}\in(\mathfrak{C}^{\rm B}_{\sim_{\rm
  g}})^{>\epsilon}_{B(R)},
\end{equation}
and $\lim_{n\to\infty}\max_{\mathcal{C}}\mathrm{d}_{\rm H}(\mathcal{C},
\tau_n(\mathcal{C}))=0$. Thus, there exists a (random)
$r_*>0$ such that every $\mathcal{C}\in(\mathfrak{C}^{\rm B}_{\sim_{\rm
g}})^{>\epsilon}_{B(R)}$ contains a loop in $\mathcal{L}^{\rm B}$ with diameter
greater than $r_*$.

For any $a>0$, we define the event
\begin{equation}\label{eq:A_eps}
  \mathsf{A}^\epsilon_{R,a}:=\big\{\exists\mathcal{C}\in(\mathfrak{C}^{\rm
  B}_{\sim_{\rm g}})^{>\epsilon}_{B(R)}\ \text{such that}\
  \underline{\dim}_{\rm B}(\mathcal{C})\le\tfrac{d}{2}+1-a\big\}.
\end{equation}
It suffices to show that $\mathbb{P}(\mathsf{A}^\epsilon_{R,a})=0$ for all
$a>0$. In what follows, we fix $a>0$, and let $\eta:=\frac{a}{4}$. We write
$r_2(m)$ for the quantity $r_2$ in \eqref{eq:volume_lower} with this choice of
$\eta$, i.e., $r_2(m)=2^{-\frac{\eta m}{2(d+2)}}$. For $m,n\ge 1$, we define the
events
\begin{equation*}
  \begin{split}
  \mathsf{D}_{m,n} :=\big\{\exists\ell\in\mathcal{L}^{\delta_n}_{R,r_2(m)}\
  \text{such that}\ \mathcal{N}_{2^{-m}}(\mathcal{C}^{\delta_n}_\ell)\le 3^d
  \cdot 2^{m(\frac{d}{2}+1-\frac{a}{2})}\big\},
  \end{split}
\end{equation*} 
\begin{equation*}
	\mathsf{F}_m:=\big\{\exists\,\ell\in\mathcal{L}^{\rm B}_{R,r_2(m)}\ \text{and}\
\mathcal{C}\in(\mathfrak{C}^{\rm B}_{\sim_{\rm g}})^{>\epsilon}_{B(R)}\ \text{such
that}\ \ell\subset\mathcal{C}\ \text{and}\ \mathcal{N}_{2^{-m}}(\mathcal{C})\le
2^{m(\frac{d}{2}+1-\frac{a}{2})}\big\}.
\end{equation*}
 On the event $\mathsf{A}^\epsilon_{R,a}$, there exists $\mathcal{C}\in
(\mathfrak{C}^{\rm B}_{\sim_{\rm g}})^{>\epsilon}_{B(R)}$ with
$\underline{\dim}_{\rm B}(\mathcal{C})\le\frac{d}{2}+1-a$, and thus
$\mathcal{N}_{2^{-m}}(\mathcal{C})\le 2^{m(\frac{d}{2}+1-\frac{a}{2})}$ for infinitely many $m$. Moreover, since $r_2(m)\to 0$ as $m\to\infty$, for all sufficiently large $m$ such that $r_2(m)<r_*$, the cluster $\mathcal{C}$ contains a loop in $\mathcal{L}^{\rm B}_{R,r_2(m)}$. Therefore, we have the inclusion
\begin{equation}\label{eq:A_subset_F}
  \mathsf{A}^\epsilon_{R,a}\subset\{\mathsf{F}_m\ \text{occurs for infinitely
  many}\ m\}.
\end{equation}
Next, we fix $m\ge 1$ and compare $\mathsf{F}_m$ with $\mathsf{D}_{m,n}$, using the coupling above with $r=r_2(m)$. Suppose that $\mathcal{C}$ and $\ell\in\mathcal{L}^{\rm B}_{R,r_2(m)}$ certify $\mathsf{F}_m$. Then for all sufficiently large $n$, it follows from \eqref{eq:joint_coupling} that 
\begin{equation}\label{1046}
	\sigma_n(\ell)\in\mathcal{L}^{\delta_n}_{R, r_2(m)}\ \text{and}\ \mathcal{C}^{\delta_n}_{\sigma_n(\ell)}=\tau_n(\mathcal{C}). 
\end{equation}
In addition, by $\mathrm{d}_{\rm H}(\mathcal{C},\tau_n(\mathcal{C}))\le 2^{-m}$ and \eqref{eq:N_Hausdorff}, one has 
 \begin{equation}\label{1047}
 	\mathcal{N}_{2^{-m}} (\tau_n(\mathcal{C}))\le 3^d\mathcal{N}_{2^{-m}}(\mathcal{C}). 
 \end{equation}
Combining (\ref{1046}) and (\ref{1047}), we have $\mathsf{F}_m\subset\cup_{b\ge 1}\cap_{n>b}
  \mathsf{D}_{m,n}$. Since the event $\cap_{n>b}\mathsf{D}_{m,n}$ is decreasing in $b$, we obtain
\begin{equation}\label{eq:P_F_m}
  \mathbb{P}(\mathsf{F}_m)\le\lim\nolimits_{b\to\infty}\mathbb{P} (\cap
  _{n>b}\mathsf{D}_{m,n} )\le\liminf\nolimits_{n\to\infty}
  \mathbb{P}(\mathsf{D}_{m,n}).
\end{equation}
Combined with \eqref{eq:volume_lower}, it yields that $ \mathbb{P}(\mathsf{F}_m)\lesssim R^de^{-c_\triangle 2^{\frac{am}{8(d+2)}}}$.
Since the right-hand side is summable in $m$, the Borel--Cantelli lemma implies
that almost surely $\mathsf{F}_m$ occurs for only finitely many $m$. By
\eqref{eq:A_subset_F}, we conclude that $\mathbb{P}(\mathsf{A}^\epsilon_{R,a})
=0$. In conclusion, we have confirmed \eqref{eq:lower_dim} assuming
\eqref{eq:volume_lower}.

It remains to verify \eqref{eq:volume_lower}. Recall $r_0,\theta,r_1,r_2$ from the claim, and note that $r_0<r_1<r_2$. Let $R_j:=\delta^{-1}r_j$ for each $j\in\{0,1,2\}$. We first present the following extension of Lemma~\ref{lemma_cap_lower}, where
the volume of the cluster is replaced by $\mathcal{N}_{R_0}$. For
$z\in R_1\cdot\mathbb{Z}^d$, let $\mathcal{L}_z$ denote the collection of loops in
$\mathcal{L}$ whose ranges are contained in $B(z,10R_1)$. Note that
$\{\mathcal{L}_z\}_{z\in\mathcal{Z}}$ are independent for any
$\mathcal{Z}\subset R_1\cdot\mathbb{Z}^d$ such that the balls
$\{B(z,10R_1)\}_{z\in\mathcal{Z}}$ are pairwise disjoint. Let $\mathsf{V}_{z,q}$ denote the event that some cluster $\hat{\mathcal{C}}_z$ of $\mathcal{L}_z$ satisfies  
\begin{equation}\label{eq:vol_V}
  \mathrm{cap}(\hat{\mathcal{C}}_z)\ge qR_1^{d-2}\ \ \text{and}\
  \ \mathcal{N}_{r_0}(\delta\cdot\hat{\mathcal{C}}_z)\ge qr_0^{(\frac{d}{2}+1)
  (\theta-1)}.
\end{equation}
We claim that there exist $c_\diamond,c_\dagger>0$ such that for any $m,n$ as
above and $z\in R_1\cdot\mathbb{Z}^d$,
\begin{equation}\label{eq:vol_V_prob}
  \mathbb{P}(\mathsf{V}_{z,c_\diamond})\ge c_\dagger.
\end{equation}
This follows from the argument in the proof of Lemma~\ref{lemma_cap_lower} with
$\nu=0$, $x=z$ and $N=10R_1$, where the lower bound
$\#\mathcal{C}^{\nu,\partial B(x,N)}_y\ge c_8N^{\frac{d}{2}+1}$ on the volume is
replaced by the lower bound
\begin{equation}\label{eq:vol_N_lower}
  \mathcal{N}_{R_0}\big(\mathcal{C}^{0,\partial B(x,N)}_y\big)\ge c_\diamond
  (N/R_0)^{\frac{d}{2}+1}\asymp r_0^{(\frac{d}{2}+1)(\theta-1)}.
\end{equation}
The only new input is the following counterpart of (\ref{good4.17}): for all $1\le R'\le N$,
\begin{equation}\label{eq:vol_4.17}
  \mathbb{P}\big(y\leftrightarrow\partial B(y,c'N),\,\mathcal{N}_{R'}
  (\mathcal{C}_y)<c''(N/R')^{\frac{d}{2}+1}\big)\le\tfrac{1}{2}CN^{-\frac{d}{2}+1}.
\end{equation}

Now consider a loop $\ell\in\mathcal{L}$ with diameter at least $R_2$ whose range is contained in $B(\delta^{-1}R)$. Such a loop must intersect at least $K=cR_2/R_1$ balls of the form
$B(z,R_1)$, where $z\in R_1\cdot\mathbb{Z}^d$ and the balls of the form $B(z,10R_1)$ are pairwise disjoint. Note that $\ell$ is not contained in any of these balls $B(z,10R_1)$, and thus is independent of the corresponding $\mathcal{L}_z$. Every time when $\ell$ crosses such a ball, by \eqref{eq:vol_V_prob} and the lower bound on the capacity in \eqref{eq:vol_V}, it has a uniformly positive probability $c_\ddagger>0$ to intersect a cluster $\hat{\mathcal{C}}_z$ satisfying $\mathcal{N}_{r_0}(\delta\cdot
\hat{\mathcal{C}}_z)\ge c_\diamond r_0^{(\frac{d}{2}+1)(\theta-1)}$. If it does intersect such a cluster, then the cluster $\mathcal{C}_\ell$ containing $\ell$ satisfies 
\begin{equation}\label{eq:vol_C_ell}
  \mathcal{N}_{r_0}(\delta\cdot\mathcal{C}_\ell)\ge\mathcal{N}_{r_0}(\delta\cdot
  \hat{\mathcal{C}}_z)\ge c_\diamond r_0^{(\frac{d}{2}+1)(\theta-1)}=
  c_\diamond r_0^{-(\frac{d}{2}+1)+\frac{\eta}{2}}>r_0^{-(\frac{d}{2}+1)+\eta}.
\end{equation}
Therefore, by the independence of $\{\mathcal{L}_z\}$ for disjoint balls, for
such a loop $\ell$, the conditional probability of $\{\mathcal{N}_{r_0}(\delta
\cdot\mathcal{C}_\ell)\le r_0^{-(\frac{d}{2}+1)+\eta}\}$ is at most
$(1-c_\ddagger)^K\le e^{-cr_0^{-\theta/2}}$. Since the expected number of loops
$\ell\in\mathcal{L}$ with $|\ell|\ge R_2$ and $\mathrm{ran}(\ell)\subset
B(\delta^{-1}R)$ is $O(R^dr_2^{-d})$ (by (\ref{bound2.12})), it follows from the Mecke
formula for Poisson point processes (see e.g.,\ \cite{last2017lectures}) that the left-hand side of \eqref{eq:volume_lower} is at most of order
\begin{equation*}
  R^dr_2^{-d}e^{-cr_0^{-\theta/2}}\lesssim R^de^{-c'r_0^{-\theta/2}}=
  R^de^{-c'2^{\frac{\eta m}{2(d+2)}}}.
\end{equation*}
This completes the proof of \eqref{eq:volume_lower}, and hence that of Item (ii) of Theorem~\ref{theorem_dimension_gap}. \qed

  \section*{Acknowledgments}

  We warmly thank Michael Aizenman, Itai Benjamini, Nicholas Crawford, Gady Kozma, Ron Peled, Wendelin Werner and Ofer Zeitouni for fruitful discussions. J. Ding is supported by the National Natural Science Foundation of China (Grant No. 12231002, 12595284, 12595280), and by the New Cornerstone Science Foundation through the New Cornerstone Investigator Program and XPLORER PRIZE.
  
  \vspace{0.2cm}
  
 \textbf{Statement of AI use.} The main idea and proof strategy of this paper were developed by the authors and had been fully formulated by June 2026, independently of any AI tools. The detailed proofs were written by the authors, with AI tools used to improve their wording and presentation. In addition, AI tools contributed to a few technical points: (1) they pointed us to \cite{chang2014markov} for the change of the law of the GFF under the Doob $h$-transform (see Section \ref{subsection_discuss_conformal}); (2) they suggested constructing the coupling in Proposition \ref{prop_compare_relation} via the extension of Strassen's theorem in \cite{leskela2010stochastic} (see Lemma \ref{lemma_Strassen}), which simplified the derivation; (3) they pointed us to the gluing lemma in \cite{ambrosio2005gradient} (see Lemma \ref{lemma_gluing}), which led to a more direct argument in the construction of the scaling limit in Section \ref{subsection5.3_construction}. AI tools also assisted in searching the literature. The authors take full responsibility for the content of this paper.




\end{document}